\documentclass[10pt]{amsart}

\usepackage[centertags]{amsmath}
\usepackage{amsfonts}
\usepackage{amssymb}
\usepackage{amsthm}
\usepackage[english]{babel}
\usepackage[active]{srcltx}
\usepackage{enumerate}
\usepackage{mathrsfs}
\usepackage{esint}
\usepackage[utf8]{inputenc}

\usepackage[colorlinks=true, linkcolor=blue, citecolor=red, urlcolor=red, backref=page]{hyperref}
\usepackage{tikz}

\numberwithin{equation}{section}

\newtheorem{thm}{Theorem}[section]
\newtheorem{cor}[thm]{Corollary}
\newtheorem{lem}[thm]{Lemma}
\newtheorem{prop}[thm]{Proposition}

\newtheorem{ex}[thm]{Example}

\newtheorem{defn}[thm]{Definition}
\newtheorem{notation}[thm]{Notation}

\theoremstyle{remark}
\newtheorem{rem}[thm]{Remark}

\def\R{\mathbb{R}}
\def\Z{\mathbb{Z}}
\def\N{\mathbb{N}}

\def\Td{\mathbb{T}}
\def\BB{\mathbf{B}}
\def\L{\mathcal{L}}

\def\T{\mathfrak{T}}
\def\L{\text{Lip}}

\begin{document}

\title[]{
Truncated smooth function spaces}

\author{Oscar Dom\'inguez}
\address{O. Dom\'inguez, Departamento de An\'alisis Matem\'atico y Matem\'atica Aplicada, Facultad de Matem\'aticas, Universidad Complutense de Madrid, Plaza de Ciencias 3, 28040 Madrid, Spain}
\email{oscar.dominguez@ucm.es}

\author{Sergey Tikhonov}

\thanks{
 This research was partially supported by French National Research Agency (ANR-10-LABX-0070), (ANR-11-IDEX-0007) (O.D.) and by
PID2020-114948GB-I00,  2017 SGR 358, AP 14870758, 
 the CERCA Programme of the Generalitat de Catalunya, and
 Severo Ochoa and Mar\'{i}a de Maeztu Program for Centers and Units of Excellence in R$\&$D (CEX2020-001084-M) (S.T.).
}

 \address{S. Tikhonov, Centre de Recerca Matem\`{a}tica\\
Campus de Bellaterra, Edifici C
08193 Bellaterra (Barcelona), Spain;
ICREA, Pg. Llu\'{i}v Companys 23, 08010 Barcelona, Spain,
 and Universitat Aut\`{o}noma de Barcelona.}
\email{stikhonov@crm.cat}



\begin{abstract}
We introduce truncated Besov and Triebel--Lizorkin function spaces and investigate their main properties: embeddings, interpolation, duality, lifting,  traces.
These new scales allow us to improve several known results in functional analysis and PDE's.
\end{abstract}

\maketitle

\setcounter{tocdepth}{1}
\tableofcontents


\newpage
\section{Introduction}

Smooth function spaces are basic notion in  analysis. The main examples 
 are
Besov and Triebel--Lizorkin spaces, since those scales contain various distinguished spaces of distributions such as standard Sobolev, Gagliardo--Slobodecki\u{\i}, fractional Sobolev, Bessel potential spaces, H\"{o}lder--Zygmund spaces, spaces of BMO-type, Hardy spaces, and Lebesgue spaces.  Another important scale is Lipschitz spaces, in particular containing the space of bounded variation.
All mentioned scales 
 have shown to be crucial in functional and harmonic analysis and PDE's, among other fields.
Key properties of these spaces, such as duality, interpolation, embeddings, characterizations, traces, and lifting,  have been extensively developed since the middle of the last century. See e.g. the monographs \cite{BennettSharpley, BerghLofstrom, Besov, Nikolskii, Peetre76, SchmeisserTriebel, Stein, Triebel, Triebel83, Triebel92, Triebel01}.


{We also stress that in recent times a big deal of attention has been devoted to function spaces with generalized smoothness. Indeed, in many questions arising in different areas it has become apparent that the scope of classical smoothness is rather limited and thus the finer tuning given by generalized smoothness becomes very useful. Here we mention pointwise multipliers \cite{Li}, elliptic PDE's and geometric measure theory \cite{Molero}, probability theory and stochastic processes \cite{FarkasLeopold, Kempka}, sharp embedding theorems \cite{Haroske}, fractal analysis and a related spectral theory \cite{Triebel97}, capacity theory \cite{Liu}, regularity properties of Euler solutions \cite{Ogawa, Chae}, regularity of flows for the transport and continuity equation \cite{Crippa, Leger, Brue}, or optimal recovery \cite{Hinrichs}.} The prototype of function spaces with generalized smoothness is   the logarithmic Besov spaces  $B^{s, b}_{p,q}(\R^d)$ and Triebel--Lizorkin spaces $F^{s, b}_{p,q}(\R^d)$, which  can be defined through Fourier-analytical decompositions as follows:
\begin{equation}\label{IntroBesov}
	\|f\|_{B^{s, b}_{p,q}(\R^d)} = \bigg(\sum_{j=0}^\infty 2^{j s q} (1 + j)^{b q}  \|(\varphi_j \widehat{f})^\vee\|_{L_p(\mathbb{R}^d)}^q \bigg)^{\frac{1}{q}}
\end{equation}
and
\begin{equation}\label{IntroTriebelLizorkin}
	\|f\|_{F^{s, b}_{p,q}(\R^d)} = \bigg\| \bigg(\sum_{j=0}^\infty 2^{j s q} (1 + j)^{b q} |(\varphi_j \widehat{f})^\vee|^q \bigg)^{1/q} \bigg\|_{L_p(\R^d)}.
\end{equation}
In the special case $b=0$ one recovers the classical spaces $B^s_{p, q}(\R^d)$ and $F^s_{p, q}(\R^d)$, respectively.
Interestingly, the most unexpected effects can be discovered dealing with zero smoothness, that is, $s=0.$ In particular, the space $B^{0, b}_{p,q}(\R^d)$ does not coincide with its counterpart  ${\bf{B}}^{0, b}_{p,q}(\R^d)$ defined by differences. In fact, the lack of adequate Littlewood--Paley decompositions makes the analysis of $\mathbf{B}^{0, b}_{p, q}(\R^d)$ rather tricky; we refer the  reader to \cite{DominguezTikhonov} for a detailed study of $\mathbf{B}^{0, b}_{p, q}(\R^d)$, including some discussions on applications of these spaces to different areas of analysis (see \cite[Section 1.2]{DominguezTikhonov}).

The main goal of this work is to introduce the {\it truncated scale} of smooth function spaces. Let us illustrate this idea by presenting truncated Besov and Triebel--Lizorkin spaces:
  \begin{equation}\label{Tr1In}
  \|f\|_{T^b_r B^{s}_{p,q}(\R^d)} = \left(\sum_{j=0}^\infty 2^{j b r} \bigg(\sum_{\nu=2^j-1}^{2^{j+1}-2} 2^{\nu s q} \|(\varphi_\nu \widehat{f})^\vee\|_{L_p(\mathbb{R}^d)}^q\bigg)^{r/q}\right)^{1/r},
   \end{equation}
    \begin{equation}\label{Tr2In}
    	\|f\|_{T^b_r F^{s}_{p, q}(\R^d)} = \left(\sum_{j=0}^\infty 2^{j b r}  \bigg\| \bigg(\sum_{\nu=2^j-1}^{2^{j+1}-2} 2^{\nu s q} |(\varphi_\nu \widehat{f})^\vee|^q \bigg)^{1/q} \bigg\|_{L_p(\R^d)}^r \right)^{1/r}.
\end{equation}
Roughly speaking, the additional summability parameter $r$ enables to refine classical norms via truncations with respect to dyadic decomposition.

There are several aspects of the matter worth pointing out, which show that truncated spaces are extremely important in the theory of function spaces, PDE's, interpolation theory, potential theory and applications. Next we highlight some of them.

1. {\it A unified approach to classical spaces and applications to PDE's.}
As one may expect, choosing appropriately the smoothness  and  summability parameters $b$ and  $r$, respectively,  in $T^b_r X$ we recover the prototypical space  $X$. For example,
$T^b_r B^{s}_{p,q}(\R^d)=B^{s,b}_{p,q}(\R^d)$ for $q=r$ and, in particular, $T^0_q B^{s}_{p,q}(\R^d)=B^{s}_{p,q}(\R^d)$.
More than that, specifying parameters in the definitions of truncated Besov and  Triebel-Lizorkin spaces, one can recover in a unifying fashion
all classical scales of smooth function spaces: Besov $B^{s, b}_{p,q}(\R^d)$, Lipschitz $\text{Lip}^{s, b}_{p, q}(\R^d)$, Triebel--Lizorkin $F^{s, b}_{p, q}(\R^d)$, Besov spaces with zero smoothness ${{B}}^{0, b}_{p, q}(\R^d)$ and ${\textbf{B}}^{0, b}_{p, q}(\R^d)$. In particular, we obtain
$$
		T^{b+1/q}_q F^{0}_{p, 2}(\R^d)  = {\textbf{B}}^{0, b}_{p, q} (\R^d), \qquad b > -1/q,
	$$
and	$$
		T^{b+1/q}_q F^{s}_{p, 2}(\R^d) = \text{Lip}^{s, b}_{p, q} (\R^d), \qquad b < -1/q;
	$$
	cf. Proposition \ref{PropositionCoincidences}.

	Another motivation to introduce truncated smooth function spaces comes from applications to PDE's. In his fundamental paper \cite{Vishik}, Vishik was able to extend the classical Yudovich theory  for uniqueness of non-Lipschitz $2$D Euler flows via Besov-type $B_\Gamma$ spaces defined in terms of growth properties of partial sums involving $L_\infty$-norms of dyadic frequencies:
	$$
		\|f\|_{B_\Gamma(\R^d)} = \sup_{j \geq 0} \, \frac{1}{\Gamma(j)} \, \sum_{\nu=0}^j \|(\varphi_\nu \widehat{f})^\vee\|_{L_\infty(\R^d)},
	$$
	where $\Gamma$ is a positive and increasing function such that $\sum_{j=0}^\infty \frac{1}{\Gamma(j)} = \infty$. Thus the main result of \cite[Theorem 7.1]{Vishik} guarantees uniqueness of the Euler flow for related vorticities that are bounded in $B_\Gamma(\R^d)$, with $\Gamma(j) = j$. In particular, for this special choice of $\Gamma$, the  embeddings
	$$
		bmo(\R^d) \hookrightarrow B^0_{\infty, \infty}(\R^d) \hookrightarrow B_\Gamma(\R^d),
	$$
	leads to uniqueness of Euler flows for $bmo(\R^d)$ vorticities, extending thus the classical $L_\infty$-based theory due to Yudovich \cite{Yudovich63, Yudovich95}.
	
	Observe that  $B_\Gamma(\R^d), \, \Gamma(j) = j^\alpha, \, \alpha \leq 1,$ is a special case of truncated Besov spaces, namely,
	$$
		B_\Gamma(\R^d) = T^{-\alpha}_\infty B^0_{\infty, 1}(\R^d),
	$$
	see Proposition \ref{PropEquiQN}. Accordingly, truncated constructions \eqref{Tr1In} and \eqref{Tr2In} arise naturally in the analysis of PDE's.

2. {\it Embeddings.} Truncated smooth function spaces allow us to sharpen classical embeddings. To clarify this, we mention the well-known embedding between Besov and  Triebel--Lizorkin (in particular, Sobolev) spaces
$$
	B^s_{p, \min\{p, q\}}(\R^d) \hookrightarrow F^s_{p, q}(\R^d) \hookrightarrow B^s_{p, \max\{p, q\}} (\R^d),
$$
which is not sharp anymore in the scale of truncated spaces. Namely, we have (cf. Corollary \ref{Corollary3.6})
$$		B^s_{p, \min\{p, q\}} (\R^d) \hookrightarrow	 T_{\min\{p, q\}} F^s_{p, q}(\R^d)
		 \hookrightarrow F^s_{p, q}(\R^d)  \hookrightarrow T_{\max\{p, q\}} F^s_{p, q}(\R^d)  \hookrightarrow B^s_{p, \max\{p, q\}}(\R^d)
$$
and any of these embeddings are sharp. Here, $T_r F^s_{p, q}(\R^d) = T^0_r F^s_{p, q}(\R^d)$.

{As already mentioned above, the scales of $\mathbf{B}^{0, b}_{p, q}(\R^d)$ and $\text{Lip}^{s, b}_{p, q}(\R^d)$ spaces are not contained in the scales of Besov and Triebel--Lizorkin spaces (cf. \eqref{IntroBesov}-\eqref{IntroTriebelLizorkin}).
The known embeddings between them involve log-shifts of smoothness:
\begin{equation}\label{B1new}
	 B^{0, b + 1/\min\{2, p, q\}}_{p, q}(\R^d)  \hookrightarrow \textbf{B}^{0, b}_{p, q}(\R^d) \hookrightarrow B^{0, b + 1/\max\{2, p, q\}}_{p, q}(\R^d)
\end{equation}
and
\begin{equation}\label{Lip1new}
	 B^{s, b + 1/\min\{2, p, q\}}_{p, q}(\R^d)  \hookrightarrow \text{Lip}^{s, b}_{p, q}(\R^d) \hookrightarrow B^{s, b + 1/\max\{2, p, q\}}_{p, q}(\R^d).
\end{equation}
Furthermore, these embeddings are optimal within the scale of classical Besov spaces; cf. \cite{DominguezHaroskeTikhonov} and \cite{DominguezTikhonov}.  As a byproduct of general embedding theorems between truncated spaces, we are now able to show that both \eqref{B1new} and \eqref{Lip1new} can be considerably sharpened by dealing with the refined scale of truncated spaces. Namely, we establish (cf. Corollaries \ref{Cor1} and \ref{Cor2})
	$$
		 T^{b+1/q}_q B^{0}_{p, \min\{2, p\}}(\R^d) \hookrightarrow \mathbf{B}^{0, b}_{p, q}(\R^d) \hookrightarrow T^{b+1/q}_q B^{0}_{p, \max\{2, p\}}(\R^d),
	$$
	and
	 $$
	 	T^{b+1/q}_q B^{s}_{p, \min\{2, p\}}(\R^d) \hookrightarrow \text{Lip}^{s, b}_{p, q}(\R^d) \hookrightarrow T^{b+1/q}_q B^{s}_{p, \max\{2, p\}}(\R^d).
	 $$
	 Furthermore, they are optimal.}

3. {\it Interpolation.}
One of the main properties of Besov spaces is that they are closed under real interpolation. A more complete result states that, for $A, \widetilde{A} \in \{B, F\}$,
\begin{equation}\label{Besov}
	(A^{s_0}_{p, q_0}(\R^d) , \widetilde{A}^{s_1}_{p, q_1}(\R^d) )_{\theta, r}= B^{s}_{p, r}(\R^d),
\end{equation}
where $s= (1-\theta) s_0 + \theta s_1$ and $s_0 \neq s_1$.
It is a natural question to understand whenever this property holds for truncated spaces.
 We positively answer this question (cf. Theorem \ref{ThmIntNew}):
$$		(T^{b_0}_{r_0} A^{s_0}_{p, q_0}(\R^d) , T^{b_1}_{r_1}\widetilde{A}^{s_1}_{p, q_1}(\R^d) )_{\theta, r} = B^{s, b}_{p, r}(\R^d) , \qquad s_0\ne s_1.
$$
Moreover, in the borderline case $s_0 = s_1=s$, we obtain that (cf. Theorems \ref{TheoremInterpolation2} and  \ref{TheoremInterpolation2pspapsa})
$$		(T^{b_0}_{r_0} A^{s}_{p,q}(\R^d) , T^{b_1}_{r_1}A^{s}_{p,q}(\R^d) )_{\theta,r} = T^b_r A^{s}_{p,q}(\R^d), \qquad b_0\ne b_1.
$$

Another intrinsic question can be stated as follows. Can one obtain  the truncated spaces as a result of interpolation between the classical ones?
An answer is again positive and rather surprising since modern interpolation methods are naturally appearing:   \emph{limiting interpolation}. In particular, we have
for $s_0 > s$ and $b > 0$
 \begin{equation}\label{Intro1}
  	T^b_r B^{s}_{p,q}(\R^d)  = (B^s_{p,q}(\R^d), A^{s_0}_{p,q_1}(\R^d))_{(0,b-1/r),r},
	\end{equation}
and for $s_0 < s$ and $b < 0$
\begin{equation}\label{Intro2}
	 T^b_r B^{s}_{p,q}(\R^d)  = (A^{s_0}_{p,q_0}(\R^d), B^s_{p,q}(\R^d) )_{(1,b-1/r),r};
\end{equation}
see Theorem \ref{Thm4.2}.  A similar result holds for truncated Triebel--Lizorkin spaces (cf. Theorem \ref{TheoremInterpolationF}). As special cases, we can characterize zero-smoothness Besov  and Lipschitz spaces via limiting interpolation:
$$			\mathbf{B}^{0, b}_{p, q}(\R^d)  = (L_p(\R^d), A^{s_0}_{p, q_1}(\R^d))_{(0, b), q},\qquad b > -1/q,
	$$
$$			\text{Lip}^{s, b}_{p, q}(\R^d)  = (A^{s_0}_{p, q_0}(\R^d), H^s_p(\R^d))_{(1, b), q},\qquad b < -1/q.
	$$
{The formulae \eqref{Intro1}-\eqref{Intro2} will play a key role in our arguments, since they put truncated function spaces into the powerful framework provided by limiting interpolation. In particular, this enables to transfer many properties of classical spaces to truncated spaces.}

{4. {\it Lifting properties. 
} Consider the lifting operator
$$
	I_\sigma = (\text{id}-\Delta)^{\frac{\sigma}{2}}, \qquad \sigma \in \R.
$$
In particular, if $\sigma < 0$ then $I_\sigma$ is the classical \emph{Bessel potential}. A crucial property of Besov and Triebel--Lizorkin spaces is their stability under liftings. More precisely, for $A \in \{B, F\}$, $I_\sigma$ acts as an isomorphism from $A^{s, b}_{p, q}(\R^d)$ onto $A^{s-\sigma, b}_{p, q}(\R^d)$ and
\begin{equation*}
	\|I_\sigma f\|_{A^{s-\sigma, b}_{p, q}(\R^d)} \asymp \|f\|_{A^{s, b}_{p, q}(\R^d)}.
\end{equation*}
On the other hand, the mapping properties of liftings on Besov spaces of zero smoothness and Lipschitz spaces are much more delicate and not completely understood. For instance, a partial answer asserts that
\begin{equation*}
I_\sigma : \mathbf{B}^{0, b}_{p, q}(\R^d) \to B^{-\sigma, b+ 1/\max\{2, p, q\}}_{p, q}(\R^d)
\end{equation*}
and
\begin{equation*}
	I_\sigma : B^{\sigma, b+ 1/\min\{2, p, q\}}_{p, q}(\R^d) \to \mathbf{B}^{0, b}_{p, q}(\R^d).
\end{equation*}
Furthermore, the shifts given by $\frac{1}{\max\{2, p, q\}}$ and $\frac{1}{\min\{2, p, q\}}$ are optimal. In particular, this result tells us that classical Besov spaces are not sharp enough to characterize lifting properties of  $\mathbf{B}^{0, b}_{p, q}(\R^d)$. This drawback can be now overcome with the help of truncated spaces.}

In Section \ref{Section8.2}, we show that  $I_\sigma$ acts as an isomorphism from $\mathbf{B}^{0, b}_{p, q}(\R^d)$ onto $T^{b+1/q}_ q F^{-\sigma}_{p, 2}(\R^d)$ and
	$$
		\|I_\sigma f\|_{T^{b+1/q}_q F^{-\sigma}_{p, 2}(\R^d)} \asymp \|f\|_{\mathbf{B}^{0, b}_{p, q}(\R^d)}.
	$$
Recall that $\mathbf{B}^{0, b}_{p, q}(\R^d) = T^{b+1/q}_ q F^{0}_{p, 2}(\R^d)$. Accordingly, lifting properties of $\mathbf{B}^{0, b}_{p, q}(\R^d)$ should be understood within the more general setting provided by truncated spaces. A similar statement also holds true for $\text{Lip}^{s, b}_{p, q}(\R^d)$.

5. {\it Duality.}
In Section \ref{SectionDuality}, we obtain that the scale of truncated spaces is closed under duality,  extending then  the well-known duality assertions related to classical Besov and Triebel--Lizorkin spaces. Specifically (cf. Theorem \ref{ThmDualNewBesovSpaces} for precise statement) we have
	\begin{equation*}
		(T^b_r A^s_{p, q}(\mathbb{R}^d))' = T^{-b}_{r'} A^{-s}_{p',q'}(\R^d), \qquad A \in \{B, F\}.
	\end{equation*}
	In particular, this yields a complete answer to the open question of characterizing the dual spaces of $\mathbf{B}^{0, b}_{p, q}(\R^d)$ and $\text{Lip}^{s, b}_{p, q}(\R^d)$:
$$
	(\mathbf{B}^{0, b}_{p, q}(\R^d))' = T^{-b-1/q}_{q'} F^{0}_{p', 2}(\R^d),   \qquad
		(T^{-b-1/q}_{q'} F^{0}_{p', 2}(\R^d))' = \mathbf{B}^{0, b}_{p, q}(\R^d)
$$
and
$$
   (\text{Lip}^{s, b}_{p, q}(\R^d))' = T^{-b-1/q}_{q'} F^{-s}_{p', 2}(\R^d), \qquad
		 (T^{-b-1/q}_{q'}F^{-s}_{p', 2}(\R^d))' = \text{Lip}^{s, b}_{p, q}(\R^d).
$$

6.  {\it Applications.} We would like to mention here at least three new applications of truncated spaces in different areas.

(i) The trace operator can be defined  (at least formally) by
\begin{equation*}
    (\text{Tr } f)(x) = (\text{Tr }_{\{x_d=0\}} f)(x) = f(x',0), \qquad x=(x',x_d), \quad x' \in \mathbb{R}^{d-1}, \qquad d \geq 2.
\end{equation*}
The well-known  trace theorem for Besov spaces states that
$$
    \text{Tr }: B^s_{p,q}(\mathbb{R}^d) \longrightarrow
    B^{s-1/p}_{p,q}(\mathbb{R}^{d-1})
$$
for $p \geq 1, \, s > 1/p$, and $0 < q \leq \infty$.
Furthermore,  $\text{Tr}$ admits a right inverse $\text{Ex}$, which is a
bounded linear operator from $B^{s-1/p}_{p,q}(\mathbb{R}^{d-1})$ into
$B^s_{p,q}(\mathbb{R}^d)$. See e.g. \cite[Section 2.7.2]{Triebel83}.

In the limiting case $p=1$ and/or $s=1/p$, we  mention the classical results by Gagliardo \cite{Gagliardo}:
\begin{equation}\label{TraceGagliardo--}
	\text{Tr}: \text{BV}(\R^{d+1}_+) \to L_1(\R^d)
\end{equation}
and 
 by   Peetre \cite{Peetre}, Agmon--H\"ormander \cite{AgmonHormander}, and Frazier--Jawerth \cite{FrazierJawerth}:
\begin{equation}\label{TraceGagliardo---}
    \text{Tr }: B^{1/p}_{p,1}(\mathbb{R}^d) \longrightarrow L_p(\mathbb{R}^{d-1})
\end{equation}
and there exist bounded extension operators. Importantly, there does not exist a bounded linear extension related to  \eqref{TraceGagliardo--} or \eqref{TraceGagliardo---}, which yields serious obstructions;  cf. \cite{Peetre79} (see also \cite{Triebel83} for further historical comments).

Motivated by the above discussion, very recently Mal\'y, Shanmugalingam and Snipes \cite{MSS} considered the Besov space 
with zero smoothness\footnote{The corresponding norm appears frequently in the literature under the name of \emph{Dini-type condition}.} $\mathbf{B}^{0}_{1,1} = \mathbf{B}^{0, 0}_{1, 1}$, which is a slightly smaller space than $L_1$. It turns out that in this case 
   a bounded linear extension operator does  exist:
\begin{equation}\label{MSSIntro}
\text{Ex} : \mathbf{B}^{0}_{1,1} (\partial \Omega) \to \text{BV}(\Omega)
\end{equation}
such that $\text{Tr} \circ \text{Ex} = \text{id}$. Here $\Omega$ is a smooth bounded domain in $\R^d$.

Working with truncated spaces allows us to give a full solution to the trace/extension problem in the limiting case $s=1/p$ (cf. \eqref{TraceGagliardo--} and \eqref{TraceGagliardo---}). In particular, we can substantially  improve the above mentioned  result by Mal\'y, Shanmugalingam and Snipes as follows (cf. Theorem \ref{Theorem 2.1})
$$
    \text{Tr }:
    T^{b+1/q}_q B^{1/p}_{p, 1}(\mathbb{R}^d) \longrightarrow \mathbf{B}^{0,b}_{p,q}(\mathbb{R}^{d-1})
$$
and there exists a linear extension operator $\text{Ex}$ which is
        continuous from $\mathbf{B}^{0,b}_{p,q}(\mathbb{R}^{d-1})$
        into $T^{b+1/q}_q B^{1/p}_{p,1}(\mathbb{R}^d)$ with
           $ \text{Tr } \circ \text{Ex } =
            \text{id}.$ Specializing this result with $b=0$ and $p=q=1$, we obtain
   $$
    \text{Tr }: B^{1, 1}_{1, 1}(\R^d) \longrightarrow \mathbf{B}^0_{1, 1}(\R^{d-1})
   $$ (cf. \eqref{TraceGagliardo--}-(\ref{TraceGagliardo---})) and there exists a bounded linear extension operator
$$
	\text{Ex}: \mathbf{B}^{0}_{1,1}(\mathbb{R}^{d-1}) \to B^{1,1}_{1, 1}(\mathbb{R}^d)
$$
with $\text{Tr } \circ \text{Ex } =
            \text{id}.$ Note that $B^{1, 1}_{1, 1}(\R^d) \subsetneq \text{BV}(\R^d)$ (see \eqref{MSSIntro}).

(ii) An important embedding for  elliptic and parabolic equations is the prominent result by Trudinger \cite{Trudinger}:
$$
	W^{d/p}_p(\mathbb{T}^d) \hookrightarrow L_\infty (\log L)_{-1/p'} (\mathbb{T}^d), \qquad p \in (1, \infty).
$$ 
Later on this result was sligtly improved with respect to both domain and target spaces: 
$$
	B^{d/p}_{p,q}(\mathbb{T}^d) \hookrightarrow L_{\infty, q}(\log L)_{-1}(\mathbb{T}^d).
$$
In its turn, this  embedding  can be  naturally considered in the more general setting of Besov spaces of logarithmic smoothness. Namely, if 
 $b < -1/q$ then
\begin{equation}\label{CriticalEmbeddingLog--}
	B^{d/p, b + 1/\min\{1, q\}}_{p,q}(\mathbb{T}^d) \hookrightarrow L_{\infty, q}(\log L)_{b}(\mathbb{T}^d).
\end{equation}
The last embedding  has been the object of intensive research since 1979 (see the paper \cite{DeVore}) until now (see  \cite{CaetanoMoura, CaetanoFarkas, CaetanoLeopold, Martin, MouraNevesPiotrowski}). It is known that
  \eqref{CriticalEmbeddingLog--} is optimal within the scale of Besov spaces of logarithmic smoothness.
However, we will show that one can  sharpen \eqref{CriticalEmbeddingLog--}  with the help of the truncated Besov spaces $T^b_r B^{s}_{p, q}(\mathbb{T}^d)$.
Namely, we show (cf. Theorem \ref{TheoremEmbedinngCritical})
$$
		T^{b+1/q}_q B^{d/p}_{p, 1}(\mathbb{T}^d) \hookrightarrow L_{\infty, q}(\log L)_b(\mathbb{T}^d)
	$$ and this embedding is optimal (cf. Remark \ref{Remark6.5}).
	
{The counterpart of \eqref{CriticalEmbeddingLog--} with $b > -1/q$ is formulated in terms of the zero-smoothness Besov space $\BB_{\infty, q}^{0,b}(\R^d)$. More precisely, if $b > -1/q$ then
	\begin{equation}\label{ContinuityEnvelope1Intro}
	B^{d/p, b + 1/\min\{1,q\}}_{p, q}(\R^d) \hookrightarrow \BB_{\infty, q}^{0,b}(\R^d).
\end{equation}
This embedding is a crux of matter in function spaces. It served as a basis for the theory of continuity envelopes developed in detail by Haroske--Triebel \cite{Triebel01, Haroske} and it has been further investigated in \cite{MouraNevesPiotrowski09, MouraNevesSchneider11, MouraNevesSchneider14} and the references within. Similarly as \eqref{CriticalEmbeddingLog--}, the embedding \eqref{ContinuityEnvelope1Intro} is optimal within the scale of classical Besov spaces,
but we can now prove that it admits the following improvements within the scale of truncated spaces (cf. Theorem \ref{TheoremEmbedinngCriticalContinuity})
$$
T^{b+1/q}_q B^{d/p}_{p, 1}(\R^d) \hookrightarrow \BB^{0,b}_{\infty, q}(\R^d)
$$
and this is optimal.}

(iii) The celebrated Bourgain--Brezis--Mironescu formula \cite{Bourgain} states that
$$\lim_{s \to 1-} (1-s)^{1/p} \|f\|_{\dot{W}^{s, p}(\R^d)} = c_{d, p} \, \|\nabla f\|_{L_p(\R^d)}.
$$
In some sense it fixes the well-known defect of the fractional Sobolev (Gagliardo--Slobodecki\u{\i}) seminorms 
 \begin{equation*}
		\|f\|_{\dot{W}^{s, p}(\R^d)} = \ \bigg(\int_{\R^d} \int_{\R^d} \frac{|f(x)-f(y)|^p}{|x-y|^{d + s p}} \, dx \, dy \bigg)^{1/p}, \quad s \in (0,1), 
\end{equation*}
that they do not converge to the standard Sobolev seminorm $\|\nabla f\|_{L_p(\R^d)}$ as $s \to 1-$.
 The explicit value of the constant $c_{d, p}$ is known, depending only on $d$ and $p$.
{This result has been widely considered recently
 with respect to
    different perspectives: interpolation theory \cite{Milman}, higher-order Besov spaces \cite{KaradzhovMilmanXiao}, Triebel--Lizorkin spaces \cite{Brazke}, anisotropic Sobolev spaces \cite{Ludwig}, etc.}

We will show that the family of truncated norms on Besov and Triebel--Lizorkin  spaces  satisfies the Bourgain--Brezis--Mironescu phenomenon, in the sense that under a certain normalization, one can attain the classical  norms $\|\cdot\|_{A^s_{p,q}(\R^d)}, \, A \in \{B, F\}$, via limits.
Specifically, one has (cf. Theorem \ref{ThmLimitsBBM})
$$		\lim_{b \to 0-}  (1-2^{b r})^{1/r} \, \|f\|^*_{T^b_r A^{s}_{p, q }(\R^d)} = \|f\|_{A^s_{p, q}(\R^d)}.
$$

{To close this introduction, we would like to note that truncated Besov and Triebel--Lizorkin spaces introduced in this paper can be understood within the abstract setting provided by interpolation theory.
In the forthcoming work \cite{DomTik22}, we propose a novel interpolation method,  
 \emph{the truncated interpolation}, which contains as distinguished examples truncated Besov and Triebel--Lizorkin spaces. It is worth mentioning 
 that this methodology goes further beyond the setting of smooth function spaces and it enables us to define many other truncated scales $T^b_r X$ with the special choice of a space $X$.}

\vskip 0.4cm

\emph{Notation.}
Given two non-negative  quantities $A$ and $B$, we use $A \lesssim B$ for the estimate $A \leq C B$\index{\bigskip\textbf{Numbers and relations}!$A \lesssim B$}\label{AB}, where $C$ is a positive constant which is independent of all essential parameters in $A$ and $B$. If $A \lesssim B \lesssim A$, then we write $A \asymp B$\index{\bigskip\textbf{Numbers and relations}!$A \asymp B$}\label{ASYMP}.

Let $X$ and $Y$ be quasi-Banach spaces. We write $X \hookrightarrow Y$ if the identity operator from $X$ into $Y$ is continuous, that is, $\|f\|_Y \lesssim \|f\|_X$\index{\bigskip\textbf{Numbers and relations}!$X \hookrightarrow Y$}\label{XY} for $f \in X$. By $X=Y$ we mean that $X \hookrightarrow Y \hookrightarrow X$, i.e., $\|f\|_X \asymp \|f\|_Y$.  We use $X'$ to denote the dual space of $X$\index{\bigskip\textbf{Spaces}!$X'$}\label{DUAL}  and $p'$ is the dual exponent of $p \in (0,\infty)$, i.e.,  $\frac{1}{p} + \frac{1}{p'}=1$ if $p \geq 1$ and $p'=\infty$ if $p \leq 1$\index{\bigskip\textbf{Numbers and relations}!$p'$}\label{p'}.

As usual, $\R^d$ denotes the $d$-dimensional Euclidean space\index{\bigskip\textbf{Sets}!$\R^d$}\label{SETR}, $\R = \R^1$\index{\bigskip\textbf{Sets}!$\R$}\label{SETR1}, $\mathbb{C}$ is the set of all complex numbers\index{\bigskip\textbf{Sets}!$\mathbb{C}$}\label{SETC}, $\mathbb{T}^d$ is the $d$-dimensional torus\index{\bigskip\textbf{Sets}!$\mathbb{T}^d$}\label{SETTD}, $\mathbb{T} = \mathbb{T}^1$\index{\bigskip\textbf{Sets}!$\mathbb{T}$}\label{SETT}, $\N$ is the collection of all natural numbers\index{\bigskip\textbf{Sets}!$\N$}\label{SETN}, $\N_0 = \N \cup \{0\}$\index{\bigskip\textbf{Sets}!$\N_0$}\label{SETN0}, $\Z^d$ is the $d$-dimensional integer lattice\index{\bigskip\textbf{Sets}!$\Z^d$}\label{SETZ}, $\Z = \Z^1$\index{\bigskip\textbf{Sets}!$\Z$}\label{SETZ1}. For some $a \in \R$, we let $a_+ = \max\{a, 0\}.$ \index{\bigskip\textbf{Numbers and relations}!$a_+$}\label{a_+}

Given a quasi-Banach space $X$ and $p \in (0, \infty)$, the Bochner space $L_p(\R^d;X)$\index{\bigskip\textbf{Spaces}!$L_p(\R^d; X)$}\label{LEBX} is formed by all strongly measurable functions $f: \R^d \to X$ such that $$\|f\|_{L_p(\R^d;X)} = \bigg(\int_{\R^d} \|f(\xi)\|_X^p \, d\xi \bigg)^{1/p} < \infty.$$ In the special case $X = \R$ or $X = \mathbb{C}$, we simply write $L_p(\R^d)$, the classical Lebesgue space\index{\bigskip\textbf{Spaces}!$L_p(\R^d)$}\label{LEB}.

Let $s \in \mathbb{R}, q \in (0,\infty]$ and let $X$ be a quasi-Banach space. Then $\ell^s_q(X)$ is the sequence space formed by all those  $x= (x_j) \subset X$ such that\index{\bigskip\textbf{Spaces}!$\ell^s_q(X)$}\label{LEBSEQ}
$$\|x\|_{\ell^s_q(X)} = \bigg(\sum_{j=0}^\infty [2^{j s} \|x_j\|_X]^q \bigg)^{1/q} < \infty$$ (with the usual modification if $q=\infty$). In the special case $X = \mathbb{C}$ (or $X= \R$), we simply write $\ell^s_q$\index{\bigskip\textbf{Spaces}!$\ell^s_q(X)$}\label{LEBSEQS}. In addition, if $s=0$ then we recover $\ell_q$\index{\bigskip\textbf{Spaces}!$\ell_q$}\label{LEBC}.

\newpage

\section{A quick review on classical function spaces: Besov, Triebel--Lizorkin and Lipschitz}\label{SectionClassicSpaces}

We collect here basic definitions and standard notations in the theory of function spaces. By $\mathcal{S}(\R^d)$\index{\bigskip\textbf{Spaces}!$\mathcal{S}(\R^d)$}\label{Schwartz} we denote the Schwartz space of all complex-valued, infinitely differentiable and rapidly decreasing functions on $\R^d$ and by $\mathcal{S}'(\R^d)$  the dual space of all tempered distributions in $\R^d$\index{\bigskip\textbf{Spaces}!$\mathcal{S}'(\R^d)$}\label{S'}. If $\varphi \in \mathcal{S}(\R^d)$ then
\begin{equation}\label{FT}
	\widehat{\varphi} (\xi) = \int_{\R^d} e^{-i x \cdot \xi} \varphi(x) \, dx, \qquad \xi \in \R^d,
\end{equation}
denotes the Fourier transform of $\varphi$\index{\bigskip\textbf{Operators}!$\widehat{f}$}\label{FT}. The symbol $\varphi^{\vee}$\index{\bigskip\textbf{Operators}!$\varphi^\vee$}\label{IFT} stands for the inverse Fourier transform and is given by the right-hand side of \eqref{FT} with $i$ in place of $-i$. Both the Fourier transform and the inverse Fourier transform are extended to $\mathcal{S}'(\R^d)$ in the usual way. Let  $(\varphi_j)_{j=0}^\infty$ be the usual dyadic resolution of unity in $\R^d$.

Let $p, q \in (0,\infty]$ and $s, b \in \R$. Then $B^{s, b}_{p,q}(\R^d)$\index{\bigskip\textbf{Spaces}!$B^{s, b}_{p, q}(\R^d)$}\label{BESOVF} is the \emph{Besov space} formed by all $f \in \mathcal{S}'(\R^d)$ such that
\begin{equation}\label{Besov}
	\|f\|_{B^{s, b}_{p,q}(\R^d)} = \left(\sum_{j=0}^\infty 2^{j s q} (1 + j)^{b q}  \|(\varphi_j \widehat{f})^\vee\|_{L_p(\mathbb{R}^d)}^q \right)^{\frac{1}{q}} < \infty
\end{equation}
(with the usual convention if $q=\infty$). The \emph{Triebel--Lizorkin space} $F^{s, b}_{p, q}(\R^d), \, p < \infty,$\index{\bigskip\textbf{Spaces}!$F^{s, b}_{p, q}(\R^d)$}\label{TL} is defined as the collection of all $f \in  \mathcal{S}'(\R^d)$ such that
\begin{equation}\label{TriebelLizorkin}
	\|f\|_{F^{s, b}_{p,q}(\R^d)} = \bigg\| \bigg(\sum_{j=0}^\infty 2^{j s q} (1 + j)^{b q} |(\varphi_j \widehat{f})^\vee|^q \bigg)^{1/q} \bigg\|_{L_p(\R^d)} < \infty
\end{equation}
(with the usual modification made if $q=\infty$). Standard references for the theory of these spaces are \cite{KalyabinLizorkin}, \cite{Moura} and \cite{FarkasLeopold}; see also \cite{DominguezTikhonov} for a recent account. In particular, in the absence of logarithmic smoothness (i.e., $b=0$) one recovers  classical Besov spaces, $B^{s}_{p, q}(\R^d)$, and Triebel--Lizorkin spaces, $F^s_{p, q}(\R^d)$, studied systematically in the monographs \cite{Nikolskii, Peetre76, Triebel83, Triebel01}.

For $s \in \R$ and $p \in (1, \infty)$, the \emph{(fractional) Sobolev space} $H^s_p(\R^d)$\index{\bigskip\textbf{Spaces}!$H^s_p(\R^d)$}\label{SOB} (also known as \emph{Bessel-potential space}) is formed by all $f \in \mathcal{S}'(\R^d)$ such that
$$
	\|f\|_{H^s_p(\R^d)} = \|((1+ |\xi|^2)^{s/2} \widehat{f})^\vee\|_{L_p(\R^d)}< \infty.
$$
Recall that $H^s_p(\R^d) = F^s_{p, 2}(\R^d)$ (with equivalence of norms) and, in particular, $W^k_p(\R^d) = F^k_{p, 2}(\R^d), \, k \in \N_0$, the classical \emph{Sobolev space} formed by all $f \in L_p(\R^d)$ whose weak derivatives $D^l f, \, |l| \leq k$, belong to $L_p(\R^d)$ and equipped with the norm\index{\bigskip\textbf{Spaces}!$W^k_p(\R^d)$}\label{SOBCLAS}
$$
	\|f\|_{W^k_p(\R^d)} := \sum_{|l| \leq k} \| |D^l f| \|_{L_p(\R^d)}.
$$

It is well known that Besov--Triebel--Lizorkin spaces can be equivalently introduced in terms of differences, see e.g.  \cite[2.5.12]{Triebel83}, \cite{KalyabinLizorkin} and \cite{HaroskeMoura}. For $k \in \N$ and $f \in L_p(\R^d), \, 1 \leq p \leq \infty$, the \emph{$k$-th order modulus of smoothness} $\omega_k(f,t)_p$ is defined by\index{\bigskip\textbf{Functionals and functions}!$\omega_k(f, t)_p$}\label{MODK}
\begin{equation}\label{DefModuli}
	\omega_k(f,t)_p = \sup_{|h| \leq t} \|\Delta^k_h f\|_{L_p(\R^d)}, \qquad t > 0,
\end{equation}
where $\Delta^k_h$ is the \emph{$k$-th difference with step $h \in \R^d$} given by\index{\bigskip\textbf{Functionals and functions}!$\Delta^k_h$}\label{DELTA}
$$
	(\Delta^1_h f)(x)= (\Delta_h f)(x) = f(x+h)-f(x), \qquad x \in \R^d,
$$
and $\Delta_h^{k+1} = \Delta_h^k \Delta_h$. It is clear that
\begin{equation}\label{HigherDiff}
	\Delta^k_h f (x) = \sum_{j=0}^k (-1)^j {k \choose j} f(x + (k-j) h).
\end{equation}

Let $1 \leq p \leq \infty, 0 < q \leq \infty, 0 \leq  s < k$ and $b \in \R$. Then $\mathbf{B}^{s, b}_{p, q}(\R^d)$ is the set of all $f \in L_p(\R^d)$ having finite quasi-norm\index{\bigskip\textbf{Spaces}!$\mathbf{B}^{s, b}_{p, q}(\R^d)$}\label{BESOVDIFF}
\begin{equation}\label{BesovDifDef}
	\|f\|_{\mathbf{B}^{s, b}_{p, q}(\R^d)} = \|f\|_{L_p(\R^d)} +  \bigg(\int_0^1 (t^{-s} (1-\log t)^b \omega_k(f,t)_p)^q \frac{dt}{t} \bigg)^{1/q}
\end{equation}
(with the usual modification if $q=\infty$). It is well known that
\begin{equation}\label{BesovDif}
	\|f\|_{\mathbf{B}^{s, b}_{p,q}(\R^d)} \asymp \|f\|_{L_p(\R^d)} + \bigg(\int_{|h| < 1} |h|^{-s q - d} (1-\log |h|)^{b q} \|\Delta^k_h f\|_{L_p(\R^d)}^q dh \bigg)^{1/q}
\end{equation}
and different values of $k$ with $k > s$ in the right-hand sides of \eqref{BesovDifDef} and \eqref{BesovDif} give equivalent quasi-norms on $\mathbf{B}^{s, b}_{p,q}(\R^d)$. We also mention that these spaces admit extensions to the case $p \in (0, 1)$. For the special choice $b=0$, we deal with classical spaces $\mathbf{B}^s_{p, q}(\R^d)$.

The limiting parameter $s=0$ in $\mathbf{B}^{s, b}_{p, q}(\R^d)$ deserves special attention. In this case, we may assume $b \geq -1/q \, (b > 0 \text{ if } q=\infty)$. The latter condition on $b$ is not restrictive since $\mathbf{B}^{0, b}_{p, q}(\R^d) = L_p(\R^d)$ for $b < -1/q$. In particular, the elements in $\mathbf{B}^{0}_{p, 1}(\R^d)$ are defined in terms of the validity of the classical \emph{$L_p$-Dini condition}
\begin{equation}\label{Dini}
	\int_0^1 \omega_k(f, t)_p \frac{dt}{t} < \infty.
\end{equation}
This condition is widely used in functional analysis, probability theory, harmonic analysis and PDE's. The logarithmic refinement of \eqref{Dini} provided by \eqref{BesovDifDef} (with $s=0$) is also important in applications. In this regard we only refer to the survey contained in \cite[Section 1.2]{DominguezTikhonov}.

As already mentioned above, a classical result in the theory of function spaces asserts that
\begin{equation}\label{BesovModuli}
	B^{s, b}_{p,q}(\R^d) = \mathbf{B}^{s, b}_{p,q}(\R^d), \qquad s \in (0, k), \qquad p \in [1,\infty].
\end{equation}
However, the limiting case $s = 0$ in \eqref{BesovModuli} fails to be true.
Similarly, putting
$s = k$ in \eqref{BesovDifDef} we do not recover
$B^{s, b}_{p,q}(\R^d)$. In fact,
 the expression given in the right-hand side of  \eqref{BesovDifDef} with $s=k$ corresponds to a special case of Lipschitz space. More precisely, for $1 \leq p \leq \infty, 0 < q \leq \infty, s > 0$ and $b < -1/q \, (b \leq 0 \text{ if } q=\infty)$, the space $\text{Lip}^{s, b}_{p, q}(\R^d)$ is formed by all those $f \in L_p(\R^d)$ such that
$$
	\int_0^1 t^{-s q} (1-\log t)^{b q} \omega_s(f, t)_p^q \frac{dt}{t} < \infty
$$
(with the usual change if $q=\infty$). Here $\omega_s(f,t)_p$ denotes the classical (if $s \in \N$) or fractional modulus of smoothness, which can be naturally introduced after replacing the classical differences $\Delta^k_h$ (cf. \eqref{HigherDiff}) in \eqref{DefModuli} by \index{\bigskip\textbf{Functionals and functions}!$\Delta^s_h$}\label{DELTAFRAC}
$$
	\Delta^s_h f (x) = \sum_{j=0}^\infty (-1)^j {s \choose j} f(x + (s-j) h)
$$
with ${s \choose j} = \frac{s (s-1) \ldots (s-j+1)}{j!}$ and ${s \choose 0} =1$, more precisely,\index{\bigskip\textbf{Functionals and functions}!$\omega_s(f,t)_p$}\label{MODA}
$$
	\omega_s(f, t)_p = \sup_{|h| \leq t} \|\Delta^s_h f\|_{L_p(\R^d)}.
$$
The space $\text{Lip}^{s, b}_{p, q}(\R^d)$ is endowed with the quasi-norm\index{\bigskip\textbf{Spaces}!$\text{Lip}^{s, b}_{p, q}(\R^d)$}\label{LOGLIPSCHITZ}
$$
	\|f\|_{\text{Lip}^{s, b}_{p, q}(\R^d)} = \|f\|_{L_p(\R^d)} + \bigg( \int_0^1 t^{-s q} (1-\log t)^{b q} \omega_s(f, t)_p^q \frac{dt}{t}  \bigg)^{1/q}.
$$
The assumption $b < -1/q$ is natural, otherwise the space $\text{Lip}^{s, b}_{p, q}(\R^d)$ becomes trivial (in the sense that it is formed only by the zero element). These spaces play a key role in embedding theorems and PDE's and they were intensively investigated in \cite{EdmundsHaroske, EdmundsHaroske00, Haroske} (and the references quoted there) if $s=1$ and  in \cite{DominguezHaroskeTikhonov} for general $s > 0$. In particular, a bounded function $f$ belongs to $\text{Lip}^{1, b}_{\infty, \infty}(\R^d)$ if and only if
$$
	|f(x+h)-f(h)| \lesssim |h| (1 - \log |h|)^{-b} \qquad \text{for all} \qquad |h| < 1.
$$
In the special case $b=0$ one recovers the classical Lipschitz condition. On the other hand,
\begin{equation}\label{LipSobFract}
	\text{Lip}^{s, 0}_{p, \infty}(\R^d) = H^s_p(\R^d).
\end{equation}

As already mentioned above
$$
	\textbf{B}^{0, b}_{p, q}(\R^d) \neq B^{0, b}_{p, q}(\R^d) \qquad \text{and} \qquad \text{Lip}^{k, b}_{p,q}(\R^d) \neq B^{k, b}_{p, q}(\R^d)
$$
(cf. \eqref{BesovModuli}). However, it was recently shown in \cite[Theorem 4.1]{DominguezHaroskeTikhonov}  and \cite[Section 9.1]{DominguezTikhonov} that it is still possible to establish relations between these three scales of function spaces. Namely, let $p \in (1, \infty)$ and $q \in (0, \infty]$ then
\begin{equation}\label{EmbBBzero}
	 B^{0, b + 1/\min\{2, p, q\}}_{p, q}(\R^d)  \hookrightarrow \textbf{B}^{0, b}_{p, q}(\R^d) \hookrightarrow B^{0, b + 1/\max\{2, p, q\}}_{p, q}(\R^d), \qquad \text{if} \qquad b > - 1/q,
\end{equation}
and, if $s > 0$ then
\begin{equation}\label{EmbBL}
	 B^{s, b + 1/\min\{2, p, q\}}_{p, q}(\R^d)  \hookrightarrow \text{Lip}^{s, b}_{p, q}(\R^d) \hookrightarrow B^{s, b + 1/\max\{2, p, q\}}_{p, q}(\R^d), \qquad \text{if} \qquad b < -1/q.
\end{equation}
Furthermore, these embeddings are optimal within the class of Besov spaces. We mention that  versions of \eqref{EmbBBzero} and \eqref{EmbBL} in the limiting cases $p=1, \infty$ are also available in \cite{DominguezHaroskeTikhonov, DominguezTikhonov}.

\newpage
\section{Truncated Besov and Triebel--Lizorkin spaces: definitions and basic properties}

The main objects that we study in this paper are new spaces of smooth functions, which are obtained via truncation methods of the classical quasi-norms \eqref{Besov} and \eqref{TriebelLizorkin} with $b=0$. In particular, these spaces contain as distinguished examples the Besov spaces $B^{s, b}_{p, q}(\R^d)$, the Triebel--Lizorkin spaces $F^{s, b}_{p, q}(\R^d)$, the Besov spaces of smoothness near zero $\mathbf{B}^{0, b}_{p, q}(\R^d)$ and the Lipschitz spaces $\text{Lip}^{s, b}_{p, q}(\R^d)$.

\index{\bigskip\textbf{Spaces}!$T^b_r B^{s}_{p,q}(\mathbb{R}^d)$}\label{TRUNB}
\index{\bigskip\textbf{Spaces}!$T^b_r F^{s}_{p,q}(\mathbb{R}^d)$}\label{TRUNF}
\index{\bigskip\textbf{Spaces}!$\mathfrak{T}^b_r F^{s}_{p,q}(\mathbb{R}^d)$}\label{TRUNF2}

\begin{defn}\label{DefinitionNewBesov}
    \begin{enumerate}[\upshape(i)]
    \item Let $p, q, r \in (0,\infty]$ and $s,b \in\mathbb{R}$. The truncated Besov
    space $T^b_r B^{s}_{p,q}(\mathbb{R}^d)$ is formed by all $f \in
    \mathcal{S}'(\R^d)$ for which
    \begin{equation*}
        \|f\|_{T^b_r B^{s}_{p,q}(\mathbb{R}^d)} = \left(\sum_{j=0}^\infty 2^{j b r} \bigg(\sum_{\nu=2^j-1}^{2^{j+1}-2} 2^{\nu s q} \|(\varphi_\nu \widehat{f})^\vee\|_{L_p(\mathbb{R}^d)}^q\bigg)^{r/q}\right)^{1/r}
    \end{equation*}
    is finite (with the usual convention if $q=\infty$ and/or $r=\infty$).
    \item 	Let $p \in (0, \infty), q, r \in (0, \infty]$ and $s, b \in \R$. The truncated Triebel--Lizorkin space $T^b_r F^{s}_{p, q}(\R^d)$ is formed by all $f \in
    \mathcal{S}'(\R^d)$ for which
    $$
    	\|f\|_{T^b_r F^{s}_{p, q}(\R^d)} = \left(\sum_{j=0}^\infty 2^{j b r}  \bigg\| \bigg(\sum_{\nu=2^j-1}^{2^{j+1}-2} 2^{\nu s q} |(\varphi_\nu \widehat{f})^\vee|^q \bigg)^{1/q} \bigg\|_{L_p(\R^d)}^r \right)^{1/r}
    $$
    is finite (with the usual convention if $q=\infty$ and/or $r=\infty$).
    \item Let $p \in (0, \infty), q, r \in (0, \infty]$ and $s, b \in \R$.  The inner truncated Triebel--Lizorkin space $\mathfrak{T}^b_r F^{s}_{p, q}(\R^d)$ is formed by all $f \in
    \mathcal{S}'(\R^d)$ for which
    $$\|f\|_{\mathfrak{T}^b_r F^{s}_{p, q}(\R^d)} = \bigg\| \bigg(\sum_{j=0}^\infty 2^{j b r} \bigg(\sum_{\nu=2^{j}-1}^{2^{j+1}-2} 2^{\nu s q}  |(\varphi_\nu \widehat{f})^\vee|^q   \bigg)^{r/q} \bigg)^{1/r} \bigg\|_{L_p(\R^d)}$$
      is finite (with the usual convention if $q=\infty$ and/or $r=\infty$).
    \end{enumerate}
\end{defn}

\begin{notation}
	\emph{For $A \in \{B, F\}$, we let $T_r A^{s}_{p, q}(\R^d) := T^0_r A^{s}_{p, q}(\R^d)$ and $\mathfrak{T}_r F^{s}_{p, q}(\R^d) := \mathfrak{T}^0_r F^{s}_{p, q}(\R^d)$.}
\end{notation}

\begin{rem}
\begin{enumerate}[\upshape(i)]
	\item Let $A \in \{B, F\}$. It is plain to see that $T^b_r A^{s}_{p,q}(\mathbb{R}^d)$ and $\mathfrak{T}^b_r F^s_{p, q}(\R^d)$ are quasi-Banach spaces (Banach space if $p, q, r \geq 1$).
	\item One can also introduce the counterparts of the spaces $T^b_r A^{s}_{p,q}(\mathbb{R}^d)$ and $\T^b_r F^s_{p, q}(\R^d)$  related to general smooth resolutions of unity  in the sense of \cite[Section 2.3.1, page 45]{Triebel83}. However the definition of these spaces do not depend (up to equivalence of quasi-norms) on the choice of the resolution of unity. The proof of this fact is a simple consequence of multiplier assertions (cf. \cite[Section 2.3.2, page 46]{Triebel83} for further details).
	\item The periodic counterparts $T^b_r A^s_{p, q}(\mathbb{T}^d), \, A \in \{B, F\},$ and $\mathfrak{T}^b_r F^{s}_{p, q}(\mathbb{T}^d)$ can be introduced in the same fashion. \index{\bigskip\textbf{Spaces}!$T^b_r B^{s}_{p,q}(\mathbb{T}^d)$}\label{TRUNBPER}
\index{\bigskip\textbf{Spaces}!$T^b_r F^{s}_{p,q}(\mathbb{T}^d)$}\label{TRUNFPER}
\index{\bigskip\textbf{Spaces}!$\mathfrak{T}^b_r F^{s}_{p,q}(\mathbb{T}^d)$}\label{TRUNF2PER}
\end{enumerate}
\end{rem}

\begin{prop}
Let $A \in \{B, F\}, p, q, r \in (0, \infty] \, (p < \infty \text{ if } A=F)$ and $s, b \in \R$. Then
	\begin{equation}\label{RemFM}
		\|f\|_{T^b_r A^s_{p, q}(\R^d)} \asymp \bigg(\sum_{j=0}^\infty 2^{j b r} \bigg\|\sum_{\nu=2^j-1}^{2^{j+1}-2} (\varphi_\nu \widehat{f} )^\vee \bigg\|^r_{A^s_{p, q}(\R^d)} \bigg)^{1/r}.
	\end{equation}
	\end{prop}
	\begin{proof}
	Set $\varphi_{-1} \equiv 0$ and $\widetilde{\varphi}_l :=\varphi_{l-1} + \varphi_l + \varphi_{l+1}$ for $l \in \N_0$,. By basic Fourier multiplier assertions (cf. \cite[Theorem 1.5.2]{Triebel83}),
	\begin{align*}
	\bigg\|\sum_{\nu=2^j-1}^{2^{j+1}-2} (\varphi_\nu \widehat{f})^\vee \bigg\|_{B^s_{p, q}(\R^d)}^q &= \sum_{l=0}^\infty 2^{l s q} \bigg\| \bigg(\varphi_l \sum_{\nu=2^j-1}^{2^{j+1}-2} \varphi_\nu \widehat{f} \bigg)^\vee  \bigg\|_{L_p(\R^d)}^q \\
	&= \sum_{l=2^j-2}^{2^{j+1}-1} 2^{l s q} \bigg\| \bigg(\varphi_l \sum_{\nu=2^j-1}^{2^{j+1}-2} \varphi_\nu \widehat{f} \bigg)^\vee  \bigg\|_{L_p(\R^d)}^q \\
	& \lesssim \sum_{l=2^j-2}^{2^{j+1}-1} 2^{l s q} \| (\varphi_l \widetilde{\varphi}_l \widehat{f})^\vee \|_{L_p(\R^d)}^q \\
	& = \sum_{l=2^j-2}^{2^{j+1}-1} 2^{l s q} \| (\varphi_l \widehat{f})^\vee \|_{L_p(\R^d)}^q.
	\end{align*}
	On the other hand, for $j \geq 2$,
	\begin{align*}
		\sum_{l=2^j-2}^{2^{j+1}-1} 2^{l s q} \bigg\| \bigg(\varphi_l \sum_{\nu=2^j-1}^{2^{j+1}-2} \varphi_\nu \widehat{f} \bigg)^\vee  \bigg\|_{L_p(\R^d)}^q  & \geq  \sum_{l=2^j}^{2^{j+1}-3} 2^{l s q} \bigg\| \bigg(\varphi_l \sum_{\nu=2^j-1}^{2^{j+1}-2} \varphi_\nu \widehat{f} \bigg)^\vee  \bigg\|_{L_p(\R^d)}^q \\
		& =  \sum_{l=2^j}^{2^{j+1}-3} 2^{l s q} \| (\varphi_l \widehat{f})^\vee \|_{L_p(\R^d)}^q.
	\end{align*}
	This gives \eqref{RemFM} with $A=B$. The corresponding assertion with $A=F$ follows similar ideas as above but now relying on vector-valued multiplier theorems, as can be found e.g. in \cite[Theorem 1.6.3]{Triebel83}
\end{proof}

We provide some elementary equivalent quasi-norms on truncated function spaces.

\begin{prop}\label{PropEquiQN}
	Let $p, q, r \in (0, \infty], s \in \R$ and $b \in \R \backslash \{0\}$.
	\begin{enumerate}[\upshape(i)]
	\item Assume $b > 0$. Then
	          \begin{equation*}
        \|f\|_{T^b_r B^{s}_{p,q}(\mathbb{R}^d)} \asymp  \|f\|^*_{T^b_r B^{s}_{p,q}(\mathbb{R}^d)} := \left(\sum_{j=0}^\infty 2^{j b r} \bigg(\sum_{\nu= 2^{j}-1}^\infty 2^{\nu s q} \|(\varphi_\nu \widehat{f})^\vee\|_{L_p(\mathbb{R}^d)}^q\bigg)^{r/q}\right)^{1/r},
    \end{equation*}
    and if additionally $p < \infty$ then
    $$
    	\|f\|_{T^b_r F^{s}_{p, q}(\R^d)} \asymp    	\|f\|^*_{T^b_r F^{s}_{p, q}(\R^d)} := \left(\sum_{j=0}^\infty 2^{j b r}  \bigg\| \bigg(\sum_{\nu= 2^{j}-1}^\infty 2^{\nu s q} |(\varphi_\nu \widehat{f})^\vee|^q \bigg)^{1/q} \bigg\|_{L_p(\R^d)}^r \right)^{1/r},
    $$
      $$\|f\|_{\mathfrak{T}^b_r F^{s}_{p, q}(\R^d)} \asymp \|f\|^*_{\mathfrak{T}^b_r F^{s}_{p, q}(\R^d)} := \bigg\| \bigg(\sum_{j=0}^\infty 2^{j b r} \bigg(\sum_{\nu= 2^{j}-1}^\infty 2^{\nu s q}  |(\varphi_\nu \widehat{f})^\vee|^q   \bigg)^{r/q} \bigg)^{1/r} \bigg\|_{L_p(\R^d)}.$$
      	\item Assume $b < 0$. Then
	          \begin{equation*}
        \|f\|_{T^b_r B^{s}_{p,q}(\mathbb{R}^d)} \asymp  \|f\|^*_{T^b_r B^{s}_{p,q}(\mathbb{R}^d)} := \left(\sum_{j=0}^\infty 2^{j b r} \bigg(\sum_{\nu= 0}^{2^j} 2^{\nu s q} \|(\varphi_\nu \widehat{f})^\vee\|_{L_p(\mathbb{R}^d)}^q\bigg)^{r/q}\right)^{1/r},
    \end{equation*}
    and if additionally $p < \infty$ then
    $$
    	\|f\|_{T^b_r F^{s}_{p, q}(\R^d)} \asymp    	\|f\|^*_{T^b_r F^{s}_{p, q}(\R^d)} := \left(\sum_{j=0}^\infty 2^{j b r}  \bigg\| \bigg(\sum_{\nu= 0}^{2^j} 2^{\nu s q} |(\varphi_\nu \widehat{f})^\vee|^q \bigg)^{1/q} \bigg\|_{L_p(\R^d)}^r \right)^{1/r},
    $$
      $$\|f\|_{\mathfrak{T}^b_r F^{s}_{p, q}(\R^d)} \asymp \|f\|^*_{\mathfrak{T}^b_r F^{s}_{p, q}(\R^d)} := \bigg\| \bigg(\sum_{j=0}^\infty 2^{j b r} \bigg(\sum_{\nu= 0}^{2^j} 2^{\nu s q}  |(\varphi_\nu \widehat{f})^\vee|^q   \bigg)^{r/q} \bigg)^{1/r} \bigg\|_{L_p(\R^d)}.$$
      \end{enumerate}
	\end{prop}
	
	\begin{rem}
		We stress that the definitions of $\|\cdot\|_{T^b_r A^s_{p, q}(\R^d)}^*$ for $A \in \{B, F\}$ and $\|\cdot\|_{\mathfrak{T}^b_r F^{s}_{p, q}(\R^d)}$ depend on the sign of $b$.
	\end{rem}
	
	\begin{proof}[Proof of Proposition \ref{PropEquiQN}]
		We will only prove
		\begin{equation*}
				\|f\|_{T^b_r F^{s}_{p, q}(\R^d)} \asymp    	\|f\|^*_{T^b_r F^{s}_{p, q}(\R^d)}.
		\end{equation*}
		The corresponding assertions for $T^b_r B^{s}_{p,q}(\mathbb{R}^d)$ and $\mathfrak{T}^b_r F^{s}_{p, q}(\R^d)$ follow similar ideas.
		
		Obviously $\|f\|_{T^b_r F^{s}_{p, q}(\R^d)} \leq    	\|f\|^*_{T^b_r F^{s}_{p, q}(\R^d)}$.  To show the converse estimate, we can argue as follows. Let $b > 0$, by triangle inequality and Hardy's inequality \eqref{H2},
		\begin{align*}
			\|f\|^*_{T^b_r F^{s}_{p, q}(\R^d)} & =  \left(\sum_{j=0}^\infty 2^{j b r}  \bigg\| \sum_{\nu= 2^{j}-1}^\infty 2^{\nu s q} |(\varphi_\nu \widehat{f})^\vee|^q  \bigg\|_{L_{p/q}(\R^d)}^{r/q} \right)^{1/r} \\
			& \leq  \left(\sum_{j=0}^\infty 2^{j b r}  \bigg(\sum_{l=j}^\infty \bigg\|\sum_{\nu= 2^{l}-1}^{2^{l+1}-2} 2^{\nu s q} |(\varphi_\nu \widehat{f})^\vee|^q  \bigg\|^{\min\{1, p/q\}}_{L_{p/q}(\R^d)} \bigg)^{r/(q \min\{1, p/q\})} \right)^{1/r} \\
			& \asymp  \left(\sum_{j=0}^\infty 2^{j b r}   \bigg\|\bigg(\sum_{\nu= 2^{j}-1}^{2^{j+1}-2} 2^{\nu s q} |(\varphi_\nu \widehat{f})^\vee|^q \bigg)^{1/q} \bigg\|^{r}_{L_{p}(\R^d)} \right)^{1/r} = \|f\|_{T^b_r F^{s}_{p, q}(\R^d)}.
		\end{align*}
		The case $b < 0$ can be obtained similarly but now invoking \eqref{H1}.
	\end{proof}

\begin{rem}\label{Remark36}
	The case $b=0$ in $T^b_r A^s_{p, q}(\R^d), \, A \in \{B, F\},$ and $\T^b_r F^{s}_{p, q}(\R^d)$ needs special care. This assertion can be justified by the following two facts.  Firstly, Proposition \ref{PropEquiQN} provides two natural candidates for quasi-norms on $T^0_r B^s_{p, q}(\R^d)$, namely
	\begin{equation}\label{31}
	 \left(\sum_{j=0}^\infty \bigg(\sum_{\nu= 2^{j}-1}^\infty 2^{\nu s q} \|(\varphi_\nu \widehat{f})^\vee\|_{L_p(\mathbb{R}^d)}^q\bigg)^{r/q}\right)^{1/r}
	\end{equation}
	and
	\begin{equation}\label{32}
	\left(\sum_{j=0}^\infty \bigg(\sum_{\nu= 0}^{2^j} 2^{\nu s q} \|(\varphi_\nu \widehat{f})^\vee\|_{L_p(\mathbb{R}^d)}^q\bigg)^{r/q}\right)^{1/r}.
	\end{equation}
	However \eqref{32} is not a meaningful condition since it is only satisfied by zero function:
	\begin{equation}\label{TrivialityZero}
		\sum_{j=0}^\infty \bigg(\sum_{\nu= 0}^{2^j} 2^{\nu s q} \|(\varphi_\nu \widehat{f})^\vee\|_{L_p(\mathbb{R}^d)}^q\bigg)^{r/q} < \infty \iff f \equiv 0.
	\end{equation}
	  Indeed, the trivial estimate
	$$
		\sum_{j=0}^\infty \bigg(\sum_{\nu= 0}^{2^j} 2^{\nu s q} \|(\varphi_\nu \widehat{f})^\vee\|_{L_p(\mathbb{R}^d)}^q\bigg)^{r/q} \geq \sum_{j=0}^\infty \|(\varphi_0 \widehat{f})^\vee\|_{L_p(\mathbb{R}^d)}
	$$
	holds.
	On the other hand, condition \eqref{31} does make sense and then $T^*_r B^s_{p, q}(\R^d)$ is formed by all $f \in \mathcal{S}'(\R^d)$ such that\index{\bigskip\textbf{Spaces}!$T^*_r B^{s}_{p,q}(\mathbb{R}^d)$}\label{TRUNBL}
	\begin{equation}\label{33}
		\|f\|_{T^*_r B^s_{p, q}(\R^d)}  :=   \left(\sum_{j=0}^\infty \bigg(\sum_{\nu= 2^{j}-1}^\infty 2^{\nu s q} \|(\varphi_\nu \widehat{f})^\vee\|_{L_p(\mathbb{R}^d)}^q\bigg)^{r/q}\right)^{1/r} < \infty.
	\end{equation}
	In particular, in the special case $r=q$ we have\index{\bigskip\textbf{Spaces}!$B^{s, b, \xi}_{p,q}(\mathbb{R}^d)$}\label{BESOVFLOG}
	\begin{equation}\label{newqq}
		\|f\|_{T^*_q B^s_{p, q}(\R^d)} \asymp \|f\|_{B^{s, 0, 1/q}_{p, q}(\R^d)}:= \left(\sum_{j=0}^\infty [2^{j s} (1+ \log (1+ j))^{1/q}
        \| (\varphi_j \widehat{f})^\vee \|_{L_p(\mathbb{R}^d)}]^q  \right)^{1/q},
	\end{equation}
	a \emph{Besov space of iterated logarithmic smoothness} (cf. \cite[p. 32]{DominguezTikhonov}). Analogously, the space $T^*_r F^{s}_{p, q}(\R^d)$ is defined by\index{\bigskip\textbf{Spaces}!$T^*_r F^{s}_{p,q}(\mathbb{R}^d)$}\label{TRUNFL}
	\begin{equation}\label{34}
	\|f\|_{T^*_r F^{s}_{p, q}(\R^d)}  := \left(\sum_{j=0}^\infty  \bigg\| \bigg(\sum_{\nu= 2^{j}-1}^\infty 2^{\nu s q} |(\varphi_\nu \widehat{f})^\vee|^q \bigg)^{1/q} \bigg\|_{L_p(\R^d)}^r \right)^{1/r} < \infty.
	\end{equation}
	Clearly
	\begin{equation}\label{35}
		T^*_\infty A^{s}_{p, q}(\R^d) = A^{s}_{p, q}(\R^d), \qquad A \in \{B, F\}.
	\end{equation}
	
	Secondly
	$$
		T^*_r A^s_{p, q}(\R^d) \neq T_r A^s_{p, q}(\R^d), \qquad A \in \{B, F\}.
	$$
	Indeed, to fix some ideas consider first the special case $r=q$ and $A=B$, by above considerations (cf. \eqref{newqq}) we have $T^*_q B^s_{p, q}(\R^d) = B^{s, 0, 1/q}_{p, q}(\R^d)$, but on the other hand it is clear that $T_q B^s_{p, q}(\R^d) = B^s_{p, q}(\R^d)$ (we will see in Proposition \ref{PropositionCoincidences} below further relationships between truncated and classical function spaces). Note that $B^{s, 0, 1/q}_{p, q}(\R^d) \subsetneq B^s_{p, q}(\R^d)$. In fact, this particular case illustrates what happens in the general setting. More precisely, it is clear that $T^*_r A^s_{p, q}(\R^d) \hookrightarrow T_r A^{s}_{p, q}(\R^d)$. However,  this embedding is strict. To see this one may consider lacunary Fourier series of the form
	\begin{equation}\label{Ex67}
f(x) \sim \sum_{j=0}^\infty a_j e^{i 2^j x_1} \psi(x), \qquad x = (x_1, \ldots, x_d) \in \R^d,
\end{equation}
where $\psi \in \mathcal{S}(\mathbb{R}^d) \backslash \{0\}$ with $\text{supp } \psi \subset \mathbb{T}^d$ and $(a_j)_{j \in \N_0}$ is a scalar-valued sequence. It is plain to check that
\begin{equation}\label{lac}
	(\varphi_j \widehat{f})^\vee(x) = a_j e^{i (2^j-2) x_1} \psi(x), \qquad j \geq 0.
\end{equation}
	In particular, if $a_j =  2^{-j s} j^{-1/q} (1 + \log (1+j))^{-\varepsilon}$ where $\max\{1/r, 1/q\} < \varepsilon < 1/r + 1/q$, then
\begin{align*}
	\|f\|_{T^*_r A^s_{p, q}(\R^d)} &\asymp \left(\sum_{j=0}^\infty \bigg(\sum_{\nu=2^{j}}^{\infty} (1+ \log \nu)^{-\varepsilon q} \frac{1}{\nu}  \bigg)^{r/q}  \right)^{1/r} \\
	&  \asymp \bigg(\sum_{j=0}^\infty j^{-\varepsilon r + r/q} \bigg)^{1/r} = \infty
\end{align*}
but
$$
	\|f\|_{T_r A^{s}_{p, q}(\R^d)} \asymp \left( \sum_{j=0}^\infty j^{-\varepsilon r} \right)^{1/r} < \infty.
	$$
	Hence $f \in T_r A^{s}_{p, q}(\R^d) \backslash T^*_r A^s_{p, q}(\R^d)$.

\end{rem}

\subsection{Relationships between truncated and classical spaces} The following result provides a table of coincidences between the scale of spaces $T^b_r A^{s}_{p, q}(\R^d), \, A \in \{B, F\}, \T^b_r F^s_{p, q}(\R^d)$, introduced in Definition  \ref{DefinitionNewBesov} and the classical spaces considered in Section \ref{SectionClassicSpaces}.

\begin{prop}\label{PropositionCoincidences}
	\begin{enumerate}[\upshape(i)]
	\item Let $p, q \in (0,\infty]$ and $s, b \in \R$. Then
	$$
		T^b_q B^{s}_{p,q}(\mathbb{R}^d) = B^{s, b}_{p,q}(\R^d).
	$$
	\item Let $p \in (0, \infty), q \in (0, \infty]$ and $s, b \in \R$. Then
	$$
		\T^b_q F^{s}_{p, q}(\R^d) = F^{s, b}_{p, q}(\R^d).
	$$
	\item Let $p \in (1, \infty), q \in (0, \infty]$ and $b > -1/q$. Then
	$$
		T^{b+1/q}_q F^{0}_{p, 2} (\R^d) = \emph{\textbf{B}}^{0, b}_{p, q}(\R^d).
	$$
	\item Let $p \in (1, \infty), q \in (0, \infty], s > 0$ and $b < -1/q$. Then
	$$
		T^{b+1/q}_q F^{s}_{p, 2}(\R^d) = \emph{Lip}^{s, b}_{p, q}(\R^d).
	$$
	\end{enumerate}
\end{prop}

\begin{proof}
(i): We have
\begin{equation*}
	\|f\|_{T^b_q B^{s}_{p,q}(\mathbb{R}^d)}  \asymp \left(\sum_{j=0}^\infty \sum_{\nu=2^j-1}^{2^{j+1}-2} 2^{\nu s q} (1 + \nu)^{b q} \|(\varphi_\nu \widehat{f})^\vee\|_{L_p(\mathbb{R}^d)}^q \right)^{1/q} = \|f\|_{B^{s, b}_{p, q}(\R^d)}.
\end{equation*}

The item (ii) can be obtained in a similar fashion as (i).

(iii): It was shown in \cite[Theorem 4.3]{CobosDominguezTriebel} that the space $\textbf{B}^{0, b}_{p, q}(\R^d)$ can be equivalently defined through the Fourier transform, namely,
\begin{equation}\label{CDT}
	\|f\|_{\textbf{B}^{0, b}_{p, q}(\R^d)} \asymp \left(\sum_{j=0}^\infty (1 + j)^{b q} \bigg\| \bigg(\sum_{\nu=j}^\infty |(\varphi_\nu \widehat{f})^\vee|^2 \bigg)^{1/2} \bigg\|_{L_p(\R^d)}^q \right)^{1/q}.
\end{equation}
Therefore the proof will be completed if we show that the right-hand side of the previous expression is equivalent to $\|f\|_{T^{b+1/q}_q F^{0}_{p, 2}(\R^d)}$. To proceed with, we first note that, by  \eqref{CDT} and monotonicity properties,
\begin{equation}\label{CDT1}
		\|f\|_{\textbf{B}^{0, b}_{p, q}(\R^d)}  \asymp \left(\sum_{j=0}^\infty 2^{j (b + 1/q)q}  \bigg\| \bigg(\sum_{\nu=2^j-1}^{\infty} |(\varphi_\nu \widehat{f})^\vee|^2 \bigg)^{1/2} \bigg\|_{L_p(\R^d)}^q \right)^{1/q}
	 \end{equation}
	 and thus  $\|f\|_{T^{b+1/q}_q F^{0}_{p, 2}(\R^d)} \lesssim \|f\|_{\textbf{B}^{0, b}_{p, q}(\R^d)}$. To prove the converse estimate, according to \eqref{CDT1}, the fact that $\ell_1 \hookrightarrow \ell_2$ and the triangle inequality, we obtain
	 \begin{align}
	 	\|f\|_{\textbf{B}^{0, b}_{p, q}(\R^d)}  & \asymp  \left(\sum_{j=0}^\infty 2^{j (b + 1/q)q}  \bigg\| \bigg(\sum_{l=j}^\infty \sum_{\nu=2^l-1}^{2^{l+1}-2} |(\varphi_\nu \widehat{f})^\vee|^2 \bigg)^{1/2} \bigg\|_{L_p(\R^d)}^q \right)^{1/q} \nonumber  \\
		& \leq  \left(\sum_{j=0}^\infty 2^{j (b + 1/q)q}  \bigg(\sum_{l=j}^\infty \bigg\| \bigg(\sum_{\nu=2^l-1}^{2^{l+1}-2} |(\varphi_\nu \widehat{f})^\vee|^2 \bigg)^{1/2} \bigg\|_{L_p(\R^d)} \bigg)^q \right)^{1/q}. \label{CDT2}
	 \end{align}
	
	 If $q \leq 1$ then, using \eqref{CDT2}, the embedding $\ell_q \hookrightarrow \ell_1$ and changing the order of summation, we get
	 \begin{align*}
	 	\|f\|_{\textbf{B}^{0, b}_{p, q}(\R^d)} & \lesssim \left(\sum_{j=0}^\infty 2^{j (b + 1/q)q}  \sum_{l=j}^\infty \bigg\| \bigg(\sum_{\nu=2^l-1}^{2^{l+1}-2} |(\varphi_\nu \widehat{f})^\vee|^2 \bigg)^{1/2} \bigg\|_{L_p(\R^d)}^q  \right)^{1/q} \\
		& \asymp \bigg(\sum_{l=0}^\infty 2^{l(b+1/q) q} \bigg\| \bigg(\sum_{\nu=2^l-1}^{2^{l+1}-2} |(\varphi_\nu \widehat{f})^\vee|^2 \bigg)^{1/2} \bigg\|_{L_p(\R^d)}^q \bigg)^{1/q} = \|f\|_{T^{b+1/q}_q F^{0}_{p, 2}(\R^d)}.
	 \end{align*}
	 On the other hand, if $q > 1$ we can apply Hardy's inequality in \eqref{CDT2} (since $b + 1/q > 0$) resulting in
	 \begin{equation*}
	 \|f\|_{\textbf{B}^{0, b}_{p, q}(\R^d)}  \lesssim  \left(\sum_{j=0}^\infty 2^{j (b + 1/q)q}  \bigg\| \bigg(\sum_{\nu=2^j-1}^{2^{j+1}-2} |(\varphi_\nu \widehat{f})^\vee|^2 \bigg)^{1/2} \bigg\|_{L_p(\R^d)}^q \right)^{1/q} = \|f\|_{T^{b+1/q}_q F^{0}_{p, 2}(\R^d)}.
	 \end{equation*}
	 	
	 (iv): See Theorem \ref{ThmLipFourier} in Appendix A.

\end{proof}

We continue with some elementary relations for $T^b_r A^{s}_{p, q}(\R^d)$ for $A \in \{B, F\}$ and $\T^b_r F^s_{p, q}(\R^d)$. In particular, we show the important property that embeddings between classical spaces  $A^s_{p, q}(\R^d)$ are preserved under the action of  truncation $T^b_r$. To be more precise, we establish the following result.

\begin{prop}\label{PropPreser}
	Let $A, \widetilde{A} \in \{B, F\}$. Let $p_i \in (0,\infty] \, (p_0 \in (0, \infty) \text{ if } A = F \text{ and } p_1 \in (0, \infty) \text{ if } \widetilde{A} = F), q_i \in (0,\infty]$ and $s_i \in\mathbb{R},$ where $i=0, 1$. Assume that
$$
	A^{s_0}_{p_0, q_0}(\R^d) \hookrightarrow \widetilde{A}^{s_1}_{p_1, q_1}(\R^d).
$$
Then, for $r \in (0, \infty]$ and $b \in \R$, we have
\begin{equation*}
T^b_r A^{s_0}_{p_0,q_0}(\mathbb{R}^d) \hookrightarrow  T^{b}_r \widetilde{A}^{s_1}_{p_1,q_1}(\mathbb{R}^d).
\end{equation*}
\end{prop}

\begin{proof}
	Apply \eqref{RemFM}.
\end{proof}

\begin{prop}\label{PropositionElementary}
\begin{enumerate}[\upshape(i)]
\item Let $A \in \{B, F\}$. Let $p \in (0,\infty] \, (p \in (0, \infty) \text{ if } A = F), q, r \in (0,\infty], s \in\mathbb{R}$ and $b > 0$. Then
\begin{equation}\label{cncnnc}
T^b_r A^{s}_{p,q}(\mathbb{R}^d) \hookrightarrow A^s_{p,q}(\R^d) \hookrightarrow T^{-b}_r A^{s}_{p,q}(\mathbb{R}^d)
\end{equation}
and
 \begin{equation}\label{cncnnc2}
 \T^b_r F^{s}_{p,q}(\mathbb{R}^d) \hookrightarrow F^s_{p,q}(\R^d) \hookrightarrow \T^{-b}_r F^{s}_{p,q}(\mathbb{R}^d).
 \end{equation}

\item Let $A, \widetilde{A} \in \{B, F\}$. Let $p \in (0,\infty] \, (p \in (0, \infty) \text{ if } A = F \text{ or } \widetilde{A} = F), q_0, q_1, r \in (0,\infty], s,b \in\mathbb{R}$ and $\varepsilon > 0$. Then
\begin{equation}\label{cncnnc--}
T^b_r A^{s+\varepsilon}_{p,q_0}(\mathbb{R}^d) \hookrightarrow \widetilde{A}^s_{p,q_1}(\R^d) \hookrightarrow T^b_r A^{s-\varepsilon}_{p,q_0}(\mathbb{R}^d),
\end{equation}
 \begin{equation}\label{cncnnc2--}
 \T^b_r F^{s+\varepsilon}_{p,q_0}(\mathbb{R}^d) \hookrightarrow A^s_{p,q_1}(\R^d) \hookrightarrow \T^{b}_r F^{s-\varepsilon}_{p,q_0}(\mathbb{R}^d).
 \end{equation}

\item Let $p \in (0, \infty), q, r \in (0, \infty]$ and $s, b \in \R$. Then
\begin{equation}\label{PropElem1*}
T^b_r B^{s}_{p,\min\{p, q\}}(\R^d)	\hookrightarrow T^b_r F^{s}_{p, q}(\R^d) \hookrightarrow T^b_r B^{s}_{p,\max\{p, q\}}(\R^d).
\end{equation}
In particular
$$
	T^b_r F^{s}_{p, p}(\R^d) = T^b_r B^{s}_{p, p}(\R^d).
$$
\item Let $p \in (0, \infty), q, r \in (0, \infty]$ and $s, b \in \R$. Then
\begin{equation}\label{PropElem1}
	T^b_{ \min\{p, r\}} F^{s}_{p, q}(\R^d) \hookrightarrow  \T^b_r F^{s}_{p, q}(\R^d) \hookrightarrow T^b_{\max\{p, r\}} F^{s, b}_{p, q}(\R^d).
\end{equation}
In particular
$$
	T^b_p F^{s}_{p, q}(\R^d) = \T^b_p F^{s}_{p, q}(\R^d).
$$
\item Let $p \in (0, \infty)$ and $s, b \in \R$. Then
$$
	T^b_p B^{s}_{p, p}(\R^d) = T^b_p	F^{s}_{p, p}(\R^d)= \T^b_p F^{s}_{p, p}(\R^d).
$$
\end{enumerate}
\end{prop}

Before we give the proof of this proposition, we observe that the new scale of spaces $T^b_r A^{s}_{p, q}(\R^d)$ enables to sharpen classical embeddings. For instance, let us illustrate this phenomenon with the well-known embeddings
\begin{equation}\label{EmbBFnwnqnq}
	B^s_{p, \min\{p, q\}}(\R^d) \hookrightarrow F^s_{p, q}(\R^d) \hookrightarrow B^s_{p, \max\{p, q\}}(\R^d),
\end{equation}
where $p \in (0, \infty), q \in (0, \infty]$ and $s \in \R$,
and their special case (taking $q=2$), for $1 < p < \infty$ and $s \in \R$,
\begin{equation}\label{EmbBFnwnqnq2}
B^s_{p, \min\{p, 2\}}(\R^d) \hookrightarrow H^s_{p}(\R^d) \hookrightarrow B^s_{p, \max\{p, 2\}}(\R^d).
\end{equation}

\begin{cor}\label{Corollary3.6}
	\begin{enumerate}[\upshape(i)]
	\item Let $s \in \R, p \in (0, \infty)$ and $q \in (0, \infty]$. Then
	\begin{align}
		B^s_{p, \min\{p, q\}}(\R^d) &\hookrightarrow	 T_{\min\{p, q\}} F^s_{p, q}(\R^d) \nonumber \\
		& \hookrightarrow F^s_{p, q}(\R^d) \hookrightarrow T_{\max\{p, q\}} F^s_{p, q}(\R^d) \hookrightarrow B^s_{p, \max\{p, q\}}(\R^d).\label{FFSharpening}
	\end{align}
	\item Let $s \in \R$ and $p \in (1, \infty)$. Then
	\begin{align*}
	B^s_{p, \min\{p, 2\}}(\R^d) &\hookrightarrow T_{\min\{p, 2\}} F^s_{p, 2}(\R^d) \\
	&	\hookrightarrow H^s_p(\R^d) \hookrightarrow T_{\max\{p, 2\}} F^s_{p, 2}(\R^d) \hookrightarrow B^s_{p, \max\{p, 2\}}(\R^d).
	\end{align*}
	\end{enumerate}
\end{cor}
\begin{proof}
(i): Since $F^s_{p, q}(\R^d) = \T_q F^s_{p, q}(\R^d)$ and $B^s_{p, q}(\R^d) = T_q B^s_{p, q}(\R^d)$ (cf. Proposition \ref{PropositionCoincidences}), it follows from \eqref{PropElem1*}-\eqref{PropElem1} that
$$
	F^s_{p, q}(\R^d) \hookrightarrow T_{\max\{p, q\}} F^s_{p, q}(\R^d) \hookrightarrow  T_{\max\{p, q\}} B^s_{p, \max\{p, q\}}(\R^d) = B^s_{p, \max\{p, q\}}(\R^d)
$$
and
$$
	F^s_{p, q}(\R^d) \hookleftarrow T_{\min\{p, q\}} F^s_{p, q}(\R^d) \hookleftarrow T_{\min\{p, q\}} B^s_{p, \min\{p, q\}}(\R^d) = B^s_{p, \min\{p, q\}}(\R^d).
$$

Specializing (i) with $q=2$ and $p \in (1, \infty)$ one gets (ii).
\end{proof}

\begin{rem}
	Embeddings given in Corollary \ref{Corollary3.6} are non trivial improvements of the classical ones \eqref{EmbBFnwnqnq} and \eqref{EmbBFnwnqnq2}. For instance, assume $p > q$ then by \eqref{FFSharpening},
	$$
		F^{s}_{p, q}(\R^d) \hookrightarrow T_p F^s_{p, q}(\R^d) \hookrightarrow B^s_{p, p}(\R^d).
	$$
	Furthermore these embeddings are strict. Indeed, consider the lacunary Fourier series $f$ defined by \eqref{Ex67}. Since \eqref{lac} holds, we have
$$
	\|f\|_{F^s_{p, q}(\R^d)} \asymp \bigg(\sum_{j=0}^\infty 2^{j s q} |a_j|^q \bigg)^{1/q},
$$
$$
	\|f\|_{T_p F^s_{p, q}(\R^d)} \asymp \bigg(\sum_{k=0}^\infty \bigg(\sum_{j=2^k-1}^{2^{k+1}-2} 2^{j s q} |a_j|^q \bigg)^{p/q} \bigg)^{1/p},
$$
and
$$
	\|f\|_{B^s_{p, p}(\R^d)} \asymp \bigg(\sum_{j=0}^\infty 2^{j s p} |a_j|^p \bigg)^{1/p}.
$$
Taking $a_j = 2^{-j s} (j+1)^{-1/q} (1+ \log (1+j))^{-\varepsilon}$ where $1/p < \varepsilon < 1/q$, we have
$$
	\|f\|_{F^s_{p, q}(\R^d)}^q \asymp \sum_{j=1}^\infty (1+ \log j)^{-\varepsilon q} \frac{1}{j} = \infty
$$
and
$$
	\|f\|_{T_p F^s_{p, q}(\R^d)}^p \asymp \sum_{k=1}^\infty k^{-\varepsilon p} < \infty.
$$
On the other hand, if $a_j = 2^{-j s} (1+j)^{-\varepsilon}$ with $1/p < \varepsilon < 1/q$ then
$$
	\|f\|_{B^s_{p, p}(\R^d)}^p \asymp \sum_{j=1}^\infty j^{-\varepsilon p} <  \infty
$$
and
$$
	\|f\|_{T_p F^s_{p, q}(\R^d)}^p \asymp \sum_{k=0}^\infty \bigg(\sum_{j=2^k}^{2^{k+1}} j^{- \varepsilon q} \bigg)^{p/q} \asymp \sum_{k=0}^\infty 2^{-k p(\varepsilon - 1/q)} = \infty.
$$

The above argument also shows that the embeddings
$$
	B^s_{p, p}(\R^d) \subsetneqq T_p F^s_{p, q}(\R^d) \subsetneqq F^s_{p, q}(\R^d), \qquad p < q
$$
are strict. 
\end{rem}

\begin{proof}[Proof of Proposition \ref{PropositionElementary}]
	(i): We start by proving the right-hand side embedding with $A \in \{B, F\}$. It is clear that
	$$
		\|f\|_{T^{-b}_r A^{s}_{p,q}(\mathbb{R}^d)} \leq \bigg(\sum_{j=0}^\infty 2^{-j b r} \bigg)^{1/r} \|f\|_{A^s_{p,q}(\R^d)}
	$$
	where the last sum is convergent since $b > 0$. We also have
	$$
		 \bigg(\sum_{j=0}^\infty 2^{-j b r} \bigg(\sum_{\nu=2^{j}-1}^{2^{j+1}-2} 2^{\nu s q}  |(\varphi_\nu \widehat{f})^\vee (x)|^q   \bigg)^{r/q} \bigg)^{1/r} \lesssim  \bigg(\sum_{\nu=0}^\infty 2^{\nu s q}  |(\varphi_\nu \widehat{f})^\vee (x)|^q   \bigg)^{1/q}
	$$
	for $x \in \R^d$ and thus $\|f\|_{\T^{-b}_r F^{s}_{p, q}(\R^d)} \lesssim \|f\|_{F^s_{p, q}(\R^d)}.$
	
	Concerning the left-hand side embedding of \eqref{cncnnc} with $A=B$, we have
	$$
		 \sum_{\nu=2^j-1}^{2^{j+1}-2} 2^{\nu s q} \|(\varphi_\nu \widehat{f})^\vee\|_{L_p(\mathbb{R}^d)}^q \leq 2^{-j b q} \, \|f\|^q_{T^b_r B^{s}_{p, q}(\R^d)}, \qquad j \geq 0,
	$$
	for $f \in T^b_r B^{s}_{p, q}(\R^d)$. Summing up the previous estimates, we obtain
	\begin{align*}
		\|f\|^q_{B^s_{p, q}(\R^d)} &= \sum_{j=0}^\infty \sum_{\nu=2^j-1}^{2^{j+1}-2} 2^{\nu s q} \|(\varphi_\nu \widehat{f})^\vee\|_{L_p(\mathbb{R}^d)}^q \\
		& \leq \|f\|^q_{T^b_r B^{s}_{p, q}(\R^d)} \sum_{j=0}^\infty 2^{-j b q} \asymp \|f\|^q_{T^b_r B^{s}_{p, q}(\R^d)}.
	\end{align*}
	Similarly, to deal with $A = F$ in \eqref{cncnnc}, we note that
		\begin{equation}\label{new1}
		 \bigg\| \bigg(\sum_{\nu=2^j-1}^{2^{j+1}-2} 2^{\nu s q} |(\varphi_\nu \widehat{f})^\vee|^q \bigg)^{1/q} \bigg\|_{L_p(\R^d)} \leq 2^{-j b} \, \|f\|_{T^b_r F^{s}_{p, q}(\R^d)}, \qquad j \geq 0,
	\end{equation}
	for $f \in T^b_r F^{s}_{p, q}(\R^d)$. Furthermore, we claim that
	\begin{equation}\label{new2new}
		\|f\|_{F^s_{p, q}(\R^d)} \leq  \bigg(\sum_{j=0}^\infty \bigg\| \bigg(\sum_{\nu=2^j-1}^{2^{j+1}-2} 2^{\nu s q} |(\varphi_\nu \widehat{f})^\vee|^q \bigg)^{1/q} \bigg\|_{L_p(\R^d)}^{\min\{p, q\}} \bigg)^{1/\min\{p, q\}}.
	\end{equation}
	Assuming momentarily that \eqref{new2new} is true, then the embedding $T^b_r F^{s}_{p, q}(\R^d) \hookrightarrow F^s_{p, q}(\R^d)$ follows directly from \eqref{new1}.
	
	To prove \eqref{new2new}, we shall consider two possible cases. Firstly, if $p \geq q$ we can estimate $\|f\|_{F^s_{p, q}(\R^d)}$ using the Minkowski's inequality, specifically,
	\begin{align*}
		\|f\|_{F^s_{p, q}(\R^d)} & =  \bigg\| \bigg(\sum_{j=0}^\infty \sum_{\nu=2^j-1}^{2^{j+1}-2} 2^{\nu s q} |(\varphi_\nu \widehat{f})^\vee|^q \bigg)^{1/q} \bigg\|_{L_p(\R^d)} \\
		& \leq \bigg(\sum_{j=0}^\infty \bigg\|\bigg( \sum_{\nu=2^j-1}^{2^{j+1}-2} 2^{\nu s q} |(\varphi_\nu \widehat{f})^\vee|^q \bigg)^{1/q}  \bigg\|_{L_p(\R^d)}^q \bigg)^{1/q}.
			\end{align*}
			Secondly, if $p < q$ then we make use of the fact that $\ell_p \hookrightarrow \ell_q$ and apply Fubini's theorem in order to get
			$$
				\|f\|_{F^s_{p, q}(\R^d)} \leq  \bigg(\sum_{j=0}^\infty \bigg\| \bigg(\sum_{\nu=2^j-1}^{2^{j+1}-2} 2^{\nu s q} |(\varphi_\nu \widehat{f})^\vee|^q \bigg)^{1/q} \bigg\|_{L_p(\R^d)}^p \bigg)^{1/p}.
			$$
			Hence \eqref{new2new} holds.

			The proof of the left-hand side embedding in \eqref{cncnnc2}, i.e., $ \T^b_r F^{s}_{p,q}(\mathbb{R}^d) \hookrightarrow F^s_{p,q}(\R^d)$ follows also from \eqref{new2new} since (cf. \eqref{new1})
			\begin{equation*}
		 \bigg\| \bigg(\sum_{\nu=2^j-1}^{2^{j+1}-2} 2^{\nu s q} |(\varphi_\nu \widehat{f})^\vee|^q \bigg)^{1/q} \bigg\|_{L_p(\R^d)} \leq 2^{-j b} \, \|f\|_{\T^b_r F^{s}_{p, q}(\R^d)}, \qquad j \geq 0,
	\end{equation*}
	for $f \in \T^b_r F^{s}_{p, q}(\R^d)$.

(ii): Concerning the left-hand side embedding in \eqref{cncnnc--}, in light of \eqref{EmbBFnwnqnq} and Proposition \ref{PropPreser}, it is enough to prove that
\begin{equation*}
	T^b_r B^{s+\varepsilon}_{p, \infty}(\R^d) \hookrightarrow B^{s}_{p, \xi}(\R^d)
\end{equation*}
for some $\xi > 0$ sufficiently small.  Applying H\"older's inequality, we get
\begin{align*}
	\|f\|_{B^{s}_{p, \xi}(\R^d)} & = \bigg(\sum_{j=0}^\infty \sum_{\nu=2^j-1}^{2^{j+1}-2} (2^{\nu s} \|(\varphi_\nu \widehat{f})^\vee\|_{L_p(\R^d)})^\xi \bigg)^{1/\xi} \\
	& \lesssim \bigg(\sum_{j=0}^\infty 2^{-2^j \varepsilon \xi} \sup_{\nu = 2^j-1, \ldots, 2^{j+1}-2} \,  \{ 2^{\nu(s+\varepsilon)} \|(\varphi_\nu \widehat{f})^\vee\|_{L_p(\R^d)} \}^\xi \bigg)^{1/\xi} \\
	& \lesssim \bigg(\sum_{j=0}^\infty 2^{-2^j \varepsilon \xi} 2^{-j b \xi} \bigg)^{1/\xi} \, \|f\|_{T^b_\infty B^{s+\varepsilon}_{p, \infty}(\R^d)} \\
	& \lesssim \|f\|_{T^b_r B^{s+\varepsilon}_{p, \infty}(\R^d)}.
\end{align*}
On the other hand, in a similar fashion as above, the right-hand side embedding in \eqref{cncnnc--} would follow from the special case
$$
	B^s_{p, \infty}(\R^d) \hookrightarrow T^b_r B^{s-\varepsilon}_{p, \xi}(\R^d), \qquad \xi > 0.
$$
To prove the latter embedding, one can reason as follows:
\begin{align*}
	\|f\|_{ T^b_r B^{s-\varepsilon}_{p, \xi}(\R^d)} & =  \left(\sum_{j=0}^\infty 2^{j b r} \bigg(\sum_{\nu=2^j-1}^{2^{j+1}-2} 2^{\nu (s-\varepsilon) \xi} \|(\varphi_\nu \widehat{f})^\vee\|_{L_p(\mathbb{R}^d)}^\xi\bigg)^{r/\xi}\right)^{1/r} \\
	& \lesssim \bigg(\sum_{j=0}^\infty 2^{-2^j \varepsilon r} 2^{j b r}  \bigg)^{1/r} \,  \|f\|_{B^s_{p, \infty}(\R^d)} \asymp \|f\|_{B^s_{p, \infty}(\R^d)}.
\end{align*}

Next we turn our attention to the left-hand side embedding in \eqref{cncnnc2--}. Note that it is enough to show that
\begin{equation}\label{TbrF1}
	 \T^b_r F^{s+\varepsilon}_{p,q_0}(\mathbb{R}^d) \hookrightarrow B^s_{p,\xi}(\R^d)
\end{equation}
for some $\xi > 0$ sufficiently small; cf. \eqref{EmbBFnwnqnq}. Let $j \in \N_0$ and $\nu \in \{2^j-1, \ldots, 2^{j+1}-2\}$. Since
\begin{align*}
	 \|2^{\nu s} (\varphi_\nu \widehat{f})^\vee\|_{L_p(\R^d)} & \lesssim 2^{-2^j \varepsilon} \,  \bigg\| \bigg(\sum_{\nu=2^j-1}^{2^{j+1}-2} 2^{\nu (s + \varepsilon) q_0} |(\varphi_\nu \widehat{f})^\vee|^{q_0} \bigg)^{1/q_0} \bigg\|_{L_p(\R^d)} \\
	 & \lesssim 2^{-2^j \varepsilon} \, 2^{-j b} \,   \|f\|_{ \T^b_r F^{s+\varepsilon}_{p,q_0}(\mathbb{R}^d) },
\end{align*}
we have
\begin{align*}
	\|f\|_{B^s_{p,\xi}(\R^d)} & = \bigg(\sum_{j=0}^\infty \sum_{\nu=2^j-1}^{2^{j+1}-2} \|2^{\nu s} (\varphi_\nu \widehat{f})^\vee \|_{L_p(\R^d)}^\xi \bigg)^{1/\xi} \\
	& \lesssim  \bigg(\sum_{j=0}^\infty 2^{-2^j \varepsilon \xi} \, 2^{j (1- b \xi)} \bigg)^{1/\xi}  \,  \|f\|_{ \T^b_r F^{s+\varepsilon}_{p,q_0}(\mathbb{R}^d)} \\
	& \asymp  \|f\|_{ \T^b_r F^{s+\varepsilon}_{p,q_0}(\mathbb{R}^d)}.
\end{align*}
This completes the proof of \eqref{TbrF1}. Furthermore, similar ideas as above can be applied to show the right-hand side embedding in \eqref{cncnnc2--}. This is left to the interested reader.


			(iii): It follows immediately from Proposition \ref{PropPreser} and \eqref{EmbBFnwnqnq}.
			
			
			(iv): We only provide the proof of the left-hand side embedding of \eqref{PropElem1}. If $p \geq r$ then we can apply Minkowski's inequality so that
			\begin{align*}
				\|f\|_{\T^b_r F^{s}_{p, q}(\R^d)} & =  \bigg\| \bigg(\sum_{j=0}^\infty 2^{j b r} \bigg(\sum_{\nu=2^{j}-1}^{2^{j+1}-2} 2^{\nu s q}  |(\varphi_\nu \widehat{f})^\vee|^q   \bigg)^{r/q} \bigg)^{1/r} \bigg\|_{L_p(\R^d)} \\
				& \leq \bigg(\sum_{j=0}^\infty 2^{j b r} \bigg\| \bigg(\sum_{\nu=2^{j}-1}^{2^{j+1}-2} 2^{\nu s q}  |(\varphi_\nu \widehat{f})^\vee|^q   \bigg)^{1/q} \bigg\|_{L_p(\R^d)}^r \bigg)^{1/r} = \|f\|_{T^b_r F^{s}_{p, q}(\R^d)}.
			\end{align*}
			On the other hand, if $r \geq p$ then
			\begin{align*}
				\|f\|_{\T^b_r F^{s}_{p, q}(\R^d)} &\leq \bigg\| \bigg(\sum_{j=0}^\infty 2^{j b p} \bigg(\sum_{\nu=2^{j}-1}^{2^{j+1}-2} 2^{\nu s q}  |(\varphi_\nu \widehat{f})^\vee|^q   \bigg)^{p/q} \bigg)^{1/p} \bigg\|_{L_p(\R^d)} \\
				&= \bigg(\sum_{j=0}^\infty 2^{j b p} \bigg\| \bigg(\sum_{\nu=2^{j}-1}^{2^{j+1}-2} 2^{\nu s q}  |(\varphi_\nu \widehat{f})^\vee|^q   \bigg)^{1/q}  \bigg\|_{L_p(\R^d)}^p \bigg)^{1/p} = \|f\|_{T^b_p F^{s}_{p, q}(\R^d)}.
			\end{align*}

			Item (v) is an immediate consequences of Fubini's theorem.
		
\end{proof}

\begin{rem}
 	A detailed study of embedding theorems between $T^b_r A^{s}_{p, q}(\R^d)$-spaces will be postponed until Sections \ref{Section10}--\ref{SectionFJ}. In particular, the optimality of \eqref{PropElem1*} will be addressed in Corollary \ref{CorTruncBTLFixedps} below, namely,
	$$
		T^b_r B^{s}_{p, u}(\R^d) \hookrightarrow T^b_r F^{s}_{p, q}(\R^d) \iff u \leq \min \{p, q\}
	$$
	and
	$$
		 T^b_r F^{s}_{p, q}(\R^d) \hookrightarrow T^b_r B^{s}_{p, u}(\R^d) \iff u \geq \max\{p, q\}.
	$$
\end{rem}

\newpage

\section{Characterizations by interpolation}\label{SectionLimIntGeneral}


\subsection{Limiting interpolation. General facts. }\label{Section:LimitingInterpolation}

The goal of this section is to show that the spaces $T^b_r A^{s}_{p,q}(\mathbb{R}^d), \\ A \in \{B, F\}$, and $\T^b_r F^s_{p, q}(\R^d)$ can be generated as interpolation spaces of the classical spaces $B^s_{p, q}(\R^d)$ and $F^s_{p, q}(\R^d)$ (cf. Theorems \ref{TheoremInterpolation}, \ref{TheoremInterpolationF} and \ref{TheoremInterpolationFrakF} below). As a consequence, we shall conclude not only that these new scales are closed under interpolation (cf. Theorems \ref{TheoremInterpolation2} and \ref{TheoremInterpolation2pspapsa} below), but also new interpolation formulae involving $\mathbf{B}^{0, b}_{p, q}(\R^d)$ and $\text{Lip}^{s, b}_{p, q}(\R^d)$ (cf. Corollary \ref{CorIntBLip}). To get all these results we shall apply the machinery provided by limiting interpolation. Next we briefly recall the construction of this method.

Let $(A_0,A_1)$ be a couple of quasi-Banach spaces. Given $t >0$ and $f \in A_0+A_1$, the \emph{Peetre's $K$-functional} is
defined by\index{\bigskip\textbf{Functionals and functions}!$K(t, f)$}\label{K}
\begin{equation}\label{DefPeetreKFunct}
    K(t,f) := K(t,f; A_0,A_1) := \inf \{\|f_0\|_{A_0} + t \|f_1\|_{A_1} : f = f_0 + f_1, \, f_j \in
    A_j\}.
\end{equation}
For $\theta \in (0,1), \, q \in (0,\infty]$ and $b \in \R$, the interpolation space $(A_0,A_1)_{\theta, q, b}$ is formed by all those $f \in A_0 + A_1$ such that\index{\bigskip\textbf{Spaces}!$(A_0, A_1)_{\theta, q, b}$}\label{REAL}
\begin{equation}\label{ClassicInt2}
	\|f\|_{(A_0,A_1)_{\theta,q, b}} := \bigg(\int_0^\infty (t^{-\theta} (1+ |\log t|)^b K(t,f))^q \frac{dt}{t} \bigg)^{1/q} < \infty
\end{equation}
(with the usual modification if $q=\infty$). See \cite{Gustavsson}, \cite{DeVore} and \cite{GogatishviliOpicTrebels}. In the special case $b=0$ one recovers the classical interpolation space $(A_0, A_1)_{\theta, q}$, i.e.,
\begin{equation}\label{ClassicInt}
	\|f\|_{(A_0,A_1)_{\theta,q}} := \bigg(\int_0^\infty (t^{-\theta} K(t,f))^q \frac{dt}{t} \bigg)^{1/q}
\end{equation}
 cf. \cite{BerghLofstrom}, \cite{BennettSharpley} and \cite{Triebel}. Since $K(t, f; A_0, A_1) = t K(t^{-1}, f; A_1, A_0)$, one has
 \begin{equation}\label{Sym}
 	(A_0, A_1)_{\theta, q, b} = (A_1, A_0)_{1-\theta, q, b}
 \end{equation}
 with equality of quasi-norms.

The introduction of limiting interpolation spaces requires some non-trivial modifications of the quasi-norms \eqref{ClassicInt2} and \eqref{ClassicInt}. Indeed, note that setting $\theta=0$ or $\theta=1$ we only obtain in general the trivial space $\{0\}$. To overcome this issue we make the elementary observation that, working with ordered couples (i.e., $A_1 \hookrightarrow A_0$),
\begin{equation}\label{Int2}
	\|f\|_{(A_0,A_1)_{\theta,q}}  \asymp  \bigg(\int_0^1 (t^{-\theta} K(t,f))^q \frac{dt}{t} \bigg)^{1/q}, \qquad \theta \in (0,1).
	\end{equation}
This follows from the fact $K(t,f) \asymp \|f\|_{A_0}$ for $t > 1$. Accordingly, one may consider the integral given in the right-hand side of \eqref{Int2} to introduce the space $(A_0,A_1)_{\theta,q}$. This together with the additional refinement given by logarithmic weights allow us to define limiting interpolation spaces. More precisely, given an ordered couple of quasi-Banach spaces $(A_0, A_1)$, for $\theta \in \{0, 1\}, q \in (0,\infty]$ and $b \in \R$, the \emph{limiting interpolation space} $(A_0,A_1)_{(\theta,b),q}$ is the set of all those $f \in A_0$ for which\index{\bigskip\textbf{Spaces}!$(A_0, A_1)_{(\theta, b), q}$}\label{REALLIM}
\begin{equation}\label{DefLimInterpolation}
	\|f\|_{(A_0,A_1)_{(\theta,b),q}} = \bigg(\int_0^1 (t^{-\theta}(1-\log t)^{b} K(t,f))^q \frac{dt}{t} \bigg)^{1/q}
\end{equation}
is finite (with the usual modification if $q=\infty$). See \cite{JawerthMilman, GogatishviliOpicTrebels, CobosDominguez, Astashkin} and the references given there. It is not hard to check that if $b < -1/q \, (b \leq 0 \text{ if } q=\infty)$ then $(A_0,A_1)_{(0,b),q} = A_0$ and if $b \geq -1/q \, (b > 0 \text{ if } q=\infty)$ then $(A_0,A_1)_{(1,b),q} = \{0\}$. Accordingly, we may assume without loss of generality that $b \geq -1/q \, (b > 0 \text{ if } q=\infty)$ if $\theta = 0$ and $b < -1/q \, (b \leq 0 \text{ if } q=\infty)$ if $\theta=1$. Under these assumptions, the $((\theta,b),q)$-method produces intermediate spaces between $A_0$ and $A_1$, that is,
\begin{equation}\label{EmbeddingsLimInt}
	A_1 \hookrightarrow (A_0,A_1)_{(\theta,b),q} \hookrightarrow A_0,
\end{equation}
and it satisfies the interpolation property for bounded linear operators. Furthermore, the left-hand side embedding in \eqref{EmbeddingsLimInt} is dense whenever $q < \infty$. We remark that it is possible to introduce limiting interpolation spaces for general couples of quasi-Banach spaces (not necessarily ordered) with the help of adequate slowly varying functions, however for the purposes of this paper is enough to restrict ourselves to the ordered case.

The following relations between classical and limiting interpolation will be useful; cf.  \cite[Corollaries 7.8 and 7.11]{EvansOpicPick} and \cite[Lemma 2.2]{DominguezTikhonov}.

\begin{lem}\label{LemmaReiteration}
	Let $\theta \in (0,1), \, q, r, r_0, r_1 \in (0,\infty]$ and $\xi \in \R$.
	\begin{enumerate}[\upshape(i)]
	\item If $b \geq -1/r$ then
	\begin{equation*}
		(A_0, (A_0,A_1)_{\theta,q, \xi})_{(0,b),r} = (A_0,A_1)_{(0,b),r}.
	\end{equation*}
	\item If $b < -1/r$ then
	\begin{equation*}
		((A_0,A_1)_{\theta,q, \xi}, A_1)_{(1,b),r} = (A_0,A_1)_{(1,b),r}.
	\end{equation*}
	\item If $0 < b_0 + 1/r_0 < b_1 + 1/r_1$ then
	$$
		((A_0,A_1)_{(0,b_0), r_0}, (A_0,A_1)_{(0,b_1),r_1})_{\theta,q} = (A_0,A_1)_{(0, \gamma), q}
	$$
	where $\gamma = (1-\theta) (b_0 + 1/r_0) + \theta (b_1+1/r_1) -1/q$.
	\item If $b_1 + 1/r_1 < b_0 + 1/r_0 < 0$ then
	$$
		((A_0,A_1)_{(1,b_0), r_0}, (A_0,A_1)_{(1,b_1),r_1})_{\theta,q} = (A_0,A_1)_{(1, \gamma), q}
	$$
	where $\gamma = (1-\theta) (b_0 + 1/r_0) + \theta (b_1+1/r_1) -1/q$.
		\item If $b > -1/r$ then
		$$
			((A_0, A_1)_{(0, b), r}, A_1)_{\theta, q} = (A_0, A_1)_{\theta, q, (1-\theta) (b+1/r)}.
		$$
		\item If $b < -1/r$ then
		$$
		(A_0, (A_0, A_1)_{(1, b), r})_{\theta, q} = (A_0, A_1)_{\theta, q, \theta (b+ 1/r)}.
		$$
	\end{enumerate}
\end{lem}


\subsection{$T^b_r B^{s}_{p, q}(\R^d)$ spaces via interpolation}

\begin{thm}[Characterization of $T^b_r B^{s}_{p, q}(\R^d)$ via limiting interpolation]\label{TheoremInterpolation}\label{Thm4.2}
	 Let $A \in \{B, F\}$. Let $p, q, q_0, r \in (0, \infty] \, (p < \infty \text{ if } A = F), s, s_0 \in \R,  s \neq s_0$, and $b \in \R, \, b \neq 0$.
	 \begin{enumerate}[\upshape(i)]
	 \item If $s_0 > s$ and $b > 0$ then
	 \begin{equation}\label{Interpolation1B}
	 	T^b_r B^{s}_{p,q}(\mathbb{R}^d) = (B^s_{p,q}(\R^d), A^{s_0}_{p,q_0}(\R^d))_{(0,b-1/r),r}.
	 \end{equation}
	 \item If $s_0 < s$ and $b < 0$ then
	  \begin{equation}\label{Interpolation1Bnew}
	 T^b_r B^{s}_{p,q}(\mathbb{R}^d) = (A^{s_0}_{p,q_0}(\R^d), B^s_{p,q}(\R^d))_{(1,b-1/r),r}.
	 \end{equation}
	 \end{enumerate}
\end{thm}

\begin{rem}
	Note that the assumption $s_0 > s$ in (i) implies that $ A^{s_0}_{p,q_0}(\R^d) \hookrightarrow B^s_{p,q}(\R^d)$ (cf. \cite[Section 2.3.2, page 47]{Triebel83}). A similar comment applies to (ii).
\end{rem}

\begin{rem}
	The case $b=0$ in Theorem \ref{TheoremInterpolation} is more delicate since the outcome is not the expected spaces $T_r B^{s}_{p, q}(\R^d)$ but $T^*_r B^{s}_{p, q}(\R^d)$ (cf. \eqref{33}). This will be treated in detail in  Remark \ref{Remark4.4} below.
\end{rem}

\begin{proof}[Proof of Theorem \ref{TheoremInterpolation}]
	 The proof relies on the retraction method (see
    \cite[Section 1.2.4]{Triebel} for details on this method) and limiting interpolation techniques. Assume first $p \in [1,\infty]$ and $A=B$. In this case, it is well known that
    $B^s_{p,q}(\mathbb{R}^d)$ is a retract of $\ell^s_q(L_p(\mathbb{R}^d))$ and
    $B^{s_0}_{p,q}(\mathbb{R}^d)$ is a retract of
    $\ell^{s_0}_q(L_p(\mathbb{R}^d))$ with the corresponding co-retraction operator being
    \begin{equation}\label{Retraction}
    \mathfrak{J} f := ((\varphi_j \widehat{f})^\vee)_{j=0}^\infty.
    \end{equation}
     Furthermore, elementary computations lead to
    \begin{equation}\label{Kfunct}
        K(t,x; \ell^s_q(L_p(\mathbb{R}^d)),
        \ell^{s_0}_q(L_p(\mathbb{R}^d))) \asymp \bigg(
        \sum_{\nu=0}^\infty [\min (2^{\nu s}, t 2^{\nu s_0}) \|x_\nu \|_{L_p(\mathbb{R}^d)}]^q \bigg)^{1/q}
    \end{equation}
    for $t > 0$ and $x=(x_\nu)_{\nu=0}^\infty \in \ell^s_q(L_p(\mathbb{R}^d)) +
        \ell^{s_0}_q(L_p(\mathbb{R}^d))$.

        (i): Assume $s_0 > s$.  According to \eqref{Kfunct}, for every $j \geq 0$,
    \begin{align}
        K(2^{-j (s_0-s)},x; \ell^s_q(L_p(\mathbb{R}^d)),
        \ell^{s_0}_q(L_p(\mathbb{R}^d))) \nonumber \\
        &  \hspace{-5cm}\asymp
        2^{-j(s_0-s)} \bigg(\sum_{\nu=0}^j [2^{\nu s_0} \|x_\nu\|_{L_p(\mathbb{R}^d)}]^q \bigg)^{1/q} + \bigg( \sum_{\nu = j+1}^\infty [2^{\nu s}
        \|x_\nu \|_{L_p(\mathbb{R}^d)}]^q \bigg)^{1/q}. \label{412}
    \end{align}
    Therefore,
    \begin{align}
        \|x\|_{(\ell^s_q(L_p(\mathbb{R}^d)),
        \ell^{s_0}_q(L_p(\mathbb{R}^d)))_{(0,b-1/r),r}} & \nonumber \\
        & \hspace{-4cm}\asymp \left(\sum_{j=0}^\infty [(1+j)^{b-1/r}  K(2^{-j (s_0-s)},x;\ell^s_q(L_p(\mathbb{R}^d)),
        \ell^{s_0}_q(L_p(\mathbb{R}^d)))]^r\right)^{1/r}  \nonumber\\
        & \hspace{-4cm} \asymp \left(\sum_{j=0}^\infty (1+j)^{(b-1/r) r}  2^{-j(s_0-s) r} \bigg(\sum_{\nu=0}^j [2^{\nu s_0} \|x_\nu\|_{L_p(\mathbb{R}^d)}]^q \bigg)^{r/q} \right)^{1/r} \label{22}\\
        & \hspace{-3cm} + \left(\sum_{j=0}^\infty (1+j)^{(b-1/r) r} \bigg( \sum_{\nu = j}^\infty [2^{\nu s}
        \|x_\nu \|_{L_p(\mathbb{R}^d)}]^q \bigg)^{r/q}  \right)^{1/r} \nonumber\\
        & \hspace{-4cm} =: I + II. \nonumber
    \end{align}
   Furthermore,
    \begin{equation}\label{23}
        I \lesssim II.
    \end{equation}
    Indeed, it follows from Hardy's inequality \eqref{H1} that
    \begin{equation*}
        I  \lesssim \left(\sum_{j=0}^\infty \left[(1+j)^{b-1/r}   2^{j s}
        \|x_j \|_{L_p(\mathbb{R}^d)}\right]^r\right)^{1/r} \leq II.
    \end{equation*}

    Combining \eqref{22} and \eqref{23}, we obtain that
    \begin{equation}\label{24}
       \|x\|_{(\ell^s_q(L_p(\mathbb{R}^d)),
        \ell^{s_0}_q(L_p(\mathbb{R}^d)))_{(0,b-1/r),r}}  \asymp \left(\sum_{j=0}^\infty (1+j)^{(b-1/r) r} \bigg( \sum_{\nu = j}^\infty [2^{\nu s}
        \|x_\nu \|_{L_p(\mathbb{R}^d)}]^q \bigg)^{r/q}  \right)^{1/r}.
    \end{equation}
    In virtue of the retraction method and \eqref{24} we get
    \begin{align}
    	\|f\|_{ (B^s_{p,q}(\R^d), B^{s_0}_{p,q}(\R^d))_{(0,b-1/r),r}} &\asymp \|\mathfrak{J} f\|_{(\ell^s_q(L_p(\mathbb{R}^d)),
        \ell^{s_0}_q(L_p(\mathbb{R}^d)))_{(0,b-1/r),r}} \nonumber\\
        &\hspace{-3cm} \asymp \left(\sum_{j=0}^\infty (1+j)^{(b-1/r) r} \bigg( \sum_{\nu = j}^\infty [2^{\nu s}
        \| (\varphi_\nu \widehat{f})^\vee \|_{L_p(\mathbb{R}^d)}]^q \bigg)^{r/q}  \right)^{1/r}. \label{new2}
    \end{align}

      Assume further $b > 0$. We claim
    \begin{equation}\label{new3}
     \left(\sum_{j=0}^\infty (1+j)^{(b-1/r) r} \bigg( \sum_{\nu = j}^\infty [2^{\nu s}
        \| (\varphi_\nu \widehat{f})^\vee \|_{L_p(\mathbb{R}^d)}]^q \bigg)^{r/q}  \right)^{1/r} \asymp \|f\|_{T^b_r B^{s}_{p, q}(\R^d)}.
    \end{equation}
    Assuming momentarily the validity of the previous assertion, we conclude by \eqref{new2} that (i) holds for $A=B, \, p \in [1,\infty]$ and $q_0=q \in (0,\infty]$. Next we show \eqref{new3}. By monotonicity properties and Hardy's inequality (note that $b > 0$; cf. \eqref{H2})
    \begin{align}
    	  \left(\sum_{j=0}^\infty (1+j)^{(b-1/r) r} \bigg( \sum_{\nu = j}^\infty [2^{\nu s}
        \| (\varphi_\nu \widehat{f})^\vee \|_{L_p(\mathbb{R}^d)}]^q \bigg)^{r/q}  \right)^{1/r} & \asymp \nonumber  \\
        & \hspace{-6cm} \left(\sum_{j=0}^\infty 2^{j b r}  \bigg( \sum_{\nu = 2^j}^\infty [2^{\nu s}
        \| (\varphi_\nu \widehat{f})^\vee \|_{L_p(\mathbb{R}^d)}]^q \bigg)^{r/q} \right)^{1/r} \nonumber  \\
        & \hspace{-6cm} \asymp \left(\sum_{j=0}^\infty 2^{j b r}  \bigg( \sum_{\nu = 2^j-1}^{2^{j+1}-2} [2^{\nu s}
        \| (\varphi_\nu \widehat{f})^\vee \|_{L_p(\mathbb{R}^d)}]^q \bigg)^{r/q} \right)^{1/r}  = \|f\|_{T^b_r B^{s}_{p, q}(\R^d)}.\label{new4}
    \end{align}

    The non-diagonal case $q_0 \neq q$ can reduced to the previous case via the reiteration formula given in Lemma \ref{LemmaReiteration}(i). Indeed, let $s_1$ be such that $s_0 \in (s,s_1)$. By the well-known interpolation properties of Besov spaces (cf. \cite[Theorem 6.4.5]{BerghLofstrom}, \cite[2.4.1]{Triebel} and \cite[2.4.2]{Triebel83})
    \begin{equation}\label{gdgsadg}
    	B^{s_0}_{p,q_0}(\R^d) = (B^s_{p,q}(\R^d), B^{s_1}_{p,q}(\R^d))_{\theta,q_0}
    \end{equation}
    where $\theta \in (0,1)$ is given by $s_0 = (1-\theta) s + \theta s_1$. Applying now Lemma \ref{LemmaReiteration}(i) we have
    \begin{align*}
    	 (B^s_{p,q}(\R^d), B^{s_0}_{p,q_0}(\R^d))_{(0,b-1/r),r} & =  (B^s_{p,q}(\R^d), (B^s_{p,q}(\R^d), B^{s_1}_{p,q}(\R^d))_{\theta,q_0})_{(0,b-1/r),r} \\
	 & = (B^s_{p,q}(\R^d), B^{s_1}_{p,q}(\R^d))_{(0,b-1/r),r} = T^b_r B^{s}_{p,q}(\R^d).
    \end{align*}
   This completes the proof of (i) in the case $A=B$ and $p \in [1,\infty]$. Furthermore, the case $p \in (0,1)$ can be done in a similar fashion but now replacing $L_p(\R^d)$ by the Hardy space $h_p(\R^d)$ so that the operator $\mathfrak{J}$ (cf. \eqref{Retraction}) is well defined and the retraction method can be also applied.

    The formula (i) with $A=F$ follows immediately from the case $A=B$ and the well-known embeddings
    \begin{equation}\label{RelationsBF}
    	B^{s_0}_{p,\min\{p,q_0\}}(\R^d) \hookrightarrow F^{s_0}_{p,q_0}(\R^d) \hookrightarrow B^{s_0}_{p,\max\{p,q_0\}}(\R^d)
    \end{equation}
    (see, e.g., \cite[(9), page 47]{Triebel83}).

    (ii): Assume $s_0 < s$ and $b < 0$.  According to \eqref{Kfunct}, for every $j \geq 0$,
    \begin{align*}
        K(2^{-j(s-s_0)}, x; \ell^{s_0}_{q}(L_p(\R^d)), \ell^s_q(L_p(\R^d))) & \\
        &  \hspace{-5cm}\asymp
        2^{-j(s-s_0)} \bigg(\sum_{\nu=0}^j [2^{\nu s} \|x_\nu\|_{L_p(\mathbb{R}^d)}]^q \bigg)^{1/q} + \bigg( \sum_{\nu = j+1}^\infty [2^{\nu s_0}
        \|x_\nu \|_{L_p(\mathbb{R}^d)}]^q \bigg)^{1/q}.
    \end{align*}
    Thus, by Hardy's inequality \eqref{H2}, we have
    \begin{align*}
        \|x\|_{(\ell^{s_0}_q(L_p(\mathbb{R}^d)),
        \ell^{s}_q(L_p(\mathbb{R}^d)))_{(1,b-1/r),r}} &  \\
        & \hspace{-4cm}\asymp \left(\sum_{j=0}^\infty [2^{j(s-s_0)}(1+j)^{b-1/r}  K(2^{-j (s-s_0)},x;\ell^{s_0}_q(L_p(\mathbb{R}^d)),
        \ell^{s}_q(L_p(\mathbb{R}^d)))]^r\right)^{1/r}  \\
        & \hspace{-4cm} \asymp \left(\sum_{j=0}^\infty (1+j)^{(b-1/r) r}   \bigg(\sum_{\nu=0}^j [2^{\nu s} \|x_\nu\|_{L_p(\mathbb{R}^d)}]^q \bigg)^{r/q} \right)^{1/r}\\
        & \hspace{-3cm} + \left(\sum_{j=0}^\infty 2^{j(s-s_0) r}(1+j)^{(b-1/r) r} \bigg( \sum_{\nu = j}^\infty [2^{\nu s_0}
        \|x_\nu \|_{L_p(\mathbb{R}^d)}]^q \bigg)^{r/q}  \right)^{1/r} \\
        & \hspace{-4cm} \asymp \left(\sum_{j=0}^\infty (1+j)^{(b-1/r) r}   \bigg(\sum_{\nu=0}^j [2^{\nu s} \|x_\nu\|_{L_p(\mathbb{R}^d)}]^q \bigg)^{r/q} \right)^{1/r}.
    \end{align*}
    Invoking now the retraction method we establish that
    $$
    	\|f\|_{ (B^{s_0}_{p,q}(\R^d), B^{s}_{p,q}(\R^d))_{(1,b-1/r),r}}
        \asymp \left(\sum_{j=0}^\infty (1+j)^{(b-1/r) r} \bigg( \sum_{\nu = 0}^j [2^{\nu s}
        \| (\varphi_\nu \widehat{f})^\vee \|_{L_p(\mathbb{R}^d)}]^q \bigg)^{r/q}  \right)^{1/r}
    $$
    and thus a similar argument as in \eqref{new4} now with $b < 0$ yields that
    $$
    	\|f\|_{ (B^{s_0}_{p,q}(\R^d), B^{s}_{p,q}(\R^d))_{(1,b-1/r),r}} \asymp \|f\|_{T^b_r B^{s}_{p, q}(\R^d)}
    $$
    or equivalently,
    \begin{equation}\label{210}
    	(B^{s_0}_{p,q}(\R^d), B^{s}_{p,q}(\R^d))_{(1,b-1/r),r} = T^b_r B^{s}_{p,q}(\mathbb{R}^d).
    \end{equation}
    This covers (ii) with $A=B, p \in [1,\infty]$ and $q_0=q \in (0,\infty]$.

    Assume $q_0 \in (0, \infty], \, q_0 \neq q$. Given $s_0 < s$ we take $s_1$ with $s_1 < s_0$ and define $\theta \in (0,1)$ such that $s_0= (1-\theta) s_1 + \theta s$. Since $B^{s_0}_{p,q_0}(\R^d) = (B^{s_1}_{p,q}(\R^d), B^s_{p,q}(\R^d))_{\theta,q_0}$ (cf. \eqref{gdgsadg}) we can apply \eqref{210} and Lemma \ref{LemmaReiteration}(ii) to get
    \begin{align*}
    	T^b_r B^{s}_{p,q}(\mathbb{R}^d) &= (B^{s_1}_{p,q}(\R^d), B^{s}_{p,q}(\R^d))_{(1,b-1/r),r}  \\
	& =  ( (B^{s_1}_{p,q}(\R^d), B^s_{p,q}(\R^d))_{\theta,q_0}, B^{s}_{p,q}(\R^d))_{(1,b-1/r),r} \\
	& = (B^{s_0}_{p,q_0}(\R^d), B^{s}_{p,q}(\R^d))_{(1,b-1/r),r},
    \end{align*}
    that is, (ii) holds with $A=B$ and $p \in [1,\infty]$. Furthermore, the modification of this method obtained by replacing $L_p(\R^d)$ by $h_p(\R^d)$ enables us to cover the full range $p \in (0,\infty]$. Finally, the case $A=F$ follows from the $B$-case via the embeddings \eqref{RelationsBF}.
\end{proof}

\begin{rem}\label{Remark4.4}
	Let us consider the limiting case $b = 0$ in \eqref{Interpolation1B}. Namely (cf. \eqref{new2}) we have
	\begin{equation}\label{4.4ururu}
	(B^s_{p,q}(\R^d), B^{s_0}_{p,q}(\R^d))_{(0,-1/r),r} = T^*_r B^s_{p, q}(\R^d),
	\end{equation}
	see \eqref{33}. In the special case $r=q$, we have (cf. \eqref{newqq})
	$$
	(B^s_{p,q}(\R^d), B^{s_0}_{p,q}(\R^d))_{(0,-1/q),q} = B^{s, 0, 1/q}_{p,q}(\R^d).
	$$	
\end{rem}

\begin{rem}
	The limiting cases $p=1, \infty$  in Theorem \ref{TheoremInterpolation} deserve special attention. Unlike in the case $p \in (1, \infty)$, the Sobolev spaces $W^k_1(\R^d)$ and $W^k_\infty(\R^d)$, $k \in \N_0$, do not fit into the scale of Triebel--Lizorkin spaces $F^s_{p, q}(\R^d)$. However, the interpolation formulae \eqref{Interpolation1B} and \eqref{Interpolation1Bnew} still are valid with $W^k_p(\R^d), \, p=1, \infty$. More precisely, let $p= 1, \infty$ and $q, r \in (0, \infty]$, then
	\begin{equation}\label{IntFormp=1}
		T^b_r B^{s}_{p, q}(\R^d) = (B^s_{p, q}(\R^d), W^k_p(\R^d))_{(0, b-1/r), r}, \qquad s < k, \qquad b > 0
	\end{equation}
	and
	\begin{equation}\label{IntFormp=infty}
		T^b_r B^{s}_{p, q}(\R^d) = (W^k_p(\R^d), B^s_{p, q}(\R^d))_{(1, b-1/r), r}, \qquad s > k, \qquad b < 0.
	\end{equation}
	To prove \eqref{IntFormp=1}, we make use of the well-known embeddings (cf. \cite[Section 2.5.7]{Triebel83})
	$$
	B^k_{p, 1}(\R^d)	\hookrightarrow W^k_p(\R^d) \hookrightarrow B^k_{p, \infty}(\R^d).
	$$
	By \eqref{Interpolation1B}, we derive
	\begin{align*}
	T^b_r B^{s}_{p, q}(\R^d) &= (B^s_{p, q}(\R^d), B^k_{p, 1}(\R^d))_{(0, b-1/r), r} \hookrightarrow (B^s_{p, q}(\R^d), W^k_p(\R^d))_{(0, b-1/r), r} \\
	& \hookrightarrow 	(B^s_{p, q}(\R^d), B^k_{p, \infty}(\R^d))_{(0, b-1/r), r} = T^b_r B^{s}_{p, q}(\R^d).
	\end{align*}
	The proof of \eqref{IntFormp=infty} can be obtained similarly but now invoking \eqref{Interpolation1Bnew}.
\end{rem}

\subsection{$T^b_r B^s_{p, q}(\R^d)$ spaces are closed under interpolation.}
Our next result proves that real interpolation between spaces $T^b_r B^{s}_{p,q}(\R^d)$ with $s, p$ and $q$ fixed produces a space of the same type.

\begin{thm}\label{TheoremInterpolation2}
	Let $0 < p, q, r, r_0, r_1 \leq \infty, -\infty < s, b_0, b_1 < \infty$ and $0 < \theta < 1$. Assume  further $b_0  \neq b_1$ and $b_0 b_1 > 0$ and let $b = (1-\theta) b_0 + \theta b_1$. Then
	\begin{equation*}
		(T^{b_0}_{r_0} B^{s}_{p,q}(\R^d), T^{b_1}_{r_1} B^{s}_{p,q}(\R^d))_{\theta,r} = T^{b}_r B^{s}_{p,q}(\R^d).
	\end{equation*}
\end{thm}

\begin{proof}
	The assumption $b_0 b_1 > 0$ means that both $b_0$ and $b_1$ are either positive or negative. Assume first that $0 < b_0 < b_1$. Let $s_0 > s$. According to Theorem \ref{TheoremInterpolation}(i),
	$$
		T^{b_i}_{r_i} B^{s}_{p,q}(\mathbb{R}^d) = (B^s_{p,q}(\R^d), B^{s_0}_{p,q}(\R^d))_{(0,b_i-1/r_i),r_i}, \qquad i = 0,1.
	$$
	Let  $\gamma := b -1/r$. In light of Lemma \ref{LemmaReiteration}(iii), one has
	\begin{align*}
		(T^{b_0}_{r_0} B^{s}_{p,q}(\R^d), T^{b_1}_{r_1} B^{s}_{p,q}(\R^d))_{\theta,r} & = \\
		& \hspace{-3.5cm} ((B^s_{p,q}(\R^d), B^{s_0}_{p,q}(\R^d))_{(0,b_0-1/r_0),r_0}, (B^s_{p,q}(\R^d), B^{s_0}_{p,q}(\R^d))_{(0,b_1-1/r_1),r_1})_{\theta,r} \\
		& \hspace{-3.5cm} = (B^s_{p,q}(\R^d), B^{s_0}_{p,q}(\R^d))_{(0,\gamma),r} = T^b_r B^{s}_{p,q}(\R^d),
	\end{align*}
	where the last step follows again from Theorem \ref{TheoremInterpolation}(i).
	
	The case $b_0  > b_1 > 0$ can be reduced to the previous case. More precisely, basic properties of the real interpolation method (cf. \eqref{Sym}) yield that
	$$
		(T^{b_0}_{r_0} B^{s}_{p,q}(\R^d), T^{b_1}_{r_1} B^{s}_{p,q}(\R^d))_{\theta,r} = (T^{b_1}_{r_1}B^{s}_{p,q}(\R^d), T^{b_0}_{r_0}B^{s}_{p,q}(\R^d))_{1-\theta,r}  =  T^b_r B^{s}_{p,q}(\R^d).
	$$
	This completes the proof under $b_0, b_1 > 0$. The situation $b_0, b_1 < 0$ can be settled following similar ideas but now employing Theorem \ref{TheoremInterpolation}(ii) and Lemma \ref{LemmaReiteration}(iv).
	
\end{proof}

In order to show that  $T^b_r F^s_{p, q}(\R^d)$ spaces are closed under interpolation, we will need the results from the next section.

\subsection{$T^b_r F^{s}_{p, q}(\R^d)$ spaces via interpolation}
The goal of this section is to show that the scale $\{T^b_r F^{s}_{p,q}(\mathbb{R}^d)\}$ can be generated as limiting interpolation involving the classical Triebel--Lizorkin spaces $F^s_{p, q}(\R^d)$. To be more precise, we establish the following

\begin{thm}[Characterization of $T^b_r F^{s}_{p, q}(\R^d)$ via limiting interpolation]\label{TheoremInterpolationF}
	 Let $A \in \{B, F\}$. Let $p \in (1, \infty), q, q_0, r \in (0, \infty], s, s_0 \in \R,  s \neq s_0$, and $b \in \R \backslash \{0\}$.
	 \begin{enumerate}[\upshape(i)]
	 \item If $s_0 > s$ and $b > 0$ then
	 \begin{equation}\label{TheoremInterpolationFDisplay1}
	 	T^b_r F^{s}_{p,q}(\mathbb{R}^d) = (F^s_{p,q}(\R^d), A^{s_0}_{p,q_0}(\R^d))_{(0,b-1/r),r}.
	 \end{equation}
	 \item If $s_0 < s$ and $b < 0$ then
	 \begin{equation}\label{TheoremInterpolationFDisplay1new}
	 T^b_r F^{s}_{p,q}(\mathbb{R}^d) = (A^{s_0}_{p,q_0}(\R^d), F^s_{p,q}(\R^d))_{(1,b-1/r),r}.
	 \end{equation}
	 \end{enumerate}
\end{thm}

The proof of Theorem \ref{TheoremInterpolationF} will rely on two interpolation lemmas, which are of independent interest.

 \begin{lem}
	Let $(A_0, A_1)$ be a quasi-Banach pair and let $0 < p < \infty$. Then
	 \begin{equation}\label{ProofIntF1}
 	K(t, f; L_p(\R^d; A_0), L_p(\R^d; A_1)) \asymp \bigg(\int_{\R^d} K(t, f(x); A_0, A_1)^p \, dx \bigg)^{1/p}
 \end{equation}
 for $t > 0$ and $f \in L_p(\R^d; A_0) + L_p(\R^d; A_1)$.
\end{lem}

 \begin{proof}
 	 This result was already stated without proof in \cite[p. 218]{Persson}. For the sake of completeness, next we provide details of the proof.
	
	 Without loss of generality, we may assume that $f(x) = \sum_{j=1}^N a_j \chi_{\Omega_j}$ where $a_j \in A_0 \cap A_1$, the $\Omega_j$'s are measurable sets in $\R^d$ of finite measure with $\Omega_i \cap \Omega_j = \emptyset$ if $i \neq j$ and $\chi_{\Omega_j}$ are the related characteristic functions.
This assumption allows us to avoid unnecessary  complications related to  measurability. Then
 \begin{align*}
 	K(t, f; L_p(\R^d; A_0), L_p(\R^d; A_1))^p &\asymp \inf_{f = f_0 + f_1} (\|f_0\|_{L_p(\R^d; A_0)}^p + t^p \|f_1\|^p_{L_p(\R^d; A_1)}) \\
	& =  \inf_{f = f_0 + f_1}  \int_{\R^d} (\|f_0(x)\|_{A_0}^p + t^p \|f_1(x)\|_{A_1}^p) \, dx.
 \end{align*}
 Assume $f = f_0 + f_1$ with $f_i \in L_p(\R^d; A_i), \, i= 0, 1$. In particular, $f(x) = f_0(x) + f_1(x)$ which implies
\begin{align*}
	K(t, f(x); A_0, A_1)^p &\asymp \inf_{f(x) = a_0 + a_1} (\|a_0\|_{A_0}^p + t^p \|a_1\|_{A_1}^p) \\
	& \leq \|f_0(x)\|_{A_0}^p + t^p \|f_1(x)\|_{A_1}^p
\end{align*}
and integrating this inequality
$$
	\int_{\R^d} K(t, f(x); A_0, A_1)^p \, dx \lesssim \int_{\R^d} (\|f_0(x)\|_{A_0}^p + t^p \|f_1(x)\|_{A_1}^p) \, dx.
$$
Taking now the infimum over all possible decompositions $f=f_0 + f_1$, we arrive at
$$
\int_{\R^d} K(t, f(x); A_0, A_1)^p \, dx \lesssim K(t, f; L_p(\R^d; A_0), L_p(\R^d; A_1))^p.
$$

Conversely, given $x \in \R^d$, consider any decomposition $f(x) = f_0(x) + f_1(x)$ with $f_i(x) \in A_i, \, i=0, 1$, such that
$$
	\|f_0(x)\|_{A_0}^p + t^p \|f_1(x)\|_{A_1}^p \lesssim K(t, f(x); A_0, A_1)^p.
$$
 Integrating over all $x \in \R^d$, we derive
 $$
 	\int_{\R^d} (\|f_0(x)\|_{A_0}^p + t^p \|f_1(x)\|_{A_1}^p) \, dx \lesssim \int_{\R^d} K(t, f(x); A_0, A_1)^p \, dx.
 $$
 In particular,
 $$
 	K(t, f; L_p(\R^d; A_0), L_p(\R^d; A_1))^p \lesssim \int_{\R^d} K(t, f(x); A_0, A_1)^p \, dx.
 $$
 \end{proof}

\begin{lem}\label{LemmaNew23}
	Let $0 < p < \infty$ and $0 < q, r \leq \infty$.
	\begin{enumerate}[\upshape(i)]
	\item Assume $-\infty < s < s_0 < \infty$ and $b > 0$. Then
	    $$
    	\|f\|_{(L_p(\R^d; \ell^s_q), L_p(\R^d; \ell^{s_0}_{q}))_{(0, b-1/r), r}} \asymp \bigg(\sum_{j=0}^\infty 2^{j b r} \bigg(\int_{\R^d} \bigg( \sum_{\nu=2^{j}}^{2^{j+1}-1} [2^{\nu s} |f_\nu (x)|]^q \bigg)^{p/q} \, dx \bigg)^{r/p}   \bigg)^{1/r}
    $$
    (with the corresponding modifications if $q=\infty$ and/or $r=\infty$).
    \item Assume $-\infty < s_0 < s  < \infty$ and $b < 0$. Then
	    $$
    	\|f\|_{(L_p(\R^d; \ell^{s_0}_q), L_p(\R^d; \ell^{s}_{q}))_{(1, b-1/r), r}} \asymp \bigg(\sum_{j=0}^\infty 2^{j b r} \bigg(\int_{\R^d} \bigg( \sum_{\nu=2^{j}}^{2^{j+1}-1} [2^{\nu s} |f_\nu (x)|]^q \bigg)^{p/q} \, dx \bigg)^{r/p}   \bigg)^{1/r}
    $$
    (with the corresponding modifications if $q=\infty$ and/or $r=\infty$).
    \end{enumerate}
\end{lem}
\begin{proof}
   (i): Specializing \eqref{ProofIntF1} with $(A_0, A_1) = (\ell^s_q, \ell^{s_0}_q)$ and using (cf. \eqref{Kfunct})
     \begin{equation}\label{ProofIntF2}
        K(t,x; \ell^s_q,
        \ell^{s_0}_q) \asymp \bigg(
        \sum_{\nu=0}^\infty [\min (2^{\nu s}, t 2^{\nu s_0}) |x_\nu|]^q \bigg)^{1/q},
    \end{equation}
     we derive
    \begin{equation}\label{ProofIntF2*}
    	K(t, f; L_p(\R^d; \ell^s_q), L_p(\R^d; \ell^{s_0}_{q})) \asymp  \bigg(\int_{\R^d} \bigg( \sum_{\nu=0}^\infty [\min (2^{\nu s}, t 2^{\nu s_0}) |f_\nu (x)|]^q \bigg)^{p/q} \, dx \bigg)^{1/p}.
    \end{equation}
    Inserting this formula into the definition of limiting interpolation space \eqref{DefLimInterpolation},
    \begin{align*}
    	\|f\|_{(L_p(\R^d; \ell^s_q), L_p(\R^d; \ell^{s_0}_{q}))_{(0, b-1/r), r}} &\asymp \\
	& \hspace{-5cm} \bigg(\sum_{j=0}^\infty [(1+j)^{b-1/r} K(2^{-j(s_0-s)},f; L_p(\R^d; \ell^s_q), L_p(\R^d; \ell^{s_0}_{q})]^r  \bigg)^{1/r} \\
	& \hspace{-5cm} \asymp \bigg(\sum_{j=0}^\infty (1+j)^{(b-1/r) r} \bigg(\int_{\R^d} \bigg( \sum_{\nu=0}^\infty [2^{\nu s} \min (1,  2^{(\nu-j) (s_0-s)}) |f_\nu (x)|]^q \bigg)^{p/q} \, dx \bigg)^{r/p}  \bigg)^{1/r} \\
	& \hspace{-5cm} \asymp \bigg(\sum_{j=0}^\infty 2^{j (s-s_0) r} (1+j)^{(b-1/r) r} \bigg(\int_{\R^d} \bigg( \sum_{\nu=0}^j [2^{\nu s_0} |f_\nu (x)|]^q \bigg)^{p/q} \, dx \bigg)^{r/p}  \bigg)^{1/r} \\
	& \hspace{-4.5cm} +  \bigg(\sum_{j=0}^\infty (1+j)^{(b-1/r) r} \bigg(\int_{\R^d} \bigg( \sum_{\nu=j}^\infty [2^{\nu s} |f_\nu (x)|]^q \bigg)^{p/q} \, dx \bigg)^{r/p}   \bigg)^{1/r} \\
	& \hspace{-5cm} =: I + II.
    \end{align*}

    Next we prove
    \begin{equation}\label{ProofIntF3}
    	I \lesssim II,
    \end{equation}
    which immediately leads to
    \begin{equation}\label{ProofIntF4}
    \|f\|_{(L_p(\R^d; \ell^s_q), L_p(\R^d; \ell^{s_0}_{q}))_{(0, b-1/r), r}} \asymp II.
    \end{equation}
    To proceed with \eqref{ProofIntF3}, we consider two cases. Firstly, we assume $q \geq p$. Then, for each $x \in \R^d$ and $j \in \N_0$,
    $$
    	\bigg( \sum_{\nu=0}^j [2^{\nu s_0} |f_\nu (x)|]^q \bigg)^{1/q} \leq \bigg( \sum_{\nu=0}^j [2^{\nu s_0} |f_\nu (x)|]^p \bigg)^{1/p}
    $$
    which yields, by Hardy's inequality \eqref{H1} (note that $s_0 > s$),
    \begin{align*}
    	I &\leq  \bigg(\sum_{j=0}^\infty 2^{j (s-s_0) r} (1+j)^{(b-1/r) r} \bigg( \sum_{\nu=0}^j 2^{\nu s_0 p} \|f_\nu\|_{L_p(\R^d)}^p \bigg)^{r/p}  \bigg)^{1/r} \\
	& \lesssim  \bigg(\sum_{j=0}^\infty 2^{j s r} (1+j)^{(b-1/r) r} \|f_j\|_{L_p(\R^d)}^r   \bigg)^{1/r} \leq II.
    \end{align*}
    Secondly, if $q < p$ we can apply Minkowski's inequality to get
    $$
    	 \bigg(\int_{\R^d} \bigg( \sum_{\nu=0}^j [2^{\nu s_0} |f_\nu (x)|]^q \bigg)^{p/q} \, dx \bigg)^{1/p} \leq \bigg(\sum_{\nu=0}^j 2^{\nu s_0 q} \|f_\nu\|_{L_p(\R^d)}^{q} \bigg)^{1/q}.
    $$
    Accordingly, by Hardy's inequality \eqref{H1},
    \begin{align*}
    	I &\leq  \bigg(\sum_{j=0}^\infty 2^{j (s-s_0) r} (1+j)^{(b-1/r) r}  \bigg(\sum_{\nu=0}^j 2^{\nu s_0 q} \|f_\nu\|_{L_p(\R^d)}^{q} \bigg)^{r/q}  \bigg)^{1/r} \\
	& \lesssim \bigg(\sum_{j=0}^\infty 2^{j s r} (1+j)^{(b-1/r) r}  \|f_j\|^r_{L_p(\R^d)} \bigg)^{1/r} \leq II.
    \end{align*}
    This completes the proof of \eqref{ProofIntF3}.

    Furthermore, we make the following claim
    \begin{equation}\label{ProofIntF5}
    	II \asymp  \bigg(\sum_{j=0}^\infty 2^{j b r} \bigg(\int_{\R^d} \bigg( \sum_{\nu=2^{j}}^{2^{j+1}-1} [2^{\nu s} |f_\nu (x)|]^q \bigg)^{p/q} \, dx \bigg)^{r/p}   \bigg)^{1/r}.
    \end{equation}
    Indeed, by basic monotonicity properties,
    \begin{equation*}
    	II \asymp   \bigg(\sum_{j=0}^\infty 2^{j b r} \bigg(\int_{\R^d} \bigg( \sum_{\nu=2^{j}}^\infty [2^{\nu s} |f_\nu (x)|]^q \bigg)^{p/q} \, dx \bigg)^{r/p}   \bigg)^{1/r}
    \end{equation*}
    and thus the estimate $\gtrsim$ in \eqref{ProofIntF5} is clear. Conversely, one can rewrite the previous estimate as
    $$
    	II \asymp \bigg(\sum_{j=0}^\infty 2^{j b r} \bigg(\int_{\R^d} \bigg(\sum_{l=j}^\infty F_l(x)^q  \bigg)^{p/q} \, dx \bigg)^{r/p}   \bigg)^{1/r}
    $$
    where
    $$
    	F_l (x) :=  \bigg(\sum_{\nu=2^{l}}^{2^{l+1}-1} [2^{\nu s} |f_\nu (x)|]^q  \bigg)^{1/q}, \qquad l \in \N_0, \qquad x \in \R^d.
    $$
    If $q \geq p$ then the embedding $\ell_p \hookrightarrow \ell_q$ and Hardy's inequality \eqref{H2} (note that $b > 0$) imply
    \begin{align*}
    	II &\lesssim \bigg(\sum_{j=0}^\infty 2^{j b r} \bigg( \sum_{l=j}^\infty  \|F_l\|_{L_p(\R^d)}^p \bigg)^{r/p}   \bigg)^{1/r} \\
	& \lesssim \bigg(\sum_{j=0}^\infty 2^{j b r} \|F_l\|_{L_p(\R^d)}^r \bigg)^{1/r}
    \end{align*}
    which gives \eqref{ProofIntF5}.
    On the other hand, if $q < p$ we can apply Minkowski's inequality together with Hardy's inequality \eqref{H2} to estimate
    \begin{align*}
    	II &\lesssim \bigg(\sum_{j=0}^\infty 2^{j b r} \bigg(\sum_{l=j}^\infty \|F_l\|_{L_p(\R^d)}^q \bigg)^{r/q} \bigg)^{1/r} \\
	& \lesssim \bigg(\sum_{j=0}^\infty 2^{j b r} \|F_j\|_{L_p(\R^d)}^r  \bigg)^{1/r}.
    \end{align*}
    This finishes the proof of \eqref{ProofIntF5}.

    Putting together \eqref{ProofIntF4} and \eqref{ProofIntF5},
    $$
    	\|f\|_{(L_p(\R^d; \ell^s_q), L_p(\R^d; \ell^{s_0}_{q}))_{(0, b-1/r), r}} \asymp \bigg(\sum_{j=0}^\infty 2^{j b r} \bigg(\int_{\R^d} \bigg( \sum_{\nu=2^{j}}^{2^{j+1}-1} [2^{\nu s} |f_\nu (x)|]^q \bigg)^{p/q} \, dx \bigg)^{r/p}   \bigg)^{1/r}.
    $$

    (ii):     It follows from \eqref{ProofIntF2*} (after replacing the roles played by $s$ and $s_0$) that
        \begin{align*}
    	\|f\|_{(L_p(\R^d; \ell^{s_0}_q), L_p(\R^d; \ell^{s}_{q}))_{(1, b-1/r), r}} &\asymp \\
	& \hspace{-5cm} \bigg(\sum_{j=0}^\infty [2^{j(s-s_0)}(1+j)^{b-1/r} K(2^{-j(s-s_0)},f; L_p(\R^d; \ell^{s_0}_q), L_p(\R^d; \ell^{s}_{q})]^r  \bigg)^{1/r} \\
	& \hspace{-5cm} \asymp \bigg(\sum_{j=0}^\infty 2^{j(s-s_0) r} (1+j)^{(b-1/r) r} \bigg(\int_{\R^d} \bigg( \sum_{\nu=0}^\infty [2^{\nu s_0} \min (1,  2^{(\nu-j) (s-s_0)}) |f_\nu (x)|]^q \bigg)^{p/q} \, dx \bigg)^{r/p}  \bigg)^{1/r} \\
	& \hspace{-5cm} \asymp \bigg(\sum_{j=0}^\infty (1+j)^{(b-1/r) r} \bigg(\int_{\R^d} \bigg( \sum_{\nu=0}^j [2^{\nu s} |f_\nu (x)|]^q \bigg)^{p/q} \, dx \bigg)^{r/p}  \bigg)^{1/r} \\
	& \hspace{-4.5cm} +  \bigg(\sum_{j=0}^\infty 2^{j(s-s_0) r}  (1+j)^{(b-1/r) r} \bigg(\int_{\R^d} \bigg( \sum_{\nu=j}^\infty [2^{\nu s_0} |f_\nu (x)|]^q \bigg)^{p/q} \, dx \bigg)^{r/p}   \bigg)^{1/r} \\
	& \hspace{-5cm} =: \mathcal{I} + \mathcal{II}.
    \end{align*}
        Furthermore, we claim
    $$
    	\mathcal{II} \lesssim \mathcal{I}
    $$
    and
    $$
    	\mathcal{I} \asymp \bigg(\sum_{j=0}^\infty 2^{j b r} \bigg(\int_{\R^d} \bigg( \sum_{\nu=2^{j}}^{2^{j+1}-1} [2^{\nu s} |f_\nu (x)|]^q \bigg)^{p/q} \, dx \bigg)^{r/p}   \bigg)^{1/r}.
    $$
     We omit the proofs of these estimates since they are obvious modifications of those in \eqref{ProofIntF3} and \eqref{ProofIntF5}, respectively. Hence
     $$
     	\|f\|_{(L_p(\R^d; \ell^{s_0}_q), L_p(\R^d; \ell^{s}_{q}))_{(1, b-1/r), r}} \asymp  \bigg(\sum_{j=0}^\infty 2^{j b r} \bigg(\int_{\R^d} \bigg( \sum_{\nu=2^{j}}^{2^{j+1}-1} [2^{\nu s} |f_\nu (x)|]^q \bigg)^{p/q} \, dx \bigg)^{r/p}   \bigg)^{1/r}.
     $$
    \end{proof}

    We are now ready to give the

\begin{proof}[Proof of Theorem \ref{TheoremInterpolationF}]
 (i): We start by proving \eqref{TheoremInterpolationFDisplay1} with $A = F$ and $q_0 = q$, i.e.
 \begin{equation}\label{ProofIntF1*}
 	T^b_r F^{s}_{p,q}(\mathbb{R}^d) = (F^s_{p,q}(\R^d), F^{s_0}_{p,q}(\R^d))_{(0,b-1/r),r}.
\end{equation}
 Since $F^s_{p, q}(\R^d)$ is a retract of $L_p(\R^d; \ell^s_q)$, the interpolation space $(F^s_{p,q}(\R^d), F^{s_0}_{p,q}(\R^d))_{(0,b-1/r),r}$ can be identified, via the retraction method, with $(L_p(\R^d; \ell^s_q), L_p(\R^d; \ell^{s_0}_{q}))_{(0, b-1/r), r}$. This space was computed in  Lemma \ref{LemmaNew23}(i). Accordingly, the desired formula \eqref{ProofIntF1*} follows.

    The formula \eqref{TheoremInterpolationFDisplay1} for $A=B$ and $0 < q_0 \leq \infty$ follows from the previous case (i.e., \eqref{ProofIntF1*}) and the well-known interpolation formula (cf. \cite[Theorem 2.4.2, p. 64]{Triebel83})
    \begin{equation}\label{ProofIntF6}
    	(F^s_{p, q}(\R^d), F^{s_0}_{p, q}(\R^d))_{\theta, q_0} = B^{(1-\theta) s + \theta s_0}_{p, q_0}(\R^d)
    \end{equation}
    for $\theta \in (0, 1)$. More precisely, given $s_0 > s$ we choose $s_1 > s_0$ and $\theta \in (0, 1)$ such that $s_0 = (1-\theta) s + \theta s_1$. Then, by \eqref{ProofIntF6}, Lemma \ref{LemmaReiteration}(i), and \eqref{ProofIntF1*},
    \begin{align*}
    	(F^s_{p,q}(\R^d), B^{s_0}_{p,q_0}(\R^d))_{(0,b-1/r),r} &= (F^s_{p,q}(\R^d), (F^s_{p, q}(\R^d), F^{s_1}_{p, q}(\R^d))_{\theta, q_0} )_{(0,b-1/r),r} \\
	&\hspace{-3cm} =  (F^s_{p,q}(\R^d), F^{s_1}_{p, q}(\R^d) )_{(0,b-1/r),r}  = T^b_r F^{s}_{p, q}(\R^d).
    \end{align*}

    It remains to show \eqref{TheoremInterpolationFDisplay1} with $A = F$ and $0 < q_0 \leq \infty$ with $q_0 \neq q$. This is a simple consequence of the embeddings \eqref{RelationsBF} and the corresponding assertion with $A = B$.

   (ii): The proof of \eqref{TheoremInterpolationFDisplay1new} follows similar ideas as those given in \eqref{TheoremInterpolationFDisplay1}. Accordingly, we will only provide a sketch of the argument. We first concentrate on the case $A = F$ with $q_0 = q$ and we claim that
     \begin{equation}\label{ProofIntF7}
	 T^b_r F^s_{p, q}(\mathbb{R}^d) = (F^{s_0}_{p,q}(\R^d), F^s_{p,q}(\R^d))_{(1,b-1/r),r}.
	 \end{equation}
	Indeed, this is an immediate consequence of Lemma \ref{LemmaNew23}(ii) and the retraction method.     Once formula \eqref{ProofIntF7} has been established, its extension from $F^{s_0}_{p, q}(\R^d)$ to any $A^{s_0}_{p, q_0}(\R^d)$ with $s_0 < s$ and $0 < q_0 \leq \infty$ can be obtained via Lemma \ref{LemmaReiteration}(ii) and  \eqref{RelationsBF}.

\end{proof}

\begin{rem}
	Similarly as in Remark \ref{Remark4.4}, the proof of Theorem \ref{TheoremInterpolationF} can be applied for the limiting case $b=0$. Namely, if $p \in (1, \infty), q \in (0, \infty], r \in (0, \infty), -\infty < s < s_0 < \infty$, then (cf. \eqref{34})
	\begin{equation}\label{435}
		(F^s_{p,q}(\R^d), F^{s_0}_{p,q}(\R^d))_{(0,-1/r),r}  = T^*_r F^{s}_{p, q}(\R^d).
	\end{equation}
	Recall that $T_r F^{s}_{p, q}(\R^d) \neq  T^*_r F^{s}_{p, q}(\R^d)$; see Remark \ref{Remark36}.
\end{rem}

Next we write down two important special cases of Theorem \ref{TheoremInterpolationF}. Namely
	\begin{equation}\label{RemindNew}
		T^{b+1/q}_q F^{0}_{p, 2} (\R^d) = \mathbf{B}^{0, b}_{p, q}(\R^d) \qquad \text{and} \qquad   T^{b+1/q}_q F^{s}_{p, 2}(\R^d) =  \text{Lip}^{s, b}_{p, q}(\R^d),
	\end{equation}
	cf. Proposition \ref{PropositionCoincidences}(iii) and (iv) for assumptions on the involved parameters. Accordingly, the following result is an immediate application of Theorem \ref{TheoremInterpolationF}.

\begin{cor}[$\mathbf{B}^{0, b}_{p, q}(\R^d)$ and $\text{Lip}^{s, b}_{p, q}(\R^d)$ via limiting interpolation] Let $A \in \{B, F\}$.
	\begin{enumerate}[\upshape(i)]
		\item Let $1 < p < \infty, 0 < q, q_1 \leq \infty, s > 0$ and $b > -1/q$. Then
		\begin{equation}\label{IntLimB0}
			\mathbf{B}^{0, b}_{p, q}(\R^d) = (L_p(\R^d), A^s_{p, q_1}(\R^d))_{(0, b), q}.
		\end{equation}
		\item Let $1 < p < \infty, 0 < q, q_0 \leq \infty, s > 0, s_0 < s$ and $b < -1/q$. Then
		\begin{equation}\label{IntLimLip}
			\emph{Lip}^{s, b}_{p, q}(\R^d) = (A^{s_0}_{p, q_0}(\R^d), H^s_p(\R^d))_{(1, b), q}.
		\end{equation}
	\end{enumerate}
\end{cor}

\begin{rem}\label{RemarkLimBesov}
	The formula \eqref{IntLimB0} was already obtained in \cite[Theorem 3.4]{CobosDominguezTriebel} using the well-known fact
	$$
	K(t^k, f; L_p(\R^d), W^k_p(\R^d)) \asymp t^k \|f\|_{L_p(\R^d)} + \omega_k(f, t)_p, \qquad t \in (0, 1),
	$$
	and reiteration formulae for limiting interpolation (cf. Lemma \ref{LemmaReiteration}). Furthermore this method also works with the limiting values $p=1, \infty$ in \eqref{IntLimB0}.
	
	On the other hand, \eqref{IntLimLip} with $A^{s_0}_{p, q_0}(\R^d) = L_p(\R^d)$ (i.e., $s_0=0$ and $q_0 = 2$) reads
	\begin{equation}\label{IntLimLip2}
		\text{Lip}^{s, b}_{p, q}(\R^d) = (L_p(\R^d), H^s_p(\R^d))_{(1, b), q};
	\end{equation}
	see \cite[(2.13)]{DominguezHaroskeTikhonov} for  an alternative approach to \eqref{IntLimLip2} based on $K$-functionals and (fractional) moduli of smoothness.
\end{rem}

$$		(T^{b_0}_{r_0} A^{s_0}_{p, q_0}(\R^d), T^{b_1}_{r_1}\tilde{A}^{s_1}_{p, q_1}(\R^d))_{\theta, r} = B^{s, b}_{p, r}(\R^d),\qquad s_0\ne s_1,
$$

$$		(T^{b_0}_{r_0} F^{s}_{p,q}(\R^d), T^{b_1}_{r_1}F^{s}_{p,q}(\R^d))_{\theta,r} = T^b_r F^{s}_{p,q}(\R^d), \qquad b_0\ne b_1
$$

\subsection{$T^b_r F^{s}_{p, q}(\R^d)$ spaces are closed under interpolation.} As an application of Theorem \ref{TheoremInterpolationF}, we establish the closedness of the scale of spaces $T^b_r F^s_{p, q}(\R^d)$ under classical real interpolation. More precisely, we establish the following

\begin{thm}\label{TheoremInterpolation2pspapsa}
	Let $1 < p < \infty, 0 < q, r, r_0, r_1 \leq \infty, -\infty < s, b_0, b_1 < \infty$ and $0 < \theta < 1$. Assume  further $b_0  \neq b_1$ and $b_0 b_1 > 0$ and let $b=(1-\theta)b_0 + \theta b_1$. Then
	\begin{equation*}
		(T^{b_0}_{r_0} F^{s}_{p,q}(\R^d), T^{b_1}_{r_1}F^{s}_{p,q}(\R^d))_{\theta,r} = T^b_r F^{s}_{p,q}(\R^d).
	\end{equation*}
\end{thm}

\begin{proof}
Assume e.g. that $b_0, b_1 < 0$. It is not restrictive to assume further $b_0 < b_1$. Let $s_0 <s$. By Theorem \ref{TheoremInterpolationF}(ii),
$$
	T^{b_j}_{r_j} F^{s}_{p,q}(\R^d) = (F^{s_0}_{p,q}(\R^d), F^{s}_{p,q}(\R^d))_{(1,b_j-1/r_j),r_j}, \qquad j = 0, 1.
$$
These formulae and Lemma \ref{LemmaReiteration}(iv) imply
\begin{align*}
	(T^{b_0}_{r_0} F^{s}_{p,q}(\R^d), T^{b_1}_{r_1}F^{s}_{p,q}(\R^d))_{\theta,r} & = \\
	& \hspace{-4cm} ((F^{s_0}_{p,q}(\R^d), F^{s}_{p,q}(\R^d))_{(1,b_0-1/r_0),r_0}, (F^{s_0}_{p,q}(\R^d), F^{s}_{p,q}(\R^d))_{(1,b_1-1/r_1),r_1})_{\theta, r} \\
	& \hspace{-4cm} = (F^{s_0}_{p,q}(\R^d), F^{s}_{p,q}(\R^d))_{(1, b-1/r), r} = T^b_r F^{s}_{p, q}(\R^d)
\end{align*}
where the last step follows from Theorem \ref{TheoremInterpolationF}(ii).

The case $b_0, b_1 > 0$ can be treated analogously, but now applying Theorem  \ref{TheoremInterpolationF}(i).
\end{proof}

\begin{rem}
 Theorems \ref{TheoremInterpolationF} and \ref{TheoremInterpolation2pspapsa} can be extended to deal with  $0 < p < \infty$. We refer to Remark \ref{RemarkIntp} below.
\end{rem}

\begin{rem}
	Theorems \ref{TheoremInterpolation} and \ref{TheoremInterpolation2} as well as
 Theorems \ref{TheoremInterpolationF} and \ref{TheoremInterpolation2pspapsa}
 show some striking differences between classical interpolation (cf. \eqref{ClassicInt2} and \eqref{ClassicInt}) and limiting interpolation (cf. \eqref{DefLimInterpolation}). Let $A, \tilde{A} \in \{B, F\}$. The following holds
	$$
		(A^s_{p, q}(\R^d), \tilde{A}^{s_0}_{p, q_0}(\R^d))_{\theta, r, b} = B^{(1-\theta) s + \theta s_0, b}_{p, r}(\R^d)
	$$
	provided that $s, s_0, b \in \R, \, s \neq s_0, \theta \in (0, 1), 0 < p < \infty, 0 < q, q_0, r \leq \infty$; cf. \cite[Section 2.4.2]{Triebel83} if $b=0$ and \cite{CobosFernandez} if $b \in \R$. In particular, this formula tells us that any interpolation pair formed by Besov or Triebel--Lizorkin spaces with fixed integrability parameter $p$ produces as an outcome the Besov space $B^{(1-\theta) s + \theta s_0, b}_{p, r}(\R^d)$. Note that classical interpolation does not depend on the secondary parameters $q$ and $q_0$. Importantly, these assertions may fail to be true dealing with limiting interpolation, as can be illustrated by
		\begin{align*}
		T^{b+1/r}_r B^s_{p, q}(\R^d) &= (B^s_{p, q}(\R^d), B^{s_0}_{p, q_0}(\R^d))_{(0, b), r} \\
		& \neq (F^s_{p, q}(\R^d), B^{s_0}_{p, q_0}(\R^d))_{(0, b), r} =T^{b+1/r}_r F^s_{p, q}(\R^d),
	\end{align*}
	where $q \neq p$, and
		\begin{align*}
		T^{b+1/r}_r B^s_{p, q}(\R^d) & =(B^s_{p, q}(\R^d), B^{s_0}_{p, q_0}(\R^d))_{(0, b), r}  \\
		& \neq (B^s_{p, \tilde{q}}(\R^d), B^{s_0}_{p, q_0}(\R^d))_{(0, b), r} = T^{b+1/r}_r B^s_{p, \tilde{q}}(\R^d),
	\end{align*}
	where $q \neq \tilde{q}$.
\end{rem}

\subsection{Interpolation in the case of non-fixed smoothness}
Our next result extends
 the classical embedding
$$
	(A^{s_0}_{p, q_0}(\R^d), \tilde{A}^{s_1}_{p, q_1}(\R^d))_{\theta, r}= B^{s}_{p, r}(\R^d),\qquad A, \tilde{A} \in \{B, F\},
	$$
 and, on the other hand, it complements
 Theorems \ref{TheoremInterpolation2} and \ref{TheoremInterpolation2pspapsa} to pairs of spaces $(T^b_r A^{s_0}_{p, q}(\R^d), T^b_r A^{s_1}_{p, q}(\R^d))$ with $s_0 \neq s_1$. In sharp contrast with the case $s_0 = s_1$ studied in these theorems, the resulting interpolation space with $s_0 \neq s_1$ is always a classical Besov space. In particular, we are now able to compute the interpolation space relative to $(\mathbf{B}^{0, b}_{p, q}(\R^d), \text{Lip}^{s, b}_{p, q}(\R^d))$, which seems to be unknown in the literature; see Corollary \ref{CorIntBLip} below.

\begin{thm}\label{ThmIntNew}
	Let $A, \tilde{A} \in \{B, F\}$. Assume $0 < p, q_0, q_1, r, r_0, r_1 \leq \infty \,  (p < \infty \text{ when } A = F \text{ or } \tilde{A} = F),  b_0, b_1 \in \R \backslash \{0\}, s_0, s_1 \in \R, s_0 \neq s_1, 0 < \theta < 1$.
	Let
	$$
		s = (1-\theta) s_0 + \theta s_1 \qquad \text{and} \qquad b = (1-\theta) b_0 + \theta b_1.
	$$
	Then
	\begin{equation}\label{ThmIntNew3}
		(T^{b_0}_{r_0} A^{s_0}_{p, q_0}(\R^d), T^{b_1}_{r_1}\tilde{A}^{s_1}_{p, q_1}(\R^d))_{\theta, r} = B^{s, b}_{p, r}(\R^d).
	\end{equation}
	In particular, if $\frac{b_0}{b_0-b_1} \in (0, 1)$ then
		$$
		(T^{b_0}_{r_0} A^{s_0}_{p, q_0}(\R^d), T^{b_1}_{r_1}\tilde{A}^{s_1}_{p, q_1}(\R^d))_{\frac{b_0}{b_0-b_1}, r} = B^{s}_{p, r}(\R^d).
	$$
\end{thm}

\begin{rem}
	The outcome in \eqref{ThmIntNew3} does not depend on $q_0$ and $q_1$. This is in contrast with Theorems \ref{TheoremInterpolation2} and \ref{TheoremInterpolation2pspapsa}.
\end{rem}

\begin{proof}[Proof of Theorem \ref{ThmIntNew}]
%
	We will make use of embeddings between truncated and classical function spaces given in Corollaries \ref{TheoremEmbeddings1} and \ref{CorollaryEmbFBNew} below. More precisely, we have
	$$
	B^{s_i, b_i}_{p, \min\{q_i, r_i\}}(\R^d) \hookrightarrow	T^{b_i}_{r_i} A^{s_i}_{p, q_i}(\R^d) \hookrightarrow B^{s_i, b_i}_{p, \max\{q_i, r_i\}}(\R^d), \qquad i= 0, 1.
	$$
	Therefore
	\begin{align*}
	(T^{b_0}_{r_0}A^{s_0}_{p, q_0}(\R^d), T^{b_1}_{r_1}\tilde{A}^{s_1}_{p, q_1}(\R^d))_{\theta, r} &\hookrightarrow  \\
	& \hspace{-3cm}(B^{s_0, b_0}_{p, \max\{q_0, r_0\}}(\R^d), B^{s_1, b_1}_{p, \max\{q_1, r_1\}}(\R^d))_{\theta, r} = B^{s, b}_{p, r}(\R^d)
	\end{align*}
	where we have applied well-known interpolation properties of Besov spaces of logarithmic smoothness in the last step (cf. \cite[Theorem 5.3]{CobosFernandez} and \cite[Lemma 2.3]{DominguezTikhonov}). Analogously, one can prove
	$$
	 B^{s, b}_{p, r}(\R^d) \hookrightarrow (T^{b_0}_{r_0} A^{s_0}_{p, q_0}(\R^d), T^{b_1}_{r_1} \tilde{A}^{s_1}_{p, q_1}(\R^d))_{\theta, r}.
	$$	
\end{proof}

\begin{cor}[Interpolation involving $\mathbf{B}^{0, b}_{p, q}(\R^d)$ and $\text{Lip}^{s, b}_{p, q}(\R^d)$]\label{CorIntBLip}
	Let $0 < \theta < 1, 1 < p < \infty, s > 0,  0 < q, q_0, q_1 \leq \infty, b_0 > -1/q_0$ and $b_1 < -1/q_1$.
	\begin{enumerate}[\upshape(i)]
	\item
	\begin{equation}\label{CorIntBLip1new}
		(\mathbf{B}^{0, b_0}_{p, q_0}(\R^d), \emph{Lip}^{s, b_1}_{p, q_1}(\R^d))_{\theta, q} = B^{\theta s, (1-\theta) (b_0 + 1/q_0) + \theta (b_1 + 1/q_1) }_{p, q}(\R^d).
	\end{equation}
	In particular, if $\theta = \frac{b_0 + \frac{1}{q_0}}{b_0 + \frac{1}{q_0} - (b_1 + \frac{1}{q_1})}$ then
	$$
			(\mathbf{B}^{0, b_0}_{p, q_0}(\R^d), \emph{Lip}^{s, b_1}_{p, q_1}(\R^d))_{\theta, q} = B^{\theta s}_{p, q}(\R^d).
	$$
	\item
	$$
		(L_p(\R^d), \emph{Lip}^{s, b_1}_{p, q_1}(\R^d))_{\theta, q} = B^{\theta s, \theta (b_1 + 1/q_1)}_{p, q}(\R^d).
	$$
	\item $$
		(\mathbf{B}^{0, b_0}_{p, q_0}(\R^d), H^s_p(\R^d))_{\theta, q} = B^{\theta s, (1-\theta) (b_0 + 1/q_0)}_{p, q}(\R^d).
	$$
	\end{enumerate}
\end{cor}

\begin{proof}
	(i): According to \eqref{RemindNew}, we have
$$
	T^{b_0+1/r_0}_{r_0} F^{0}_{p, 2}(\R^d) = \mathbf{B}^{0, b_0}_{p, r_0}(\R^d) \qquad \text{and} \qquad T^{b_1+1/r_1}_{r_1} F^{s}_{p, 2}(\R^d) = \text{Lip}^{s, b_1}_{p, r_1}(\R^d).
$$
Then \eqref{CorIntBLip1new} follows from Theorem \ref{ThmIntNew}.

(ii): Combining \eqref{IntLimLip} (see also \eqref{IntLimLip2}) and Lemma \ref{LemmaReiteration}(vi), we get
\begin{align*}
	(L_p(\R^d), \text{Lip}^{s, b_1}_{p, q_1}(\R^d))_{\theta, q}  & = (L_p(\R^d), (L_p(\R^d), H^s_p(\R^d))_{(1, b_1), q_1})_{\theta, q} \\
	&\hspace{-3cm} = (L_p(\R^d), H^s_p(\R^d))_{\theta, q, \theta (b_1 + 1/q_1)} = B^{\theta s, \theta(b_1 + 1/q_1)}_{p, q}(\R^d)
\end{align*}
where the last step makes use of well-known interpolation properties of classical Besov and Triebel--Lizorkin spaces (see, e.g., \cite[Theorem 5.3]{CobosFernandez} and \cite[(2.37)]{DominguezTikhonov}).

(iii): By \eqref{IntLimB0} (with $A= F$ and $q_1=2$) and Lemma \ref{LemmaReiteration}(v),
\begin{align*}
	(\mathbf{B}^{0, b_0}_{p, q_0}(\R^d), H^s_p(\R^d))_{\theta, q}  & = ((L_p(\R^d), H^s_p(\R^d))_{(0, b_0), q_0}, H^s_p(\R^d))_{\theta, q} \\
	& \hspace{-3cm} = (L_p(\R^d), H^s_p(\R^d))_{\theta, q, (1-\theta) (b_0 + 1/q_0)} = B^{\theta s, (1-\theta)(b_0 + 1/q_0)}_{p, q}(\R^d).
\end{align*}
\end{proof}

\subsection{$\T^b_r F^{s}_{p, q}(\R^d)$ spaces via  interpolation}

Recall that the operator $\mathfrak{J}$ is defined in \eqref{Retraction}. Since $\mathfrak{J}$ is a retract from $B^s_{p, q}(\R^d)$ to $\ell^s_q(L_p(\R^d))$ and from $F^s_{p, q}(\R^d)$ to $L_p(\R^d; \ell^s_q)$, Theorem \ref{Thm4.2} (see, e.g., \eqref{Interpolation1B} with $A = B$) can be interpreted as
$$
	f \in T^b_r B^s_{p, q}(\R^d) \iff \mathfrak{J}(f) \in (\ell^s_q(L_p(\R^d)), \ell^{s_0}_{q_0} (L_{p}(\R^d)))_{(0, b-1/r), r}
$$
and
$$
	\|f\|_{ T^b_r B^s_{p, q}(\R^d) } \asymp \| \mathfrak{J}(f)\|_{(\ell^s_q(L_p(\R^d)), \ell^{s_0}_{q_0} (L_{p}(\R^d)))_{(0, b-1/r), r}}.
$$
In the same fashion, Theorem \ref{TheoremInterpolationF} (specifically, \eqref{TheoremInterpolationFDisplay1} with $A=F$) can be rewritten as
$$
	f \in T^b_r F^s_{p, q}(\R^d) \iff  \mathfrak{J}(f) \in (L_p(\R^d; \ell^s_q), L_p(\R^d; \ell^{s_0}_{q_0}))_{(0, b-1/r), r}
$$
and
$$
	\|f\|_{T^b_r F^s_{p, q}(\R^d)} \asymp \| \mathfrak{J}(f)\|_{(L_p(\R^d; \ell^s_q), L_p(\R^d; \ell^{s_0}_{q_0}))_{(0, b-1/r), r}}.
$$
Our next result shows that $\T^b_r F^{s}_{p, q}(\R^d)$ can be generated in terms of $L_p(\R^d; (\ell^s_q, \ell^{s_0}_{q_0})_{(0, b-1/r), r})$, i.e., after changing the roles of $L^p(\R^d)$ and $((0, b-1/r), r)$-interpolation in the construction of truncated $F$-spaces given above.

\begin{thm}[Characterization of $\T^b_r F^{s}_{p, q}(\R^d)$ via limiting interpolation]\label{TheoremInterpolationFrakF}
	 Let $p \in (0, \infty), q, q_0, r \in (0, \infty], s, s_0 \in \R, s \neq s_0$, and $b \in \R \backslash \{0\}$.
	  \begin{enumerate}[\upshape(i)]
	 \item If $s_0 > s$ and $b > 0$ then
	 $$f \in \T^b_r F^{s}_{p, q}(\R^d) \quad \text{if and only if} \quad \mathfrak{J} (f)  \in L^p(\R^d; (\ell^s_q, \ell^{s_0}_{q_0})_{(0, b-1/r), r}).$$
	 Furthermore
	 $$
	 	\|f\|_{\T^b_r F^{s}_{p, q}(\R^d)} \asymp \|\mathfrak{J} (f) \|_{L^p(\R^d; (\ell^s_q, \ell^{s_0}_{q_0})_{(0, b-1/r), r})}.
	 $$
	  \item If $s_0 < s$ and $b < 0$ then
	 $$f \in \T^b_r F^{s}_{p, q}(\R^d) \quad \text{if and only if} \quad \mathfrak{J} (f)  \in L^p(\R^d; (\ell^{s_0}_{q_0}, \ell^{s}_{q})_{(1, b-1/r), r}).$$
	 Furthermore
	 $$
	 	\|f\|_{\T^b_r F^{s}_{p, q}(\R^d)} \asymp \|\mathfrak{J} (f) \|_{L^p(\R^d; (\ell^{s_0}_{q_0}, \ell^{s}_{q})_{(1, b-1/r), r})}.
	 $$
	  \end{enumerate}
\end{thm}

\begin{proof}
The proof follows similar ideas as the proof of Theorem \ref{TheoremInterpolation}. More precisely, it was already shown in \eqref{24} and \eqref{new3} (modulo replacing $\|x_\nu\|_{L_p(\R^d)}$ by $|x_\nu|$) that
	\begin{equation}\label{ProofIntThmFFrak1}
		\|x\|_{(\ell^s_q, \ell^{s_0}_{q})_{(0, b-1/r), r}} \asymp  \bigg(\sum_{j=0}^\infty 2^{j b r} \bigg(\sum_{\nu=2^{j}-1}^{2^{j+1}-2} 2^{\nu s q}  |x_\nu|^q   \bigg)^{r/q} \bigg)^{1/r}
	\end{equation}
	provided that $s_0 > s$ and $b > 0$. This proves (i) with $q_0 = q$. The general case $q_0 \in (0, \infty]$ follows from reiteration formulas. Indeed, we make use of the well-known fact that (cf. \cite[Theorem 5.6.1]{BerghLofstrom} and \cite[1.18.2]{Triebel})
	$$
		\ell^{s_0}_{q_0} =(\ell^s_q, \ell^{s_1}_q)_{\theta, q_0},
	$$
	where $s < s_0 < s_1$ and $\theta  \in (0, 1)$ such that $s_0 = (1-\theta) s + \theta s_1$.
	
\end{proof}

\newpage

\section{Characterizations of truncated Besov spaces in terms of approximation}\label{SectionApproximation}

The goal of this section is to establish characterizations of the spaces $T^b_r B^s_{p, q}(\mathbb{R}^d)$  via approximation procedures. We restrict our attention to approximation procedures in $L_p(\R^d)$ with $p \in [1, \infty]$ (and thus smoothness parameter $s > 0$). However, many of the results presented in this section admit analogues in $L_p(\R^d)$ with  $p > 0$ (and so $s > d \max\{\frac{1}{p} - 1, 0\}$).

Let us recall two well-known characterization of Besov spaces (see, e.g., \cite{Pietsch}, \cite[Section 2.5.3]{Triebel83}). Namely, if $p \in [1, \infty], q \in (0,\infty], s > 0$ and $b \in \R$ then
\begin{equation}\label{BesovPi0}
	\|f\|_{B^{s,b}_{p,q}(\R^d)} \asymp \bigg(\sum_{j=0}^\infty  [2^{j s} (1 + j)^b E_{2^j-1}(f)_{L_p(\R^d)}]^q \bigg)^{1/q}
\end{equation}
and if, additionally, $\alpha > 0$ and $r \in (0, \infty]$ then
\begin{equation}\label{BesovPi}
		 \|f\|_{B^{s+ \alpha,b}_{p,r}(\mathbb{R}^d)} \asymp \bigg(\sum_{j=0}^\infty [2^{j \alpha} (1 + j)^b E_{2^j-1}(f)_{B^s_{p,q}(\R^d)}]^r \bigg)^{1/r}.
	\end{equation}
	Here, for $X=L_p(\R^d), B^s_{p,q}(\R^d)$, we use the notation\index{\bigskip\textbf{Functionals and functions}!$E_j(f)_X$}\label{ERROR}
	$$
	E_j(f)_{X} := \inf_{g} \|f-g\|_{X}, \qquad j \geq 1,
	$$
	where the infimum is taken over all  entire functions $g$ of exponential type $j$ and
	 $E_0(f)_{X} := \|f\|_{X}$. The limiting case $s = 0$ in \eqref{BesovPi0} characterizes $\mathbf{B}^{0, b}_{p, q}(\R^d)$ (cf. \cite[Section 10.4]{DominguezTikhonov}). Our next result deals with the case $\alpha=0$ in \eqref{BesovPi} by showing the connection with the spaces $T^b_r B^s_{p, q}(\mathbb{R}^d)$.
	
	 Let $\mu_j := 2^{2^j}$ for $j \geq 1$ and $\mu_0 = 0$.

\begin{thm}\label{ThmApprox}
	Let $p \in [1, \infty], q, r \in (0,\infty]$ and $s > 0$.
	\begin{enumerate}[\upshape(i)]
	\item If $b > 0$ then
	\begin{equation}\label{ThmApprox1}
		 \|f\|_{T^b_r B^s_{p, q}(\mathbb{R}^d)} \asymp \bigg(\sum_{j=0}^\infty [2^{j b} E_{\mu_j}(f)_{B^s_{p,q}(\R^d)}]^r  \bigg)^{1/r}.
	\end{equation}
	\item If $b \in \R \backslash \{0\}$ then
	\begin{equation}\label{ThmApprox2}
		 \|f\|_{T^b_r B^s_{p, q}(\mathbb{R}^d)} \asymp \bigg(\sum_{j=0}^\infty 2^{j b r} \bigg(\sum_{\nu=2^j-1}^{2^{j+1}-2} (2^{\nu s} E_{2^\nu-1}(f)_{L_p(\R^d)})^q \bigg)^{r/q} \bigg)^{1/r}.
	\end{equation}
	\end{enumerate}
\end{thm}

\begin{rem}
	 Under the assumptions of Theorem \ref{ThmApprox}(ii), letting $q=r$ in \eqref{ThmApprox2}, we have
	\begin{equation*}
		\|f\|_{T^b_q B^{s}_{p,q}(\mathbb{R}^d)}
		 \asymp \bigg(\sum_{\nu=0}^{\infty} (2^{\nu s} (1+\nu)^b E_{2^\nu-1}(f)_{L_p(\R^d)})^q  \bigg)^{1/q} \asymp    \|f\|_{B^{s,b}_{p,q}(\R^d)},
	\end{equation*}
	where the last estimate follows from \eqref{BesovPi0}. Compare with Proposition \ref{PropositionCoincidences}.
	
\end{rem}

\begin{proof}[Proof of Theorem \ref{ThmApprox}]
	(i): Let $u > 0$. In virtue of Theorem \ref{TheoremInterpolation}(i), we have
	\begin{equation}\label{IntForApprox}
	T^b_r B^s_{p, q}(\mathbb{R}^d) = (B^s_{p,q}(\R^d), B^{s+ 1/u}_{p,u}(\R^d))_{(0,b-1/r),r}.
	\end{equation}
	Making use of the known estimate (see \cite[Proposition 6.3]{Nilsson})
	$$
		K(2^{j/u}, f; B^{s+ 1/u}_{p,u}(\R^d), B^s_{p,q}(\R^d)) \asymp \bigg(\sum_{\nu=0}^{j} 2^\nu E_{2^\nu - 1}(f)_{B^s_{p,q}(\R^d)}^{u}  \bigg)^{1/u}, \qquad j \in \mathbb{N}_0,
	$$
	we can rewrite \eqref{IntForApprox} as follows
	\begin{align*}
		\|f\|_{T^b_r B^s_{p, q}(\mathbb{R}^d)} & \asymp \bigg(\sum_{j=0}^\infty [(1 + j)^{b-1/r} K(2^{-j/u}, f; B^s_{p,q}(\R^d), B^{s+ 1/u}_{p,u}(\R^d))]^r  \bigg)^{1/r} \\
		& \asymp \left(\sum_{j=0}^\infty 2^{-j r/u} (1 +  j)^{(b-1/r) r}  \bigg(\sum_{\nu=0}^{j} 2^\nu E_{2^\nu - 1}(f)_{B^s_{p,q}(\R^d)}^{u}  \bigg)^{r/u}  \right)^{1/r} \\
		& \asymp \bigg(\sum_{j=0}^\infty [(1 +  j)^{b-1/r} E_{2^j -1}(f)_{B^s_{p,q}(\R^d)}]^r \bigg)^{1/r} \\
		& \asymp \bigg(\sum_{j=0}^\infty [2^{j b} E_{\mu_j}(f)_{B^s_{p,q}(\R^d)}]^r \bigg)^{1/r},
	\end{align*}
	where we have applied  Hardy's inequality \eqref{H1} in the penultimate step. The proof of \eqref{ThmApprox1} is complete.
	
	(ii): The equivalence \eqref{ThmApprox2} with $b > 0$ follows from \eqref{ThmApprox1} and the fact that
	$$
		E_{j-1}(f)_{B^s_{p,q}(\R^d)} \asymp \bigg(\sum_{\nu=j}^\infty (\nu^s E_{\nu-1}(f)_{L_p(\R^d)})^q \frac{1}{\nu} \bigg)^{1/q}, \qquad j \in \mathbb{N}_0
	$$
	(cf. \cite[Proposition 6.3]{Nilsson}). To be more precise, by Hardy's inequality \eqref{H2} (note that $b > 0$) and basic monotonicity properties, we have
	\begin{align*}
		\|f\|_{T^b_r B^s_{p, q}(\mathbb{R}^d)} & \asymp \bigg(\sum_{j=0}^\infty [2^{j b} E_{\mu_j}(f)_{B^s_{p,q}(\R^d)}]^r  \bigg)^{1/r} \\
		& \asymp \bigg(\sum_{j=0}^\infty 2^{j b r}  \bigg(\sum_{\nu=\mu_j}^\infty (\nu^s E_{\nu-1}(f)_{L_p(\R^d)})^q \frac{1}{\nu} \bigg)^{r/q}  \bigg)^{1/r} \\
		& \asymp \bigg(\sum_{j=0}^\infty 2^{j b r} \bigg(\sum_{\nu=\mu_j}^{\mu_{j+1}-1} (\nu^s E_{\nu-1}(f)_{L_p(\R^d)})^q \frac{1}{\nu} \bigg)^{r/q} \bigg)^{1/r} \\
		& \asymp \bigg(\sum_{j=0}^\infty 2^{j b r} \bigg(\sum_{\nu=2^j-1}^{2^{j+1}-2} 2^{\nu s q} E_{2^\nu-1}(f)_{L_p(\R^d)}^q \bigg)^{r/q} \bigg)^{1/r}.
	\end{align*}
	
	Next we deal with \eqref{ThmApprox2} under $b < 0$.  Since (cf. \cite[p. 322]{Nilsson})
	$$
		K(2^{j s}, f; L_p(\R^d), B^s_{p,q}(\R^d))  \asymp \bigg(\sum_{\nu=0}^{j}  2^{\nu s q} E_{2^\nu - 1}(f)_{L_{p}(\R^d)}^{q} \bigg)^{1/q}, \qquad j \in \mathbb{N}_0,
	$$
	it follows from Theorem \ref{TheoremInterpolation}(ii) (with $A^{s_0}_{p, q_0}(\R^d) = F^0_{p, 2}(\R^d) = L_p(\R^d)$ if $p \in (1, \infty)$) and \eqref{IntFormp=infty} (with $k=0$ if $p=1, \infty$) that
	\begin{align*}
		 \|f\|_{T^b_r B^s_{p, q}(\mathbb{R}^d)} &\asymp \bigg(\sum_{j=0}^\infty [(1 + j)^{b-1/r} K(2^{j s},f; B^s_{p,q}(\R^d), L_p(\R^d))]^r \bigg)^{1/r} \\
		 & \asymp \bigg(\sum_{j=0}^\infty (1 + j)^{b r}  \bigg(\sum_{\nu=0}^{j}  2^{\nu s q} E_{2^\nu - 1}(f)_{L_{p}(\R^d)}^{q} \bigg)^{r/q} \frac{1}{1+j} \bigg)^{1/r} \\
		 & \asymp \bigg(\sum_{j=0}^\infty 2^{j b r}  \bigg(\sum_{\nu=0}^{2^j}  2^{\nu s q} E_{2^\nu - 1}(f)_{L_{p}(\R^d)}^{q} \bigg)^{r/q} \bigg)^{1/r} \\
		 &\asymp \bigg(\sum_{j=0}^\infty 2^{j b r}  \bigg(\sum_{\nu=2^{j}-1}^{2^{j+1}-2}  2^{\nu s q} E_{2^\nu - 1}(f)_{L_{p}(\R^d)}^{q} \bigg)^{r/q} \bigg)^{1/r}
	\end{align*}
	where we have applied Hardy's inequality \eqref{H1} (taking into account that $b < 0$) in the last estimate.
\end{proof}

\begin{rem}
	The above proof also deals with the limiting case $b=0$ in \eqref{ThmApprox1} and \eqref{ThmApprox2}. More precisely, if $s > 0, 1 \leq p \leq \infty, 0 < q \leq \infty,$ and $0 < r < \infty$ then (cf. \eqref{33})
	$$
		\|f\|_{T^*_r B^s_{p, q}(\R^d)} \asymp \bigg(\sum_{j=0}^\infty [E_{\mu_j}(f)_{B^s_{p,q}(\R^d)}]^r  \bigg)^{1/r}
	$$
	and
	\begin{equation}\label{56}
		\|f\|_{T^*_r B^s_{p, q}(\R^d)} \asymp \bigg\{\sum_{j=0}^\infty \bigg[\sum_{\nu=2^j-1}^\infty (2^{\nu s} E_{2^\nu-1}(f)_{L_p(\R^d)})^q \bigg]^{r/q}  \bigg\}^{1/r}. 	
	\end{equation}
 Furthermore, if $b=0$ and $r=\infty$ then we have the trivial estimate (cf. \eqref{35})
	$$
		\|f\|_{B^s_{p, q}(\R^d)} = \sup_{j \in \N_0} E_{\mu_j}(f)_{B^s_{p,q}(\R^d)}=
 		\|f\|_{T^*_\infty B^s_{p, q}(\R^d)}\asymp \sup_{j \in \N_0} \bigg(\sum_{\nu=2^j-1}^\infty [2^{\nu s} E_{2^\nu-1}(f)_{L_p(\R^d)}]^q  \bigg)^{1/q}.
	$$
\end{rem}

For $p \in [1, \infty]$ we define the class\index{\bigskip\textbf{Sets}!$\mathfrak{U}_p$}\label{EXPFUN}
$$
	\mathfrak{U}_p := \{a = (a_k)_{k \in \N_0}: a_k \in L_p(\R^d), \quad \text{supp }\widehat{a}_k \subset \{x \in \R^d: |x| \leq 2^{k}\}\},$$
	which is formed by \emph{functions of exponential type} in $L_p(\R^d)$.
	Let $0 < q \leq \infty$ and $s > 0$. A classical representation theorem in Besov spaces asserts (see, e.g., \cite[Section 2.5.3]{Triebel83}) that $f \in B^s_{p,q}(\R^d)$ if and only if $f$ can be decomposed as
	\begin{equation}\label{Decomposition}
		f= \sum_{k=0}^\infty f_k  \quad \text{with} \quad (f_k)_{k \in \N_0} \in \mathfrak{U}_p
	\end{equation}
	(convergence being in $L_p(\R^d)$) such that
	\begin{equation}\label{Decomposition0}
		\sum_{k=0}^\infty 2^{k s q} \|f_k\|_{L_p(\R^d)}^q < \infty
	\end{equation}
	(with the usual modification if $q=\infty$). Furthermore
	\begin{equation}\label{Decomposition0nenene}
		\|f\|_{B^{s}_{p, q}(\R^d)} \asymp \inf  \left( \sum_{k=0}^\infty 2^{k s q} \|f_k\|_{L_p(\R^d)}^q  \right)^{\frac{1}{q}},
	\end{equation}
	where the infimum runs over all possible decompositions \eqref{Decomposition} satisfying \eqref{Decomposition0}.

\begin{thm}\label{TheoremRepresentation}
	Let $p \in [1, \infty], q, r \in (0,\infty], s > 0$ and $b \in \R \backslash \{0\}$.
	Then $f \in T^b_r B^s_{p, q}(\R^d)$ if and only if $f$ admits decomposition \eqref{Decomposition}
	 such that
	\begin{equation}\label{Decomposition2}
		\sum_{k=0}^\infty 2^{k b r} \bigg(\sum_{\nu= 2^k-1}^{2^{k+1}-2} (2^{\nu s} \|f_{\nu}\|_{L_p(\R^d)})^q \bigg)^{\frac{r}{q}} < \infty
	\end{equation}
	(with the usual modification if $q=\infty$ and/or $r=\infty$). Furthermore
	\begin{equation}\label{Decomposition2001001}
		\|f\|_{T^b_r B^s_{p, q}(\R^d)} \asymp \inf  \left( \sum_{k=0}^\infty 2^{k b r} \bigg(\sum_{\nu= 2^k-1}^{2^{k+1}-2} (2^{\nu s} \|f_{\nu}\|_{L_p(\R^d)})^q \bigg)^{\frac{r}{q}}  \right)^{\frac{1}{r}},
	\end{equation}
	where the infimum runs over all possible decompositions \eqref{Decomposition} satisfying \eqref{Decomposition2}.
\end{thm}

\begin{proof}
	Assume $f \in T^b_r B^s_{p, q}(\R^d)$. Let $(g_k)_{k \in \N_0} \in \mathfrak{U}_p$ be such that
	$$
		\|f-g_k\|_{L_p(\R^d)} \lesssim E_{2^k}(f)_{L_p(\R^d)}, \qquad k \in \N_0.
	$$
	Consequently $f = \lim_{k \to \infty} g_k$. Define
	$$
		f_0 = g_0 \quad \text{and} \quad f_k = g_k - g_{k-1}, \quad k \geq 1.
	$$
	Clearly $f = \sum_{k=0}^\infty f_k$ with $(f_k)_{k \in \N_0} \in \mathfrak{U}_p$. Furthermore
	$$
		\|f_0\|_{L_p(\R^d)} \lesssim \|f\|_{L_p(\R^d)} \quad \text{and} \quad   \|f_k\|_{L_p(\R^d)} \lesssim E_{2^{k-1}}(f)_{L_p(\R^d)}, \quad k \geq 1.
	$$
	Therefore, by \eqref{ThmApprox2},
	\begin{align*}
		\sum_{k=0}^\infty 2^{k b r} \bigg(\sum_{\nu= 2^k-1}^{2^{k+1}-2} (2^{\nu s} \|f_{\nu}\|_{L_p(\R^d)})^q \bigg)^{\frac{r}{q}} & \lesssim \sum_{k=0}^\infty 2^{k b r} \bigg(\sum_{\nu= 2^k-1}^{2^{k+1}-2} (2^{\nu s} E_{2^\nu-1}(f)_{L_p(\R^d)})^q \bigg)^{\frac{r}{q}} \\
		& \hspace{-4cm}\asymp \|f\|_{T^b_r B^s_{p, q}(\R^d)} < \infty.
			\end{align*}
			
			Conversely, assume $f$ can be decomposed as \eqref{Decomposition} with \eqref{Decomposition2}. Given $k \in \N_0$, we have
			$$
				E_{2^k}(f)_{L_p(\R^d)}  \lesssim \sum_{\nu=k}^\infty \|f_\nu\|_{L_p(\R^d)}.
			$$
			Inserting this estimate into \eqref{ThmApprox2} and applying Hardy's inequality \eqref{H2}, we get
			\begin{align*}
				\|f\|_{T^b_r B^s_{p, q}(\mathbb{R}^d)} &\lesssim  \bigg(\sum_{k=0}^\infty 2^{k b r} \bigg(\sum_{\nu=2^k-1}^{2^{k+1}-2} 2^{\nu s q}  \bigg(\sum_{l=\nu}^\infty \|f_l\|_{L_p(\R^d)} \bigg)^q \bigg)^{r/q} \bigg)^{1/r}  \\
				& \lesssim  \bigg(\sum_{k=0}^\infty 2^{k b r} \bigg(\sum_{\nu=2^k}^{2^{k+1}-1} 2^{\nu s q}   \|f_\nu\|_{L_p(\R^d)}^{q} \bigg)^{r/q} \bigg)^{1/r}.
			\end{align*}
			Taking now the infimum over all possible decompositions \eqref{Decomposition} with \eqref{Decomposition2}, we achieve the desired result.			
\end{proof}

\begin{rem}
	(i) Assume $b > 0$. Under the assumptions of Theorem \ref{TheoremRepresentation}, the condition \eqref{Decomposition2} implies automatically that the series $\sum_{k=0}^\infty f_k$ converges absolutely in $L_p(\R^d)$. For instance, if $q, r \geq 1$ then we can apply H\"older's inequality twice  so that
	\begin{align*}
		\sum_{k=0}^\infty \|f_k\|_{L_p(\R^d)} & = \sum_{k=0}^\infty \sum_{\nu=2^k-1}^{2^{k+1}-2} \|f_\nu\|_{L_p(\R^d)} \lesssim \sum_{k=0}^\infty \bigg(\sum_{\nu=2^k-1}^{2^{k+1}-2} (2^{\nu s} \|f_\nu\|_{L_p(\R^d)})^q \bigg)^{\frac{1}{q}} \\
		& \lesssim \bigg(\sum_{k=0}^\infty 2^{k b r} \bigg(\sum_{\nu= 2^k-1}^{2^{k+1}-2} (2^{\nu s} \|f_{\nu}\|_{L_p(\R^d)})^q \bigg)^{\frac{r}{q}} \bigg)^{\frac{1}{r}} < \infty.
	\end{align*}
	The proof is even easier if $q < 1$ or $r < 1$.
	
	(ii) Theorem \ref{TheoremRepresentation} can be complemented by the following decomposition theorem in terms of classical Besov spaces. Namely, let $p \in [1, \infty], q, r \in (0, \infty]$ and $s, b > 0$, the function $f \in T^b_r B^s_{p, q}(\R^d)$ if and only if $f$ can be decomposed as
	\begin{equation}\label{DecompositionBesov}
		f= \sum_{k=0}^\infty f_k
	\end{equation}
	(convergence in $B^s_{p, q}(\R^d)$) such that $\text{supp }\widehat{f}_k \subset \big\{x \in \R^d: |x| \leq \mu_k \big\}$ and
	\begin{equation}\label{Decomposition0Besov}
		\sum_{k=0}^\infty 2^{k b r} \|f_k\|_{B^s_{p, q}(\R^d)}^r < \infty
	\end{equation}
	(with the usual modification if $r=\infty$). Furthermore
	\begin{equation}\label{Decomposition0BesovNew}
		\|f\|_{T^b_r B^s_{p, q}(\R^d)} \asymp \inf  \left( \sum_{k=0}^\infty 2^{k b r} \|f_k\|_{B^s_{p, q}(\R^d)}^r\right)^{\frac{1}{r}},
	\end{equation}
	where the infimum runs over all possible decompositions \eqref{DecompositionBesov} satisfying \eqref{Decomposition0Besov}. This result is an immediate consequence of \eqref{ThmApprox1} and the representation theorem for approximation spaces given in \cite[Theorem 1]{FeherGrassler}.
	
	(iii) Theorem \ref{TheoremRepresentation} with $b=0$ can be obtained from similar ideas as the case $b \neq 0$ but now relying on \eqref{56}. More precisely, let $p \in [1, \infty], q, r \in (0, \infty]$ and $s > 0$. Then $f \in T^*_r B^s_{p, q}(\R^d)$ if and only if $f$ admits decomposition \eqref{Decomposition}
	 such that
	\begin{equation}\label{515}
		\sum_{k=0}^\infty \bigg(\sum_{\nu= 2^k-1}^{\infty} (2^{\nu s} \|f_{\nu}\|_{L_p(\R^d)})^q \bigg)^{\frac{r}{q}} < \infty
	\end{equation}
	(with the usual modification if $q=\infty$ and/or $r=\infty$). Furthermore
	\begin{equation}\label{516}
		\|f\|_{T^*_r B^s_{p, q}(\R^d)} \asymp \inf  \left( \sum_{k=0}^\infty \bigg(\sum_{\nu= 2^k-1}^{\infty} (2^{\nu s} \|f_{\nu}\|_{L_p(\R^d)})^q \bigg)^{\frac{r}{q}}  \right)^{\frac{1}{r}},
	\end{equation}
	where the infimum runs over all possible decompositions \eqref{Decomposition} satisfying \eqref{515}.
\end{rem}

\begin{cor}
	Let $p \in [1,\infty], \, q, r \in (0, \infty), s > 0$ and $b \in \R \backslash \{0\}$. Then entire functions of exponential type are dense in  $T^b_r B^s_{p, q}(\R^d)$ and $T^*_r B^s_{p, q}(\R^d)$.
\end{cor}

Let $(P_k)_{k \in \N}$ be a \emph{linear approximation scheme} on $L_p(\R^d)$, that is, $P_k$ is a linear operator acting on $f \in L_p(\R^d)$ such that $P_k f$ is an entire function of exponential type $k$ and
\begin{equation}\label{BestApproximation}
	\|f-P_k f\|_{L_p(\R^d)} \lesssim E_{k}(f)_{L_p(\R^d)}
\end{equation}
(i.e., $P_k f$ is an \emph{almost best approximant of $f$}\index{\bigskip\textbf{Operators}!$P_k f$}\label{BA}). Let $Q_k := P_{2^{k+1}-1}-P_{2^k}$ for $k \in \N$ and $Q_0 := P_2$. Then Theorem \ref{TheoremRepresentation} can be formulated as follows.

\begin{cor}\label{CorollaryLinearDecomposition}
		Let $p \in [1, \infty], q, r \in (0,\infty], \, s > 0$ and $b \in \R \backslash \{0\}$. Assume that $(P_k)_{k \in \N}$ is a linear approximation scheme on $L_p(\R^d)$.
	Then $f \in T^b_r B^s_{p, q}(\R^d)$ if and only if
	\begin{equation*}
		f= \sum_{k=0}^\infty Q_k f
	\end{equation*}
	(convergence being in $L_p(\R^d)$)  such that
	\begin{equation*}
		\sum_{k=0}^\infty 2^{k b r} \bigg(\sum_{\nu= 2^k-1}^{2^{k+1}-2} (2^{\nu s} \|Q_{\nu} f\|_{L_p(\R^d)})^q \bigg)^{\frac{r}{q}} < \infty
	\end{equation*}
	(with the usual modification if $q=\infty$ and/or $r=\infty$). Furthermore
	\begin{equation*}
		\|f\|_{T^b_r B^s_{p, q}(\R^d)} \asymp  \left( \sum_{k=0}^\infty 2^{k b r} \bigg(\sum_{\nu= 2^k-1}^{2^{k+1}-2} (2^{\nu s} \|Q_\nu f\|_{L_p(\R^d)})^q \bigg)^{\frac{r}{q}}  \right)^{\frac{1}{r}}.
	\end{equation*}
\end{cor}

\begin{rem}
	The corresponding result for $T^*_r B^s_{p, q}(\R^d)$ according to \eqref{516} also holds true, i.e.,
		\begin{equation*}
		\|f\|_{T^*_r B^s_{p, q}(\R^d)} \asymp  \left( \sum_{k=0}^\infty \bigg(\sum_{\nu= 2^k-1}^{\infty} (2^{\nu s} \|Q_\nu f\|_{L_p(\R^d)})^q \bigg)^{\frac{r}{q}}  \right)^{\frac{1}{r}}.
	\end{equation*}
\end{rem}

\begin{ex}[De la Valle\'e-Poussin means]
	\emph{A classical example of linear approximation scheme is given by \emph{de la Valle\'e-Poussin operators} for $f \in L_p(\R^d), \, p \in [1, \infty]$, namely,\index{\bigskip\textbf{Operators}!$V_t$}\label{VALLEEPOUSSIN}
\begin{equation}\label{DelaValleePoussin}
	V_t f(x) = (\chi (t^{-1}|\xi|) \widehat{f})^\vee(x), \qquad x \in \R^d, \qquad t > 0,
\end{equation}
where $\chi \in C^\infty[0, \infty)$ is such that $\chi(u)=1$ for $0 \leq u \leq 1$ and $\chi(u)=0$ for $u \geq 2$. Hence Corollary \ref{CorollaryLinearDecomposition} holds with $Q_k = V_{2^{k+1}-1}-V_{2^k}, \, k \geq 1,$ and $Q_0 = V_2$.}
\end{ex}

\begin{ex}[Lizorkin-type representations]
	\emph{A well-known fact in the theory of function spaces is that smooth partitions of unity in the Besov quasi-norms \eqref{Besov} can be replaced by characteristic functions related to cubes. The latter are called Lizorkin representations and have a long history, we refer the reader to \cite{Triebel83, SchmeisserTriebel} and the references within. As an application of Corollary \ref{CorollaryLinearDecomposition} we obtain Lizorkin representations for $T^b_r B^{s}_{p, q}(\mathbb{T}^d)$, the periodic counterpart of  $T^b_r B^{s}_{p, q}(\R^d)$.
 Indeed, consider the set of all trigonometric polynomials of (cubic) degree less than or equal to $k$ given by
	$$
		\sum_{|m|_\infty \leq k} c_m e^{i m \cdot x}, \qquad c_m \in \mathbb{C}
	$$
	where $m = (m_1, \ldots, m_d) \in \Z^d$ and $|m|_\infty = \max_{i = 1, \ldots, d} |m_i|$. It is well known (see, e.g., \cite[Corollary 3.5.2 and Theorem 3.5.7]{Grafakos}) that these sets determine a linear approximation scheme on $L_p(\mathbb{T}^d), \, 1 < p < \infty$. In virtue of Corollary \ref{CorollaryLinearDecomposition} (which also holds true for periodic functions and trigonometric polynomials) we obtain Lizorkin-type representations for $T^b_r B^{s}_{p, q}(\mathbb{T}^d)$ with $1 < p < \infty, 0 < q, r \leq \infty, s > 0$ and $b\neq 0$, namely,
		\begin{equation}\label{Lizorkin}
		\|f\|_{T^b_r B^{s}_{p, q}(\mathbb{T}^d)} \asymp  \left( \sum_{k=0}^\infty 2^{k b r} \bigg(\sum_{\nu= 2^k-1}^{2^{k+1}-2} 2^{\nu s q} \bigg\|\sum_{m \in K_\nu} \widehat{f}(m) e^{i m \cdot x} \bigg\|_{L_p(\Td^d)}^q \bigg)^{\frac{r}{q}}  \right)^{\frac{1}{r}}
	\end{equation}
	where
	$$
		K_\nu = \{m \in \Z^d: |m_i| < 2^{\nu +1}, i = 1, \ldots, d\} \backslash \{m \in \Z^d: |m_i| < 2^{\nu}, i = 1, \ldots, d\}
	$$
	for $\nu \in \N$ and $K_0 = \{(0, \ldots, 0)\}$. The equivalence \eqref{Lizorkin} is still valid for any $s \in \R$, but the proof will make use of lifting properties and it will be postponed until Corollary \ref{CorollaryLizorkinRep}.}
\end{ex}

\newpage

\section{Characterizations by differences}

\subsection{Characterizations of Besov spaces by moduli of smoothness}
Recall the characterization of the Fourier-analytically defined Besov spaces $B^{s, b}_{p, q}(\R^d)$ (cf. \eqref{Besov}) in terms of moduli of smoothness/differences: Let $1 \leq p \leq \infty, 0 < q \leq \infty, s > 0$ and $b \in \R$.  Then (cf. \eqref{BesovModuli})
$$
	B^{s, b}_{p,q}(\R^d) = \mathbf{B}^{s, b}_{p,q}(\R^d),
$$
i.e.,
\begin{align}
\|f\|_{B^{s, b}_{p,q}(\R^d)} &\asymp 	\|f\|_{L_p(\R^d)} +  \bigg(\int_0^1 (t^{-s} (1-\log t)^b \omega_k(f,t)_p)^q \frac{dt}{t} \bigg)^{1/q} \nonumber \\
& \asymp \|f\|_{L_p(\R^d)} + \bigg(\sum_{\nu=0}^\infty (2^{\nu s} (1 + \nu)^b \omega_k(f, 2^{-\nu})_p)^q \bigg)^{1/q}, \label{61}
\end{align}
where $k > s$. Furthermore (cf. \eqref{BesovDif})
\begin{equation}\label{62}
	\|f\|_{B^{s, b}_{p,q}(\R^d)} \asymp  \|f\|_{L_p(\R^d)} + \bigg(\int_{|h| < 1} |h|^{-s q - d} (1-\log |h|)^{b q} \|\Delta^k_h f\|_{L_p(\R^d)}^q dh \bigg)^{1/q}.
\end{equation}

In this section we extend \eqref{61}-\eqref{62} to the setting of $T^b_r B^s_{p, q}(\R^d)$. To be more precise, we obtain the following

%
%

\begin{thm}\label{TheoremModuli}
Let $1 \leq p \leq \infty, 0 < q, r \leq \infty, 0 < s < k$ and $b \in \R \backslash \{0\}$.  Then
    \begin{align}
        \|f\|_{T^b_r B^{s}_{p, q}(\R^d)}&\asymp \|f\|_{L_p(\mathbb{R}^d)} + \left(\int_0^1 (1-\log t)^{b r-1} \bigg(\int_{t^2}^t (u^{-s} \omega_k(f,u)_p)^q \frac{du}{u} \bigg)^{r/q} \frac{dt}{t} \right) \nonumber \\
        & \asymp \|f\|_{L_p(\mathbb{R}^d)} +  \left(\sum_{j=0}^\infty 2^{j b r} \bigg(\sum_{\nu=2^j}^{2^{j+1}-1}   (2^{\nu s} \omega_k(f,2^{-\nu})_p)^q \bigg)^{r/q} \right)^{1/r} \label{TheoremModuli1}
    \end{align}
    and
       \begin{equation}\label{TheoremModuli2}
        \|f\|_{T^b_r B^{s}_{p, q}(\R^d)}\asymp \|f\|_{L_p(\mathbb{R}^d)} +
        \left(\int_0^1 (1-\log t)^{b r -1} \bigg(\int_{t^2<|h| < t} |h|^{-s q -d} \|\Delta^k_h f\|_{L_p(\R^d)}^q \, dh  \bigg)^{r/q} \frac{dt}{t} \right)^{1/r}
    \end{equation}
    (with the usual modifications if $q=\infty$ and/or $r=\infty$).

\end{thm}


\begin{rem}
	Under the assumptions of Theorem \ref{TheoremModuli} with $r=q$ and applying Fubini's theorem, one recovers Proposition \ref{PropositionCoincidences}(i).
\end{rem}

\begin{proof}[Proof of Theorem \ref{TheoremModuli}]
The proof relies on limiting interpolation techniques. Assume first $b > 0$. According to Theorem \ref{TheoremInterpolation}(i) (under $p \in (1, \infty)$) and \eqref{IntFormp=1} (under $p=1, \infty$), the following formula holds
\begin{equation}\label{ProofTheoremModuli1}
	T^b_r B^s_{p, q}(\mathbb{R}^d) = (B^s_{p,q}(\R^d), W^{k}_p(\R^d))_{(0,b-1/r),r}.
\end{equation}
Furthermore, using Holmstedt's formula \cite{Holmstedt}, we can compute the $K$-functional for the couple $(B^s_{p,q}(\R^d), W^{k}_p(\R^d))$ in terms of the $\omega_k(f,t)_p$. More precisely, we have
$$
	K(t^{1-s/k}, f; B^s_{p,q}(\R^d), W^{k}_p(\R^d)) \asymp t^{1-s/k} \|f\|_{L_p(\R^d)}  + \bigg(\int_0^{t^{1/k}} (u^{-s} \omega_k(f,u)_p)^q \frac{du}{u} \bigg)^{1/q}
$$
for $t \in (0,1)$ (cf. \cite[Theorem 11.1(ii)]{DominguezTikhonov}). Inserting this into \eqref{ProofTheoremModuli1}, after a simple change of variables, yields
\begin{align}
	\|f\|_{T^b_r B^s_{p, q}(\mathbb{R}^d)} & \asymp \bigg(\int_0^1 [(1-\log t)^{b-1/r} K(t, f; B^s_{p,q}(\R^d), W^{k}_p(\R^d))]^r \frac{dt}{t} \bigg)^{1/r} \nonumber\\
	& \asymp \|f\|_{L_p(\R^d)} + \left(\int_0^1 (1-\log t)^{(b-1/r) r} \bigg(\int_0^{t} (u^{-s} \omega_k(f,u)_p)^q \frac{du}{u}\bigg)^{r/q}
        \frac{dt}{t}\right)^{1/r} \label{jsajajsajs1}
        \end{align}
        and, by monotonicity properties and Hardy's inequality \eqref{H2},
        \begin{align}
        	\|f\|_{T^b_r B^s_{p, q}(\mathbb{R}^d)} & \asymp \|f\|_{L_p(\R^d)} + \bigg(\sum_{j=0}^\infty 2^{j b r} \bigg(\int_0^{\lambda_j} (u^{-s} \omega_k(f,u)_p)^q \frac{du}{u}\bigg)^{r/q}  \bigg)^{1/q} \nonumber \\
	& \asymp \|f\|_{L_p(\R^d)} + \bigg(\sum_{j=0}^\infty 2^{j b r} \bigg( \int_{\lambda_{j+1}}^{\lambda_j} (u^{-s} \omega_k(f,u)_p)^q \frac{du}{u}\bigg)^{r/q}  \bigg)^{1/q}.  \label{jsajajsajs}
        \end{align}
        Here $\lambda_j = 2^{-2^j}$ for $j \geq 0$. After applying Lemma \ref{LemmaB1} (cf. \eqref{Disc1}), we achieve the first equivalence in \eqref{TheoremModuli1} under the assumption $b > 0$.

        Assume $b < 0$. In light of the interpolation formula (cf. Theorem \ref{TheoremInterpolation}(ii) and \eqref{IntFormp=infty})
        $$
        	 T^b_r B^s_{p, q}(\mathbb{R}^d) = (L_{p}(\R^d), B^s_{p,q}(\R^d))_{(1,b-1/r),r}
        $$
	and the fact that, for each $t \in (0,1)$,
	$$
		K(t^{s/k}, f; L_{p}(\R^d), B^s_{p,q}(\R^d)) \asymp t^{s/k} \|f\|_{L_p(\R^d)} + t^{s/k} \bigg(\int_{t^{1/k}}^1 (u^{-s} \omega_k(f,u)_p)^q \frac{dt}{t}  \bigg)^{1/q}
	$$
	(see, e.g., \cite[Theorem 11.1(i)]{DominguezTikhonov}), we obtain
	\begin{align}
		\|f\|_{T^b_r B^s_{p, q}(\mathbb{R}^d) } & \asymp \bigg(\int_0^1 [t^{-s/k} (1-\log t)^{b-1/r} K(t^{s/k}, f; L_{p}(\R^d), B^s_{p,q}(\R^d))]^r \frac{dt}{t} \bigg)^{1/r} \nonumber \\
		& \asymp \|f\|_{L_p(\R^d)} + \left(\int_0^1 (1-\log t)^{(b-1/r) r} \left(\int_t^1 (u^{-s} \omega_k(f,u)_p)^q \frac{du}{u}\right)^{r/q}
        \frac{dt}{t}\right)^{1/r}. \label{jsajajsajs2}
	\end{align}
	A similar reasoning as in \eqref{jsajajsajs} leads to
	$$
			\|f\|_{T^b_r B^s_{p, q}(\mathbb{R}^d) }  \asymp \|f\|_{L_p(\R^d)} + \bigg(\sum_{j=0}^\infty 2^{j b r} \bigg( \int_{\lambda_{j+1}}^{\lambda_j} (u^{-s} \omega_k(f,u)_p)^q \frac{du}{u}\bigg)^{r/q}  \bigg)^{1/q}
	$$
	and thus, by Lemma \ref{LemmaB1} (cf. \eqref{Disc1}),
	$$
	\|f\|_{T^b_r B^s_{p, q}(\mathbb{R}^d) }  \asymp \|f\|_{L_p(\R^d)} + \left(\int_0^1 (1-\log t)^{(b-1/r) r} \left(\int_{t^2}^t (u^{-s} \omega_k(f,u)_p)^q \frac{du}{u}\right)^{r/q}
        \frac{dt}{t}\right)^{1/r}.
	$$
	This completes the proof of the first equivalence in \eqref{TheoremModuli1}.
	
	The second equivalence in \eqref{TheoremModuli1} follows easily from the first one using monotonicity properties.

        By Lemma \ref{LemmaB1} (cf. \eqref{Disc2}), the estimate \eqref{TheoremModuli2} turns out to be equivalent to
            \begin{equation}\label{DiscNew}
        	\|f\|_{T^b_r B^s_{p, q}(\mathbb{R}^d)}  \asymp  \|f\|_{L_p(\R^d)} +  \left(\sum_{j=0}^\infty 2^{j b r} \bigg(\int_{\lambda_{j+1}<|h| \leq \lambda_j} |h|^{-s q -d} \|\Delta^k_h f\|_{L_p(\R^d)}^q \, dh  \bigg)^{r/q} \right)^{1/r}.
        \end{equation}
        Next we show \eqref{DiscNew}. To proceed with, we will make use of the following well-known estimate for the modulus of smoothness
        \begin{equation}\label{ModDiff}
        	\omega_k(f,u)_p \asymp \bigg(u^{-d} \int_{|h| \leq u} \|\Delta^k_h f\|_{L_p(\R^d)}^q \, dh \bigg)^{1/q}
        \end{equation}
        (see, e.g., \cite[(1.12)]{Kolomoitsev}). Applying Fubini's theorem, we get
        \begin{align*}
        \int_0^t (u^{-s} \omega_k(f,u)_p)^q \frac{du}{u} &\asymp \int_0^t u^{-s q-d} \int_{|h| \leq u} \|\Delta^k_h f\|_{L_p(\R^d)}^q \, dh \frac{du}{u} \\
        &\asymp \int_{|h| \leq t} \|\Delta^k_h f\|_{L_p(\R^d)}^q \int_{|h|}^t u^{-s q-d}  \frac{du}{u} \, dh \\
        & \asymp \int_{|h| \leq t} |h|^{-s q -d}  \|\Delta^k_h f\|_{L_p(\R^d)}^q \, dh.
        \end{align*}
        Combining this with \eqref{jsajajsajs1} and Hardy's inequality \eqref{H2}, if $b > 0$ then
        \begin{align*}
        	\|f\|_{T^b_r B^s_{p, q}(\mathbb{R}^d)} & \asymp  \|f\|_{L_p(\R^d)} + \left(\int_0^1 (1-\log t)^{(b-1/r) r} \bigg(\int_{|h| \leq t} |h|^{-s q -d}  \|\Delta^k_h f\|_{L_p(\R^d)}^q \, dh\bigg)^{r/q} \frac{dt}{t}\right)^{1/r}  \\
	& \asymp  \|f\|_{L_p(\R^d)} +  \left(\sum_{j=0}^\infty 2^{j b r} \bigg(\int_{\lambda_{j+1}<|h| \leq \lambda_j} |h|^{-s q -d} \|\Delta^k_h f\|_{L_p(\R^d)}^q \, dh  \bigg)^{r/q} \right)^{1/r}.
        \end{align*}

      It remains to show  \eqref{DiscNew} with $b < 0$. We can proceed as follows: For each $t \in (0,1)$, we apply \eqref{ModDiff} so that
	\begin{align*}
		\int_t^1 (u^{-s} \omega_k(f,u)_p)^q \frac{du}{u} & \asymp \int_t^1 u^{-s q-d} \int_{|h| \leq u} \|\Delta^k_h f\|_{L_p(\R^d)}^q \, dh \frac{du}{u} \\
		& \hspace{-3cm}\asymp t^{-s q -d}  \int_{|h| \leq t} \|\Delta^k_h f\|_{L_p(\R^d)}^q \, dh + \int_{t < |h| < 1} |h|^{-s q -d} \|\Delta^k_h f\|_{L_p(\R^d)}^q dh.
	\end{align*}
	Thus
	\begin{align*}
	  \left(\int_0^1 (1-\log t)^{(b-1/r) r} \left(\int_t^1 (u^{-s} \omega_k(f,u)_p)^q \frac{du}{u}\right)^{r/q}
        \frac{dt}{t}\right)^{1/r} & \asymp  \\
        & \hspace{-8cm}  \left(\int_0^1 t^{-s r} (1-\log t)^{(b-1/r) r}  \bigg(t^{-d}\int_{|h| \leq t} \|\Delta^k_h f\|_{L_p(\R^d)}^q \, dh \bigg)^{r/q}
        \frac{dt}{t}\right)^{1/r} \\
        & \hspace{-7cm} +  \left(\int_0^1 (1-\log t)^{(b-1/r) r} \left(\int_{t \leq |h| < 1} |h|^{-s q -d} \|\Delta^k_h f\|_{L_p(\R^d)}^q \, dh \right)^{r/q}
        \frac{dt}{t}\right)^{1/r} \\
        & \hspace{-8cm}=: I + II.
	\end{align*}
	Assume momentarily
	\begin{equation}\label{hshasga}
		I \lesssim II.
	\end{equation}
	By previous computations, we have
	$$
		  \left(\int_0^1 (1-\log t)^{(b-1/r) r} \left(\int_t^1 (u^{-s} \omega_k(f,u)_p)^q \frac{du}{u}\right)^{r/q}
        \frac{dt}{t}\right)^{1/r}  \asymp II
	$$
	and applying Hardy's inequality \eqref{H1} (recall that $b < 0$)
	\begin{align*}
		  \left(\int_0^1 (1-\log t)^{(b-1/r) r} \left(\int_t^1 (u^{-s} \omega_k(f,u)_p)^q \frac{du}{u}\right)^{r/q}
        \frac{dt}{t}\right)^{1/r}  &\asymp \\
        & \hspace{-8cm}  \left(\sum_{j=0}^\infty 2^{j b r} \bigg(\int_{\lambda_{j+1}<|h| \leq \lambda_j} |h|^{-s q -d} \|\Delta^k_h f\|_{L_p(\R^d)}^q \, dh  \bigg)^{r/q} \right)^{1/r}.
	\end{align*}
	Inserting now this estimate into \eqref{jsajajsajs2}, we get
	$$
		\|f\|_{T^b_r B^s_{p, q}(\R^d)} \asymp \|f\|_{L_p(\R^d)} + \left(\sum_{j=0}^\infty 2^{j b r} \bigg(\int_{\lambda_{j+1}<|h| \leq \lambda_j} |h|^{-s q -d} \|\Delta^k_h f\|_{L_p(\R^d)}^q \, dh  \bigg)^{r/q} \right)^{1/r},
	$$
	i.e.,  \eqref{DiscNew} holds.

	 To complete the proof of \eqref{DiscNew} with $b < 0$, it remains to prove \eqref{hshasga}.  Indeed, applying basic monotonicity properties and Hardy's inequality \eqref{H2},
	\begin{align*}
		I & \asymp \bigg(\sum_{l=0}^\infty 2^{l s r} (1 + l)^{(b-1/r) r} \bigg(2^{l d} \int_{|h| \leq 2^{-l}} \|\Delta^k_h f\|_{L_p(\R^d)}^q \, dh \bigg)^{r/q} \bigg)^{1/r} \\
		& =  \bigg(\sum_{l=0}^\infty 2^{l s r} (1 + l)^{(b-1/r) r} \bigg(2^{l d} \sum_{\nu=l}^\infty \int_{2^{-(\nu+1)} < |h| \leq 2^{-\nu}} \|\Delta^k_h f\|_{L_p(\R^d)}^q \, dh \bigg)^{r/q} \bigg)^{1/r} \\
		& \asymp \bigg(\sum_{l=0}^\infty 2^{l s r} (1 + l)^{(b-1/r) r} \bigg(2^{l d} \int_{2^{-(l+1)} < |h| \leq 2^{-l}} \|\Delta^k_h f\|_{L_p(\R^d)}^q \, dh  \bigg)^{r/q} \bigg)^{1/r} \\
		& \asymp \bigg(\sum_{l=0}^\infty  (1 + l)^{(b-1/r) r} \bigg( \int_{2^{-(l+1)} < |h| \leq 2^{-l}} |h|^{-s q-d} \|\Delta^k_h f\|_{L_p(\R^d)}^q \, dh  \bigg)^{r/q} \bigg)^{1/r} \\
		& \leq \bigg(\sum_{l=0}^\infty  (1 + l)^{(b-1/r) r} \bigg( \int_{2^{-(l+1)} < |h| \leq 1} |h|^{-s q-d} \|\Delta^k_h f\|_{L_p(\R^d)}^q \, dh  \bigg)^{r/q} \bigg)^{1/r} \\
		& \asymp II.
	\end{align*}
	

\end{proof}

\begin{rem}
	The method of proof can also be applied for the limiting case $b=0$ in Theorem \ref{TheoremModuli}; see Remark \ref{Remark4.4}. More precisely, if $1 \leq p \leq \infty, 0 < q \leq \infty, 0 < r < \infty$ and $0 < s < k$ then
	   \begin{equation*}
        \|f \|_{T^*_r B^s_{p, q}(\R^d)} \asymp \|f\|_{L_p(\mathbb{R}^d)} + \left(\int_0^1  \left(\int_0^t (u^{-s} \omega_k(f,u)_p)^q \frac{du}{u}\right)^{r/q}
        \frac{dt}{t (1-\log t)}\right)^{1/r}
    \end{equation*}
    and
        \begin{align*}
         \|f \|_{T^*_r B^s_{p, q}(\R^d)} & \asymp \|f\|_{L_p(\mathbb{R}^d)} \\
         &  + \left(\int_0^1\left(\int_{ |h| \leq t} |h|^{-s q -d} \|\Delta^k_h f\|_{L_p(\R^d)}^q \, dh \right)^{r/q}
        \frac{dt}{t (1-\log t)}\right)^{1/r}.
    \end{align*}
\end{rem}

\subsection{Gagliardo--Slobodecki\u{\i} characterizations}
Recall the well-known fact that the Besov spaces $B^s_{p,p}(\R^d)$ can be characterized in Gagliardo--Slobodecki\u{\i} form
\begin{equation}\label{GagliardoNorm}
		\|f\|_{B^s_{p,p}(\R^d)}  \asymp \|f\|_{L_p(\R^d)} + \bigg(\int_{\R^d} \int_{\R^d} \frac{|f(x)-f(y)|^p}{|x-y|^{d + s p}} \, dx \, dy \bigg)^{1/p}
\end{equation}
for $p \in [1, \infty)$ and $s \in (0,1)$. The space of all functions $f \in L_p(\R^d)$ such that the right-hand side of \eqref{GagliardoNorm} is finite is commonly denoted in the literature as $W^{s,p}(\R^d)$, the \emph{fractional Sobolev space}\index{\bigskip\textbf{Spaces}!$W^{s, b, p}(\R^d)$}\label{FRACTSOB}. Specifically, for $p \in [1, \infty), 0 < s < k \in \N$, and $b \in \R$,
\begin{equation}\label{GagliardoNorm535353}
		\|f\|_{W^{s, b, p}(\R^d)} := \|f\|_{L_p(\R^d)} + \bigg(\int_{\R^d} \int_{\R^d} \frac{(1+ |\log |h||)^{b p}|\Delta^k_h f (x)|^p}{|h|^{d + s p}} \, dx \, dh \bigg)^{1/p}.
\end{equation}
Recall that $\Delta^k_h$ denotes iterated differences (cf. \eqref{HigherDiff}) and one can show that finiteness of \eqref{GagliardoNorm535353} is independent of $k > s$. To simplify notation we simply write $W^{s, 0, p}(\R^d) = W^{s, p}(\R^d)$. Thus the extension of \eqref{GagliardoNorm} to higher order smoothness reads $B^s_{p, p}(\R^d) = W^{s, p}(\R^d)$, i.e.,
\begin{equation}\label{GagliardoNorm2}
		\|f\|_{B^s_{p,p}(\R^d)}  \asymp \|f\|_{L_p(\R^d)} + \bigg(\int_{\R^d} \int_{\R^d} \frac{|\Delta^k_h f (x)|^p}{|h|^{d + s p}} \, dx \, dh \bigg)^{1/p}
\end{equation}
where $k > s$.

 The characterization \eqref{GagliardoNorm} is very useful in applications to trace theory, differential geometry and PDE's (cf. \cite{Valdinoci} and the references given there). In particular, it provides a natural way to introduce Besov spaces in the more general setting of metric measure spaces \cite{Bourdon}.

Specializing now \eqref{TheoremModuli2} with $p=q$ we get the extension of  \eqref{GagliardoNorm} and \eqref{GagliardoNorm2} to the spaces $T^b_r B^{s}_{p,p}(\mathbb{R}^d)$.

\begin{thm}\label{TheoremGagliardo}
Let $1 \leq p < \infty, 0 < r \leq \infty, 0 < s < k$ and $b \in \R \backslash \{0\}$.  Let $\lambda_j = 2^{-2^j}$ for $j \geq 0$. Then
       \begin{equation*}
        \|f \|_{T^b_r B^{s}_{p,p}(\mathbb{R}^d)} \asymp \|f\|_{L_p(\mathbb{R}^d)} + \left(\int_0^1 (1-\log t)^{b r -1} \left(\int_{t^2< |h| \leq t}  \int_{\R^d} \frac{|\Delta^k_h f(x)|^p}{|h|^{d + s p}} \, dx \, dh \right)^{r/p} \frac{dt}{t}
     \right)^{1/r}
    \end{equation*}
    (with the usual modification if $r=\infty$).
\end{thm}

\begin{rem}
	The logarithmic counterpart of \eqref{GagliardoNorm2} reads
	$$
		B^{s, b}_{p, p}(\R^d) = W^{s, b, p}(\R^d), \qquad b \in \R.
	$$
	This is a simple consequence of Proposition \ref{PropositionCoincidences} and Theorem \ref{TheoremGagliardo} with $r=p$. Let $b \neq 0$, by Fubini's theorem,
	  \begin{align*}
        \|f \|_{B^{s,b}_{p,p}(\mathbb{R}^d)}& \asymp \|f\|_{L_p(\mathbb{R}^d)} + \left(\int_{|h| \leq 1}  \int_{\R^d} (1-\log |h|)^{b p}   \frac{|\Delta^k_h f(x)|^p}{|h|^{d + s p}} \, dx  \, dh\right)^{1/p} \\
	        & \asymp  \|f\|_{L_p(\mathbb{R}^d)} + \left(\int_{\R^d}  \int_{\R^d} (1+ |\log |h||)^{b p}   \frac{|\Delta^k_h f(x)|^p}{|h|^{d + s p}} \, dx  \, dh\right)^{1/p},
    \end{align*}
    where the last equivalence follows from the fact that $\|\Delta^k_h f\|_{L_p(\R^d)} \lesssim \|f\|_{L_p(\R^d)}$ (cf. \eqref{HigherDiff}).
\end{rem}

\newpage

\section{Characterization via wavelets}\label{SectionWavelets}

\subsection{Preliminaries}
We briefly discuss wavelet bases. For full treatment, we refer the reader to \cite{Daubechies}, \cite{Meyer} and \cite{Triebel08}. As usual, $C^{u}
(\mathbb{R})$ with $u \in \mathbb{N}$ collects all (complex-valued)
continuous functions on $\mathbb{R}$ having continuous bounded
derivatives up to order $u$. Let
\begin{equation}\label{4.1}
\psi_F \in C^{u} (\mathbb{R}), \quad \psi_M \in C^{u} (\mathbb{R}),
\qquad u \in \mathbb{N},
\end{equation}
be real compactly supported Daubechies wavelets with
\begin{equation*}
\int_{\mathbb{R}} \psi_M (x) \, x^v \, d x =0 \quad \text{for all
$v\in \mathbb{N}_0$ with $v<u$.}
\end{equation*}
Recall that $\psi_F$ is called the \emph{scaling function} (\emph{father wavelet})\index{\bigskip\textbf{Functionals and functions}!$\psi_F$}\label{FW}
and $\psi_M$ the \emph{associated wavelet} (\emph{mother wavelet})\index{\bigskip\textbf{Functionals and functions}!$\psi_M$}\label{MW}. The extension
of these wavelets from $\mathbb{R}$ to $\mathbb{R}^d, \, d \geq 2$, is based on the usual tensor procedure. Let
\begin{equation*}   
G = (G_1, \ldots, G_n ) \in G^0 = \{F,M \}^d,
\end{equation*}
which means that $G_r$ is either $F$ or $M$. Let
\begin{equation*}   
G = (G_1, \ldots, G_n ) \in G^j = \{F,M \}^{d *}, \qquad j \in
\mathbb{N},
\end{equation*}
which means that $G_r$ is either $F$ or $M$ where * indicates that
at least one of the components of $G$ must be an $M$. Hence $G^0$
has $2^d$ elements, whereas $G^j$ with $j \in \mathbb{N}$ has $2^d
-1$ elements. Let\index{\bigskip\textbf{Functionals and functions}!$\Psi^j_{G, m}$}\label{WAV}
\begin{equation}\label{DefinitionWavelet}
\Psi^j_{G,m} (x) = 2^{j d/2} \prod^n_{r=1} \psi_{G_r} (2^j x_r - m_r
), \quad G \in G^j, \quad m \in \mathbb{Z}^d, \quad j \in \mathbb{N}_0.
\end{equation}
We shall assume that $\psi_F$ and
$\psi_M$ in (\ref{4.1}) are normalized with respect to $L_2(\R)$. Then the system
\begin{equation*}   
\Psi =  \Big\{ \Psi^j_{G,m}: \ j \in \mathbb{N}_0, \ G\in G^j, \ m
\in \mathbb{Z}^d \Big\}
\end{equation*}
is an orthonormal basis in $L_2 (\mathbb{R}^d)$ and
\begin{equation*} 
f = \sum^\infty_{j=0} \sum_{G\in G^j} \sum_{m \in \mathbb{Z}^d}
\lambda^{j,G}_m \, 2^{-j d/2} \, \Psi^j _{G,m}
\end{equation*}
with
\begin{equation}\label{Wav1}   
\lambda^{j,G}_m = \lambda^{j,G}_m (f)=2^{j d/2} \int_{\mathbb{R}^d}
f(x) \, \Psi^j_{G,m} (x) \, d x,
\end{equation}
where\index{\bigskip\textbf{Operators}!$\lambda^{j, G}_m$}\label{WAVCOE}
$2^{-j d/2} \Psi^j_{G,m}$ are uniformly bounded functions (with
respect to $j$ and $m$).

 Let  $\chi_{j,m}$ be the characteristic function of the dyadic cube\index{\bigskip\textbf{Functionals and functions}!$\chi_{j, m}$}\label{CHARDYA}
$Q_{j,m} := 2^{-j}m + 2^{-j}(0,1)^d$ in $\mathbb{R}^d$\index{\bigskip\textbf{Sets}!$Q_{j, m}$}\label{DYACUB}  with sides of
length $2^{-j}$ parallel to the axes of coordinates and $2^{-j}m$ as
the lower left corner.

It is well known that, under certain conditions on the smoothness parameter $u$ (see \eqref{4.1}), classical Besov--Triebel--Lizorkin spaces admit characterizations via wavelet decompositions. Next we introduce the related sequence spaces. Let $p, q \in (0, \infty]$ and $s, \xi \in \R$. The space $b^{s,\xi}_{p,q}$\index{\bigskip\textbf{Spaces}!$b^{s, \xi}_{p, q}$}\label{BESOVSEQ} is the collection of all sequences $\lambda=(\lambda^{j,G}_m)$ with $j \in \mathbb{N}_0, G \in G^j$ and
$m \in \mathbb{Z}^d$ such that
\begin{equation}\label{Deffspaces**}
	\|\lambda\|_{b^{s,\xi}_{p,q}} = \left(\sum_{j=0}^\infty 2^{j(s-d/p) q} (1 + j)^{\xi q} \sum_{G \in G^j} \Big(\sum_{m \in \mathbb{Z}^d} |\lambda^{j, G}_m|^p \Big)^{q/p} \right)^{1/q} < \infty
\end{equation}
with the usual modification if $p=\infty$ and/or $q=\infty$. In the special case $\xi=0$ we simply write $b^s_{p,q}$. If $p < \infty$, we also define $f^s_{p, q}$ via\index{\bigskip\textbf{Spaces}!$f^s_{p, q}$}\label{TLSEQ}
\begin{equation}\label{DefSeqfClas}
	\|\lambda\|_{f^s_{p, q}} = \bigg\|\bigg(\sum_{j=0}^\infty \sum_{G \in G^j} \sum_{m \in \Z^d} 2^{j s q} |\lambda^{j, G}_m|^q \chi_{j, m} \bigg)^{1/q} \bigg\|_{L_p(\R^d)}
\end{equation}
with the usual modification if $q=\infty$.

For $p \in (0, \infty]$, let $\sigma_p := d (\frac{1}{p}-1)_+.$\index{\bigskip\textbf{Numbers and relations}!$\sigma_p$}\label{sigmap}

\begin{thm}\label{ThmWaveletsBesovClassic}
Let $p, q \in (0,\infty]$ and $s, \xi \in \R$. Assume that \eqref{4.1} holds with 
\begin{equation*}
u > \max\{s, \sigma_p-s\}.
\end{equation*}
Then $f \in B^{s,\xi}_{p,q}(\R^d)$ if and only if
	\begin{equation*}
	 f = \sum_{j \in \N_0,G \in G^j,m \in \Z^d} \lambda^{j,G}_m 2^{-j d/2}
    \Psi^j_{G,m},  \quad (\lambda^{j,G}_m) \in b^{s,\xi}_{p,q}
    \end{equation*}
     (unconditional convergence being in $\mathcal{S}'(\R^d)$). This representation is unique, that is, the wavelet coefficients $(\lambda^{j, G}_m)$ are given by \eqref{Wav1}, and the operator
     \begin{equation}\label{OperatorI}
     	I : f \mapsto (\lambda^{j,G}_m)
     \end{equation}
     defines an isomorphism from $B^{s,\xi}_{p,q}(\R^d)$ onto $b^{s,\xi}_{p,q}$. If, in addition, $p < \infty$ and $q < \infty$, then $\{\Psi^j_{G,m}\}$ is an unconditional basis in $B^{s,\xi}_{p,q}(\R^d)$.
\end{thm}	

The proof of the previous result may be found in \cite[Theorem 1.20]{Triebel08} (if $\xi=0$) and \cite{Almeida} (if $\xi \in \R$).

The corresponding result for $F^s_{p, q}(\R^d)$ reads as follows (cf. \cite[Theorem 1.20]{Triebel08}).

\begin{thm}\label{ThmWaveletsTLClassic}
	Let $p \in (0, \infty), \, q \in (0, \infty]$ and $s \in \R$. Assume that \eqref{4.1} holds with 
\begin{equation*}
u > \max\{s, \sigma_{p q}-s\},\qquad \sigma_{p q} = d \Big(\frac{1}{\min\{p, q\}}-1 \Big)_+.
\end{equation*}
Then $f \in F^{s}_{p,q}(\R^d)$ if and only if
	\begin{equation*}
	 f = \sum_{j \in \N_0,G \in G^j,m \in \Z^d} \lambda^{j,G}_m 2^{-j d/2}
    \Psi^j_{G,m},  \quad (\lambda^{j,G}_m) \in f^{s}_{p,q}
    \end{equation*}
     (unconditional convergence being in $\mathcal{S}'(\R^d)$). This representation is unique, that is, the wavelet coefficients $(\lambda^{j, G}_m)$ are given by \eqref{Wav1}, and the operator $I$ in \eqref{OperatorI} defines an isomorphism from $F^{s}_{p,q}(\R^d)$ onto $f^{s}_{p,q}$. If, in addition, $q < \infty$ then $\{\Psi^j_{G,m}\}$ is an unconditional basis in $F^{s}_{p,q}(\R^d)$.
\end{thm}


\subsection{Wavelet characterization of $T^\xi_r B^{s}_{p, q}(\R^d)$}The goal of this section is to provide the wavelet description of the spaces $T^\xi_r B^{s}_{p, q}(\R^d)$. With this in mind, we define the related sequence spaces $T^\xi_r b^{s}_{p,q}$ as follows.\index{\bigskip\textbf{Spaces}!$T^\xi_r b^s_{p, q}$}\label{TRUBESOVSEQ}

\begin{defn}
	Let $p, q, r \in (0,\infty]$ and $s, \xi \in \R$. The space $T^\xi_r b^s_{p, q}$ is the collection of all
	\begin{equation}\label{sequence}
	\lambda = \{\lambda^{j, G}_m \in \mathbb{C} : j \in \N_0, \quad G \in G^j, \quad m \in \Z^d \}
	\end{equation}
	 such that
	\begin{equation}\label{DefBesSeq}
		\|\lambda\|_{T^\xi_r b^s_{p, q}} = \bigg(\sum_{k=0}^\infty 2^{k \xi r} \bigg(\sum_{j=2^k-1}^{2^{k+1}-2} 2^{j(s-d/p) q} \sum_{G \in G^j} \Big(\sum_{m \in \mathbb{Z}^d} |\lambda^{j, G}_m|^p \Big)^{q/p} \bigg)^{r/q} \bigg)^{1/r} < \infty
	\end{equation}
	(with the usual modification if $p=\infty$ and/or $q=\infty$ and/or $r=\infty$).
\end{defn}

\begin{rem}\label{RemarkSequenceCoincidence}
	Setting $q=r$ in \eqref{DefBesSeq}, it is not hard to see that the sequence spaces $T^\xi_r b^s_{p, q}$ contain as special cases the classical spaces $b^{s,\xi}_{p,q}$ (cf. \eqref{Deffspaces**}), that is,
$		T^\xi_q b^{s}_{p, q} = b^{s, \xi}_{p,q}.
$	
\end{rem}

\begin{thm}\label{ThmWaveletsNewBesov}
	Let $p, q, r \in (0,\infty], s \in \R$ and $\xi \in \R \backslash \{0\}$. Assume that \eqref{4.1} holds with $u > \max\{s, \sigma_p-s\}$. Then $f \in T^\xi_r B^s_{p, q}(\R^d)$ if and only if
		\begin{equation}\label{representation}
	 f = \sum_{j \in \N_0,G \in G^j,m \in \Z^d} \lambda^{j,G}_m 2^{-j d/2}
    \Psi^j_{G,m},  \quad (\lambda^{j,G}_m) \in T^\xi_r b^s_{p, q}
    \end{equation}
    (unconditional convergence in $\mathcal{S}'(\R^d)$). This representation is unique, that is, the wavelet coefficients $(\lambda^{j, G}_m)$ are given by \eqref{Wav1}, and the operator $I$ in \eqref{OperatorI} defines an isomorphism from $T^\xi_r B^s_{p, q}(\R^d)$ onto $T^\xi_r b^s_{p, q}$. If, in addition, $p < \infty, q < \infty, r < \infty$, then $\{\Psi^j_{G,m}\}$ is an unconditional basis in $T^\xi_r B^s_{p, q}(\R^d)$.
\end{thm}

\begin{rem}
	Let $q=r$. In light of Remark \ref{RemarkSequenceCoincidence}, Theorems \ref{ThmWaveletsBesovClassic} and \ref{ThmWaveletsNewBesov}, we recover Proposition \ref{PropositionCoincidences}.
\end{rem}

\begin{proof}[Proof of Theorem \ref{ThmWaveletsNewBesov}]
 Let $\xi > 0$. By assumptions, we may choose $s_0 > s$ such that $u > \max \{\sigma_p-s, s_0\}.$ Therefore, by Theorem \ref{ThmWaveletsBesovClassic}, the operator $I$ defines an isomorphism between $B^s_{p,q}(\R^d)$ (respectively, $B^{s_0}_{p,q}(\R^d)$) and $b^s_{p,q}$ (respectively, $b^{s_0}_{p,q}$). Applying the limiting interpolation method, we derive that
$$
	I : (B^s_{p,q}(\R^d), B^{s_0}_{p,q}(\R^d))_{(0,\xi-1/r),r} \to (b^s_{p,q}, b^{s_0}_{p,q})_{(0,\xi-1/r),r}
$$
is also an isomorphism. Furthermore, by Theorem \ref{TheoremInterpolation}(i), we have
$$
(B^s_{p,q}(\R^d), B^{s_0}_{p,q}(\R^d))_{(0,\xi-1/r),r} = T^\xi_r B^s_{p, q}(\R^d).
$$
We make the following claim
\begin{equation}\label{InterpolationBesovSequences}
	(b^s_{p,q}, b^{s_0}_{p,q})_{(0,\xi-1/r),r} = T^\xi_r b^s_{p, q}.
\end{equation}
Assuming momentarily the validity of such a formula, we would establish that
$$
	I: T^\xi_r B^s_{p, q}(\R^d) \to T^\xi_r b^s_{p, q}
$$
is an isomorphism. Moreover, the uniqueness of the representation \eqref{representation} and its unconditional convergence in $\mathcal{S}'(\R^d)$  are simple consequences of the embeddings $T^\xi_r B^s_{p, q}(\R^d) \hookrightarrow B^s_{p,q}(\R^d)$ (cf. Proposition \ref{PropositionElementary}(iii)) and $T^\xi_r b^s_{p, q} \hookrightarrow b^s_{p,q}$ (cf. \eqref{InterpolationBesovSequences} and \eqref{EmbeddingsLimInt}) together with the corresponding assertions for classical Besov spaces given in Theorem \ref{ThmWaveletsBesovClassic}. The unconditionality of $\{\Psi^j_{G,m}\}$ in $T^\xi_r B^s_{p, q}(\R^d), \, p, q, r < \infty,$ follows easily from \eqref{DefBesSeq}.

It remains to prove \eqref{InterpolationBesovSequences}. We first note that $b^s_{p,q}$ (respectively, $b^{s_0}_{p,q}$) can be (isometrically) identified with the sequence space $\ell^{s-d/p}_q(\ell_q(\ell_p))$ (respectively, $\ell^{s_0-d/p}_q(\ell_q(\ell_p))$) (see \eqref{Deffspaces**}). Since
\begin{align}
	K(t, \lambda; \ell^{s-d/p}_q(\ell_q(\ell_p)), \ell^{s_0-d/p}_q( \ell_q(\ell_p)))  &\asymp  \nonumber\\
	& \hspace{-4cm} \bigg(
        \sum_{j=0}^\infty [\min (1, 2^{j (s_0-s)} t) 2^{j(s-d/p)} \| \{\lambda^{j,G}_m\} \|_{\ell_q(\ell_p)}]^q \bigg)^{1/q} \label{KfunctMixedSeqSpaces}
\end{align}
for $t > 0$ (see \eqref{Kfunct}), we can apply Hardy's inequality \eqref{H1} to obtain
\begin{align*}
	\|\lambda\|_{(\ell^{s-d/p}_q(\ell_q(\ell_p)), \ell^{s_0-d/p}_q( \ell_q(\ell_p)))_{(0,\xi-1/r),r}} & \asymp \\
	& \hspace{-4cm} \bigg(\sum_{k=0}^\infty (1+ k)^{(\xi-1/r) r} K(2^{-k (s_0-s)}, \lambda; \ell^{s-d/p}_q(\ell_q(\ell_p)), \ell^{s_0-d/p}_q( \ell_q(\ell_p)))^r \bigg)^{1/r} \\
	& \hspace{-4cm} \asymp \bigg(\sum_{k=0}^\infty (1+ k)^{(\xi-1/r) r}  \bigg(
        \sum_{j=0}^\infty [\min (1, 2^{(j-k)(s_0-s)}) 2^{j(s- d/p)} \| \{\lambda^{j,G}_m\} \|_{\ell_q(\ell_p)}]^q \bigg)^{r/q} \bigg)^{1/r} \\
        & \hspace{-4cm} \asymp  \bigg(\sum_{k=0}^\infty 2^{-k(s_0-s) r} (1+ k)^{(\xi-1/r) r}  \bigg(
        \sum_{j=0}^k [2^{j(s_0- d/p)} \| \{\lambda^{j,G}_m\} \|_{\ell_q(\ell_p)}]^q \bigg)^{r/q} \bigg)^{1/r} \\
        & \hspace{-3.5cm} +  \bigg(\sum_{k=0}^\infty (1+ k)^{(\xi-1/r) r}  \bigg(
        \sum_{j=k}^\infty [ 2^{j(s- d/p)} \| \{\lambda^{j,G}_m\} \|_{\ell_q(\ell_p)}]^q \bigg)^{r/q} \bigg)^{1/r} \\
            & \hspace{-4cm} \asymp  \bigg(\sum_{k=0}^\infty  2^{k(s- d/p) r}  (1+ k)^{(\xi-1/r) r} \| \{\lambda^{k,G}_m\} \|_{\ell_q(\ell_p)}^r  \bigg)^{1/r} \\
        & \hspace{-3.5cm} +  \bigg(\sum_{k=0}^\infty (1+ k)^{(\xi-1/r) r}  \bigg(
        \sum_{j=k}^\infty [ 2^{j(s- d/p)} \| \{\lambda^{j,G}_m\} \|_{\ell_q(\ell_p)}]^q \bigg)^{r/q} \bigg)^{1/r} \\
        & \hspace{-4cm} \asymp \bigg(\sum_{k=0}^\infty (1+ k)^{(\xi-1/r) r}  \bigg(
        \sum_{j=k}^\infty [ 2^{j(s- d/p)} \| \{\lambda^{j,G}_m\} \|_{\ell_q(\ell_p)}]^q \bigg)^{r/q} \bigg)^{1/r}.
\end{align*}
Furthermore, by Hardy's inequality \eqref{H2} (recall $\xi > 0$)
\begin{align*}
	\bigg(\sum_{k=0}^\infty (1+ k)^{(\xi-1/r) r}  \bigg(
        \sum_{j=k}^\infty [ 2^{j(s- d/p)} \| \{\lambda^{j,G}_m\} \|_{\ell_q(\ell_p)}]^q \bigg)^{r/q} \bigg)^{1/r} & \asymp  \\
        & \hspace{-8cm}\bigg(\sum_{k=0}^\infty 2^{k \xi r}  \bigg(
       \sum_{j=2^k-1}^{2^{k+1}-2} [ 2^{j(s- d/p)} \| \{\lambda^{j,G}_m\} \|_{\ell_q(\ell_p)}]^q \bigg)^{r/q} \bigg)^{1/r} = \|\lambda\|_{T^\xi_r b^{s}_{p, q}}.
\end{align*}
Taking into account this estimate, the above computations lead to
$$
\|\lambda\|_{(\ell^{s-d/p}_q(\ell_q(\ell_p)), \ell^{s_0-d/p}_q( \ell_q(\ell_p)))_{(0,\xi-1/r),r}}  \asymp 	\|\lambda\|_{T^\xi_r b^{s}_{p, q}},
$$
i.e., the desired formula \eqref{InterpolationBesovSequences} holds true.

 The proof with $\xi < 0$ can be obtained in a similar fashion, but now applying Theorem \ref{TheoremInterpolation}(ii). In particular, we have to show that
\begin{equation}\label{InterpolationBesovSequences2}
	(b^{s_0}_{p,q}, b^{s}_{p,q})_{(1,\xi-1/r),r} = T^\xi_r b^s_{p, q},
\end{equation}
where $s_0 < s$ is chosen so that $u > \max\{\sigma_p-s_0, s\}$. Indeed, it follows from \eqref{KfunctMixedSeqSpaces} and Hardy's inequalities \eqref{H1}-\eqref{H2} that
\begin{align*}
	\|\lambda\|_{(\ell^{s_0-d/p}_q(\ell_q(\ell_p)), \ell^{s-d/p}_q( \ell_q(\ell_p)))_{(1,\xi-1/r),r}} & \asymp \\
	& \hspace{-4cm} \bigg(\sum_{k=0}^\infty (1+ k)^{(\xi -1/r)r} K(2^{k (s-s_0)}, \lambda; \ell^{s-d/p}_q(\ell_q(\ell_p)), \ell^{s_0-d/p}_q( \ell_q(\ell_p)))^r \bigg)^{1/r} \\
	& \hspace{-4cm} \asymp \bigg(\sum_{k=0}^\infty (1+ k)^{(\xi-1/r) r}  \bigg(
        \sum_{j=0}^\infty [\min (1, 2^{-(j-k)(s-s_0)}) 2^{j(s- d/p)} \| \{\lambda^{j,G}_m\} \|_{\ell_q(\ell_p)}]^q \bigg)^{r/q} \bigg)^{1/r} \\
        & \hspace{-4cm} \asymp  \bigg(\sum_{k=0}^\infty  (1+ k)^{(\xi-1/r) r}  \bigg(
        \sum_{j=0}^k [2^{j(s- d/p)} \| \{\lambda^{j,G}_m\} \|_{\ell_q(\ell_p)}]^q \bigg)^{r/q} \bigg)^{1/r} \\
        & \hspace{-3.5cm} +  \bigg(\sum_{k=0}^\infty 2^{k(s-s_0) r} (1+ k)^{(\xi-1/r) r}  \bigg(
        \sum_{j=k}^\infty [2^{j(s_0- d/p)} \| \{\lambda^{j,G}_m\} \|_{\ell_q(\ell_p)}]^q \bigg)^{r/q} \bigg)^{1/r} \\
            & \hspace{-4cm} \asymp   \bigg(\sum_{k=0}^\infty  (1+ k)^{(\xi -1/r)r}  \bigg(
        \sum_{j=0}^k [2^{j(s- d/p)} \| \{\lambda^{j,G}_m\} \|_{\ell_q(\ell_p)}]^q \bigg)^{r/q} \bigg)^{1/r} \\
        & \hspace{-3.5cm} +  \bigg(\sum_{k=0}^\infty 2^{k(s-d/p) r} (1+ k)^{(\xi-1/r) r}   \| \{\lambda^{k,G}_m\} \|_{\ell_q(\ell_p)}^r  \bigg)^{1/r} \\
        & \hspace{-4cm} \asymp \bigg(\sum_{k=0}^\infty (1+ k)^{(\xi-1/r) r}  \bigg(
        \sum_{j=0}^k [ 2^{j(s- d/p)} \| \{\lambda^{j,G}_m\} \|_{\ell_q(\ell_p)}]^q \bigg)^{r/q} \bigg)^{1/r} \\
        & \hspace{-4cm} \asymp \bigg(\sum_{k=0}^\infty 2^{k \xi r} \bigg(\sum_{j=2^k-1}^{2^{k+1}-2}  [ 2^{j(s- d/p)} \| \{\lambda^{j,G}_m\} \|_{\ell_q(\ell_p)}]^q   \bigg)^{r/q} \bigg)^{1/r} = \|\lambda \|_{T^\xi_r b^s_{p, q}}.
\end{align*}
Hence \eqref{InterpolationBesovSequences2} holds true.
\end{proof}

\begin{rem}\label{RemDelicate}
	The case $\xi=0$ in Theorem  \ref{ThmWaveletsNewBesov} is more delicate. In fact, the above proof still works in this case and shows that,  under the assumptions of Theorem  \ref{ThmWaveletsNewBesov},
	$$
		f \in T^*_r B^s_{p,q}(\R^d), \qquad r < \infty,
	$$
	if and only if\index{\bigskip\textbf{Spaces}!$T^*_r b^s_{p, q}$}\label{LIMTRUNBESEQ}
	$$
	\|\lambda\|_{T^*_r b^s_{p, q}}:=	\bigg(\sum_{k=0}^\infty \bigg(
        \sum_{j=k}^\infty 2^{j(s- d/p) q}  \sum_{G \in G^j} \Big(\sum_{m \in \mathbb{Z}^d} |\lambda^{j, G}_m|^p \Big)^{q/p} \bigg)^{r/q} \frac{1}{1+k} \bigg)^{1/r} < \infty.
	$$
	Furthermore
	$$
		\|f\|_{ T^*_r B^s_{p,q}(\R^d)} \asymp \|\lambda\|_{T^*_r b^s_{p, q}}.
	$$
	Note that $T^*_r b^s_{p, q} \hookrightarrow T_r b^s_{p, q} = T_r^0 b^s_{p, q}$ (cf. \eqref{DefBesSeq}), but the converse embedding fails to be true.
\end{rem}

\begin{rem}\label{RemHaar}
	The method of proof of Theorem \ref{ThmWaveletsNewBesov} is flexible enough so that it can be applied to deal with another types of wavelets. For instance, under natural assumptions on the involved parameters, one can obtain characterizations of the spaces $T^\xi_r B^s_{p, q}(\mathbb{R}^d)$ via Haar wavelets. We shall not record here the construction of Haar wavelet bases and we refer the interested reader to \cite[Section 2.5.1]{Triebel08} and \cite[Section 2.3]{Triebel10} for further explanations and related literature. Let $\{H^j_{G,m}\}$ be an orthonormal Haar wavelet basis in $L_2(\R^d)$.\index{\bigskip\textbf{Functionals and functions}!$H^{j}_{G, m}$}\label{HAAR}
\end{rem}
	
\begin{thm}
	Let $p, q, r \in (0,\infty], s \in \R$ and $\xi \in \R \backslash \{0\}$. Assume that
	$$
		\max \Big\{d\Big(\frac{1}{p} - 1\Big), \frac{1}{p} - 1 \Big\} < s < \min \Big\{\frac{1}{p}, 1\Big\}.
	$$
	Then $f \in T^\xi_r B^s_{p, q}(\R^d)$ if and only if
		\begin{equation*}
	 f = \sum_{j \in \N_0,G \in G^j,m \in \Z^d} \lambda^{j,G}_m 2^{-j d/2}
    H^j_{G,m},  \quad (\lambda^{j,G}_m) \in T^\xi_r b^s_{p, q}
    \end{equation*}
    (unconditional convergence being in $\mathcal{S}'(\R^d)$). This representation is unique in the sense that the wavelet coefficients $(\lambda^{j, G}_m)$ are given by
    \begin{equation*}
    	\lambda^{j, G}_m = 2^{j d/2} \int_{\R^d} f(x) H^j_{G,m}(x) \, dx
    \end{equation*}
    and the operator
     \begin{equation*}
     	I : f \mapsto (\lambda^{j,G}_m)
     \end{equation*}
     defines an isomorphism from $T^\xi_r B^s_{p, q}(\R^d)$ onto $T^\xi_r b^s_{p, q}$. If, in addition, $p < \infty, q < \infty, r < \infty$, then $\{H^j_{G,m}\}$ is an unconditional basis in $T^\xi_r B^s_{p, q}(\R^d)$.
\end{thm}

The proof of this result follows the same lines as the proof of Theorem \ref{ThmWaveletsNewBesov} and can be safely left to the interested reader.

\subsection{Wavelet characterization of $T^\xi_r F^s_{p, q}(\R^d)$}

We first introduce the sequence spaces $T^\xi_r f^s_{p, q}$ related to the classical spaces $f^s_{p, q}$ (cf. \eqref{DefSeqfClas}).\index{\bigskip\textbf{Spaces}!$T^\xi_r f^s_{p, q}$}\label{TTLSEQ}

\begin{defn}\label{Definitionfseqspaces}
	Let $p \in (0, \infty), \, q, r \in (0,\infty], \, s \in \R$ and $\xi \in \R$. The space $T^\xi_r f^s_{p, q}$ is the collection of all
	\begin{equation*}
	\lambda = \{\lambda^{j, G}_m \in \mathbb{C} : j \in \N_0, \quad G \in G^j, \quad m \in \Z^d \}
	\end{equation*}
	 such that
	\begin{equation*}
		\|\lambda\|_{T^\xi_r f^{s}_{p,q}} = \bigg(\sum_{k=0}^\infty 2^{k \xi r} \bigg\|\bigg(\sum_{j=2^k-1}^{2^{k+1}-2}\sum_{G \in G^j} \sum_{m \in \Z^d} 2^{j s q} |\lambda^{j, G}_m|^q \chi_{j, m} \bigg)^{1/q}\bigg\|_{L_p(\R^d)}^{r} \bigg)^{1/r} < \infty
	\end{equation*}
	(with the usual modification if $q=\infty$ and/or $r=\infty$).
\end{defn}

Next we show that $T^\xi_r F^s_{p, q}(\R^d)$ can be identified with the sequence spaces $T^\xi_r f^s_{p, q}$ via the wavelet isomorphism.

\begin{thm}\label{ThmWaveletsNewTriebelLizorkin}
	Let $p \in (1, \infty), q, r \in (0,\infty], s \in \R$ and $\xi \in \R \backslash \{0\}$. Assume that \eqref{4.1} holds with $u > \max\{s, \sigma_{p q}-s\}$. Then $f \in T^\xi_r F^s_{p, q}(\R^d)$ if and only if
		\begin{equation*}
	 f = \sum_{j \in \N_0,G \in G^j,m \in \Z^d} \lambda^{j,G}_m 2^{-j d/2}
    \Psi^j_{G,m},  \quad (\lambda^{j,G}_m) \in T^\xi_r f^s_{p, q}
    \end{equation*}
    (unconditional convergence being in $\mathcal{S}'(\R^d)$). This representation is unique, that is, the wavelet coefficients $(\lambda^{j, G}_m)$ are given by \eqref{Wav1}, and the operator $I$ in \eqref{OperatorI} defines an isomorphism from $T^\xi_r f^s_{p, q}(\R^d)$ onto $T^\xi_r f^s_{p, q}$. If, in addition, $q < \infty, r < \infty$, then $\{\Psi^j_{G,m}\}$ is an unconditional basis in $T^\xi_r f^s_{p, q}(\R^d)$.
\end{thm}

\begin{rem}
	The previous result provides, in particular, a unified approach to the wavelet description of $\mathbf{B}^{0, \xi}_{p, q}(\R^d)$ and $\text{Lip}^{s, \xi}_{p, q}(\R^d)$ (cf. Proposition \ref{PropositionCoincidences}). These spaces were investigated from different perspectives in \cite[Theorem 5.5]{CobosDominguezTriebel} and \cite[Theorem 6.5]{DominguezHaroskeTikhonov}, respectively. Next we recall some known characterizations following the notation used there. Let $p \in (1, \infty)$ and $q \in (0, \infty]$. Then $f \in \mathbf{B}^{0, \xi}_{p, q}(\R^d), \, \xi > -1/q,$ if and only if\index{\bigskip\textbf{Spaces}!$\mathbf{b}^{0, \xi}_{p, q}$}\label{BESOVSEQZERO}
	\begin{equation*}
		\|(\lambda^{j, G}_m)\|_{\textbf{b}^{0, \xi}_{p, q}} := \bigg(\sum_{k=0}^\infty (1 + k)^{\xi q} \bigg\|\bigg(\sum_{j=k}^{\infty}\sum_{G \in G^j} \sum_{m \in \Z^d}  |\lambda^{j, G}_m|^2 \chi_{j, m} \bigg)^{1/2}\bigg\|_{L_p(\R^d)}^{q} \bigg)^{1/q}
	\end{equation*}
	is finite. In this case
	\begin{equation}\label{BWav}
		\|f\|_{ \mathbf{B}^{0, \xi}_{p, q}(\R^d)} \asymp \|(\lambda^{j, G}_m)\|_{\textbf{b}^{0, \xi}_{p, q}}.
	\end{equation}
		On the other hand, $f \in \text{Lip}^{s, \xi}_{p, q}(\R^d), \, s > 0, \, \xi < -1/q$, if and only if\index{\bigskip\textbf{Spaces}!$\text{lip}^{s, \xi}_{p, q}$}\label{LIPSEQ}
		\begin{equation*}
		\|(\lambda^{j, G}_m)\|_{\text{lip}^{s, \xi}_{p, q}} :=  \bigg(\sum_{k=0}^\infty (1+k)^{\xi q} \bigg\|\bigg(\sum_{j=0}^{k}\sum_{G \in G^j} \sum_{m \in \Z^d} 2^{j s 2}  |\lambda^{j, G}_m|^2 \chi_{j, m} \bigg)^{1/2}\bigg\|_{L_p(\R^d)}^{q} \bigg)^{1/q}
	\end{equation*}
	is finite. In this case
	\begin{equation}\label{LWav}
		\|f\|_{ \text{Lip}^{s, \xi}_{p, q}(\R^d)} \asymp \|(\lambda^{j, G}_m)\|_{\text{lip}^{s, \xi}_{p, q}}.
	\end{equation}
	
	It is not hard do check that both spaces $\textbf{b}^{0, \xi}_{p, q}$ and $\text{lip}^{s, \xi}_{p, q}$ fit into the scale of truncated $f$-spaces introduced in Definition \ref{Definitionfseqspaces}. Specifically
	$$
		T^{\xi+1/q}_q f^{s}_{p, 2} =  \left\{\begin{array}{cl} \textbf{b}^{0, \xi}_{p, q}  & \text{if} \quad s = 0 \quad \text{and }\quad  \xi > -\frac{1}{q}, \\
		 & \\
		  \text{lip}^{s, \xi}_{p, q}& \text{if} \quad s> 0 \quad \text{and}  \quad \xi < -\frac{1}{q}.
		       \end{array}
                        \right.
	$$
	Accordingly, one can combine Theorem \ref{ThmWaveletsNewTriebelLizorkin} and Proposition \ref{PropositionCoincidences} to show \eqref{BWav} and \eqref{LWav}.

\end{rem}

\begin{proof}[Proof of Theorem \ref{ThmWaveletsNewTriebelLizorkin}]
We shall only deal with the case $\xi > 0$, since the case $\xi < 0$ can be obtained following similar ideas. Without loss of generality, we may choose $s_ 0 > s$ such that $u > \max \{s_0, \sigma_{p, q}-s_0\}$. Thus, by Theorem \ref{ThmWaveletsTLClassic}, the spaces $F^s_{p, q}(\R^d)$ and $F^{s_0}_{p, q}(\R^d)$ are isomorphic to the sequence spaces $f^s_{p, q}$ and $f^{s_0}_{p, q}$, respectively, and invoking the retraction theorem for interpolation (cf. \cite[Theorem 1.2.4, p. 22]{Triebel}), the spaces $(F^{s}_{p, q}(\R^d), F^{s_0}_{p, q}(\R^d))_{(0, \xi -1/r), r}$ and $(f^s_{p, q}, f^{s_0}_{p, q})_{(0, \xi -1/r), r}$ can be identified. According to Theorem \ref{TheoremInterpolationF}, the latter can be rephrased as $T^\xi_r F^s_{p, q}(\R^d)$ can be identified with $(f^s_{p, q}, f^{s_0}_{p, q})_{(0, \xi -1/r), r}$ via the wavelet map $I$ (cf. \eqref{OperatorI}). Next we compute this interpolation space using the well-known fact that $f^s_{p, q}$ can be identified with a complemented subspace of $L_p(\R^d; \ell^s_q(\ell_q))$, where
	$$
		\|(f_{j, m})\|_{L_p(\R^d; \ell^s_q(\ell_q))} := \bigg\|\bigg(\sum_{j=0}^\infty 2^{j s q} \sum_{m \in \Z^d} |f_{j, m}|^q \bigg)^{1/q} \bigg\|_{L_p(\R^d)}.
	$$
	Hence the identification of the space $(f^s_{p, q}, f^{s_0}_{p, q})_{(0, \xi-1/r), r}$ can be reduced to compute $$(L_p(\R^d; \ell^s_q(\ell_q)), L_p(\R^d; \ell^{s_0}_q(\ell_q)))_{(0, \xi-1/r), r},$$
	see \cite[Section 1.17.1, pp. 118-119]{Triebel}. To estimate the corresponding $K$-functional, recall that (cf. \eqref{ProofIntF1} and \eqref{ProofIntF2})
	\begin{equation}\label{KfunctInt1}
		K(t, f; L_p(\R^d; A_0), L_p(\R^d; A_1)) \asymp \bigg(\int_{\R^d} K(t, f(x); A_0, A_1)^p \, dx \bigg)^{1/p}
	\end{equation}
	and
	\begin{equation}\label{KfunctInt2}
		K(t, (\xi_j); \ell^s_q(\ell_q), \ell^{s_0}_q(\ell_q)) \asymp \bigg(\sum_{j=0}^\infty [\min\{2^{j s}, t 2^{j s_0}\} \|\xi_j\|_{\ell_q(\Z^d)}]^q \bigg)^{1/q}.
	\end{equation}
	Combining \eqref{KfunctInt1} and \eqref{KfunctInt2}, we obtain
	\begin{align*}
		K(t, (f_j); L_p(\R^d; \ell^s_q(\ell_q)), L_p(\R^d; \ell^{s_0}_q(\ell_q))) &  \asymp \\
		& \hspace{-5cm} \left(\int_{\R^d} \bigg(\sum_{j=0}^\infty [\min\{2^{j s}, t 2^{j s_0}\} \|f_j(x)\|_{\ell_q(\Z^d)}]^q \bigg)^{p/q} \, dx  \right)^{1/p}
	\end{align*}
	and thus
	\begin{align*}
		\|(f_j)\|_{(L_p(\R^d; \ell^s_q(\ell_q)), L_p(\R^d; \ell^{s_0}_q(\ell_q)))_{(0, \xi-1/r), r}} & \asymp  \\
		& \hspace{-6.5cm} \left(\sum_{\nu=0}^\infty (1+ \nu)^{(\xi-1/r) r} K(2^{-\nu (s_0 - s)}, (f_j); L_p(\R^d; \ell^s_q(\ell_q)), L_p(\R^d; \ell^{s_0}_q(\ell_q)))^r \right)^{1/r} \\
		& \hspace{-6.5cm} \asymp  \left(\sum_{\nu=0}^\infty (1+ \nu)^{(\xi-1/r) r} \left(\int_{\R^d} \bigg(\sum_{j=0}^\infty [\min\{1, 2^{(j-\nu) (s_0-s)}\} 2^{j s} \|f_j(x)\|_{\ell_q(\Z^d)}]^q \bigg)^{p/q} \, dx  \right)^{r/p}   \right)^{1/r} \\
		& \hspace{-6.5cm} \asymp  \left(\sum_{\nu=0}^\infty 2^{-\nu (s_0 -s) r} (1+ \nu)^{(\xi-1/r) r} \left(\int_{\R^d} \bigg(\sum_{j=0}^\nu [2^{j s_0} \|f_j(x)\|_{\ell_q(\Z^d)}]^q \bigg)^{p/q} \, dx  \right)^{r/p}   \right)^{1/r} \\
		& \hspace{-6cm} + \left(\sum_{\nu=0}^\infty (1+ \nu)^{(\xi-1/r) r} \left(\int_{\R^d} \bigg(\sum_{j=\nu}^\infty [2^{j s} \|f_j(x)\|_{\ell_q(\Z^d)}]^q \bigg)^{p/q} \, dx  \right)^{r/p}   \right)^{1/r} \\
		& \hspace{-6.5cm} =: I + II.
	\end{align*}
	We claim that
	\begin{equation}\label{ClaimProofWavelet}
		I \lesssim II.
	\end{equation}
	Assume momentarily that \eqref{ClaimProofWavelet} is true. Therefore, monotonicity properties imply
	\begin{align*}
		\|(f_j)\|_{(L_p(\R^d; \ell^s_q(\ell_q)), L_p(\R^d; \ell^{s_0}_q(\ell_q)))_{(0, \xi-1/r), r}} & \asymp \\
		& \hspace{-5cm} \left(\sum_{\nu=0}^\infty 2^{\nu \xi r} \bigg\| \bigg(\sum_{j=2^{\nu}-1}^\infty 2^{j s q} \|f_j(x)\|_{\ell_q(\Z^d)}^q \bigg)^{1/q}  \bigg\|_{L_p(\R^d)}^r   \right)^{1/r}.
	\end{align*}
	Hence the proof will be finished if we show that
	\begin{align}
		\left(\sum_{\nu=0}^\infty 2^{\nu \xi r} \bigg\| \bigg(\sum_{j=2^{\nu}-1}^\infty 2^{j s q} \|f_j(x)\|_{\ell_q(\Z^d)}^q \bigg)^{1/q}  \bigg\|_{L_p(\R^d)}^r   \right)^{1/r} & \asymp \nonumber \\
		& \hspace{-5cm} \left(\sum_{\nu=0}^\infty 2^{\nu \xi r} \bigg\| \bigg(\sum_{j=2^{\nu}-1}^{2^{\nu + 1} -2} 2^{j s q} \|f_j(x)\|_{\ell_q(\Z^d)}^q \bigg)^{1/q}  \bigg\|_{L_p(\R^d)}^r   \right)^{1/r}. \label{ClaimProofWavelet2}
	\end{align}
	
	It remains to show \eqref{ClaimProofWavelet} and \eqref{ClaimProofWavelet2}. Concerning the former one,   we shall separate two possible cases. Firstly, if $p \geq q$ we can apply Minkowski's inequality so that
	$$
		\left(\int_{\R^d} \bigg(\sum_{j=0}^\nu [2^{j s_0} \|f_j(x)\|_{\ell_q(\Z^d)}]^q \bigg)^{p/q} \, dx  \right)^{1/p} \leq \left(\sum_{j=0}^\nu 2^{j s_0 q} \|f_j\|^q_{L^p(\R^d; \ell_{q}(\Z^d))} \right)^{1/q},
	$$
	which yields
	\begin{equation*}
		I \leq  \left(\sum_{\nu=0}^\infty 2^{-\nu (s_0 -s) r} (1+ \nu)^{(\xi-1/r) r} \left(\sum_{j=0}^\nu 2^{j s_0 q} \|f_j\|^q_{L^p(\R^d; \ell_{q}(\Z^d))} \right)^{r/q}  \right)^{1/r}.
		\end{equation*}
	It follows from Hardy's inequality \eqref{H1} (noting $s_0 > s$) that
		\begin{equation*}
		I  \lesssim \left(\sum_{\nu=0}^\infty 2^{\nu s r} (1 + \nu)^{(\xi-1/r) r} \|f_\nu\|_{L^p(\R^d; \ell_q(\Z^d))}^r \right)^{1/r} \leq II.
	\end{equation*}
	This completes the proof of \eqref{ClaimProofWavelet} under the assumption $p \geq q$.
	
	The proof of \eqref{ClaimProofWavelet} with $p < q$ follows similar ideas as above and thus only a sketchy proof will be provided. We have
	\begin{align*}
		I & \leq  \left(\sum_{\nu=0}^\infty 2^{-\nu (s_0 -s) r} (1+ \nu)^{(\xi-1/r) r} \bigg(\sum_{j=0}^\nu 2^{j s_0 p} \|f_j\|_{L_p(\R^d; \ell_q(\Z^d))}^p  \bigg)^{r/p}   \right)^{1/r}
	\end{align*}
	and, by \eqref{H1},
	$$
		I \lesssim \left(\sum_{\nu=0}^\infty 2^{\nu s r} (1+ \nu)^{(\xi-1/r) r} \|f_j\|_{L_p(\R^d; \ell_q(\Z^d))}^r   \right)^{1/r} \leq II.
	$$
	
	To show \eqref{ClaimProofWavelet2}, we can argue as follows. Using the fact that $L_p(\R^d), \, 0 < p \leq \infty,$ satisfies the $\min\{p, 1\}$-triangle inequality (i.e., $\|f_1 + f_2\|_{L_p(\R^d)}^{\min\{p, 1\}} \leq \|f_1\|_{L_p(\R^d)}^{\min\{p, 1\}} +   \|f_2\|_{L_p(\R^d)}^{\min\{p, 1\}}$) and \eqref{H2} (note that $\xi > 0$)
	\begin{align*}
	\left(\sum_{\nu=0}^\infty 2^{\nu \xi r} \bigg\| \bigg(\sum_{j=2^{\nu}-1}^\infty 2^{j s q} \|f_j(x)\|_{\ell_q(\Z^d)}^q \bigg)^{1/q}  \bigg\|_{L_p(\R^d)}^r   \right)^{1/r}  & =  \\
	& \hspace{-7cm}\left(\sum_{\nu=0}^\infty 2^{\nu \xi r} \bigg\| \sum_{j=2^{\nu}-1}^\infty 2^{j s q} \|f_j(x)\|_{\ell_q(\Z^d)}^q   \bigg\|_{L_{p/q}(\R^d)}^{r/q}   \right)^{1/r}  \\
	& \hspace{-7cm} \leq  \left(\sum_{\nu=0}^\infty 2^{\nu \xi r}\bigg( \sum_{l= \nu}^\infty \bigg\| \sum_{j=2^l-1}^{2^{l+1}-2} 2^{j s q} \|f_j(x)\|_{\ell_q(\Z^d)}^q   \bigg\|^{\min\{p/q, 1\}}_{L_{p/q}(\R^d)} \bigg)^{r/(q \min\{p/q, 1\})}   \right)^{1/r} \\
	& \hspace{-7cm } \lesssim  \bigg(\sum_{\nu=0}^\infty  2^{\nu \xi r} \bigg\| \sum_{j=2^\nu-1}^{2^{\nu+1}-2} 2^{j s q} \|f_j(x)\|_{\ell_q(\Z^d)}^q   \bigg\|_{L_{p/q}(\R^d)}^{r/q} \bigg)^{1/r} \\
	& \hspace{-7cm} =  \bigg(\sum_{\nu=0}^\infty  2^{\nu \xi r} \bigg\|\bigg( \sum_{j=2^\nu-1}^{2^{\nu+1}-2} 2^{j s q} \|f_j(x)\|_{\ell_q(\Z^d)}^q \bigg)^{1/q}  \bigg\|_{L_{p}(\R^d)}^{r} \bigg)^{1/r}.
	\end{align*}
	The proof is complete.
\end{proof}

\begin{rem}
	A similar comment to Remark \ref{RemDelicate} applies to wavelet characterizations of truncated Triebel--Lizorkin spaces. To be more precise, under the assumptions of Theorem  \ref{ThmWaveletsNewTriebelLizorkin} with $r < \infty$, $f \in T^*_r F^s_{p, q}(\R^d)$ (cf. \eqref{34}) if and only if $\lambda \in T^*_r f^s_{p, q}$ where\index{\bigskip\textbf{Spaces}!$T^*_r f^s_{p, q}$}\label{LIMTRUNTLSEQ}
	$$
		\|\lambda\|_{T^*_r f^s_{p, q}} := \bigg(\sum_{k=0}^\infty \bigg\|\bigg(\sum_{j=2^k-1}^{\infty}\sum_{G \in G^j} \sum_{m \in \Z^d} 2^{j s q} |\lambda^{j, G}_m|^q \chi_{j, m} \bigg)^{1/q}\bigg\|_{L_p(\R^d)}^{r} \bigg)^{1/r}.
	$$
	Clearly $T^*_r f^s_{p, q} \hookrightarrow T_r f^s_{p, q} = T^0_r f^s_{p, q}$ (cf. Definition \ref{Definitionfseqspaces}), but the converse embedding is no longer true.
\end{rem}

\begin{rem}\label{RemarkIntp}
	The restriction $p > 1$ in Theorem  \ref{ThmWaveletsNewTriebelLizorkin} comes only from Theorem \ref{TheoremInterpolationF}. In fact, the method of proof of Theorem  \ref{ThmWaveletsNewTriebelLizorkin} still works if $p \in (0, \infty)$ to show that
	$$
		(F^s_{p, q}(\R^d), A^{s_0}_{p, q_0}(\R^d))_{(0, \xi -1/r), r}, \qquad A \in \{B, F\},
	$$
	with
	$\xi > 0$ and $s_0 > s$,
	coincides with the set of all $f \in \mathcal{S}'(\R^d)$ which admits a wavelet representation
		\begin{equation*}
	 f = \sum_{j \in \N_0,G \in G^j,m \in \Z^d} \lambda^{j,G}_m 2^{-j d/2}
    \Psi^j_{G,m},  \quad (\lambda^{j,G}_m) \in T^\xi_r f^s_{p, q}
    \end{equation*}
    (unconditional convergence being in $\mathcal{S}'(\R^d)$). Furthermore
    $$
    	\|f\|_{(F^s_{p, q}(\R^d), A^{s_0}_{p, q_0}(\R^d))_{(0, \xi -1/r), r}} \asymp \| ( \lambda^{j,G}_m )\|_{ T^\xi_r f^s_{p, q}}.
    $$
    A similar result holds true for $(A^{s_0}_{p, q_0}(\R^d), F^s_{p, q}(\R^d))_{(1, \xi-1/r), r}$ with $\xi < 0$ and $s_0 < s$.
\end{rem}

\begin{rem}
	In the same vein as Remark \ref{RemHaar}, the method of proof of Theorem \ref{ThmWaveletsNewTriebelLizorkin} can be easily adapted to deal with Haar wavelet basis for $T^\xi_r F^s_{p, q}(\R^d)$. Next we only state the result and leave the proof to the interested reader.
\end{rem}

\begin{thm}
	Let $p \in (1, \infty), q \in (0, \infty), r \in (0,\infty], s \in \R$ and $\xi \in \R \backslash \{0\}$. Assume that
\begin{equation}\label{AssumUncond}
	\max\bigg\{d\Big(\frac{1}{q}-1 \Big), \frac{1}{q}-1, \frac{1}{p}-1\bigg\}<s < \min\bigg\{\frac{1}{p}, \frac{1}{q} \bigg\}.
\end{equation}
	Then $f \in T^\xi_r F^s_{p, q}(\R^d)$ if and only if
		\begin{equation*}
	 f = \sum_{j \in \N_0,G \in G^j,m \in \Z^d} \lambda^{j,G}_m 2^{-j d/2}
    H^j_{G,m},  \quad (\lambda^{j,G}_m) \in T^\xi_r f^s_{p, q}
    \end{equation*}
    (unconditional convergence being in $\mathcal{S}'(\R^d)$). This representation is unique in the sense that the wavelet coefficients $(\lambda^{j, G}_m)$ are given by
    \begin{equation*}
    	\lambda^{j, G}_m = 2^{j d/2} \int_{\R^d} f(x) H^j_{G,m}(x) \, dx
    \end{equation*}
    and the operator
     \begin{equation*}
     	I : f \mapsto (\lambda^{j,G}_m)
     \end{equation*}
     defines an isomorphism from $T^\xi_r F^s_{p, q}(\R^d)$ onto $T^\xi_r f^s_{p, q}$. If, in addition, $r < \infty$, then $\{H^j_{G,m}\}$ is an unconditional basis in $T^\xi_r F^s_{p, q}(\R^d)$.
\end{thm}

The assumptions \eqref{AssumUncond} in the previous result are natural, in the sense that they characterize the unconditionality of Haar wavelet basis in $F^s_{p, q}(\R^d)$ under $p \in (1, \infty)$ and $q \in (0, \infty)$. A detailed analysis of several properties of the Haar system in Besov and Triebel-Lizorkin spaces was carried out in \cite{SeegerUllrichb, Garrigos18}.

\newpage

\section{Lifting property}

If $\sigma \in \R$, then the \emph{lifting operator} $I_\sigma$ is defined by\index{\bigskip\textbf{Operators}!$I_\sigma$}\label{LIFT}
\begin{equation}\label{LiftingDef}
	I_\sigma f := [(1 + |\xi|^2)^{\frac{\sigma}{2}} \widehat{f}]^\vee, \qquad f \in \mathcal{S}'(\R^d).
\end{equation}
It is well known that $I_\sigma$ is a one-to-one mapping from $\mathcal{S}'(\R^d)$ onto $\mathcal{S}'(\R^d)$.

\subsection{Lifting in truncated Besov and Triebel-Lizorkin spaces}
A basic result in the theory of function spaces is the so-called lifting property of Besov spaces (cf. \cite[Section 2.3.8]{Triebel83}, see also \cite[Proposition 1.8]{Moura}), which asserts that $I_\sigma$ acts as an isomorphism from $B^{s, b}_{p, q}(\R^d)$ onto $B^{s-\sigma, b}_{p, q}(\R^d)$ and
\begin{equation}\label{ClassicLift}
	\|I_\sigma f\|_{B^{s-\sigma, b}_{p, q}(\R^d)} \asymp \|f\|_{B^{s, b}_{p, q}(\R^d)}.
\end{equation}
Here $s, b \in \R$ and $p, q \in (0, \infty]$. The analogue for $F$-spaces reads as follows. Let $s, b \in \R, p \in (0, \infty)$ and $q \in (0, \infty]$, then
\begin{equation}\label{ClassicLiftTL}
	\|I_\sigma f\|_{F^{s-\sigma, b}_{p, q}(\R^d)} \asymp \|f\|_{F^{s, b}_{p, q}(\R^d)}.
\end{equation}

The next result shows the lifting property for the spaces  $T^b_r B^s_{p, q}(\mathbb{R}^d)$ and $T^b_r F^s_{p, q}(\mathbb{R}^d)$.

\begin{thm}\label{TheoremLifting}
Let $A \in \{B, F\}$. Let $s, \sigma, b \in \R$ and $p, q, r \in (0, \infty] \, (p < \infty \text{ if } A = F)$. Then $I_\sigma$ acts as an isomorphism from $T^b_r A^s_{p, q}(\R^d)$ onto $T^b_r A^{s-\sigma}_{p, q}(\R^d)$ and
\begin{equation*}
	\|I_\sigma f\|_{T^b_r A^{s-\sigma}_{p, q}(\R^d)} \asymp \|f\|_{T^b_r A^s_{p, q}(\R^d)}.
\end{equation*}
\end{thm}

\begin{proof}
	
	 By \eqref{RemFM}, \eqref{ClassicLift} and \eqref{ClassicLiftTL}, we have
		\begin{align*}
		\|I_\sigma f\|_{T^b_r A^{s-\sigma}_{p, q}(\R^d)} & \asymp  \bigg(\sum_{j=0}^\infty 2^{j b r} \bigg\|\bigg(\sum_{\nu=2^j-1}^{2^{j+1}-2} \varphi_\nu \widehat{I_\sigma f} \bigg)^\vee \bigg\|_{A^{s-\sigma}_{p, q}(\R^d)}^r \bigg)^{1/r} \\
		& =  \bigg(\sum_{j=0}^\infty 2^{j b r} \bigg\|I_\sigma\bigg[ \bigg(\sum_{\nu=2^j-1}^{2^{j+1}-2} \varphi_\nu \widehat{f} \bigg)^\vee \bigg] \bigg\|_{A^{s-\sigma}_{p, q}(\R^d)}^r \bigg)^{1/r} \\
		& \asymp  \bigg(\sum_{j=0}^\infty 2^{j b r} \bigg\|\bigg(\sum_{\nu=2^j-1}^{2^{j+1}-2} \varphi_\nu \widehat{f} \bigg)^\vee \bigg\|_{A^{s}_{p, q}(\R^d)}^r \bigg)^{1/r} \asymp \|f\|_{T^b_r A^s_{p, q}(\R^d)}.
	\end{align*}

%
\end{proof}

\begin{rem}\label{RemarkLifting}
	The counterpart of Theorem \ref{TheoremLifting} in terms of the periodic liftings\index{\bigskip\textbf{Operators}!$\mathfrak{I}_\sigma$}\label{LIFTPER}
	$$
		\mathfrak{I}_\sigma : f \mapsto \sum_{m \in \Z^d} (1 + |m|^2)^{\sigma/2} \widehat{f}(m) e^{i m \cdot x}, \qquad \sigma \in \R,
	$$
	and the periodic spaces $T^b_r A^{s}_{p, q}(\mathbb{T}^d), \, A \in \{B, F\}$, also holds true.
\end{rem}

As an immediate consequence of Theorem \ref{TheoremLifting} we are now able to extend the Lizorkin representations \eqref{Lizorkin} for $T^b_r B^{s}_{p, q}(\mathbb{T}^d)$ from $s > 0$ to $s \in \R$.

\begin{cor}\label{CorollaryLizorkinRep}
	Let $1 < p < \infty, 0 < q, r \leq \infty, s \in \R$ and $b \neq 0$. Then
		\begin{equation*}
		\|f\|_{T^b_r B^{s}_{p, q}(\mathbb{T}^d)} \asymp  \left( \sum_{k=0}^\infty 2^{k b r} \bigg(\sum_{\nu= 2^k-1}^{2^{k+1}-2} 2^{\nu s q} \bigg\|\sum_{m \in K_\nu} \widehat{f}(m) e^{i m \cdot x} \bigg\|_{L_p(\mathbb{T}^d)}^q \bigg)^{\frac{r}{q}}  \right)^{\frac{1}{r}}
	\end{equation*}
	where
	$$
		K_\nu = \{m \in \Z^d: |m_i| < 2^{\nu +1}, i = 1, \ldots, d\} \backslash \{m \in \Z^d: |m_i| < 2^{\nu}, i = 1, \ldots, d\}
	$$
	for $\nu \in \N$ and $K_0 = \{(0, \ldots, 0)\}$.
\end{cor}

\begin{proof}
	The case $s > 0$ was already shown in \eqref{Lizorkin}, so that it only remains to deal with the case $s \leq 0$. Given $s \leq 0$, we choose $\sigma \in \R$ such that $s > \sigma$ and invoke Theorem \ref{TheoremLifting} (cf. also Remark \ref{RemarkLifting}) to get
	\begin{align*}
		\|f\|_{T^b_r B^{s}_{p, q}(\mathbb{T}^d)} &\asymp \|\mathfrak{I}_\sigma f\|_{T^b_r B^{s-\sigma}_{p, q}(\mathbb{T}^d)}  \\
		&\hspace{-1cm} \asymp   \left( \sum_{k=0}^\infty 2^{k b r} \bigg(\sum_{\nu= 2^k-1}^{2^{k+1}-2} 2^{\nu (s-\sigma) q} \bigg\|\sum_{m \in K_\nu}  (1 + |m|^2)^{\sigma/2} \widehat{f}(m) e^{i m \cdot x} \bigg\|_{L_p(\mathbb{T}^d)}^q \bigg)^{\frac{r}{q}}  \right)^{\frac{1}{r}} \\
		&\hspace{-1cm} \asymp   \left( \sum_{k=0}^\infty 2^{k b r} \bigg(\sum_{\nu= 2^k-1}^{2^{k+1}-2} 2^{\nu s q} \bigg\|\sum_{m \in K_\nu}  \widehat{f}(m) e^{i m \cdot x} \bigg\|_{L_p(\mathbb{T}^d)}^q \bigg)^{\frac{r}{q}}  \right)^{\frac{1}{r}}.
	\end{align*}
\end{proof}

\subsection{Lifting in $\mathbf{B}^{0, b}_{p, q}(\R^d)$ and $\text{Lip}^{s, b}_{p, q}(\R^d)$}\label{Section8.2}

Unlike $B^{s, b}_{p, q}(\R^d)$ (cf. \eqref{ClassicLift}), the lifting property in the setting of spaces  $\mathbf{B}^{0, b}_{p, q}(\R^d)$ is a delicate issue, which was already considered in \cite[Section 13.2]{DominguezTikhonov}. There it is shown that, for $p \in (1, \infty), q \in (0, \infty], b > -1/q$ and $\sigma \in \R$,
\begin{equation}\label{LiftB01}
I_\sigma : \mathbf{B}^{0, b}_{p, q}(\R^d) \to B^{-\sigma, b+ 1/\max\{2, p, q\}}_{p, q}(\R^d)
\end{equation}
and
\begin{equation}\label{LiftB02}
	I_\sigma : B^{\sigma, b+ 1/\min\{2, p, q\}}_{p, q}(\R^d) \to \mathbf{B}^{0, b}_{p, q}(\R^d).
\end{equation}
Furthermore, as shown in \cite[Propositions 13.8 and 13.9]{DominguezTikhonov}, these results are sharp in the sense that the additional logarithmic smoothness $1/\max\{2, p, q\}$ and $1/\min\{2, p, q\}$ can not be improved. In particular, these assertions tell us that classical scale formed by $B^{s, b}_{p, q}(\R^d)$ does not yield optimal lifting properties for $\mathbf{B}^{0, b}_{p, q}(\R^d)$. Next we are able to remedy this defect using the new scale of spaces $T^b_r F^s_{p, q}(\R^d)$. To be more precise, according to Theorem \ref{TheoremLifting}(ii) and Proposition \ref{PropositionCoincidences}, we obtain the following

\begin{cor}[Lifting in $\mathbf{B}^{0, b}_{p, q}(\R^d)$]\label{CorollaryLiftBesov}
	Let $\sigma \in \R, p \in (1, \infty), q \in (0, \infty]$ and $b > -1/q$. Then $I_\sigma$ acts as an isomorphism from $\mathbf{B}^{0, b}_{p, q}(\R^d)$ onto $T^{b+1/q}_ q F^{-\sigma}_{p, 2}(\R^d)$ and
	$$
		\|I_\sigma f\|_{T^{b+1/q}_q F^{-\sigma}_{p, 2}(\R^d)} \asymp \|f\|_{\mathbf{B}^{0, b}_{p, q}(\R^d)}.
	$$
\end{cor}

Among other relations, it will be established in Corollary \ref{CorollaryEmbFBNew} below that
$$
	 B^{-\sigma, b + 1/\min\{2, p, q\}}_{p, q}(\R^d) \hookrightarrow T^{b+1/q}_q F^{-\sigma}_{p, 2}(\R^d) \hookrightarrow B^{-\sigma, b + 1/\max\{2, p, q\}}_{p, q}(\R^d),
$$
so that both \eqref{LiftB01} and  \eqref{LiftB02} are immediate consequences of the stronger assertion given in Corollary \ref{CorollaryLiftBesov}.

To the best of our knowledge, lifting assertions for $\text{Lip}^{s, b}_{p, q}(\R^d)$ have not been considered so far in the literature. Now we are in a position to give a precise answer via Theorem \ref{TheoremLifting}(ii) and Proposition \ref{PropositionCoincidences}. Namely, we obtain the following

\begin{cor}[Lifting in $\text{Lip}^{s, b}_{p, q}(\R^d)$]\label{CorollaryLiftLip}
	Let $\sigma \in \R, s > 0, p \in (1, \infty), q \in (0, \infty]$ and $b < -1/q$. Then $I_\sigma$ acts as an isomorphism from $\emph{Lip}^{s, b}_{p, q}(\R^d)$ onto $T^{b+1/q}_q F^{s-\sigma}_{p, 2}(\R^d)$ and
	$$
		\|I_\sigma f\|_{T^{b+1/q}_q F^{s-\sigma}_{p, 2}(\R^d)} \asymp \|f\|_{\emph{Lip}^{s, b}_{p, q}(\R^d)}.
	$$
\end{cor}

\subsection{Sobolev-type characterizations}
We start by recalling the well-known characterizations of classical Besov and Triebel--Lizorkin spaces in terms of derivatives.

\begin{prop}\label{SobProp}
	Let $A \in \{B, F\}$. Let $s, b \in \R, 0 < p, q \leq \infty \, (p < \infty \text{ if } A=F)$ and $m \in \N$. Then
	\begin{equation}\label{Der0}
		\|f\|_{A^{s, b}_{p, q}(\R^d)} \asymp \sum_{|\alpha| \leq m} \|D^\alpha f\|_{A^{s-m, b}_{p, q}(\R^d)} \asymp \|f\|_{A^{s-m, b}_{p, q}(\R^d)}  + \sum_{l=1}^d  \bigg\|\frac{\partial^m f}{\partial x_l^m} \bigg\|_{A^{s-m, b}_{p, q}(\R^d)}.
	\end{equation}
\end{prop}

For the proof of this result, we refer to \cite[Theorem 2.3.8]{Triebel83} if $b=0$ and the general case $b \in \mathbb{R}$ may be obtained from $b=0$ via standard lifting arguments \cite[Proposition 3.2]{CaetanoMoura04} (see also \cite[Lemma 12.5]{DominguezTikhonov}).

Next we show that Proposition \ref{SobProp} can be put into the more general perspective of the spaces $T^b_r B^s_{p, q}(\R^d)$ and $T^b_r F^s_{p, q}(\R^d)$.

\begin{prop}\label{SobProp2}
	Let $A \in \{B, F\}$. Let $s, b \in \R, 0 < p, q, r \leq \infty \, (p < \infty \text{ if } A=F)$ and $m \in \N$. Then
	\begin{equation*}
		\|f\|_{T^b_r A^s_{p, q}(\R^d)} \asymp \sum_{|\alpha| \leq m} \|D^\alpha f\|_{T^b_r A^{s-m}_{p, q}(\R^d)} \asymp \|f\|_{T^b_r A^{s-m}_{p, q}(\R^d)}  + \sum_{l=1}^d  \bigg\|\frac{\partial^m f}{\partial x_l^m} \bigg\|_{T^b_r A^{s-m}_{p, q}(\R^d)}.
	\end{equation*}
\end{prop}

In particular, setting $q=r$ and $A=B$ in Proposition \ref{SobProp2} one recovers \eqref{Der0} (cf. Proposition \ref{PropositionCoincidences}).

\begin{proof}[Proof of Proposition \ref{SobProp2}]
According to \eqref{RemFM} and Proposition \ref{SobProp}, we get
\begin{align*}
\|f\|_{T^b_r A^s_{p, q}(\R^d)}^r &\asymp \sum_{j=0}^\infty 2^{j b r} \bigg\|\bigg(\sum_{\nu=2^j-1}^{2^{j+1}-2} \varphi_\nu \widehat{f} \bigg)^\vee \bigg\|^r_{A^s_{p, q}(\R^d)} \\
& \asymp \sum_{|\alpha| \leq m} \sum_{j=0}^\infty 2^{j b r} \bigg\|D^\alpha \bigg(\bigg(\sum_{\nu=2^j-1}^{2^{j+1}-2} \varphi_\nu \widehat{f} \bigg)^\vee \bigg) \bigg\|_{A^{s-m}_{p, q}(\R^d)}^r  \\
& = \sum_{|\alpha| \leq m} \sum_{j=0}^\infty 2^{j b r} \bigg\|\bigg(\sum_{\nu=2^j-1}^{2^{j+1}-2} \varphi_\nu \widehat{D^\alpha f} \bigg)^\vee  \bigg\|_{A^{s-m}_{p, q}(\R^d)}^r  \\
& \asymp \sum_{|\alpha| \leq m} \|D^\alpha f\|_{T^b_r A^{s-m}_{p, q}(\R^d)}^r
\end{align*}
and
\begin{align*}
	\|f\|_{T^b_r A^s_{p, q}(\R^d)}^r &\asymp \sum_{j=0}^\infty 2^{j b r} \bigg\|\bigg(\sum_{\nu=2^j-1}^{2^{j+1}-2} \varphi_\nu \widehat{f} \bigg)^\vee \bigg\|^r_{A^s_{p, q}(\R^d)} \\
	& \asymp \sum_{j=0}^\infty 2^{j b r} \bigg\|\bigg(\sum_{\nu=2^j-1}^{2^{j+1}-2} \varphi_\nu \widehat{f} \bigg)^\vee \bigg\|^r_{A^{s-m}_{p, q}(\R^d)} \\
	& \hspace{0.5cm} + \sum_{l=1}^d \sum_{j=0}^\infty 2^{j b r} \bigg\|\frac{\partial^m}{\partial x_l^m} \bigg(
 \bigg(\sum_{\nu=2^j-1}^{2^{j+1}-2} \varphi_\nu \widehat{f} \bigg)^\vee\bigg) \bigg\|^r_{A^{s-m}_{p, q}(\R^d)} \\
 & = \sum_{j=0}^\infty 2^{j b r} \bigg\|\bigg(\sum_{\nu=2^j-1}^{2^{j+1}-2} \varphi_\nu \widehat{f} \bigg)^\vee \bigg\|^r_{A^{s-m}_{p, q}(\R^d)} \\
	& \hspace{0.5cm} + \sum_{l=1}^d \sum_{j=0}^\infty 2^{j b r} \bigg\|
 \bigg(\sum_{\nu=2^j-1}^{2^{j+1}-2} \varphi_\nu \widehat{\frac{\partial^m f}{\partial x_l^m}} \bigg)^\vee \bigg\|^r_{A^{s-m}_{p, q}(\R^d)} \\
 & \asymp \|f\|_{T^b_r A^{s-m}_{p, q}(\R^d)}^r + \sum_{l=1}^d \bigg\|\frac{\partial^m f}{\partial x_l^m} \bigg\|_{T^b_r A^{s-m}_{p, q}(\R^d)}^r.
\end{align*}

\end{proof}

Writing down Proposition \ref{SobProp2} for $\mathbf{B}^{0, b}_{p, q}(\R^d) = T^{b+1/q}_q F^{0}_{p, 2}(\R^d)$, we obtain the following

\begin{cor}[Sobolev characterization in terms of $\mathbf{B}^{0, b}_{p, q}(\R^d)$]\label{CorSobBzero}
	Let $1 < p < \infty, 0 < q \leq \infty$ and $b > -1/q$. Let $m \in \N$. Then
	$$
		\|f\|_{T^{b+1/q}_q F^{m}_{p, 2}(\R^d) } \asymp \sum_{|\alpha| \leq m} \|D^\alpha f\|_{\mathbf{B}^{0, b}_{p, q}(\R^d)} \asymp \|f\|_{\mathbf{B}^{0, b}_{p, q}(\R^d)}  + \sum_{l=1}^d  \bigg\|\frac{\partial^m f}{\partial x_l^m} \bigg\|_{\mathbf{B}^{0, b}_{p, q}(\R^d)}.
	$$
\end{cor}

In particular, Corollary \ref{CorSobBzero} improves some estimates in \cite[p. 70]{DeVore} and \cite[Theorems 12.6 and 12.11]{DominguezTikhonov} involving $\mathbf{B}^{0, b}_{p, q}(\R^d)$. Namely, it was shown there that
$$
	f \in B^{m, b+ 1/\min\{2, p, q\}}_{p, q}(\R^d) \implies D^\alpha f \in \mathbf{B}^{0, b}_{p, q}(\R^d) \quad \text{for all} \quad |\alpha| \leq m.
$$
However, this assertion can now be sharpened to
$$
	f \in T^{b+1/q}_q F^{m}_{p, 2}(\R^d) \implies D^\alpha f \in \mathbf{B}^{0, b}_{p, q}(\R^d) \quad \text{for all} \quad |\alpha| \leq m,
$$
and this estimate is in fact optimal. Note that
$$
	B^{m, b+ 1/\min\{2, p, q\}}_{p, q}(\R^d) \hookrightarrow T^{b+1/q}_q F^{m}_{p, 2}(\R^d)
$$
(see Corollary \ref{CorollaryEmbFBNew} below).

On the other hand, Proposition \ref{SobProp2} for the special choice $T^{b+1/q}_q F^{s}_{p, 2}(\R^d) = \text{Lip}^{s, b}_{p, q}(\R^d)$ (cf. Proposition \ref{PropositionCoincidences}) gives the following

\begin{cor}[Sobolev characterization for $\text{Lip}^{s, b}_{p, q}(\R^d)$]
	Let $1 < p < \infty, 0 < q \leq \infty, s > 0$ and $b < -1/q$. Let $m \in \N$. Then
	$$
		\|f\|_{\emph{Lip}^{s, b}_{p, q}(\R^d)}  \asymp \sum_{|\alpha| \leq m} \|D^\alpha f\|_{T^{b+1/q}_q F^{s-m}_{p, 2}(\R^d)} \asymp \|f\|_{T^{b+1/q}_q F^{s-m}_{p, 2}(\R^d)}   + \sum_{l=1}^d  \bigg\|\frac{\partial^m f}{\partial x_l^m} \bigg\|_{T^{b+1/q}_q F^{s-m}_{p, 2}(\R^d)}.
	$$
	In particular, if $s > m$ then
	 	$$
		\|f\|_{\emph{Lip}^{s, b}_{p, q}(\R^d)}  \asymp \sum_{|\alpha| \leq m} \|D^\alpha f\|_{\emph{Lip}^{s-m, b}_{p, q}(\R^d)} \asymp \|f\|_{\emph{Lip}^{s-m, b}_{p, q}(\R^d)}  + \sum_{l=1}^d  \bigg\|\frac{\partial^m f}{\partial x_l^m} \bigg\|_{\emph{Lip}^{s-m, b}_{p, q}(\R^d)}.
	$$
\end{cor}

\newpage

\section{Duality}\label{SectionDuality}

Recall the well-known duality property for $A^{s, b}_{p, q}(\R^d), \, A \in \{B, F\}$:
	\begin{equation}\label{DualClassicBesov}
		(A^{s,b}_{p,q}(\R^d))' = A^{-s,-b}_{p',q'}(\R^d)
	\end{equation}
	for $p, q \in (1, \infty)$ and $s, b \in \R$; cf. \cite[Section 2.11]{Triebel83} and \cite[Theorem 3.1.10]{FarkasLeopold}. Furthermore, the limiting case $p=1$ can also be incorporated in \eqref{DualClassicBesov}, as well as $q \in (0, 1]$ with $A=B$ (cf. \cite{Peetre74}).

\subsection{Dual of truncated Besov and Triebel-Lizorkin spaces}
The goal of this section is to study the duality properties of the scales of spaces $T^b_r A^{s}_{p,q}(\mathbb{R}^d), \, A \in \{B, F\}$, in the same spirit as \eqref{DualClassicBesov} for the classical setting. To make this assertion rigorous, we first observe that
\begin{equation}\label{DualPairing}
	\mathcal{S}(\R^d) \hookrightarrow T^b_r A^{s}_{p,q}(\mathbb{R}^d) \hookrightarrow \mathcal{S}'(\R^d)
\end{equation}
and the left-hand side embedding is dense if $p, q, r < \infty$. These follow immediately from Theorems \ref{TheoremInterpolation} and \ref{TheoremInterpolationF} and the corresponding assertions \eqref{DualPairing} for classical spaces $A^s_{p, q}(\R^d)$ (see, e.g., \cite[Section 2.3.3]{Triebel83}). As a consequence, it makes sense to investigate the duality properties of the spaces $T^b_r A^{s}_{p,q}(\mathbb{R}^d)$ within the dual pairing $(\mathcal{S}(\R^d),\mathcal{S}'(\R^d))$.

\begin{thm}\label{ThmDualNewBesovSpaces}
	\begin{enumerate}[\upshape(i)]
	\item Let $p, q \in [1,\infty), r \in (0, \infty), s \in \R$ and $b \in \R \backslash \{0\}$. Then
	\begin{equation}\label{ThmDualNewBesovSpaces1}
		(T^b_r B^s_{p, q}(\mathbb{R}^d))' = T^{-b}_{r'} B^{-s}_{p',q'}(\R^d).
	\end{equation}
	\item Let $p, q \in (1, \infty), r \in (0, \infty), s \in \R$ and $b \in \R \backslash \{0\}$. Then
	\begin{equation}\label{ThmDualNewBesovSpaces2}
		(T^b_r F^s_{p, q}(\mathbb{R}^d))' = T^{-b}_{r'} F^{-s}_{p',q'}(\R^d).
	\end{equation}
	\end{enumerate}
\end{thm}

The proof of Theorem \ref{ThmDualNewBesovSpaces} relies on the limiting interpolation techniques. In particular, the following characterization of the dual space of a limiting interpolation space will be useful (see \cite[Theorems 5.6 and 5.8]{CobosSegurado} and \cite[Theorem 4.3]{Besoy}).

\begin{lem}\label{LemmaDualInt}
	Let $A_0, A_1$ be Banach spaces with $A_1$ continuously and densely embedded into $A_0$. Let $\theta \in \{0, 1\}, r \in (0,\infty)$ and
	$$
	 \left\{\begin{array}{cl}  b > -\frac{1}{r} & \text{if} \quad  \theta =0, \\
			& \\
		b < -\frac{1}{r}  & \text{if} \quad \theta=1.
		       \end{array}
                        \right.
	$$
	Then
		\begin{equation}\label{LemmaDualInt1}
			((A_0,A_1)_{(\theta,b),r})' = (A_1', A_0')_{(1-\theta,-b-1/\min\{1, r\}),r'}.
		\end{equation}
		Furthermore, if $\theta=0$ and $b=-1/r$ then
		\begin{align}
			\|f\|_{((A_0, A_1)_{(0, -1/r), r})'} &\asymp \nonumber \\
			& \hspace{-2cm} \bigg(\int_0^1 [t^{-1} (1-\log t)^{\frac{1}{r} -\frac{1}{\min\{1, r\}}} (1+ \log (1-\log t))^{-\frac{1}{\min\{1, r\}}} K(t, f; A_1', A_0')]^{r'} \frac{dt}{t}  \bigg)^{\frac{1}{r'}}. \label{LemmaDualInt2}
		\end{align}
\end{lem}

\begin{proof}[Proof of Theorem \ref{ThmDualNewBesovSpaces}]
	(i): Assume first $b > 0$. According to Theorem \ref{TheoremInterpolation}(i), we have
	$$
	T^b_r B^s_{p, q}(\mathbb{R}^d) = (B^s_{p,q}(\R^d), B^{s_0}_{p,q}(\R^d))_{(0,b-1/r),r}
	$$
	for $s_0 > s$. Since $B^{s_0}_{p,q}(\R^d)$ is densely embedded into $B^s_{p,q}(\R^d)$ (noting that $q < \infty$), we can invoke \eqref{LemmaDualInt1} (with $\theta=0$) and \eqref{DualClassicBesov} to establish
	\begin{align*}
		(T^b_r B^s_{p, q}(\mathbb{R}^d))' &=  (B^s_{p,q}(\R^d), B^{s_0}_{p,q}(\R^d))_{(0,b-1/r),r}' \\
		& = (B^{-s_0}_{p',q'}(\R^d), B^{-s}_{p',q'}(\R^d))_{(1,-b+1/r-1/\min\{1, r\}),r'} \\
		& = T^{-b}_{r'} B^{-s}_{p',q'}(\R^d),
	\end{align*}
	where we have applied Theorem \ref{TheoremInterpolation}(ii) in the last step.
	
	Assume now $b < 0$. Let $s_0 < s$. By Theorem \ref{TheoremInterpolation}(ii), we have
	$$
	T^b_r B^s_{p, q}(\mathbb{R}^d) = (B^{s_0}_{p,q}(\R^d), B^{s}_{p,q}(\R^d))_{(1,b-1/r),r},
	$$
	and thus \eqref{LemmaDualInt1} (with $\theta=1$) and Theorem \ref{TheoremInterpolation}(i) imply
	\begin{align*}
		(T^b_r B^s_{p, q}(\mathbb{R}^d))'
		& = (B^{-s}_{p',q'}(\R^d), B^{-s_0}_{p',q'}(\R^d))_{(0,-b+1/r-1/\min\{1, r\}),r'} \\
		& = T^{-b}_{r'} B^{-s}_{p',q'}(\R^d).
	\end{align*}
	
	The proof of (ii) can de done similarly but now invoking Theorem \ref{TheoremInterpolationF}. This is left to the reader.
\end{proof}

\subsection{Dual of truncated Besov and Triebel-Lizorkin spaces in the limiting case $b=0$} This subsection is slightly out of the scope of our study since here we deal with the special limiting case $b=0$ and we have to consider another truncated spaces. First, it is convenient to switch from $T_r A^s_{p, q}(\R^d), \, A \in \{B, F\}$, to $T^*_r A^s_{p, q}(\R^d)$   (cf. \eqref{33} and \eqref{34}). Second, recall that (cf. \eqref{32})
$$
		\left(\sum_{j=0}^\infty \bigg(\sum_{\nu= 0}^{2^j} 2^{\nu s q} \|(\varphi_\nu \widehat{f})^\vee\|_{L_p(\mathbb{R}^d)}^q\bigg)^{r/q}\right)^{1/r} < \infty \iff f \equiv 0. 	
$$
However, the triviality of this functional can be overcome with the help of additional weights. For example, the modified functional
 $$
 		\left(\sum_{j=0}^\infty (1 + j)^{-r} \bigg(\sum_{\nu= 0}^{2^j} 2^{\nu s q} \|(\varphi_\nu \widehat{f})^\vee\|_{L_p(\mathbb{R}^d)}^q\bigg)^{r/q}\right)^{1/r}
 $$
 defines a nontrivial space of distributions. In fact, our next result proves that, in particular, this space is the dual space of $T^*_r B^s_{p, q}(\R^d)$ (recall that $T^*_r B^s_{p, q}(\R^d)$ is defined by (\ref{33})).

\begin{thm}
\begin{enumerate}[\upshape(i)]
\item Let $p, q \in [1, \infty), r \in (0, \infty)$ and $s \in \R$. Then
	\begin{align*}
		(T^*_r B^{s}_{p,q}(\mathbb{R}^d))' &= \\
		& \hspace{-2cm} \bigg\{f \in \mathcal{S}'(\R^d): \left(\sum_{j=0}^\infty (1 + j)^{-r'/\min\{1, r\}} \bigg(\sum_{\nu= 0}^{2^j} 2^{-\nu s q'} \|(\varphi_\nu \widehat{f})^\vee\|_{L_{p'}(\mathbb{R}^d)}^{q'}\bigg)^{r'/q'}\right)^{1/r'} < \infty \bigg\}
	\end{align*}
	(with the usual modification if $r \in (0, 1]$).
	\item Let $p, q \in (1, \infty), r \in (0, \infty)$ and $s \in \R$. Then
	\begin{align*}
		(T^*_r F^{s}_{p,q}(\mathbb{R}^d))' &=  \\
		& \hspace{-2cm}\bigg\{f \in \mathcal{S}'(\R^d) : \left(\sum_{j=0}^\infty (1 + j)^{-r'/\min\{1, r\}}  \bigg\| \bigg(\sum_{\nu= 0}^{2^j} 2^{-\nu s q'} |(\varphi_\nu \widehat{f})^\vee|^{q'} \bigg)^{1/q'} \bigg\|_{L_{p'}(\R^d)}^{r'} \right)^{1/r'} < \infty \bigg\}
	\end{align*}
	(with the usual modification if $r \in (0, 1]$).
	\end{enumerate}
\end{thm}

\begin{proof}
	The proof follows similar ideas as the method of proof of Theorem \ref{ThmDualNewBesovSpaces} but now relying on the duality formula \eqref{LemmaDualInt2}. According to \eqref{4.4ururu}, we have $T_r^*B^{s}_{p,q}(\mathbb{R}^d)  =  (B^s_{p,q}(\R^d), B^{s_0}_{p,q}(\R^d))_{(0,-1/r),r}, \, s_0 > s$, and thus, by \eqref{LemmaDualInt2} and \eqref{DualClassicBesov}, $\|f\|_{(T_r^*B^{s}_{p,q}(\mathbb{R}^d))'}$  is equivalent to the following interpolation norm
		\begin{equation*}
		\bigg(\int_0^1 [t^{-1} (1-\log t)^{\frac{1}{r} -\frac{1}{\min\{1, r\}}} (1+ \log (1-\log t))^{-\frac{1}{\min\{1, r\}}} K(t, f)]^{r'} \frac{dt}{t}  \bigg)^{\frac{1}{r'}},
		\end{equation*}
		where $K(t, f) = K(t, f; B^{-s_0}_{p', q'}(\R^d), B^{-s}_{p', q'}(\R^d))$.
		Applying now the methodology developed in the proof of Theorem \ref{TheoremInterpolation}, we have, by \eqref{412} and basic monotonicity properties,
		\begin{align*}
		\bigg(\int_0^1 [t^{-1} (1-\log t)^{\frac{1}{r} -\frac{1}{\min\{1, r\}}} (1+ \log (1-\log t))^{-\frac{1}{\min\{1, r\}}} K(t, f)]^{r'} \frac{dt}{t}  \bigg)^{\frac{1}{r'}} & \asymp \\
		& \hspace{-11.5cm}\bigg\{ \sum_{j=0}^\infty (1 + j)^{(\frac{1}{r} - \frac{1}{\min\{1, r\}}) r'} (1 + \log (1+j))^{-\frac{r'}{\min\{1, r\}}} \bigg(\sum_{\nu=0}^j [2^{-\nu s} \|(\varphi_\nu \widehat{f})^\vee\|_{L_{p'}(\R^d)}]^{q'} \bigg)^{\frac{r'}{q'}} \bigg\}^{\frac{1}{r'}} \\
		& \hspace{-11.5cm} \asymp \bigg\{ \sum_{j=0}^\infty  (1 + j)^{-\frac{r'}{\min\{1, r\}}} \bigg(\sum_{\nu=0}^{2^{j}} [2^{-\nu s} \|(\varphi_\nu \widehat{f})^\vee\|_{L_{p'}(\R^d)}]^{q'} \bigg)^{\frac{r'}{q'}} \bigg\}^{\frac{1}{r'}}.
		\end{align*}
		This proves (i).
		
		Next we deal with (ii). By \eqref{435},
		$$
		T^*_r F^{s}_{p, q}(\R^d) = 	(F^s_{p,q}(\R^d), F^{s_0}_{p,q}(\R^d))_{(0,-1/r),r}
		$$
		provided that $s_0 > s$. Then \eqref{LemmaDualInt2} and \eqref{DualClassicBesov} imply
		$$
			\|f\|_{(T^*_r F^{s}_{p, q}(\R^d))'} \asymp \bigg(\int_0^1 [t^{-1} (1-\log t)^{\frac{1}{r} -\frac{1}{\min\{1, r\}}} (1+ \log (1-\log t))^{-\frac{1}{\min\{1, r\}}} K(t, f)]^{r'} \frac{dt}{t}  \bigg)^{\frac{1}{r'}},
		$$
		where $K(t, f) = K(t, f; F^{-s_0}_{p', q'}(\R^d), F^{-s}_{p', q'}(\R^d))$. We can now invoke the retraction method in the same spirit as in the proof of Theorem \ref{TheoremInterpolationF} to obtain
		\begin{align*}
			\|f\|_{(T^*_r F^{s}_{p, q}(\R^d))'} & \asymp \\
			&\hspace{-2cm}\bigg\{ \sum_{j=0}^\infty (1+j)^{(\frac{1}{r} - \frac{1}{\min\{1, r\}}) r'} (1 + \log (1+j))^{-\frac{r'}{\min\{1, r\}}} \bigg\| \bigg(\sum_{\nu=0}^j [2^{-\nu s} |(\varphi_\nu \widehat{f})^\vee|]^{q'} \bigg)^{\frac{1}{q'}}\bigg\|_{L_{p'}(\R^d)}^{r'} \bigg\}^{\frac{1}{r'}} \\
			& \hspace{-2cm} \asymp  \left(\sum_{j=0}^\infty (1 + j)^{-\frac{r'}{\min\{1, r\}}}  \bigg\| \bigg(\sum_{\nu= 0}^{2^j} 2^{-\nu s q'} |(\varphi_\nu \widehat{f})^\vee|^{q'} \bigg)^{\frac{1}{q'}} \bigg\|_{L_{p'}(\R^d)}^{r'} \right)^{\frac{1}{r'}}.
		\end{align*}
\end{proof}

\subsection{Dual spaces of $\mathbf{B}^{0, b}_{p, q}(\R^d)$ and $\text{Lip}^{s, b}_{p, q}(\R^d)$} In light of \eqref{DualClassicBesov}, for $p \in [1, \infty), q \in (0, \infty)$ and $b \in \R$,
$$
	(B^{0, b}_{p, q}(\R^d))' = B^{0,-b}_{p', q'}(\R^d).
$$
However $B^{0, b}_{p, q}(\R^d) \neq \mathbf{B}^{0, b}_{p, q}(\R^d)$ and thus it becomes a rather natural question to characterize the dual space $(\mathbf{B}^{0, b}_{p, q}(\R^d))'$. Indeed, this question was first addressed  in \cite{CobosDominguez2} dealing with the case $q \geq 1$ and later extended in \cite{Besoy} to cover also the case $q < 1$. Next we recall these results.

\begin{thm}[{\cite[Theorem 5.2]{Besoy} and \cite[Theorem 4.3]{CobosDominguez2}}]\label{ThmDualBLip}
	Let $1 < p < \infty, 0 < q < \infty$ and $b > -1/q$. Then $f \in (\mathbf{B}^{0, b}_{p, q}(\R^d))'$ if and only if $f \in H^{-1}_{p'}(\R^d)$ with $I_{-1} f \in \emph{Lip}^{1,-b-1/\min\{1, q\}}_{p', q'}(\R^d)$. Furthermore
	$$
		\|f\|_{(\mathbf{B}^{0, b}_{p, q}(\R^d))'} \asymp \|I_{-1} f\|_{ \emph{Lip}^{1,-b-1/\min\{1, q\}}_{p', q'}(\R^d)}.
	$$
\end{thm}

The previous result gives an interesting duality connection between  $\mathbf{B}^{0, b}_{p, q}(\R^d)$ and $\text{Lip}^{1, b}_{p, q}(\R^d)$, however it is still not satisfactory due to the following reasons. Theorem \ref{ThmDualBLip} is formulated in terms of the lifting operator $I_{-1}$ (cf. \eqref{LiftingDef}), but a careful examination of the methodology proposed in \cite{CobosDominguez2} shows that the analogues of  Theorem \ref{ThmDualBLip} in terms of $I_{-s}$ and $\text{Lip}^{s, b}_{p, q}(\R^d)$ for any $s > 0$ also hold true, i.e.,
\begin{equation}\label{DualBLipGen}
		\|f\|_{(\mathbf{B}^{0, b}_{p, q}(\R^d))'} \asymp \|I_{-s} f\|_{ \text{Lip}^{s,-b-1/\min\{1, q\}}_{p', q'}(\R^d)}.
	\end{equation}
Accordingly, $(\mathbf{B}^{0, b}_{p, q}(\R^d))'$ is not uniquely determined by $\text{Lip}^{1, b}_{p, q}(\R^d)$. Next we are able to remedy this defect via Theorem \ref{ThmDualNewBesovSpaces}.

\begin{thm}[Dual space of $\mathbf{B}^{0, b}_{p, q}(\R^d)$]\label{ThmDualBFinal}
Let $1 < p < \infty, 0 < q < \infty$ and $b > -1/q$. Then
$$
	(\mathbf{B}^{0, b}_{p, q}(\R^d))' = T^{-b-1/q}_{q'} F^{0}_{p', 2}(\R^d)
$$
and
$$
	(T^{-b-1/q}_{q'} F^{0}_{p', 2}(\R^d))' = \mathbf{B}^{0, b}_{p, q}(\R^d).
$$
\end{thm}

\begin{rem}
	Note that Theorem \ref{ThmDualBLip} (or more generally, \eqref{DualBLipGen}) follows immediately from Theorem \ref{ThmDualBFinal}. Indeed, by Theorem \ref{TheoremLifting} and Proposition \ref{PropositionCoincidences}, we have
	\begin{equation*}
		\|f\|_{T^{-b-1/q}_{q'} F^{0}_{p', 2}(\R^d)} \asymp \|I_{-s} f\|_{T^{-b-1/q}_{q'}F^{s}_{p', 2}(\R^d)} \asymp \|I_{-s} f\|_{\text{Lip}^{s, -b-1/\min\{1, q\}}_{p', q'}(\R^d)}.
	\end{equation*}
\end{rem}

We also apply our method to characterize dual spaces of $\text{Lip}^{s, b}_{p, q}(\R^d)$.

\begin{thm}[Dual space of  $\text{Lip}^{s, b}_{p, q}(\R^d)$]\label{ThmDualLipFinal}
	Let $1 < p < \infty, 0 < q < \infty, s > 0$ and $b < -1/q$. Then
	$$
		(\emph{Lip}^{s, b}_{p, q}(\R^d))' = T^{-b-1/q}_{q'} F^{-s}_{p', 2}(\R^d)
	$$
	and
	$$
		 (T^{-b-1/q}_{q'}F^{-s}_{p', 2}(\R^d))' = \emph{Lip}^{s, b}_{p, q}(\R^d).
	$$
\end{thm}

\begin{proof}[Proof of Theorems \ref{ThmDualBFinal} and \ref{ThmDualLipFinal}]
	Apply \eqref{ThmDualNewBesovSpaces2} with $T^{b+1/q}_q F^{0}_{p, 2}(\R^d) = \mathbf{B}^{0, b}_{p, q}(\R^d)$ and $T^{b+1/q}_q F^{s}_{p, 2}(\R^d) = \text{Lip}^{s, b}_{p, q}(\R^d)$ (cf. Proposition \ref{PropositionCoincidences}).
\end{proof}

\newpage

\section{Embeddings between truncated Besov spaces
}\label{Section10}

\begin{thm}\label{TheoremEmbeddingsBBCharacterization}
Let $0 <  p_i, q_i, r_i \leq \infty, \, s_i \in \mathbb{R}$ and $b_i \in \R \backslash \{0\}$ for $i = 0, 1$.  Then
\begin{equation}\label{EmbeddingBB}
	T^{b_0}_{r_0} B^{s_0}_{p_0,q_0}(\mathbb{R}^d) \hookrightarrow T^{b_1}_{r_1} B^{s_1}_{p_1,q_1}(\mathbb{R}^d)
\end{equation}
if and only if $0 < p_0 \leq p_1 \leq \infty$ and one of the following five conditions is satisfied
\begin{enumerate}[\upshape(i)]
	\item $s_0 - \frac{d}{p_0} > s_1 - \frac{d}{p_1}$.
	\item $s_0 - \frac{d}{p_0} = s_1 - \frac{d}{p_1}, \quad 0 < q_0 \leq q_1 \leq \infty$ \quad and \quad  $b_0  > b_1$.
	\item $s_0 - \frac{d}{p_0} = s_1 - \frac{d}{p_1}, \quad 0 < q_0 \leq q_1 \leq \infty, \quad b_0 = b_1$ \quad and \quad $0 < r_0 \leq r_1 \leq \infty$.
	\item $s_0 - \frac{d}{p_0} = s_1 - \frac{d}{p_1}, \quad 0 < q_1 < q_0 \leq \infty$ \quad and \quad $b_0 + \frac{1}{q_0} > b_1  + \frac{1}{q_1}$.
	\item $s_0 - \frac{d}{p_0} = s_1 - \frac{d}{p_1}, \quad 0 < q_1 < q_0 \leq \infty, \quad b_0 + \frac{1}{q_0} = b_1 + \frac{1}{q_1}$ \quad and \quad $0 < r_0 \leq r_1 \leq \infty$.
\end{enumerate}
\end{thm}

\begin{proof}
	According to Theorem \ref{ThmWaveletsNewBesov} the validity of \eqref{EmbeddingBB} is equivalent to the  corresponding embedding between related sequence spaces introduced in \eqref{DefBesSeq}, i.e.,
	\begin{equation}\label{ProofEmbeddingsBB1*}
		T^{b_0}_{r_0} b^{s_0}_{p_0,q_0} \hookrightarrow T^{b_1}_{r_1} b^{s_1}_{p_1,q_1}.
	\end{equation}

	\emph{Sufficient conditions} (i): Let $p_0 \leq p_1$ and $s_0 - \frac{d}{p_0} > s_1 - \frac{d}{p_1}$. We will split the embedding \eqref{ProofEmbeddingsBB1*} as (recall \eqref{Deffspaces**})
	\begin{equation}\label{ProofEmbeddingsBB1*nsnan}
	T^{b_0}_{r_0} b^{s_0}_{p_0,q_0} \hookrightarrow b^{s_0, b_0 - (1/r_0 -1/q_0)_+}_{p_0, r_0} \hookrightarrow   b^{s_1, b_1 + (1/q_1-1/r_1)_+}_{p_1, r_1} \hookrightarrow T^{b_1}_{r_1} b^{s_1}_{p_1,q_1}.
	\end{equation}
	
	We start by proving the third embedding in \eqref{ProofEmbeddingsBB1*nsnan}. Let $\lambda = \{\lambda^{j, G}_m\}$ be a sequence given by \eqref{sequence}. For $k \in \N_0$, we claim that
	\begin{align}\label{gagags}
		 \bigg(\sum_{j=2^k-1}^{2^{k+1}-2} 2^{j(s_1-d/p_1) q_1} \sum_{G \in G^j} \Big(\sum_{m \in \mathbb{Z}^d} |\lambda^{j, G}_m|^{p_1} \Big)^{q_1/p_1} \bigg)^{1/q_1} & \lesssim \\
		 &\hspace{-6cm} 2^{k (1/q_1-1/r_1)_+}  \bigg(\sum_{j=2^k-1}^{2^{k+1}-2} 2^{j(s_1-d/p_1) r_1} \sum_{G \in G^j} \Big(\sum_{m \in \mathbb{Z}^d} |\lambda^{j, G}_m|^{p_1} \Big)^{r_1/p_1} \bigg)^{1/r_1}.
	\end{align}
	Indeed, this estimate is a consequence of the embedding $\ell_{r_1} \hookrightarrow \ell_{q_1}$ if $r_1 \leq q_1$ and H\"older's inequality if $r_1 > q_1$. Hence (cf. \eqref{Deffspaces**})
	\begin{align*}
		\|\lambda\|_{T^{b_1}_{r_1} b^{s_1}_{p_1, q_1}}^{r_1} &\lesssim \sum_{k=0}^\infty 2^{k (b_1 + (1/q_1 -1/r_1)_+) r_1}  \sum_{j=2^k-1}^{2^{k+1}-2} 2^{j(s_1-d/p_1) r_1} \sum_{G \in G^j} \Big(\sum_{m \in \mathbb{Z}^d} |\lambda^{j, G}_m|^{p_1} \Big)^{r_1/p_1} \\
		& \asymp \sum_{j=0}^\infty  2^{j(s_1-d/p_1) r_1} (1 + j)^{(b_1 + (1/q_1 -1/r_1)_+) r_1} \sum_{G \in G^j} \Big(\sum_{m \in \mathbb{Z}^d} |\lambda^{j, G}_m|^{p_1} \Big)^{r_1/p_1} \\
		& = \|\lambda\|^{r_1}_{b^{s_1, b_1 + (1/q_1 - 1/r_1)_+}_{p_1, r_1}}.
	\end{align*}

	Next we focus our attention to the first embedding in \eqref{ProofEmbeddingsBB1*nsnan}. Indeed, in the same spirit as \eqref{gagags} we have
	\begin{align*}
		2^{-k (1/r_0-1/q_0)_+} \bigg(\sum_{j=2^k-1}^{2^{k+1}-2} 2^{j(s_0-d/p_0) r_0} \sum_{G \in G^j} \Big(\sum_{m \in \mathbb{Z}^d} |\lambda^{j, G}_m|^{p_0} \Big)^{r_0/p_0} \bigg)^{1/r_0} & \lesssim \\
		 &\hspace{-7cm}   \bigg(\sum_{j=2^k-1}^{2^{k+1}-2} 2^{j(s_0-d/p_0) q_0} \sum_{G \in G^j} \Big(\sum_{m \in \mathbb{Z}^d} |\lambda^{j, G}_m|^{p_0} \Big)^{q_0/p_0} \bigg)^{1/q_0},
	\end{align*}
	which yields
	$$
		\|\lambda\|_{b^{s_0, b_0 -(1/r_0-1/q_0)_+}_{p_0, r_0}} \lesssim \|\lambda\|_{T^{b_0}_{r_0} b^{s_0}_{p_0, q_0}(\R^d)}.
	$$

	The second embedding in \eqref{ProofEmbeddingsBB1*nsnan} is well known, in fact a more general assertion holds true
	\begin{equation}\label{TrivialEmbedding8}
	b^{s_0, \xi}_{p_0, u} \hookrightarrow b^{s_1, \eta}_{p_1, v}
	\end{equation}
	under $p_0 \leq p_1, \, s_0 -d/p_0 > s_1-d/p_1, 0 < u, v \leq \infty$ and $\xi, \eta \in \R$. For the sake of completeness, we provide below the short proof of \eqref{TrivialEmbedding8}.  We have
	\begin{align*}
		\|\lambda\|_{b^{s_1, \eta}_{p_1, v}} & \leq \left(\sum_{j=0}^\infty 2^{j(s_1-d/p_1) v} (1 + j)^{\eta v} \sum_{G \in G^j} \Big(\sum_{m \in \mathbb{Z}^d} |\lambda^{j, G}_m|^{p_0} \Big)^{v/p_0} \right)^{1/v} \\
		& \leq  \bigg(\sum_{j=0}^\infty 2^{j(s_1-d/p_1-s_0+d/p_0) v} (1+j)^{(\eta-\xi) v} \bigg)^{1/v} \|\lambda\|_{b^{s_0, \xi}_{p_0, \infty}}
	\end{align*}
	where the last sum is convergent since $s_0 -d/p_0 > s_1-d/p_1$. Clearly this implies \eqref{TrivialEmbedding8}.

	\vspace{2mm}
	\emph{Sufficient conditions} (ii)--(iii): 	Let $p_0 \leq p_1, \, s_0 - \frac{d}{p_0} = s_1 - \frac{d}{p_1}$ and $q_0 \leq q_1$. 	The desired embedding \eqref{ProofEmbeddingsBB1*} will be achieved via the splitting
	\begin{equation}\label{ProofEmbeddingsBB2}
		T^{b_0}_{r_0} b^{s_0}_{p_0, q_0} \hookrightarrow T^{b_1}_{r_1} b^{s_0}_{p_0, q_0} \hookrightarrow  T^{b_1}_{r_1} b^{s_1}_{p_1, q_1}.
	\end{equation}
	
	The second embedding in \eqref{ProofEmbeddingsBB2} is clear:
	\begin{align*}
		\|\lambda\|_{T^{b_1}_{r_1} b^{s_1}_{p_1,q_1}} &= \bigg(\sum_{k=0}^\infty 2^{k b_1 r_1} \bigg(\sum_{j= 2^k - 1}^{2^{k+1}-2} 2^{j(s_1-\frac{d}{p_1}) q_1} \sum_{G \in G^j} \Big(\sum_{m \in \mathbb{Z}^d} |\lambda^{j, G}_m|^{p_1} \Big)^{\frac{q_1}{p_1}} \bigg)^{\frac{r_1}{q_1}} \bigg)^{\frac{1}{r_1}} \nonumber \\
		& \hspace{-.75cm}\leq \bigg(\sum_{k=0}^\infty 2^{k b_1  r_1} \bigg(\sum_{j= 2^k - 1}^{2^{k+1}-2}  2^{j(s_0-\frac{d}{p_0}) q_0} \sum_{G \in G^j} \Big(\sum_{m \in \mathbb{Z}^d} |\lambda^{j, G}_m|^{p_0} \Big)^{\frac{q_0}{p_0}} \bigg)^{\frac{r_1}{q_0}} \bigg)^{\frac{1}{r_1}} = \|\lambda\|_{T^{b_1}_{r_1} b^{s_0}_{p_0, q_0}}.
	\end{align*}
	On the other hand, the first embedding in \eqref{ProofEmbeddingsBB2} follows immediately from $\ell_{r_0} \hookrightarrow \ell_{r_1}$ if $b_0  \geq b_1$ and $r_0 \leq r_1$ and from H\"older's inequality if
	$
		b_0  > b_1$ and $r_0 > r_1$.

	\vspace{2mm}
	\emph{Sufficient conditions} (iv)-(v): 	Let $p_0 \leq p_1, \, s_0 - \frac{d}{p_0} = s_1 - \frac{d}{p_1}$ and $q_1 < q_0$. For each $k \geq 0$, by H\"older's inequality, we have
	\begin{align*}
		\bigg(\sum_{j=2^k-1}^{2^{k+1}-2} 2^{j(s_1-\frac{d}{p_1}) q_1} \sum_{G \in G^j} \Big(\sum_{m \in \mathbb{Z}^d} |\lambda^{j, G}_m|^{p_1} \Big)^{\frac{q_1}{p_1}} \bigg)^{\frac{1}{q_1}}  & \lesssim \\
		& \hspace{-5cm} 2^{k(\frac{1}{q_1}- \frac{1}{q_0})} \bigg(\sum_{j=2^k-1}^{2^{k+1}-2} 2^{j(s_0-\frac{d}{p_0}) q_0} \sum_{G \in G^j} \Big(\sum_{m \in \mathbb{Z}^d} |\lambda^{j, G}_m|^{p_0} \Big)^{\frac{q_0}{p_0}} \bigg)^{\frac{1}{q_0}}.
	\end{align*}
	Therefore,
	\begin{align*}
	\|\lambda\|_{T^{b_1}_{r_1} b^{s_1}_{p_1,q_1}} &\lesssim \bigg(\sum_{k=0}^\infty 2^{k(b_1+ \frac{1}{q_1} - \frac{1}{q_0}) r_1} \bigg(\sum_{j= 2^k - 1}^{2^{k+1}-2} 2^{j(s_0-\frac{d}{p_0}) q_0} \sum_{G \in G^j} \Big(\sum_{m \in \mathbb{Z}^d} |\lambda^{j, G}_m|^{p_0} \Big)^{\frac{q_0}{p_0}} \bigg)^{\frac{r_1}{q_0}} \bigg)^{\frac{1}{r_1}} \\
	& \lesssim \|\lambda\|_{T^{b_0}_{r_0} b^{s_0}_{p_0,q_0}},
	\end{align*}
	where the last step is a simple application of H\"older's inequality if $r_0 > r_1$ (and thus, $b_0 + \frac{1}{q_0} > b_1 +  \frac{1}{q_1}$) and the fact that $\ell_{r_0} \hookrightarrow \ell_{r_1}$  if $r_0 \leq r_1$.
	
	\vspace{2mm}
	\emph{Necessary condition for $p_0 \leq p_1$}: We will show that \eqref{ProofEmbeddingsBB1*} implies  $p_0 \leq p_1$. Indeed, if $p_0 > p_1$ then we let
		\begin{equation*}
		\lambda^{j,G}_m = \left\{\begin{array}{cl}  l^{-\varepsilon}, & \quad j =0, \quad m= (l, 0, \ldots, 0), \quad l \in \N,  \\
		& \qquad G = (M, \ldots, M),  \\
		0, & \text{otherwise},
		       \end{array}
                        \right.
	\end{equation*}
	where $\frac{1}{p_0} < \varepsilon < \frac{1}{p_1}$. The sequence $\lambda = (\lambda^{j,G}_m)$ satisfies
	$$
		\|\lambda\|_{T^{b_0}_{r_0} b^{s_0}_{p_0,q_0}}  \asymp \bigg(\sum_{l=1}^\infty l^{-\varepsilon p_0} \bigg)^{\frac{1}{p_0}} < \infty
	$$
	but
		$$
		\|\lambda\|_{T^{b_1}_{r_1} b^{s_1}_{p_1,q_1}}  \asymp \bigg(\sum_{l=1}^\infty l^{-\varepsilon p_1} \bigg)^{\frac{1}{p_1}} = \infty.
	$$
	
	\vspace{2mm}
	\emph{Necessary condition for $s_0 - \frac{d}{p_0} \geq s_1 - \frac{d}{p_1}$}: We proceed by contradiction, that is, we assume the validity of \eqref{ProofEmbeddingsBB1*} under the assumption $s_0 - \frac{d}{p_0} < s_1 - \frac{d}{p_1}$. Let
		\begin{equation*}
		\lambda^{j,G}_m = \left\{\begin{array}{cl}  2^{-j \varepsilon}, & \quad j \in \N_0, \quad m= (j, 0, \ldots, 0),  \\
		& \qquad G = (M, \ldots, M),  \\
		0, & \text{otherwise},
		       \end{array}
                        \right.
	\end{equation*}
	where $s_0 - \frac{d}{p_0} < \varepsilon < s_1 - \frac{d}{p_1}$. Then $\lambda = (\lambda^{j,G}_m) \in T^{b_0}_{r_0} b^{s_0}_{p_0,q_0}$ since
	\begin{align*}
		\|\lambda\|_{T^{b_0}_{r_0} b^{s_0}_{p_0,q_0}} &= \bigg(\sum_{k=0}^\infty 2^{k b_0 r_0} \bigg(\sum_{j=2^k-1}^{2^{k+1}-2} 2^{j(s_0-\frac{d}{p_0}-\varepsilon) q_0}  \bigg)^{\frac{r_0}{q_0}} \bigg)^{\frac{1}{r_0}}  \\
		& \asymp \bigg(\sum_{k=0}^\infty 2^{k b_0 r_0}  2^{2^k(s_0-\frac{d}{p_0}-\varepsilon) r_0}  \bigg)^{\frac{1}{r_0}} < \infty.
	\end{align*}
	However, similar computations lead to
	\begin{align*}
		\|\lambda\|_{T^{b_1}_{r_1} b^{s_1}_{p_1,q_1}} \asymp \bigg(\sum_{k=1}^\infty 2^{k b_1 r_1}  2^{2^{k}(s_1-\frac{d}{p_1}-\varepsilon) r_1}  \bigg)^{\frac{1}{r_1}} = \infty.
	\end{align*}
	
	\vspace{2mm}
	\emph{Necessary condition for $b_0 + \frac{1}{q_0} \geq b_1 + \frac{1}{q_1}$ under the assumption $s_0 - \frac{d}{p_0} = s_1 - \frac{d}{p_1}$}: Again we argue by contradiction. We assume that \eqref{ProofEmbeddingsBB1*} holds with $s_0 - \frac{d}{p_0} = s_1 - \frac{d}{p_1}$ and $b_0  + \frac{1}{q_0} < b_1 + \frac{1}{q_1}$. Define
\begin{equation*}
		\lambda^{j,G}_m = \left\{\begin{array}{cl}  2^{-j (s_0 - \frac{d}{p_0} )} j^{-\varepsilon}, & \quad j \in \N_0, \quad m= (j, 0, \ldots, 0),  \\
		& \qquad G = (M, \ldots, M),  \\
		0, & \text{otherwise},
		       \end{array}
                        \right.
	\end{equation*}
	and $b_0  + \frac{1}{q_0}  < \varepsilon < b_1  + \frac{1}{q_1}$. Elementary computations lead to
	\begin{align*}
		\|\lambda\|_{T^{b_0}_{r_0} b^{s_0}_{p_0,q_0}} &=  \bigg(\sum_{k=0}^\infty 2^{k b_0 r_0} \bigg(\sum_{j=2^{k}-1}^{2^{k+1}-2}  (1 + j)^{-\varepsilon q_0}  \bigg)^{\frac{r_0}{q_0}} \bigg)^{\frac{1}{r_0}}  \\
		& \asymp \bigg(\sum_{k=0}^\infty 2^{k (b_0 - \varepsilon + \frac{1}{q_0}) r_0} \bigg)^{\frac{1}{r_0}} < \infty
	\end{align*}
	and
	\begin{align*}
		\|\lambda\|_{T^{b_1}_{r_1} b^{s_1}_{p_1,q_1}} &=  \bigg(\sum_{k=0}^\infty 2^{k b_1 r_1} \bigg(\sum_{j=2^k-1}^{2^{k+1}-2}  (1 + j)^{-\varepsilon q_1}  \bigg)^{\frac{r_1}{q_1}} \bigg)^{\frac{1}{r_1}}  \\
		& \asymp \bigg(\sum_{k=0}^\infty 2^{k (b_1 - \varepsilon + \frac{1}{q_1}) r_1} \bigg)^{\frac{1}{r_1}} = \infty.
	\end{align*}
	
	\vspace{2mm}
	\emph{Necessary condition for $b_0 + \frac{1}{q_0} > b_1 + \frac{1}{q_1}$ under the assumptions $s_0 - \frac{d}{p_0} = s_1 - \frac{d}{p_1}$ and $r_1 < r_0$}: The necessity of $b_0  + \frac{1}{q_0} \geq b_1 + \frac{1}{q_1}$ was already obtained in the previous case. It remains to show that the limiting case $b_0 + \frac{1}{q_0} = b_1  + \frac{1}{q_1}, \, s_0 - \frac{d}{p_0} = s_1 - \frac{d}{p_1}$ and $r_1 < r_0$ is not admissible in \eqref{ProofEmbeddingsBB1*}. To prove this, consider
\begin{equation*}
		\lambda^{j,G}_m = \left\{\begin{array}{cl}  2^{-j (s_0 - \frac{d}{p_0} )} (1 + j)^{-(b_0  + \frac{1}{q_0})} (1 + \log (1+j))^{-\beta}, & \quad j \in \N_0, \quad m= (j, 0, \ldots, 0),  \\
		& \qquad G = (M, \ldots, M),  \\
		0, & \text{otherwise},
		       \end{array}
                        \right.
	\end{equation*}
	and $\frac{1}{r_0} < \beta < \frac{1}{r_1}$. Thus
	\begin{align*}
		\|\lambda\|_{T^{b_0}_{r_0} b^{s_0}_{p_0,q_0}} &=  \bigg(\sum_{k=0}^\infty 2^{k b_0 r_0} \bigg(\sum_{j=2^k-1}^{2^{k+1}-2}   j^{-(b_0 + \frac{1}{q_0}) q_0} (1 + \log (1+ j))^{-\beta q_0}  \bigg)^{\frac{r_0}{q_0}} \bigg)^{\frac{1}{r_0}}  \\
		& \asymp \bigg(\sum_{k=0}^\infty (1 + k)^{-\beta r_0} \bigg)^{\frac{1}{r_0}} < \infty
	\end{align*}
	and, analogously,
	\begin{equation*}
		\|\lambda\|_{T^{b_1}_{r_1} b^{s_1}_{p_1,q_1}}  \asymp   \bigg(\sum_{k=0}^\infty (1+k)^{-\beta r_1} \bigg)^{\frac{1}{r_1}}= \infty.
	\end{equation*}
	
	\vspace{2mm}
	\emph{Necessary condition for $b_0  \geq b_1$ under the assumption $s_0 - \frac{d}{p_0} = s_1 - \frac{d}{p_1}$}: Assume $b_0  < b_1$. Consider the lacunary sequence given by
		\begin{equation*}
		\lambda^{j,G}_m = \left\{\begin{array}{cl}  2^{-2^k (s_0 - \frac{d}{p_0})} 2^{-k \varepsilon}, & \quad j = 2^k, \quad k \in \N_0, \quad m= (j, 0, \ldots, 0),  \\
		& \qquad G = (M, \ldots, M),  \\
		0, & \text{otherwise},
		       \end{array}
                        \right.
	\end{equation*}
	where $b_0  < \varepsilon < b_1$. Therefore
	\begin{equation*}
		\|\lambda\|_{T^{b_0}_{r_0} b^{s_0}_{p_0,q_0}}  \asymp \bigg(\sum_{k=0}^\infty 2^{k(b_0 - \varepsilon) r_0}  \bigg)^{\frac{1}{r_0}} < \infty
	\end{equation*}
	and
	\begin{equation*}
		\|\lambda\|_{T^{b_1}_{r_1}b^{s_1}_{p_1,q_1}}  \asymp \bigg(\sum_{k=0}^\infty 2^{k(b_1 - \varepsilon) r_1}  \bigg)^{\frac{1}{r_1}} = \infty.
	\end{equation*}
	
	\vspace{2mm}
	\emph{Necessary condition for $r_0 \leq r_1$ under the assumptions $b_0  = b_1$ and $s_0 - \frac{d}{p_0} = s_1 - \frac{d}{p_1}$}: Assume that \eqref{ProofEmbeddingsBB1*} holds with $r_0 > r_1$. Let
		\begin{equation*}
		\lambda^{j,G}_m = \left\{\begin{array}{cl}  2^{-2^k (s_0 - \frac{d}{p_0})} 2^{-k b_0} (1 + k)^{-\varepsilon}, & \quad j = 2^k, \quad k \in \N_0, \quad m= (j, 0, \ldots, 0),  \\
		& \qquad G = (M, \ldots, M),  \\
		0, & \text{otherwise},
		       \end{array}
                        \right.
	\end{equation*}
	where $\frac{1}{r_0}  < \varepsilon < \frac{1}{r_1}$. Therefore
	\begin{equation*}
		\|\lambda\|_{T^{b_0}_{r_0} b^{s_0}_{p_0,q_0}}  \asymp \bigg(\sum_{k=0}^\infty (1 + k)^{-\varepsilon r_0}  \bigg)^{\frac{1}{r_0}} < \infty
	\end{equation*}
	and
	\begin{equation*}
		\|\lambda\|_{T^{b_1}_{r_1} b^{s_1}_{p_1,q_1}}  \asymp \bigg(\sum_{k=0}^\infty (1 + k)^{-\varepsilon r_1}  \bigg)^{\frac{1}{r_1}} = \infty.
	\end{equation*}
	This gives the desired contradiction.
\end{proof}

 \subsection{Embeddings between truncated and classical  Besov spaces}\label{SectionTrunBB}
As an immediate consequence of Theorem \ref{TheoremEmbeddingsBBCharacterization} and Proposition \ref{PropositionCoincidences}, we obtain sharp relationships between classical Besov spaces $B^{s, b}_{p, q}(\R^d)$ and the new scale of spaces $T^b_r B^s_{p, q}(\R^d)$.

\begin{cor}\label{CorollaryBNewB}
	Let $0 <  p_0, p_1, q_0, q_1, r_1 \leq \infty, \, s_0, s_1 \in \mathbb{R}$ and $b_0, b_1 \in \R \backslash \{0\}$.  Then
\begin{equation*}
	B^{s_0,b_0}_{p_0,q_0}(\mathbb{R}^d) \hookrightarrow T^{b_1}_{r_1} B^{s_1}_{p_1,q_1}(\mathbb{R}^d)
\end{equation*}
if and only if  $0 < p_0 \leq p_1 \leq \infty$ and one of the following five conditions is satisfied
\begin{enumerate}[\upshape(i)]
	\item $s_0 - \frac{d}{p_0} > s_1 - \frac{d}{p_1}$.
	\item $s_0 - \frac{d}{p_0} = s_1 - \frac{d}{p_1}, \quad 0 < q_0 \leq q_1 \leq \infty$ \quad and \quad  $b_0  > b_1$.
	\item $s_0 - \frac{d}{p_0} = s_1 - \frac{d}{p_1}, \quad 0 < q_0 \leq q_1 \leq \infty, \quad b_0 = b_1$ \quad and \quad $0 < q_0 \leq r_1 \leq \infty$.
	\item $s_0 - \frac{d}{p_0} = s_1 - \frac{d}{p_1}, \quad 0 < q_1 < q_0 \leq \infty$ \quad and \quad $b_0  + \frac{1}{q_0} > b_1  + \frac{1}{q_1}$.
	\item $s_0 - \frac{d}{p_0} = s_1 - \frac{d}{p_1}, \quad 0 < q_1 < q_0 \leq \infty, \quad b_0 + \frac{1}{q_0} = b_1 + \frac{1}{q_1}$ \quad and \quad $0 < q_0 \leq r_1 \leq \infty$.
\end{enumerate}
\end{cor}

\begin{cor}\label{Corollary11.4}
	Let $0 <  p_0, p_1, q_0, q_1, r_0 \leq \infty, \, s_0, s_1 \in \mathbb{R}$ and $b_0, b_1 \in \R \backslash \{0\}$.  Then
\begin{equation*}
	T^{b_0}_{r_0} B^{s_0}_{p_0,q_0}(\mathbb{R}^d) \hookrightarrow B^{s_1,b_1}_{p_1,q_1}(\mathbb{R}^d)
\end{equation*}
if and only if $0 < p_0 \leq p_1 \leq \infty$ and one of the following five conditions is satisfied
\begin{enumerate}[\upshape(i)]
	\item $s_0 - \frac{d}{p_0} > s_1 - \frac{d}{p_1}$.
	\item $s_0 - \frac{d}{p_0} = s_1 - \frac{d}{p_1}, \quad 0 < q_0 \leq q_1 \leq \infty$ \quad and \quad  $b_0  > b_1$.
	\item $s_0 - \frac{d}{p_0} = s_1 - \frac{d}{p_1}, \quad 0 < q_0 \leq q_1 \leq \infty, \quad b_0 = b_1$ \quad and \quad $0 < r_0 \leq q_1 \leq \infty$.
	\item $s_0 - \frac{d}{p_0} = s_1 - \frac{d}{p_1}, \quad 0 < q_1 < q_0 \leq \infty$ \quad and \quad $b_0  + \frac{1}{q_0} > b_1  + \frac{1}{q_1}$.
	\item $s_0 - \frac{d}{p_0} = s_1 - \frac{d}{p_1}, \quad 0 < q_1 < q_0 \leq \infty, \quad b_0 + \frac{1}{q_0} = b_1 + \frac{1}{q_1}$ \quad and \quad $0 < r_0 \leq q_1 \leq \infty$.
\end{enumerate}
\end{cor}

\begin{rem}\label{Remark10.4}
	Specializing Corollary \ref{CorollaryBNewB} with $q_1=r_1$ and applying Proposition \ref{PropositionCoincidences}, we recover the well-known characterization for embeddings between Besov spaces of logarithmic smoothness (cf. \cite{Leopold} and \cite[Section 6]{DominguezTikhonov}). More precisely, let $0 < p_i, q_i \leq \infty, \, s_i, b_i \in \R$ for $i= 0, 1$, then
	$$
		B^{s_0, b_0}_{p_0, q_0}(\R^d) \hookrightarrow B^{s_1, b_1}_{p_1, q_1}(\R^d)
	$$
	if and only if $0 < p_0 \leq p_1 \leq \infty$ and one of the following three conditions is satisfied
	\begin{enumerate}[\upshape(i)]
		\item $s_0 - \frac{d}{p_0} > s_1 - \frac{d}{p_1}$.
		\item $s_0 - \frac{d}{p_0} = s_1 - \frac{d}{p_1}, \quad 0 < q_0 \leq q_1 \leq \infty$ \quad and \quad  $b_0 \geq b_1$.
		\item $s_0 - \frac{d}{p_0} = s_1 - \frac{d}{p_1}, \quad 0 < q_1 < q_0 \leq \infty$ \quad and \quad  $b_0 + \frac{1}{q_0} > b_1 + \frac{1}{q_1}$.
	\end{enumerate}

\end{rem}

Writing down Corollaries \ref{CorollaryBNewB} and \ref{Corollary11.4} in the special case $p_0 = p_1$ and $s_0 = s_1$, we establish the following embedding result.

\begin{cor}\label{TheoremEmbeddings1}
	   Let $0 <  p, q, r \leq \infty, s \in \mathbb{R}$ and $b \in \R \backslash \{0\}$. Then
	       \begin{equation}\label{TheoremEmbeddings2}
        B^{s, b-1/r+ 1/\min\{q,r\}}_{p,r}(\mathbb{R}^d) \hookrightarrow
        T^b_r B^s_{p, q}(\mathbb{R}^d) \hookrightarrow
       	B^{s, b-1/r+ 1/\max\{q,r\}}_{p,r}(\mathbb{R}^d)
    \end{equation}
    and
        \begin{equation}\label{TheoremEmbeddings3}
        B^{s, b}_{p,\min\{q,r\}}(\mathbb{R}^d) \hookrightarrow
        T^b_r B^s_{p, q}(\mathbb{R}^d) \hookrightarrow
       	B^{s, b}_{p,\max\{q, r\}}(\mathbb{R}^d).
    \end{equation}
\end{cor}

\begin{rem}
	All the parameters in Corollary \ref{TheoremEmbeddings1} are sharp.
\end{rem}

\newpage

\section{Embeddings between truncated Besov and Triebel--Lizorkin spaces
with fixed integrability parameter
}

\begin{thm}\label{TheoremFB1new}
	Let $0 < p < \infty, 0 < q_i, r_i \leq \infty, s_i \in \R$, and $b_i \in \R\backslash\{0\}$ for $i=0, 1$. Then
	\begin{equation}\label{TheoremFB1newSta}
		T^{b_0}_{r_0} F^{s_0}_{p, q_0}(\R^d) \hookrightarrow T^{b_1}_{r_1} B^{s_1}_{p, q_1}(\R^d)
	\end{equation}
	if and only if one of the following conditions is satisfied
	\begin{enumerate}[\upshape(i)]
		\item $s_0 > s_1$.
		\item $s_0 = s_1, \quad q_1 \geq \max\{p, q_0\}$, \qquad and \qquad $b_0 > b_1$.
		\item $s_0=s_1, \quad q_1 \geq \max\{p, q_0\}, \quad b_0 = b_1$, \qquad and \qquad $r_0 \leq r_1$.
		\item $s_0 = s_1, \quad q_1 < \max \{p, q_0\},$ \qquad and \qquad $b_0 + \frac{1}{\max\{p, q_0\}} > b_1 + \frac{1}{q_1}$.
		\item $s_0=s_1, \quad q_1 < \max \{p, q_0\}, \quad b_0 + \frac{1}{\max\{p, q_0\}} = b_1 + \frac{1}{q_1}$, \qquad and \qquad $r_0 \leq r_1$.
	\end{enumerate}
\end{thm}

\begin{thm}\label{TheoremFB2new}
	Let $0 < p < \infty, 0 < q_i, r_i \leq \infty, s_i \in \R$, and $b_i \in \R \backslash \{0\}$ for $i=0, 1$. Then
	$$
		T^{b_0}_{r_0} B^{s_0}_{p, q_0}(\R^d) \hookrightarrow T^{b_1}_{r_1} F^{s_1}_{p, q_1}(\R^d)
	$$
	if and only if one of the following conditions is satisfied
	\begin{enumerate}[\upshape(i)]
	\item $s_0 > s_1$.
		\item $s_0 = s_1, \quad q_0 \leq \min\{p, q_1\}$, \qquad and \qquad $b_0 > b_1$.
		\item $s_0=s_1, \quad q_0 \leq \min\{p, q_1\}, \quad b_0 = b_1$, \qquad and \qquad $r_0 \leq r_1$.
		\item $s_0 = s_1, \quad q_0 > \min \{p, q_1\},$ \qquad and \qquad $b_0 + \frac{1}{q_0} > b_1 + \frac{1}{\min\{p, q_1\}}$.
		\item $s_0=s_1, \quad q_0 > \min \{p, q_1\}, \quad b_0 + \frac{1}{q_0} = b_1 + \frac{1}{\min\{p, q_1\}}$, \qquad and \qquad $r_0 \leq r_1$.
	\end{enumerate}
\end{thm}

\begin{proof}[Proof of Theorem \ref{TheoremFB1new}]
	\emph{Sufficient conditions:} We claim that the following embedding is true
	\begin{equation}\label{ProofTheoremFB1new1}
		T^b_r F^s_{p, q}(\R^d) \hookrightarrow T^b_r B^{s}_{p, \max\{p, q\}}(\R^d).
	\end{equation}
	Indeed, using that $\ell_q \hookrightarrow \ell_p$ if $q \leq p$ and Minkowski's inequality if $q > p$, we have
	$$
		\bigg(\sum_{\nu=2^j-1}^{2^{j+1}-2} 2^{\nu s \max\{p, q\}}  \|(\varphi_\nu \widehat{f})^\vee\|^{\max\{p, q\}}_{L_p(\mathbb{R}^d)} \bigg)^{1/\max\{p, q\}} \lesssim \bigg\| \bigg(\sum_{\nu=2^j-1}^{2^{j+1}-2} 2^{\nu s q} |(\varphi_\nu \widehat{f})^\vee|^q \bigg)^{1/q} \bigg\|_{L_p(\R^d)}
	$$
	for any $j \in \N_0$. Consequently, \eqref{ProofTheoremFB1new1} is obtained.
	
	Under one of the conditions stated in (i)--(v), it follows from Theorem \ref{TheoremEmbeddingsBBCharacterization} (with $p_0=p_1=p$) that
	\begin{equation}\label{ProofTheoremFB1new2}
		T^{b_0}_{r_0} B^{s_0}_{p, \max\{p, q_0\}}(\R^d) \hookrightarrow T^{b_1}_{r_1}B^{s_1}_{p, q_1}(\R^d).
	\end{equation}
	
	Combining \eqref{ProofTheoremFB1new1} and \eqref{ProofTheoremFB1new2}, we achieve  $T^{b_0}_{r_0}F^{s_0}_{p, q_0}(\R^d) \hookrightarrow T^{b_1}_{r_1}B^{s_1}_{p, q_1}(\R^d).$
	
	\emph{Necessary conditions $s_0 \geq s_1$; $b_0 \geq b_1$ under $s_0 = s_1$; $r_0 \leq r_1$ under $s_0 = s_1$ and $b_0 = b_1$:} Assume \eqref{TheoremFB1newSta} holds, or equivalently (cf. Theorems \ref{ThmWaveletsNewBesov} and \ref{ThmWaveletsNewTriebelLizorkin}),
	\begin{equation}\label{ProofTheoremFB1new3}
		T^{b_0}_{r_0}f^{s_0}_{p, q_0} \hookrightarrow T^{b_1}_{r_1}b^{s_1}_{p, q_1}.
	\end{equation}
	Let $\beta = (\beta_k)_{k \in \N_0}$ be a scalar-valued sequence and define $\lambda = (\lambda^{j, G}_m)$ by
		\begin{equation*}
		\lambda^{j,G}_{m} = \left\{\begin{array}{cl}  \beta_k, & \quad j = 2^k, \quad m= (0, \ldots, 0),  \\
		& \qquad G = (M, \ldots, M),  \\
		0, & \text{otherwise}.
		       \end{array}
                        \right.
	\end{equation*}
	Simple computations show
	$$
		\|\lambda\|_{T^{b_1}_{r_1} b^{s_1}_{p, q_1}} \asymp \bigg(\sum_{k=0}^\infty 2^{k b_1 r_1} 2^{2^k (s_1-d/p) r_1} |\beta_k|^{r_1} \bigg)^{1/r_1}
	$$
	and
	$$
		\|\lambda\|_{T^{b_0}_{r_0}f^{s_0}_{p, q_0}} \asymp  \bigg(\sum_{k=0}^\infty 2^{k b_0 r_0} 2^{2^k (s_0-d/p) r_0} |\beta_k|^{r_0} \bigg)^{1/r_0},
	$$
	thus \eqref{ProofTheoremFB1new3} implies
	\begin{equation}\label{ProofTheoremFB1new4}
		\bigg(\sum_{k=0}^\infty 2^{k b_1 r_1} 2^{2^k (s_1-d/p) r_1} |\beta_k|^{r_1} \bigg)^{1/r_1} \lesssim \bigg(\sum_{k=0}^\infty 2^{k b_0 r_0} 2^{2^k (s_0-d/p) r_0} |\beta_k|^{r_0} \bigg)^{1/r_0}.
	\end{equation}
	
	For every $N \in \N_0$, we set $\beta$ with $\beta_N = 1$ and $\beta_k = 0$ if $k \neq N$. According to \eqref{ProofTheoremFB1new4},
	$$
		2^{N b_1} 2^{2^N s_1} \lesssim 2^{N b_0} 2^{2^N s_0}, \qquad N \in \N.
	$$
	Therefore $s_0 \geq s_1$ and $b_0 \geq b_1$ provided that $s_0 = s_1$. On the other hand, if $s= s_0 = s_1$ and $b=b_0 = b_1$, then letting $\beta_k = 2^{-k b} 2^{-2^k (s-d/p)} \xi_k$ in \eqref{ProofTheoremFB1new4}, we arrive at $r_0 \leq r_1$.
	
	\emph{Necessary condition $b_0 + 1/p \geq b_1 + 1/q_1$ under $s= s_0 = s_1$:} For every $N \in \N$, we let		\begin{equation*}
		\lambda^{j,G}_{m} = \left\{\begin{array}{cl}  2^{-j(s-d/p)}, & \quad j \in \{2^{N}-1, \cdots, 2^{N+1}-2\}, \quad m= (0, \ldots, 0),  \\
		& \qquad G = (M, \ldots, M),  \\
		0, & \text{otherwise}.
		       \end{array}
                        \right.
	\end{equation*}
	Clearly
	\begin{equation}\label{ProofFJGen5new}
		\|\lambda\|_{T^{b_1}_{r_1} b^{s}_{p, q_1}} \asymp 2^{N (b_1 + 1/q_1)}.
	\end{equation}
	On the other hand
	\begin{equation}\label{ProofTheoremFB1new5}
		\|\lambda\|_{T^{b_0}_{r_0} f^{s}_{p, q_0}} = 2^{N b_0} \bigg\| \bigg(\sum_{j=2^N - 1}^{2^{N+1}-2} 2^{ j d q_0/p} \chi_{j, (0, \ldots, 0)} \bigg)^{1/q_0} \bigg\|_{L_p(\R^d)}.
	\end{equation}
		Furthermore, we claim that, for every $x \in \R^d,$
	\begin{equation}\label{ProofFJGen4}
		 \bigg(\sum_{j=2^N-1}^{2^{N+1}-2} 2^{j d q_0/p} \chi_{j, (0, \ldots, 0)}(x) \bigg)^{1/q_0} \asymp  \bigg(\sum_{j=2^N-1}^{2^{N+1}-2} 2^{j d} \chi_{j, (0, \ldots, 0)} (x) \bigg)^{1/p}.
	\end{equation}
	Indeed, this becomes obvious if $x \not \in Q_{2^N-1, (0, \ldots, 0)}$, so that the only non-trivial case to check occurs if $x \in Q_{2^N-1, (0, \ldots, 0)}$. In this case, we denote by $j_x \in \{2^N-1, \cdots, 2^{N+1}-3\}$ such that $x \in Q_{j_x, (0, \ldots, 0)}$ but $x \not \in Q_{j_x + 1, (0, \ldots, 0)}$ and we set $j_x = 2^{N+1}-2$ if $x \in Q_{2^{N+1}-2, (0, \ldots, 0)}$. Accordingly
	\begin{align*}
	 \bigg(\sum_{j=2^N-1}^{2^{N+1}-2} 2^{j d q_0/p} \chi_{j, (0, \ldots, 0)}(x) \bigg)^{1/q_0} & =  \bigg(\sum_{j=2^N-1}^{j_x} 2^{j d q_0/p}  \bigg)^{1/q_0} \\
	 &\hspace{-5cm} \asymp 2^{j_x d/p} \asymp \bigg(\sum_{j=2^N-1}^{j_x} 2^{j d}  \bigg)^{1/p} =  \bigg(\sum_{j=2^N-1}^{2^{N+1}-2} 2^{j d} \chi_{j, (0, \ldots, 0)}(x) \bigg)^{1/p}.
	\end{align*}
	Inserting \eqref{ProofFJGen4} into \eqref{ProofTheoremFB1new5}, we get
	\begin{equation}\label{ProofFJGen6}
		\|\lambda\|_{T^{b_0}_{r_0} f^{s}_{p, q_0}} \asymp 2^{N (b_0 + 1/p)}.
	\end{equation}
	
	As a combination of \eqref{ProofTheoremFB1new3}, \eqref{ProofFJGen5new} and \eqref{ProofFJGen6},
	$$
		2^{N (b_1 + 1/q_1)} \lesssim 2^{N (b_0 + 1/p)}, \qquad N \in \N.
	$$
	This implies $b_1 + 1/q_1 \leq b_0 + 1/p$.
	
	\emph{Necessary condition $r_0 \leq r_1$ under $s= s_0 = s_1$ and $b_0 + 1/p = b_1 + 1/q_1$:} We shall proceed by contradiction, i.e., we assume the validity of \eqref{ProofTheoremFB1new3} with $r_0 > r_1$ and we will arrive at a contradiction. Indeed, consider the sequence $\lambda$ given by
		\begin{equation*}
		\lambda^{j,G}_{m} = \left\{\begin{array}{cl}  2^{-j(s-d/p)} (1+j)^{-(b_1 + 1/q_1)} (1 + \log (1 + j))^{-\varepsilon}, & \quad j \in \N_0, \quad m= (0, \ldots, 0),  \\
		& \qquad G = (M, \ldots, M),  \\
		0, & \text{otherwise},
		       \end{array}
                        \right.
	\end{equation*}
	where $1/r_0 < \varepsilon < 1/r_1$. Then
	$$
		\|\lambda\|_{T^{b_1}_{r_1} b^{s}_{p, q_1}} \asymp \bigg(\sum_{k=0}^\infty  (1+ k)^{-\varepsilon r_1} \bigg)^{1/r_1} = \infty
	$$
	and, applying a similar reasoning as in \eqref{ProofFJGen4}, we have
	\begin{align*}
		\|\lambda\|_{T^{b_0}_{r_0} f^{s}_{p, q_0}} & \asymp \bigg(\sum_{k=0}^\infty 2^{k b_0 r_0} \bigg(\sum_{j=2^k-1}^{2^{k+1}-2} (1+j)^{-(b_1 + 1/q_1) p} (1 + \log (1+j))^{-\varepsilon p} \bigg)^{r_0/p} \bigg)^{1/r_0} \\
		&\hspace{-1.5cm} \asymp \bigg(\sum_{k=0}^\infty 2^{k b_0 r_0} 2^{-k(b_1 + 1/q_1) r_0} (1 + k)^{-\varepsilon r_0} 2^{k r_0/p} \bigg)^{1/r_0} = \bigg(\sum_{k=0}^\infty (1+k)^{-\varepsilon r_0} \bigg)^{1/r_0} < \infty.
	\end{align*}
	This yields the desired contradiction.
	
	\emph{Necessary conditions $b_0 + 1/q_0 \geq b_1 + 1/q_1$ under $s= s_0 = s_1$:} Let $(b_j)_{j \in \N_0}$ be a scalar-valued sequence and consider the related lacunary Fourier series given by
	\begin{equation*}
	f(x) \sim \sum_{j=3}^\infty b_j e^{i (2^j -2) x_1} \psi (x), \qquad x = (x_1,\ldots, x_d) \in \R^d,
	\end{equation*}
	where $\psi \in \mathcal{S}(\R^d) \backslash \{0\}$ is fixed. It is plain to check that
	$$
		(\varphi_j \widehat{f})^\vee (x) = b_j e^{i (2^j-2) x_1} \psi(x)
	$$
	for $j \in \N_0$. Therefore
	\begin{equation}\label{ProofFJGen8}
		\|f\|_{T^{b_0}_{r_0} F^{s}_{p, q_0}(\R^d)} \asymp \bigg(\sum_{j=0}^\infty 2^{j b_0 r_0} \bigg(\sum_{\nu=2^j-1}^{2^{j+1}-2} 2^{\nu s q_0} |b_\nu|^{q_0} \bigg)^{r_0/q_0}  \bigg)^{1/r_0}
	\end{equation}
	and
	\begin{equation}\label{ProofFJGen9}
		\|f\|_{T^{b_1}_{r_1} B^{s}_{p, q_1}(\R^d)} \asymp \bigg(\sum_{j=0}^\infty 2^{j b_1 r_1} \bigg(\sum_{\nu=2^j-1}^{2^{j+1}-2} 2^{\nu s q_1} |b_\nu|^{q_1} \bigg)^{r_1/q_1} \bigg)^{1/r_1}.
	\end{equation}
	It follows from \eqref{ProofTheoremFB1new3}, \eqref{ProofFJGen8} and \eqref{ProofFJGen9} that
	\begin{equation}\label{ProofFJGen10}
	 \bigg(\sum_{j=0}^\infty 2^{j b_1 r_1} \bigg(\sum_{\nu=2^j-1}^{2^{j+1}-2} 2^{\nu s q_1} |b_\nu|^{q_1} \bigg)^{r_1/q_1} \bigg)^{1/r_1} \lesssim \bigg(\sum_{j=0}^\infty 2^{j b_0 r_0} \bigg(\sum_{\nu=2^j-1}^{2^{j+1}-2} 2^{\nu s q_0} |b_\nu|^{q_0} \bigg)^{r_0/q_0}  \bigg)^{1/r_0}.
	\end{equation}
	Evaluating this inequality for the sequences $b^N = (b_\nu^N)_{\nu \in \N_0}, \, N \in \N_0,$ defined by
	$$
		b_\nu =  \left\{\begin{array}{cl}  2^{-\nu s}, & \quad \nu \in \{2^N-1, \ldots, 2^{N+1}-2\}, \\
		0, & \text{otherwise,}
		       \end{array}
                        \right.
	$$
	we obtain
	$$
		2^{N (b_1 + 1/q_1)} \lesssim 2^{N(b_0 + 1/q_0)}.
	$$
	Hence $b_1 + 1/q_1 \leq b_0 + 1/q_0$.
	
	\emph{Necessary condition $r_0 \leq r_1$ under $s = s_0 = s_1$ and $b_0 + 1/q_0 = b_1+ 1/q_1$:} Let $\lambda = (\lambda_j)_{j \in \N_0}$ an arbitrary sequence. We apply the method constructed above (cf. \eqref{ProofFJGen10}) for the special sequence $b_\nu = 2^{-\nu s} 2^{-j(b_0 + 1/q_0)} \lambda_j$ for $\nu \in \{2^j-1, \ldots, 2^{j+1}-2\}$ and $j \in \N_0$. Thus,
	$$
		\bigg(\sum_{j=0}^\infty |\lambda_j|^{r_1}  \bigg)^{1/r_1} \lesssim \bigg(\sum_{j=0}^\infty |\lambda_j|^{r_0} \bigg)^{1/r_0}
	$$
	which yields $r_0 \leq r_1$.
\end{proof}

\begin{proof}[Proof of Theorem \ref{TheoremFB2new}]
	The proof follows the same lines as the proof of Theorem \ref{TheoremFB1new}.
\end{proof}

As an immediate consequence of Theorems  \ref{TheoremFB1new} and \ref{TheoremFB2new}, we obtain (see also \eqref{PropElem1*})

\begin{cor}\label{CorTruncBTLFixedps}
Let $0 < p < \infty, 0 < q, r \leq \infty,$ and $s, b \in \R$. Then
	 $$
	 	T^b_r B^{s}_{p, \min\{p, q\}} (\R^d) \hookrightarrow T^b_r F^s_{p, q}(\R^d) \hookrightarrow T^b_r B^{s}_{p, \max\{p, q\}}(\R^d).
	 $$
\end{cor}

\begin{rem}
	The exponents $\min\{p, q\}$ and $\max\{p, q\}$ in previous embeddings are sharp.
\end{rem}

\subsection{Embeddings between truncated Triebel--Lizorkin spaces and classical Besov spaces with fixed integrability parameter}\label{SectionEmbTrunTLBFixed} In Section \ref{SectionTrunBB} we have investigated the relations between truncated and classical Besov spaces. Now we turn our attention to the corresponding question for truncated Triebel--Lizorkin spaces and classical Besov spaces. More precisely, letting $r_1 = q_1$ and $r_0 = q_0$ in Theorems \ref{TheoremFB1new} and \ref{TheoremFB2new}, respectively, we get the following

\begin{cor}\label{Cornew12}
	Let $0 < p < \infty, 0 < q_i, r_i \leq \infty, s_i \in \R$, and $b_i \in \R\backslash\{0\}$ for $i=0, 1$. Then
	\begin{equation*}
		T^{b_0}_{r_0} F^{s_0}_{p, q_0}(\R^d) \hookrightarrow B^{s_1, b_1}_{p, q_1}(\R^d)
	\end{equation*}
	if and only if one of the following conditions is satisfied
	\begin{enumerate}[\upshape(i)]
		\item $s_0 > s_1$.
		\item $s_0 = s_1, \quad q_1 \geq \max\{p, q_0\}$, \qquad and \qquad $b_0 > b_1$.
		\item $s_0=s_1, \quad q_1 \geq \max\{p, q_0\}, \quad b_0 = b_1$, \qquad and \qquad $r_0 \leq q_1$.
		\item $s_0 = s_1, \quad q_1 < \max \{p, q_0\},$ \qquad and \qquad $b_0 + \frac{1}{\max\{p, q_0\}} > b_1 + \frac{1}{q_1}$.
		\item $s_0=s_1, \quad q_1 < \max \{p, q_0\}, \quad b_0 + \frac{1}{\max\{p, q_0\}} = b_1 + \frac{1}{q_1}$, \qquad and \qquad $r_0 \leq q_1$.
	\end{enumerate}
\end{cor}

\begin{cor}\label{Cornew123}
	Let $0 < p < \infty, 0 < q_i , r_i \leq \infty, \, s_i \in \mathbb{R}$ and $b_i \in \R \backslash \{0\}$ for $i=0, 1$.  Then
\begin{equation*}
	B^{s_0,b_0}_{p,q_0}(\mathbb{R}^d) \hookrightarrow T^{b_1}_{r_1} F^{s_1}_{p ,q_1}(\mathbb{R}^d)
\end{equation*}
if and only if one of the following five conditions is satisfied
	\begin{enumerate}[\upshape(i)]
	\item $s_0 > s_1$.
		\item $s_0 = s_1, \quad q_0 \leq \min\{p, q_1\}$, \qquad and \qquad $b_0 > b_1$.
		\item $s_0=s_1, \quad q_0 \leq \min\{p, q_1\}, \quad b_0 = b_1$, \qquad and \qquad $q_0 \leq r_1$.
		\item $s_0 = s_1, \quad q_0 > \min \{p, q_1\},$ \qquad and \qquad $b_0 + \frac{1}{q_0} > b_1 + \frac{1}{\min\{p, q_1\}}$.
		\item $s_0=s_1, \quad q_0 > \min \{p, q_1\}, \quad b_0 + \frac{1}{q_0} = b_1 + \frac{1}{\min\{p, q_1\}}$, \qquad and \qquad $q_0 \leq r_1$.
	\end{enumerate}
\end{cor}

Letting $s_0=s_1$ in Corollaries \ref{Cornew12} and \ref{Cornew123}, we arrive at the analogue of Corollary \ref{TheoremEmbeddings1} for truncated Triebel--Lizorkin spaces.

	 \begin{cor}\label{CorollaryEmbFBNew} Let $0 < p < \infty, 0 < q, r \leq \infty, s \in \R$ and $b \in \R \backslash \{0\}$. Then
	 $$
	   B^{s, b - 1/r + 1/\min\{p, q, r\}}_{p, r}(\R^d)	 \hookrightarrow T^b_r F^s_{p, q}(\R^d)  \hookrightarrow B^{s, b -1/r + 1/\max\{p, q, r\}}_{p, r}(\R^d)
	 $$
	 and
	 $$
	 	 B^{s, b}_{p, \min\{p, q, r\}}(\R^d) \hookrightarrow T^b_r F^s_{p, q}(\R^d)  \hookrightarrow  B^{s, b}_{p, \max\{p, q, r\}}(\R^d).
	 $$
	 \end{cor}
	
	 \begin{rem}
	All the parameters in Corollary \ref{CorollaryEmbFBNew} are sharp.
\end{rem}

\subsection{Embeddings for $\text{Lip}^{s, b}_{p, q}(\R^d)$ and $\textbf{B}^{0, b}_{p, q}(\R^d)$ spaces}
Let us show how embeddings from Section \ref{SectionEmbTrunTLBFixed} can be applied to sharpen and extend  embedding theorems for $\text{Lip}^{s, b}_{p, q}(\R^d)$ and $\textbf{B}^{0, b}_{p, q}(\R^d)$ available in the literature.

Recall the embeddings between Lipschitz and Besov spaces obtained in \cite[Theorem 4.1]{DominguezHaroskeTikhonov} (see also \eqref{EmbBL}): Let $1 < p < \infty, 0 < q \leq \infty,  s> 0$ and $b < -1/q$. Then
\begin{equation}\label{LipB1nenenenen}
	 B^{s, b + 1/\min\{2, p, q\}}_{p, q}(\R^d)  \hookrightarrow \text{Lip}^{s, b}_{p, q}(\R^d) \hookrightarrow B^{s, b + 1/\max\{2, p, q\}}_{p, q}(\R^d)
	 \end{equation}
	 and
	 \begin{equation}\label{LipB2nenenenen}
	 B^{s, b + 1/q}_{p, \min\{2, p, q\}}(\R^d)  \hookrightarrow \text{Lip}^{s, b}_{p, q}(\R^d) \hookrightarrow B^{s, b + 1/q}_{p, \max\{2, p, q\}}(\R^d).
	 \end{equation}
	 These embeddings are, in general, not comparable and they are optimal within the classical scale formed by $B^{s, b}_{p, q}(\R^d)$, cf. \cite[Remark 4.2]{DominguezHaroskeTikhonov}.
	
	 Let us recall embedding theorems related to $\textbf{B}^{0, b}_{p, q}(\R^d)$ (cf. \eqref{EmbBBzero}): If $1 < p < \infty, 0 < q \leq \infty$ and $b > -1/q$, then
	 \begin{equation}\label{LipB3nenenenen}
	 B^{0, b + 1/\min\{2, p, q\}}_{p, q}(\R^d)  \hookrightarrow \textbf{B}^{0, b}_{p, q}(\R^d) \hookrightarrow B^{0, b + 1/\max\{2, p, q\}}_{p, q}(\R^d).
\end{equation}
Furthermore, these embeddings are sharp within the classical scale of $B^{0, b}_{p, q}(\R^d)$ spaces, cf. \cite[Section 9.1]{DominguezTikhonov}.
	
	 Since (cf. Proposition \ref{PropositionCoincidences})
	 \begin{equation}\label{sjasjjsa}
	 	\text{Lip}^{s, b}_{p, q}(\R^d)  = T^{b+1/q}_q F^{s}_{p, 2}(\R^d), \, \textbf{B}^{0, b}_{p, q}(\R^d)  = T^{b+1/q}_q F^{0}_{p, 2}(\R^d), \text{ and } B^{s, b}_{p, q}(\R^d) = B^{s, b}_{p, q, q}(\R^d),
	 \end{equation}
	 the embeddings \eqref{LipB1nenenenen}--\eqref{LipB3nenenenen} are contained in Corollary \ref{CorollaryEmbFBNew}. However, these embeddings admits non-trivial improvements in terms of truncated Besov spaces via Corollary \ref{CorTruncBTLFixedps}.

	 \begin{cor}[Embeddings between $\text{Lip}^{s, b}_{p, q}(\R^d)$ and $T^b_r B^s_{p, q}(\R^d)$]\label{Cor1}
	 Let $1 < p < \infty, 0 < q \leq \infty, s > 0$, and $b < -1/q$. Then
	 $$
	 	T^{b+1/q}_q B^{s}_{p, \min\{2, p\}}(\R^d) \hookrightarrow \emph{Lip}^{s, b}_{p, q}(\R^d) \hookrightarrow T^{b+1/q}_q B^{s}_{p, \max\{2, p\}}(\R^d).
	 $$
	 \end{cor}

\begin{cor}[Embeddings between $\textbf{B}^{0, b}_{p, q}(\R^d)$ and $T^b_r B^{0}_{p, q}(\R^d)$]\label{Cor2}
	Let $1 < p < \infty, 0 < q \leq \infty$ and $b > -1/q$. Then
	$$
		 T^{b+1/q}_q B^{0}_{p, \min\{2, p\}}(\R^d) \hookrightarrow \emph{\textbf{B}}^{0, b}_{p, q}(\R^d) \hookrightarrow T^{b+1/q}_q B^{0}_{p, \max\{2, p\}}(\R^d).
	$$
\end{cor}

\begin{rem}
The embeddings given in Corollaries \ref{Cor1} and \ref{Cor2} improves  \eqref{LipB1nenenenen}--\eqref{LipB3nenenenen}. Indeed, note that (cf. Corollary \ref{TheoremEmbeddings1})
	 $$
	 	T^{b+1/q}_q B^{s}_{p, \max\{2, p\}}(\R^d) \hookrightarrow B^{s, b + 1/\max\{2, p, q\}}_{p, q}(\R^d) \cap B^{s, b + 1/q}_{p, \max\{2, p, q\}}(\R^d)
	 $$
	 and
	 $$
	  B^{s, b + 1/\min\{2, p, q\}}_{p, q}(\R^d) + B^{s, b + 1/q}_{p, \min\{2, p, q\}}(\R^d)	\hookrightarrow  T^{b+1/q}_q B^{s}_{p, \min\{2, p\}}(\R^d).
	 $$
	 Furthermore, all the parameters involved in the embeddings given in Corollaries \ref{Cor1} and \ref{Cor2} are optimal.
\end{rem}

\newpage
\section{Embeddings between truncated Triebel--Lizorkin spaces
}

\begin{thm}\label{ThmFFp0p1}
	Let $0 < p_0 \leq p_1 < \infty, 0 < q_i , r_i \leq \infty, s_i \in \R,$ and $b_i \in \R \backslash \{0\}$ for $i=0, 1$. Then
	\begin{equation}\label{ThmFFp0p1State}
		T^{b_0}_{r_0} F^{s_0}_{p_0, q_0}(\R^d) \hookrightarrow T^{b_1}_{r_1} F^{s_1}_{p_1, q_1}(\R^d)
	\end{equation}
	if and only if one of the following conditions is satisfied
	\begin{enumerate}[\upshape(i)]
		\item $s_0 - \frac{d}{p_0} > s_1 - \frac{d}{p_1}$.
		\item $s_0 - \frac{d}{p_0} = s_1 - \frac{d}{p_1}, \quad p_0 < p_1,$ \qquad and \qquad $b_0 > b_1$.
		\item $s_0 - \frac{d}{p_0} = s_1 - \frac{d}{p_1}, \quad p_0 < p_1, \quad b_0 = b_1$, \qquad and \qquad $r_0 \leq r_1$.
		\item  $s_0 - \frac{d}{p_0} = s_1 - \frac{d}{p_1}, \quad p_0 = p_1, \quad q_0 \leq q_1,$ \qquad and \qquad $b_0 > b_1$.
		\item $s_0 - \frac{d}{p_0} = s_1 - \frac{d}{p_1}, \quad p_0 = p_1, \quad q_0 \leq q_1, \quad b_0 = b_1$\qquad and \qquad $r_0 \leq r_1$.
		\item $s_0 - \frac{d}{p_0} = s_1 - \frac{d}{p_1}, \quad p_0 = p_1, \quad q_1 < q_0,$\qquad and \qquad$b_0 + \frac{1}{q_0} > b_1 + \frac{1}{q_1}$.
		\item $s_0 - \frac{d}{p_0} = s_1 - \frac{d}{p_1}, \quad p_0 = p_1, \quad q_1 < q_0, \quad b_0 + \frac{1}{q_0} = b_1 + \frac{1}{q_1}$\qquad and \qquad$r_0 \leq r_1$.
	\end{enumerate}
\end{thm}

Before we give the proof of Theorem \ref{ThmFFp0p1}, we write down some of its immediate consequences. Namely, since (cf. Proposition \ref{PropositionCoincidences})
\begin{equation}\label{LipFnew1}
\text{Lip}^{s, b}_{p, q}(\R^d) = T^{b+1/q}_q F^{s}_{p, 2}(\R^d),
\end{equation}
 as a special case of Theorem \ref{ThmFFp0p1} we derive sharp embeddings between Lipschitz spaces.

\begin{cor}[Embeddings between Lipschitz spaces]\label{CorLipLip}
Let $1 < p_0 \leq p_1 < \infty, 0 < q_i \leq \infty, s_i > 0$ and $b_i < - 1/q_i$ for $i=0, 1$. Then
	$$
		\emph{Lip}^{s_0, b_0}_{p_0, q_0}(\R^d) \hookrightarrow  \emph{Lip}^{s_1, b_1}_{p_1, q_1}(\R^d)
	$$
	if and only if one of the following conditions is satisfied
	\begin{enumerate}[\upshape(i)]
		\item $s_0 - \frac{d}{p_0} > s_1 - \frac{d}{p_1}$.
		\item $s_0 - \frac{d}{p_0} = s_1 - \frac{d}{p_1}$, \qquad and \qquad $b_0 + \frac{1}{q_0} > b_1 + \frac{1}{q_1}$.
		 \item $s_0 - \frac{d}{p_0} = s_1 - \frac{d}{p_1}, \quad b_0 + \frac{1}{q_0} = b_1 + \frac{1}{q_1}$ \qquad and \qquad $q_0 \leq q_1$.
	\end{enumerate}
	\end{cor}
	
	This result extends \cite[Theorem 5.1]{DominguezHaroskeTikhonov} where the proposed  methodology restricts to the case $p_0=p_1$. See also \cite[Theorem 5.3]{DominguezHaroskeTikhonov}. The limiting cases $p_0 = 1$ and/or $p_1=\infty$ in Corollary \ref{CorLipLip} are also of great interest since they are related to the classical Lipschitz space $\text{Lip}(\R^d)$ and the space $\text{BV}(\R^d)$ formed by bounded variation functions; these cases were already investigated in detail in \cite[Theorems 5.5, 5.7, 5.9 and 5.10]{DominguezHaroskeTikhonov}.

Another consequence of Theorem \ref{ThmFFp0p1} comes from the combination of (cf. Proposition \ref{PropositionCoincidences})
$$
	\textbf{B}^{0, b}_{p, q}(\R^d) = T^{b+1/q}_q F^{0}_{p, 2}(\R^d)
$$
and \eqref{LipFnew1}. To be more precise, we can establish the following

\begin{cor}[Embeddings between Lipschitz spaces and Besov spaces of smoothness near zero]
	Let $1 < p_0 \leq p_1 < \infty, 0 < q_i \leq \infty, s_0 > 0, b_0 < -1/q_0$ and $b_1 > -1/q_1$. Then
	$$
		\emph{Lip}^{s_0, b_0}_{p_0, q_0}(\R^d) \hookrightarrow \mathbf{B}^{0, b_1}_{p_1, q_1}(\R^d)
	$$
	if and only if
	$$
		s_0 -\frac{d}{p_0} > -\frac{d}{p_1}.
	$$
\end{cor}

\begin{proof}[Proof of Theorem \ref{ThmFFp0p1}]
\emph{Sufficient condition} (i): Let $s_0 -d/p_0 > s_1-d/p_1$. According to Corollary \ref{CorTruncBTLFixedps} and Theorem \ref{TheoremEmbeddingsBBCharacterization}, we have
$$
	T^{b_0}_{r_0} F^{s_0}_{p_0, q_0}(\R^d) \hookrightarrow T^{b_0}_{r_0} B^{s_0}_{p_0, \max\{p_0, q_0\}}(\R^d) \hookrightarrow T^{b_1}_{r_1} B^{s_1}_{p_1, \min\{p_1, q_1\}}(\R^d) \hookrightarrow T^{b_1}_{r_1}F^{s_1}_{p_1, q_1}(\R^d).	
$$

\emph{Sufficient conditions} (ii)-(iii): Assume $s_0 - d/p_0 = s_1-d/p_1$ and $p_0 < p_1$ (in particular, $s_0 > s_1$). We will show that, for every $j \in \N_0$,
\begin{equation}\label{ThmFFp0p1Proof0}
	\bigg\|\bigg(\sum_{\nu=2^j-1}^{2^{j+1}-2} 2^{\nu s_1 q_1} |(\varphi_\nu \widehat{f})^\vee(x)|^{q_1} \bigg)^{\frac{1}{q_1}} \bigg\|_{L_{p_1}(\R^d)} \lesssim \bigg\| \sup_{2^{j}-1 \leq \nu \leq 2^{j+1}-2} 2^{\nu s_0} |(\varphi_\nu \widehat{f})^\vee|  \bigg\|_{L_{p_0}(\R^d)}.
\end{equation}

To prove \eqref{ThmFFp0p1Proof0},  we make use of the following well-known inequality (see, e.g., \cite[Theorem 5.6.1]{BerghLofstrom}): let $\theta \in (0, 1), q_1 \in (0, \infty]$ and $s_1 = (1-\theta) s_2 + \theta s_0$, then
$$
	\bigg(\sum_{\nu=0}^\infty 2^{\nu s_1 q_1} |a_\nu|^{q_1} \bigg)^{1/q_1} \lesssim \bigg(\sup_{\nu \in \N_0} 2^{\nu s_0} |a_\nu| \bigg)^{\theta} \bigg(\sup_{\nu \in \N_0} 2^{\nu s_2} |a_\nu| \bigg)^{1-\theta}
$$
for every complex-valued sequence $(a_\nu)_{\nu \in \N_0}$. Consequently, given any $x \in \R^d$ and $j \in \N_0$, the previous inequality with $\theta = p_0/p_1$ gives
\begin{align}
	\bigg(\sum_{\nu=2^j-1}^{2^{j+1}-2} 2^{\nu s_1 q_1} |(\varphi_\nu \widehat{f})^\vee(x)|^{q_1} \bigg)^{\frac{1}{q_1}} & \lesssim \nonumber \\
	& \hspace{-4cm}\bigg(\sup_{2^{j}-1 \leq \nu \leq 2^{j+1}-2} 2^{\nu s_0} |(\varphi_\nu \widehat{f})^\vee(x)| \bigg)^{\frac{p_0}{p_1}} \bigg(\sup_{2^{j}-1 \leq \nu \leq  2^{j+1}-2} 2^{\nu s_2}  |(\varphi_\nu \widehat{f})^\vee(x)|  \bigg)^{1-\frac{p_0}{p_1}} \nonumber \\
	& \hspace{-4cm} \leq \bigg(\sup_{2^{j}-1 \leq \nu \leq 2^{j+1}-2} 2^{\nu s_0} |(\varphi_\nu \widehat{f})^\vee(x)| \bigg)^{\frac{p_0}{p_1}} \bigg(\sup_{2^{j}-1 \leq \nu \leq  2^{j+1}-2} 2^{\nu s_2}  \|(\varphi_\nu \widehat{f})^\vee\|_{L_\infty(\R^d)}  \bigg)^{1-\frac{p_0}{p_1}}.  \label{ThmFFp0p1Proof1}
\end{align}
Furthermore, by the classical Nikolskii's inequality for entire functions of exponential type (cf. \cite{Nikolskii51, Nikolskii}), we have
$$
	\|(\varphi_\nu \widehat{f})^\vee\|_{L_\infty(\R^d)} \lesssim 2^{\nu d/p_0} \|(\varphi_\nu \widehat{f})^\vee\|_{L_{p_0}(\R^d)}
$$
and thus (noting that $s_2 = s_0 - \frac{d}{p_0}$)
\begin{align}
	\sup_{2^{j}-1 \leq \nu \leq  2^{j+1}-2} 2^{\nu s_2}  \|(\varphi_\nu \widehat{f})^\vee\|_{L_\infty(\R^d)} &\lesssim \sup_{2^{j}-1 \leq \nu \leq  2^{j+1}-2} 2^{\nu s_0}  \|(\varphi_\nu \widehat{f})^\vee\|_{L_{p_0}(\R^d)} \nonumber \\
	& \leq \bigg\| \sup_{2^{j}-1 \leq \nu \leq  2^{j+1}-2}  2^{\nu s_0 q_0} |(\varphi_\nu \widehat{f})^\vee|^{q_0}  \bigg\|_{L_{p_0}(\R^d)}.  \label{ThmFFp0p1Proof2}
\end{align}

Putting together \eqref{ThmFFp0p1Proof1} and \eqref{ThmFFp0p1Proof2},
\begin{align*}
	\bigg(\sum_{\nu=2^j-1}^{2^{j+1}-2} 2^{\nu s_1 q_1} |(\varphi_\nu \widehat{f})^\vee(x)|^{q_1} \bigg)^{\frac{1}{q_1}} & \lesssim  \\
	& \hspace{-4cm} \bigg(\sup_{2^{j}-1 \leq \nu \leq 2^{j+1}-2} 2^{\nu s_0} |(\varphi_\nu \widehat{f})^\vee(x)| \bigg)^{\frac{p_0}{p_1}}   \bigg\| \sup_{2^{j}-1 \leq \nu \leq  2^{j+1}-2}  2^{\nu s_0 q_0} |(\varphi_\nu \widehat{f})^\vee|^{q_0}  \bigg\|_{L_{p_0}(\R^d)}^{1-\frac{p_0}{p_1}}
\end{align*}
and integrating on both sides of the previous estimate, we arrive at \eqref{ThmFFp0p1Proof0}.

It follows from \eqref{ThmFFp0p1Proof0} that
\begin{align}
	\|f\|_{T^{b_1}_{r_1} F^{s_1}_{p_1, q_1}(\R^d)} &=  \left(\sum_{j=0}^\infty 2^{j b_1 r_1}  \bigg\| \bigg(\sum_{\nu=2^j-1}^{2^{j+1}-2} 2^{\nu s_1 q_1} |(\varphi_\nu \widehat{f})^\vee|^{q_1} \bigg)^{1/q_1} \bigg\|_{L_{p_1}(\R^d)}^{r_1} \right)^{1/r_1} \nonumber \\
	& \lesssim \left(\sum_{j=0}^\infty 2^{j b_1 r_1}  \bigg\| \sup_{2^{j}-1 \leq \nu \leq 2^{j+1}-2} 2^{\nu s_0} |(\varphi_\nu \widehat{f})^\vee|  \bigg\|_{L_{p_0}(\R^d)}^{r_1}  \right)^{1/r_1} \nonumber \\
	& \leq  \left(\sum_{j=0}^\infty 2^{j b_1 r_1}  \bigg\|\bigg( \sum_{\nu=2^j-1}^{2^{j+1}-2} 2^{\nu s_0 q_0} |(\varphi_\nu \widehat{f})^\vee|^{q_0} \bigg)^{1/q_0}  \bigg\|_{L_{p_0}(\R^d)}^{r_1}  \right)^{1/r_1}. \label{ThmFFp0p1Proof3}
\end{align}
On the other hand, it is plain to see that
\begin{equation}\label{ThmFFp0p1Proof4}
	\bigg(\sum_{j=0}^\infty 2^{j b_1 r_1} |a_j|^{r_1} \bigg)^{1/r_1} \lesssim \bigg(\sum_{j=0}^\infty 2^{j b_0 r_0} |a_j|^{r_0} \bigg)^{1/r_0}
\end{equation}
holds provided that either
$
	b_0 > b_1
$ or $b_0=b_1$ and $r_0 \leq r_1$.

By \eqref{ThmFFp0p1Proof3} and \eqref{ThmFFp0p1Proof4}, we get
\begin{align*}
	\|f\|_{T^{b_1}_{r_1}F^{s_1}_{p_1, q_1}(\R^d)} &\lesssim \left(\sum_{j=0}^\infty 2^{j b_0 r_0}  \bigg\|\bigg( \sum_{\nu=2^j-1}^{2^{j+1}-2} 2^{\nu s_0 q_0} |(\varphi_\nu \widehat{f})^\vee|^{q_0} \bigg)^{1/q_0}  \bigg\|_{L_{p_0}(\R^d)}^{r_0}  \right)^{1/r_0} \\
	& = \|f\|_{T^{b_0}_{r_0} F^{s_0}_{p_0, q_0}(\R^d)}.
	\end{align*}

\emph{Sufficient conditions} (iv)-(vii): Let $p= p_0= p_1$ and $s= s_0 = s_1$. We have
$$
	\bigg( \sum_{\nu=2^j-1}^{2^{j+1}-2} 2^{\nu s q_1} |(\varphi_\nu \widehat{f})^\vee(x)|^{q_1} \bigg)^{1/q_1} \lesssim 2^{j(1/q_1-1/q_0)_+} \bigg( \sum_{\nu=2^j-1}^{2^{j+1}-2} 2^{\nu s q_0} |(\varphi_\nu \widehat{f})^\vee(x)|^{q_0} \bigg)^{1/q_0}
$$
for $x \in \R^d$ and $j \in \N_0$. This is clear from the embedding $\ell_{q_0} \hookrightarrow \ell_{q_1}$ if $q_0 \leq q_1$ and from H\"older's inequality if $q_1 < q_0$. Hence
\begin{align*}
	\|f\|_{T^{b_1}_{r_1} F^{s}_{p, q_1}(\R^d)} &\leq \bigg(\sum_{j=0}^\infty 2^{j (b_1 + (1/q_1-1/q_0)_+ )r_1} \bigg\|\bigg( \sum_{\nu=2^j-1}^{2^{j+1}-2} 2^{\nu s q_0} |(\varphi_\nu \widehat{f})^\vee|^{q_0} \bigg)^{1/q_0} \bigg\|_{L_p(\R^d)}^{r_1}  \bigg)^{1/r_1} \\
	& \lesssim \|f\|_{T^{b_0}_{r_0} F^{s}_{p, q_0}(\R^d)},
\end{align*}
where we have applied \eqref{ThmFFp0p1Proof4} in the last estimate.

\emph{Necessary condition $s_0 -\frac{d}{p_0} \geq s_1-\frac{d}{p_1}$ and $b_0 \geq b_1$ under $s_0 -\frac{d}{p_0} = s_1-\frac{d}{p_1}$:} Assume \eqref{ThmFFp0p1State} holds or equivalently, by Theorem \ref{ThmWaveletsNewTriebelLizorkin}, working at the sequence level
\begin{equation}\label{ThmFFp0p1Proof5}
	T^{b_0}_{r_0} f^{s_0}_{p_0, q_0} \hookrightarrow T^{b_1}_{r_1} f^{s_1}_{p_1, q_1}.
\end{equation}
	For every $k \in \N$, define the sequence $\lambda_k = (\lambda^{j,G}_{m, k})$ given by
	\begin{equation*}
		\lambda^{j,G}_{m, k} = \left\{\begin{array}{cl}  1, & \quad j = 2^k, \quad m= (0, \ldots, 0),  \\
		& \qquad G = (M, \ldots, M),  \\
		0, & \text{otherwise}.
		       \end{array}
                        \right.
	\end{equation*}
	Then
	$$
		\|\lambda_k\|_{T^{b_0}_{r_0} f^{s_0}_{p_0, q_0}} \asymp 2^{k b_0} 2^{2^k(s_0-d/p_0)} \qquad \text{and} \qquad \|\lambda_k\|_{T^{b_1}_{r_1} f^{s_1}_{p_1, q_1}} \asymp 2^{k b_1} 2^{2^k(s_1-d/p_1)}.
	$$
	The validity of \eqref{ThmFFp0p1Proof5} implies that either $s_0 - d/p_0 > s_1-d/p_1$ or $s_0 - d/p_0 = s_1-d/p_1$ and $b_0 \geq b_1$.
	
	\emph{Necessary condition $r_0 \leq r_1$ under $s_0 - \frac{d}{p_0} = s_1 - \frac{d}{p_1}$ and $b_0 = b_1$:} Let $\beta = (\beta_k)$ any scalar-valued sequence and define the related sequence
		\begin{equation*}
		\lambda^{j,G}_{m} = \left\{\begin{array}{cl} 2^{-2^k (s_0 - \frac{d}{p_0})} 2^{-k b_0} \beta_k , & \quad j = 2^k, \quad m= (0, \ldots, 0),  \\
		& \qquad G = (M, \ldots, M),  \\
		0, & \text{otherwise}.
		       \end{array}
                        \right.
	\end{equation*}
	Basic computations yield
	$$
		\|\lambda\|_{T^{b_0}_{r_0} f^{s_0}_{p_0, q_0}} \asymp \|\beta\|_{\ell_{r_0}} \qquad \text{and} \qquad \|\lambda\|_{T^{b_1}_{r_1} f^{s_1}_{p_1, q_1}} \asymp \|\beta\|_{\ell_{r_1}}.
	$$
	Therefore, by \eqref{ThmFFp0p1Proof5}, we conclude $r_0 \leq r_1$.
	
		\emph{Necessary condition $b_1 + \frac{1}{q_1} \leq b_0 + \frac{1}{q_0}$ under $s=s_0 = s_1$ and $p=p_0=p_1$:} We let $\lambda = (\lambda^{j, G}_{m })$ where
		\begin{equation*}
		\lambda^{j,G}_{m} = \left\{\begin{array}{cl}  2^{-j s_0} \beta_k, & \quad j \in \{2^k-1, \cdots, 2^{k+1}-2\}, \quad k \in \N_0, \quad m= (0, \ldots, 0),  \\
		& \qquad G = (M, \ldots, M),  \\
		0, & \text{otherwise}.
		       \end{array}
                        \right.
	\end{equation*}
	Here $\beta = (\beta_k)_{k \in \N_0}$ denotes any scalar-valued sequence to be chosen.
	
	Then
	\begin{equation}\label{ProofFJGen3new}
	\|\lambda\|_{T^{b_0}_{r_0} f^{s}_{p, q_0}} =\bigg( \sum_{k=0}^\infty 2^{k b_0 r_0} |\beta_k|^{r_0} \bigg\| \bigg(\sum_{j=2^k-1}^{2^{k+1}-2}  \chi_{j, (0, \ldots, 0)} \bigg)^{1/q_0} \bigg\|_{L_{p}(\R^d)}^{r_0} \bigg)^{1/r_0}.
	\end{equation}
	We can estimate
	\begin{align*}
		\bigg\| \bigg(\sum_{j=2^k-1}^{2^{k+1}-2}  \chi_{j, (0, \ldots, 0)} \bigg)^{1/q_0} \bigg\|_{L_{p}(\R^d)} &\geq \bigg\| \bigg(\sum_{j=2^k-1}^{2^{k+1}-2}  \chi_{j, (0, \ldots, 0)} \bigg)^{1/q_0} \bigg\|_{L_{p}(Q_{2^{k+1}-2, (0, \ldots, 0)})} \\
		&\asymp 2^{k/q_0} 2^{-2^{k+1} d/p}.
	\end{align*}
	Conversely,
	$$
	\bigg\| \bigg(\sum_{j=2^k-1}^{2^{k+1}-2}  \chi_{j, (0, \ldots, 0)} \bigg)^{1/q_0} \bigg\|_{L_{p}(\R^d)} \lesssim 2^{k/q_0} \|\chi_{2^k-1, (0, \ldots, 0)}\|_{L_{p}(\R^d)} \asymp 2^{k/q_0}  2^{-2^k d/p}.
	$$
	Inserting these estimates into \eqref{ProofFJGen3new}, we obtain
	\begin{equation}\label{ProofFJGen3new34}
		\|\lambda\|_{T^{b_0}_{r_0} f^{s}_{p, q_0}} \asymp \bigg( \sum_{k=0}^\infty 2^{k (b_0+ 1/q_0) r_0}   2^{-2^k d r_0/p} |\beta_k|^{r_0}\bigg)^{1/r_0}.
	\end{equation}
	Analogously,
	\begin{equation}\label{ProofFJGen3new35}
		\|\lambda\|_{T^{b_1}_{r_1} f^{s}_{p, q_1}} \asymp \bigg( \sum_{k=0}^\infty 2^{k (b_1+ 1/q_1) r_1}   2^{-2^k d r_1/p} |\beta_k|^{r_1}\bigg)^{1/r_1}.
	\end{equation}
	
	It follows from \eqref{ThmFFp0p1Proof5}, \eqref{ProofFJGen3new34} and \eqref{ProofFJGen3new35}, that
	\begin{equation}\label{ProofFJGen3new36}
		\bigg( \sum_{k=0}^\infty 2^{k (b_1+ 1/q_1) r_1}   2^{-2^k d r_1/p} |\beta_k|^{r_1}\bigg)^{1/r_1} \lesssim \bigg( \sum_{k=0}^\infty 2^{k (b_0+ 1/q_0) r_0}   2^{-2^k d r_0/p} |\beta_k|^{r_0}\bigg)^{1/r_0}.
	\end{equation}
	For an arbitrary $N \in \N_0$, we define
	$$
		\beta_k = \left\{\begin{array}{cl} 1, & \quad k \in \{0, \cdots, N\},  \\
		0, & k > N.
		       \end{array}
                        \right.
	$$
	By \eqref{ProofFJGen3new36}, we get
	$$
		2^{N(b_1 + 1/q_1)} \lesssim 2^{N(b_0 + 1/q_0)}
	$$
	which yields $b_1 + 1/q_1 \leq b_0 + 1/q_0$.
	
	\emph{Necessary condition $r_0 \leq r_1$ under $b_1 + \frac{1}{q_1} = b_0 + \frac{1}{q_0}, \, s=s_0 = s_1$ and $p=p_0=p_1$:} Given any sequence $(\xi_k)_{k \in \N_0}$, applying \eqref{ProofFJGen3new36} related to $(\beta_k)_{k \in \N_0} = (2^{2^k d/p} 2^{-k(b_1 + 1/q_1)} \xi_k)_{k \in \N_0}$, we obtain
	$$
		\bigg(\sum_{k=0}^\infty |\xi_k|^{r_1} \bigg)^{1/r_1} \lesssim \bigg(\sum_{k=0}^\infty |\xi_k|^{r_0} \bigg)^{1/r_0}.
	$$
	This implies $r_0 \leq r_1$.

\end{proof}

\newpage

\section{Franke--Jawerth embeddings  for truncated spaces
}\label{SectionFJ}

 We start by recalling the classical Franke--Jawerth embeddings \cite{Jawerth, Franke} (see also \cite{Marschall}, \cite{Sickel} and \cite{Vybiral08}). Let $0 < p_0 < p < p_1 \leq \infty, \, -\infty < s_1 < s < s_0 < \infty$ with
 $$s_0 - \frac{d}{p_0} = s -\frac{d}{p}= s_1-\frac{d}{p_1}.$$
  Let $0 < q, q_0, q_1 \leq \infty$. Then
\begin{equation}\label{FJClas}
	B^{s_0}_{p_0, q_0}(\R^d) \hookrightarrow F^{s}_{p, q}(\R^d)	 \hookrightarrow B^{s_1}_{p_1, q_1}(\R^d)
\end{equation}
if and only if
$$q_0 \leq p \leq q_1.$$

\begin{thm}\label{ThmFJ}
	Let $0 < p_0 < p_1 \leq \infty, -\infty < s_1 < s_0 < \infty, 0 < q_i, r_i  \leq \infty$ and $b_i \in \mathbb{R}\backslash \{0\}$ for $i=0, 1$. Then
	\begin{equation}\label{ThmFJState}
		T^{b_0}_{r_0} F^{s_0}_{p_0, q_0}(\R^d) \hookrightarrow T^{b_1}_{r_1} B^{s_1}_{p_1, q_1}(\R^d)
	\end{equation}
	if and only if one of the following conditions is satisfied
	\begin{enumerate}[\upshape(i)]
		\item $s_0 -\frac{d}{p_0} > s_1 - \frac{d}{p_1}$.
		\item $s_0 -\frac{d}{p_0} = s_1 - \frac{d}{p_1}, \quad 0 < p_0 \leq q_1 \leq \infty, \quad 0 < q_0 \leq \infty$ \quad and \quad $b_1 < b_0$.
		\item $s_0 -\frac{d}{p_0} = s_1 - \frac{d}{p_1}, \quad 0 < p_0 \leq q_1 \leq \infty, \quad 0 < q_0 \leq \infty, \quad b_1 = b_0$ \quad and \quad $r_0 \leq r_1$.
		\item $s_0 -\frac{d}{p_0} = s_1 - \frac{d}{p_1}, \quad 0 < q_1 < p_0, \quad 0 < q_0 \leq \infty$ \quad and \quad $b_1 + \frac{1}{q_1} < b_0 + \frac{1}{p_0}$.
		\item $s_0 -\frac{d}{p_0} = s_1 - \frac{d}{p_1}, \quad 0 < q_1 < p_0, \quad 0 < q_0 \leq \infty, \quad b_1 + \frac{1}{q_1} = b_0 + \frac{1}{p_0}$ and $0 < r_0 \leq r_1 \leq \infty$.
	\end{enumerate}
\end{thm}

\begin{thm}\label{ThmFJ2}
	Let $0 < p_0 < p_1 \leq \infty, -\infty < s_1 < s_0 < \infty, 0 < q_i, r_i  \leq \infty$ and $b_i \in \mathbb{R}\backslash \{0\}$ for $i=0, 1$. Then
	\begin{equation}\label{ThmFJState2}
		T^{b_0}_{r_0} B^{s_0}_{p_0, q_0}(\R^d) \hookrightarrow T^{b_1}_{r_1} F^{s_1}_{p_1, q_1}(\R^d)
	\end{equation}
	if and only if one of the following conditions is satisfied
	\begin{enumerate}[\upshape(i)]
		\item $s_0 -\frac{d}{p_0} > s_1 - \frac{d}{p_1}$.
		\item $s_0 -\frac{d}{p_0} = s_1 - \frac{d}{p_1}, \quad 0 < q_0 \leq p_1 \leq \infty, \quad 0 < q_1 \leq \infty$ \quad and \quad $b_1 < b_0$.
		\item $s_0 -\frac{d}{p_0} = s_1 - \frac{d}{p_1}, \quad 0 < q_0 \leq p_1 \leq \infty, \quad 0 < q_1 \leq \infty, \quad b_1 = b_0$ \quad and \quad $r_0 \leq r_1$.
		\item $s_0 -\frac{d}{p_0} = s_1 - \frac{d}{p_1}, \quad 0 < p_1 < q_0, \quad 0 < q_1 \leq \infty$ \quad and \quad $b_1 + \frac{1}{p_1} < b_0 + \frac{1}{q_0}$.
		\item $s_0 -\frac{d}{p_0} = s_1 - \frac{d}{p_1}, \quad 0 < p_1 < q_0, \quad 0 < q_1 \leq \infty, \quad b_1 + \frac{1}{p_1} = b_0 + \frac{1}{q_0}$ and $0 < r_0 \leq r_1 \leq \infty$.
	\end{enumerate}
\end{thm}

\begin{rem}
Let $s_0 -\frac{d}{p_0} = s_1 - \frac{d}{p_1}$ and $0 < q_0 \leq \infty$. Conditions (ii) and (iv) in Theorem \ref{ThmFJ} can be rewritten as
	$$
		  \quad b_0- b_1 > \Big(\frac{1}{q_1}-\frac{1}{p_0} \Big)_+,
	$$
	while conditions (iii) and (v) read as
	$$
		 b_0- b_1 = \Big(\frac{1}{q_1}-\frac{1}{p_0} \Big)_+ \qquad \text{and} \qquad r_0 \leq r_1.
	$$
	A similar comment applies to Theorem \ref{ThmFJ2} but now in terms of $(\frac{1}{p_1}-\frac{1}{q_0})_+$.
\end{rem}

\begin{proof}[Proof of Theorem \ref{ThmFJ}]
\emph{Sufficient condition} (i):  According to Theorems \ref{TheoremEmbeddingsBBCharacterization} and Corollary \ref{CorTruncBTLFixedps},
$$
	T^{b_0}_{r_0} F^{s_0}_{p_0, q_0}(\R^d) \hookrightarrow T^{b_0}_{r_0} B^{s_0}_{p_0, \max\{p_0, q_0\}}(\R^d) \hookrightarrow T^{b_1}_{r_1}B^{s_1}_{p_1, q_1}(\R^d).
$$

	\emph{Sufficient conditions} (ii)--(v): Under these assumptions, it follows from Theorem \ref{ThmFFp0p1} that
	$$T^{b_0}_{r_0} F^{s_0}_{p_0, q_0}(\R^d) \hookrightarrow T^{b_0}_{r_0} F^{s_0}_{p_0, \infty}(\R^d),$$
	and thus \eqref{ThmFJState} will be shown once the following
	\begin{equation}\label{ProofFJGen0}
		T^{b_1 + (1/q_1 - 1/p_0)_+}_{r_1} F^{s_0}_{p_0, \infty}(\R^d) \hookrightarrow T^{b_1}_{r_1}B^{s_1}_{p_1, q_1}(\R^d)
	\end{equation}
	is established.
	
	For $j \in \N_0$ and $x \in \R^d$, consider
	$$
		f_j (x) := \sup_{\nu = 2^{j}-1, \cdots, 2^{j+1}-2} 2^{\nu s_0} \sum_{m \in \Z^d, \, G \in G^\nu} |\lambda^{\nu, G}_m (f)| \chi_{\nu, m}(x).
	$$
	Clearly
	$$
		2^{\nu s_0} \sum_{G \in G^\nu} |\lambda^{\nu, G}_m (f)| \leq \inf_{x \in Q_{\nu, m}} f_j(x)
	$$
	for $ \nu \in \{2^j-1, \cdots, 2^{j+1}-2\}$ and $m \in \Z^d$. Therefore
	\begin{align}
		\bigg(\sum_{\nu=2^j-1}^{2^{j+1}-2} 2^{\nu (s_1-d/p_1) q_1} \sum_{G \in G^\nu} \bigg(\sum_{m \in \Z^d} |\lambda^{\nu, G}_m|^{p_1} \bigg)^{q_1/p_1} \bigg)^{1/q_1} & \lesssim \nonumber \\
		& \hspace{-7cm} \bigg(\sum_{\nu=2^j-1}^{2^{j+1}-2} 2^{\nu (s_1-s_0-d/p_1) q_1}\bigg(\sum_{m \in \Z^d} \inf_{x \in Q_{\nu, m}} f_j(x)^{p_1}  \bigg)^{q_1/p_1} \bigg)^{1/q_1} \nonumber \\
		& \hspace{-7cm} \lesssim \bigg(\sum_{\nu=2^j-1}^{2^{j+1}-2} 2^{\nu (s_1-s_0-d/p_1) q_1} \bigg(\sum_{l=1}^\infty f_j^*(2^{-\nu d} l)^{p_1} \bigg)^{q_1/p_1} \bigg)^{1/q_1} =:I_j.  \label{ProofFJGen1}
	\end{align}
	As usual, $f^*$ denotes the non-increasing rearrangement of $f$, cf. \eqref{Rearrengement} below for precise definition.
	
	To estimate $I_j$, we can make use of elementary monotonicity properties together with a simple change of variables in order to get
	$$
	I_j \asymp  \bigg(\sum_{\nu=2^j-1}^{2^{j+1}-2} 2^{\nu (s_1-s_0) q_1} \bigg(\sum_{\mu=-\nu}^\infty f_j^*(2^{\mu d})^{p_1} 2^{\mu d} \bigg)^{q_1/p_1} \bigg)^{1/q_1}
	$$
	and, since $p_0 < p_1$,
	$$
		I_j \lesssim  \bigg(\sum_{\nu=2^j-1}^{2^{j+1}-2} 2^{\nu (s_1-s_0) q_1} \bigg(\sum_{\mu=-\nu}^\infty f_j^*(2^{\mu d})^{p_0} 2^{\mu d p_0/p_1} \bigg)^{q_1/p_0} \bigg)^{1/q_1}.
	$$
	Furthermore, we claim that
	\begin{equation*}
		I_j \lesssim  2^{j(1/q_1-1/p_0)_+} \bigg(\sum_{\nu=2^j-1}^{2^{j+1}-2} 2^{\nu (s_1-s_0) p_0} \sum_{\mu=-\nu}^\infty f_j^*(2^{\mu d})^{p_0} 2^{\mu d p_0/p_1} \bigg)^{1/p_0}.
	\end{equation*}
	Indeed, this follows easily from $\ell_{p_0} \hookrightarrow \ell_{q_1}$ if $q_1 \geq p_0$ and from H\"older's inequality if $q_1 < p_0$. Changing the order of summation and taking into account that $s_0 > s_1$, we derive
	\begin{align*}
		I_j &\lesssim 2^{j(1/q_1-1/p_0)_+} \bigg(\sum_{\mu=-2^{j+1}+2}^{-2^j+1}  f_j^*(2^{\mu d})^{p_0} 2^{\mu d p_0/p_1}  \sum_{\nu = -\mu}^{2^{j+1}-2} 2^{ \nu (s_1 - s_0) p_0}  \bigg)^{1/p_0}  \\
		& \hspace{.5cm}+ 2^{j(1/q_1-1/p_0)_+} \bigg(\sum_{\mu = -2^j+1}^\infty  f_j^*(2^{\mu d})^{p_0} 2^{\mu d p_0/p_1}  \sum_{\nu = 2^{j}-1}^{2^{j+1}-2} 2^{\nu (s_1 - s_0) p_0}  \bigg)^{1/p_0} \\
		& \lesssim 2^{j(1/q_1-1/p_0)_+} \bigg(\sum_{\mu=-2^{j+1}+2}^{\infty}  f_j^*(2^{\mu d})^{p_0}  2^{ \mu p_0 (d/p_1 -s_1 + s_0)}  \bigg)^{1/p_0}  \\
		& \lesssim 2^{j(1/q_1-1/p_0)_+}  \|f_j\|_{L_{p_0}(\R^d)}.
	\end{align*}
	Inserting this into \eqref{ProofFJGen1}, we achieve
	$$
		\bigg(\sum_{\nu=2^j-1}^{2^{j+1}-2} 2^{\nu (s_1-d/p_1) q_1} \sum_{G \in G^\nu} \bigg(\sum_{m \in \Z^d} |\lambda^{\nu, G}_m|^{p_1} \bigg)^{q_1/p_1} \bigg)^{1/q_1} \lesssim  2^{j(1/q_1-1/p_0)_+}  \|f_j\|_{L_{p_0}(\R^d)}.
	$$
	Accordingly, by Theorem \ref{ThmWaveletsNewBesov} (noting that $b_1 \neq 0$),
	\begin{align*}
	\|f\|_{T^{b_1}_{r_1} B^{s_1}_{p_1, q_1}(\R^d)} & \asymp \bigg(\sum_{j=0}^\infty 2^{j b_1 r_1} \bigg(\sum_{\nu=2^j-1}^{2^{j+1}-2} 2^{\nu (s-d/p_1) q_1} \sum_{G \in G^\nu} \bigg(\sum_{m \in \Z^d} |\lambda^{\nu, G}_m|^{p_1} \bigg)^{q_1/p_1} \bigg)^{r_1/q_1} \bigg)^{1/r_1} \\
	&\hspace{-1.5cm} \lesssim \bigg(\sum_{j=0}^\infty 2^{j (b_1 + (1/q_1-1/p_0)_+) r_1}  \|f_j\|^{r_1}_{L_{p_0}(\R^d)}  \bigg)^{1/r_1} = \|f\|_{T_{r_1}^{b_1 + (1/q_1-1/p_0)_+} F^{s_0}_{p_0, \infty}(\R^d)}.
	\end{align*}
	This proves the desired embedding \eqref{ProofFJGen0}.
	
	\emph{Necessary conditions $s_0 - \frac{d}{p_0} > s_1 - \frac{d}{p_1}$ and $b_0 \geq b_1$ under $s_0 - \frac{d}{p_0} = s_1 - \frac{d}{p_1}$:} Assume that \eqref{ThmFJState} holds, or equivalently (cf. Theorems \ref{ThmWaveletsNewBesov} and \ref{ThmWaveletsNewTriebelLizorkin}),
	\begin{equation}\label{ProofFJGen2}
		T^{b_0}_{r_0} f^{s_0}_{p_0, q_0} \hookrightarrow T^{b_1}_{r_1} b^{s_1}_{p_1, q_1}.
	\end{equation}
	For every $k \in \N$, define the sequence $\lambda_k = (\lambda^{j,G}_{m, k})$ given by
	\begin{equation*}
		\lambda^{j,G}_{m, k} = \left\{\begin{array}{cl}  1, & \quad j = 2^k, \quad m= (0, \ldots, 0),  \\
		& \qquad G = (M, \ldots, M),  \\
		0, & \text{otherwise}.
		       \end{array}
                        \right.
	\end{equation*}
	Then
	$$
		\|\lambda_k\|_{T^{b_1}_{r_1}b^{s_1}_{p_1, q_1}} \asymp 2^{2^k (s_1-d/p_1)} 2^{k b_1} \qquad \text{and} \qquad \|\lambda_k\|_{T^{b_0}_{r_0} f^{s_0}_{p_0, q_0}} \asymp 2^{2^k (s_0-d/p_0)} 2^{k b_0}.
	$$
	It follows from \eqref{ProofFJGen2} that one of the following conditions is satisfied
	$$
	 \left\{\begin{array}{cl} s_0 - d/p_0 > s_1 - d/p_1, \\
			& \\
		s_0 - d/p_0 = s_1 - d/p_1  & \text{and} \quad b_0 \geq b_1.
		       \end{array}
                        \right.
	$$
	
	\emph{Necessary condition $r_0 \leq r_1$ under $s_0 - \frac{d}{p_0} = s_1 - \frac{d}{p_1}$ and $b_0 = b_1$:} Let $\beta = (\beta_k)$ any scalar-valued sequence and define the related sequence
		\begin{equation*}
		\lambda^{j,G}_{m} = \left\{\begin{array}{cl} 2^{-2^k (s_0 - \frac{d}{p_0})} 2^{-k b_0} \beta_k , & \quad j = 2^k, \quad m= (0, \ldots, 0),  \\
		& \qquad G = (M, \ldots, M),  \\
		0, & \text{otherwise}.
		       \end{array}
                        \right.
	\end{equation*}
	Elementary computations show that
	$$
		\|(\lambda^{j, G}_m)\|_{T^{b_1}_{r_1} b^{s_1}_{p_1, q_1}} \asymp \|\beta\|_{\ell_{r_1}} \qquad \text{and} \qquad \|(\lambda^{j, G}_m)\|_{T^{b_0}_{r_0} f^{s_0}_{p_0, q_0}} \asymp \|\beta\|_{\ell_{r_0}}.
	$$
	Since \eqref{ProofFJGen2} holds, we conclude that necessarily $r_0 \leq r_1$.
	
	\emph{Necessary condition $b_1 + \frac{1}{q_1} \leq b_0 + \frac{1}{p_0}$ under $s_0 - \frac{d}{p_0} = s_1 - \frac{d}{p_1}$:} For every $k \in \N$, we let $\lambda_k = (\lambda^{j, G}_{m , k})$ where
		\begin{equation*}
		\lambda^{j,G}_{m, k} = \left\{\begin{array}{cl}  2^{-j (s_1 - \frac{d}{p_1})}, & \quad j \in \{2^k-1, \cdots, 2^{k+1}-2\}, \quad m= (0, \ldots, 0),  \\
		& \qquad G = (M, \ldots, M),  \\
		0, & \text{otherwise}.
		       \end{array}
                        \right.
	\end{equation*}
	Then
	$$
		\|\lambda_k\|_{T^{b_1}_{r_1} b^{s_1}_{p_1, q_1}} \asymp 2^{k (b_1+1/q_1)}.
	$$
	On the other hand
	\begin{equation}\label{ProofFJGen3}
	\|\lambda_k\|_{T^{b_0}_{r_0} f^{s_0}_{p_0, q_0}} = 2^{k b_0} \bigg\| \bigg(\sum_{j=2^k-1}^{2^{k+1}-2} 2^{j d q_0/p_0} \chi_{j, (0, \ldots, 0)} \bigg)^{1/q_0} \bigg\|_{L_{p_0}(\R^d)}.
	\end{equation}
	Furthermore, we claim that, for every $x \in \R^d,$
	\begin{equation}\label{ProofFJGen4}
		 \bigg(\sum_{j=2^k-1}^{2^{k+1}-2} 2^{j d q_0/p_0} \chi_{j, (0, \ldots, 0)}(x) \bigg)^{1/q_0} \asymp  \bigg(\sum_{j=2^k-1}^{2^{k+1}-2} 2^{j d} \chi_{j, (0, \ldots, 0)} (x) \bigg)^{1/p_0}.
	\end{equation}
	Indeed, this becomes obvious if $x \not \in Q_{2^k-1, (0, \ldots, 0)}$, so that the only non-trivial case to check occurs if $x \in Q_{2^k-1, (0, \ldots, 0)}$. In this case, we denote by $j_x \in \{2^k-1, \cdots, 2^{k+1}-3\}$ such that $x \in Q_{j_x, (0, \ldots, 0)}$ but $x \not \in Q_{j_x + 1, (0, \ldots, 0)}$ and we set $j_x = 2^{k+1}-2$ if $x \in Q_{2^{k+1}-2, (0, \ldots, 0)}$. Accordingly,
	\begin{align*}
	 \bigg(\sum_{j=2^k-1}^{2^{k+1}-2} 2^{j d q_0/p_0} \chi_{j, (0, \ldots, 0)}(x) \bigg)^{1/q_0} & =  \bigg(\sum_{j=2^k-1}^{j_x} 2^{j d q_0/p_0}  \bigg)^{1/q_0} \\
	 &\hspace{-5cm} \asymp 2^{j_x d/p_0} \asymp \bigg(\sum_{j=2^k-1}^{j_x} 2^{j d}  \bigg)^{1/p_0} =  \bigg(\sum_{j=2^k-1}^{2^{k+1}-2} 2^{j d} \chi_{j, (0, \ldots, 0)}(x) \bigg)^{1/p_0}.
	\end{align*}
	Plugging \eqref{ProofFJGen4} into \eqref{ProofFJGen3} we arrive at
	$$
			\|\lambda_k\|_{T^{b_0}_{r_0} f^{s_0}_{p_0, q_0}} \asymp 2^{k (b_0+ 1/p_0)}.
	$$
	Then the validity of \eqref{ProofFJGen2} implies, in particular,
	$$
		2^{k (b_1+1/q_1)} \lesssim  2^{k (b_0+ 1/p_0)}, \qquad k \in \N,
	$$
	and thus $b_1 + 1/q_1 \leq b_0 + 1/p_0$.
	
	\emph{Necessary condition $r_0 \leq r_1$ under $s_0 - \frac{d}{p_0} = s_1 - \frac{d}{p_1}$ and $b_1 + \frac{1}{q_1} = b_0 + \frac{1}{p_0}:$} We shall proceed by contradiction, i.e., we shall prove that  \eqref{ProofFJGen2} with $r_0 > r_1$ is not true via a counterexample. Let
		\begin{equation*}
		\lambda^{j,G}_{m} = \left\{\begin{array}{cl}  2^{-j (s_1 - \frac{d}{p_1})} (1+j)^{-(b_1+1/q_1)} (1+ \log (j+1))^{-\varepsilon}, & \quad j \in \N_0, \quad m= (0, \ldots, 0),  \\
		& \qquad G = (M, \ldots, M),  \\
		0, & \text{otherwise},
		       \end{array}
                        \right.
	\end{equation*}
	where $1/r_0 < \varepsilon < 1/r_1$. We have
	\begin{align}
		\|\lambda\|_{T^{b_1}_{r_1} b^{s_1}_{p_1, q_1}}  &= \bigg( \sum_{k=0}^\infty 2^{k b_1 r_1} \bigg(\sum_{j=2^k-1}^{2^{k+1}-2} (1+j)^{-(b_1+1/q_1) q_1}  (1+ \log (j+1))^{-\varepsilon q_1}\bigg)^{r_1/q_1}\bigg)^{1/r_1} \nonumber \\
		& \asymp \bigg( \sum_{k=0}^\infty (1+k)^{-\varepsilon r_1} \bigg)^{1/r_1} = \infty. \label{ProofFJGen5}
	\end{align}
	Furthermore, similarly as in \eqref{ProofFJGen4}, one can show that, for $x \in \R^d$,
	\begin{align*}
		\bigg(\sum_{j=2^k-1}^{2^{k+1}-2} 2^{j d q_0/p_0}  (1+j)^{-(b_1+1/q_1)q_0} (1+ \log (j+1))^{-\varepsilon q_0} \chi_{j, (0, \ldots, 0)} (x) \bigg)^{1/q_0} &\asymp \\
		& \hspace{-10cm} \bigg(\sum_{j=2^k-1}^{2^{k+1}-2} 2^{j d}  (1+j)^{-(b_1+1/q_1)p_0} (1+ \log (j+1))^{-\varepsilon p_0} \chi_{j, (0, \ldots, 0)} (x) \bigg)^{1/p_0}
	\end{align*}
	which implies
	\begin{align*}
		\|\lambda\|_{T^{b_0}_{r_0} f^{s_0}_{p_0, q_0}} &\asymp \bigg(\sum_{k=0}^\infty 2^{k b_0 r_0} \bigg(\sum_{j=2^k-1}^{2^{k+1}-2}  (1+j)^{-(b_1+1/q_1)p_0} (1+ \log (j+1))^{-\varepsilon p_0} \bigg)^{r_0/p_0}  \bigg)^{1/r_0} \\
		& \asymp \bigg(\sum_{k=0}^\infty  (1+ k)^{-\varepsilon r_0} \bigg)^{1/r_0} < \infty.
	\end{align*}
	However, this contradicts \eqref{ProofFJGen5} if \eqref{ProofFJGen2} holds.
\end{proof}

\begin{proof}[Proof of Theorem \ref{ThmFJ2}]

\emph{Sufficient condition} (i): It follows from Theorems \ref{TheoremEmbeddingsBBCharacterization} and Corollary \ref{CorTruncBTLFixedps} that
$$
	T^{b_0}_{r_0} B^{s_0}_{p_0, q_0}(\R^d) \hookrightarrow T^{b_1}_{r_1} B^{s_1}_{p_1, \min\{p_1, q_1\}}(\R^d) \hookrightarrow T^{b_1}_{r_1} F^{s_1}_{p_1, q_1}(\R^d).
$$

\emph{Sufficient conditions} (ii)--(v): To establish \eqref{ThmFJState2} under one of the assumptions given in (ii)--(v), it will be enough to show that
\begin{equation}\label{ThmFJ2Claim1}
	T^{b_1+ (1/p_1-1/q_0)_+}_{r_1} B^{s_0}_{p_0, q_0} (\R^d) \hookrightarrow T^{b_1}_{r_1} F^{s_1}_{p_1, q_1}(\R^d)
\end{equation}
provided that $q_1 \leq \min\{p_0, p_1, q_0\}$; see Theorems \ref{TheoremEmbeddingsBBCharacterization} and \ref{ThmFFp0p1}.

Assume $q_1 \leq \min\{p_0, p_1, q_0\}$ and denote by $\alpha, \beta$ and $\gamma$ the dual exponents of $p_0/q_1, p_1/q_1$ and $q_0/q_1$, respectively, i.e.,
$$
	\frac{q_1}{p_0} + \frac{1}{\alpha} = 1, \qquad \frac{q_1}{q_0} + \frac{1}{\gamma} = 1, \qquad \frac{q_1}{p_1} + \frac{1}{\beta} = 1.
$$

For $k \in \N_0$, we can apply basic properties of rearrangements (cf. \cite[Chapter 2, Theorems 3.4 and 4.6]{BennettSharpley}) to obtain
\begin{align}
	\bigg\| \bigg(\sum_{j=2^k-1}^{2^{k+1}-2} \sum_{G \in G^j} \sum_{m \in \Z^d} 2^{j s_1 q_1} |\lambda^{j, G}_m|^{q_1} \chi_{j, m} \bigg)^{1/q_1} \bigg\|_{L_{p_1}(\R^d)} & = \bigg\| \sum_{j=2^k-1}^{2^{k+1}-2} \sum_{G \in G^j} \sum_{m \in \Z^d} 2^{j s_1 q_1} |\lambda^{j, G}_m|^{q_1} \chi_{j, m}  \bigg\|_{L_{p_1/q_1}(\R^d)}^{1/q_1} \nonumber \\
	& \hspace{-6cm} \leq \bigg\| \sum_{j=2^k-1}^{2^{k+1}-2} 2^{j s_1 q_1} \sum_{G \in G^j} \bigg(\sum_{m \in \Z^d} |\lambda^{j, G}_m|^{q_1} \chi_{j, m} \bigg)^*  \bigg\|_{L_{p_1/q_1}(0, \infty)}^{1/q_1} \nonumber \\
	&\hspace{-6cm} = \bigg\|  \sum_{j=2^k-1}^{2^{k+1}-2} 2^{j s_1 q_1} \sum_{G \in G^j} \sum_{l=0}^\infty (\tilde{\lambda}^{j, G}_l)^{q_1} \tilde{\chi}_{j, l}   \bigg\|_{L_{p_1/q_1}(0, \infty)}^{1/q_1} \label{ThmFJ2Proof1}
\end{align}
where, for each $j \in \{2^k-1, \cdots, 2^{k+1}-2\}$ and $G \in G^j$, the sequence $(\tilde{\lambda}^{j, G}_l)_{l \in \N_0}$ denotes the non-increasing rearrangement of $(|\lambda^{j, G}_m|)_{m \in \Z^d}$ and $\tilde{\chi}_{j, l}$ is the characteristic function relative to the interval $[2^{-j d} l, 2^{-j d}(l+1))$.

By duality,
\begin{align}
	 \bigg\|  \sum_{j=2^k-1}^{2^{k+1}-2} 2^{j s_1 q_1} \sum_{G \in G^j} \sum_{l=0}^\infty (\tilde{\lambda}^{j, G}_l)^{q_1} \tilde{\chi}_{j, l}   \bigg\|_{L_{p_1/q_1}(0, \infty)} &= \sup \int_0^\infty g(x) \bigg( \sum_{j=2^k-1}^{2^{k+1}-2} 2^{j s_1 q_1} \sum_{G \in G^j} \sum_{l=0}^\infty (\tilde{\lambda}^{j, G}_l)^{q_1} \tilde{\chi}_{j, l} (x) \bigg) \, dx \nonumber \\
	 & \hspace{-5cm}= \sup \sum_{j=2^k-1}^{2^{k+1}-2} \sum_{G \in G^j} \sum_{l=0}^\infty (\tilde{\lambda}^{j, G}_l)^{q_1} 2^{j (s_1 q_1- d)} g_{j, l} \label{ThmFJ2Proof2}
\end{align}
where the supremum runs over all non-increasing and non-negative functions $g \in L_\beta(0, \infty)$ with $\|g\|_{L_\beta(0, \infty)} \leq 1$ and $g_{j, l} := 2^{j d} \int_0^\infty g(x) \tilde{\chi}_{j, l}(x) \, dx$. For any $g$ satisfying these assumptions, we can apply H\"older's inequality with exponents $\alpha$ and $\gamma$ such that
\begin{align}
	 \sum_{j=2^k-1}^{2^{k+1}-2} \sum_{G \in G^j} \sum_{l=0}^\infty (\tilde{\lambda}^{j, G}_l)^{q_1} 2^{j(s_1 q_1 -d)} g_{j, l} &\leq \sum_{j=2^k-1}^{2^{k+1}-2} 2^{j(s_1 q_1- d)} \sum_{G \in G^j} \bigg(\sum_{l=0}^\infty (\tilde{\lambda}^{j, G}_l)^{p_0} \bigg)^{q_1/p_0} \bigg(\sum_{l=0}^\infty g_{j, l}^\alpha \bigg)^{1/\alpha} \nonumber \\
	 &\hspace{-4cm}\leq  \bigg(\sum_{j = 2^k-1}^{2^{k+1}-2} 2^{j (s_0-d/p_0) q_0} \sum_{G \in G^j} \bigg(\sum_{m \in \Z^d}^\infty |\lambda^{j, G}_m|^{p_0} \bigg)^{q_0/p_0} \bigg)^{q_1/q_0}  \nonumber \\
	 & \hspace{-2cm} \times \bigg(\sum_{j=2^k-1}^{2^{k+1}-2} 2^{-j d \gamma/\beta} \bigg(\sum_{l=0}^\infty g_{j, l}^\alpha \bigg)^{\gamma/\alpha}   \bigg)^{1/\gamma}. \label{ThmFJ2Proof3}
\end{align}

Next we show
\begin{equation}\label{ThmFJ2Proof4}
	\bigg(\sum_{j=2^k-1}^{2^{k+1}-2} 2^{-j d \gamma/\beta} \bigg(\sum_{l=0}^\infty g_{j, l}^\alpha \bigg)^{\gamma/\alpha}   \bigg)^{1/\gamma} \lesssim 2^{k (1/q_0 - 1/p_1)_+ q_1}.
\end{equation}
To be more precise, we will prove that
\begin{equation}\label{ThmFJ2Proof5}
	\bigg(\sum_{j=2^k-1}^{2^{k+1}-2} 2^{-j d \gamma/\beta} \bigg(\sum_{l=0}^\infty g_{j, l}^\alpha \bigg)^{\gamma/\alpha}   \bigg)^{1/\gamma} \lesssim 2^{k (1/q_0 - 1/p_1)_+ q_1}  \bigg(\sum_{j=2^k-1}^{2^{k+1}-2} 2^{-j d} \bigg(\sum_{l=0}^\infty g_{j, l}^\alpha \bigg)^{\beta/\alpha}   \bigg)^{1/\beta}
\end{equation}
and
\begin{equation}\label{ThmFJ2Proof6}
	 \bigg(\sum_{j=2^k-1}^{2^{k+1}-2} 2^{-j d} \bigg(\sum_{l=0}^\infty g_{j, l}^\alpha \bigg)^{\beta/\alpha}   \bigg)^{1/\beta} \lesssim 1.
\end{equation}
Clearly \eqref{ThmFJ2Proof5} and \eqref{ThmFJ2Proof6} imply \eqref{ThmFJ2Proof4}.

The proof of  \eqref{ThmFJ2Proof6} follows from monotonicity properties of $g$. More precisely, we have
\begin{align*}
	 \sum_{l=0}^\infty g_{j, l}^\alpha  & = \sum_{l=0}^\infty \bigg( 2^{j d} \int_{2^{-j d} l}^{2^{- j d} (l+1)} g(x) \, dx \bigg)^\alpha \\
	 & = \sum_{\mu=0}^\infty \sum_{l=2^{\mu d}-1}^{2^{(\mu + 1)d}-2} \bigg( 2^{j d} \int_{2^{-j d} l}^{2^{- j d} (l+1)} g(x) \, dx \bigg)^\alpha \\
	 & \lesssim  \sum_{\mu=0}^\infty 2^{\mu d} \bigg( 2^{j d} \int_{2^{-j d} (2^{\mu d}-1)}^{2^{- j d}2^{\mu d}} g(x) \, dx \bigg)^\alpha \\
	 & \leq   \sum_{\mu=0}^\infty 2^{\mu d} g^{**}(2^{(\mu- j) d})^\alpha
\end{align*}
and since $p_0 < p_1$ (or equivalently, $\beta < \alpha$)
\begin{align*}
	\bigg(\sum_{j=2^k-1}^{2^{k+1}-2} 2^{-j d} \bigg(\sum_{l=0}^\infty g_{j, l}^\alpha \bigg)^{\beta/\alpha}   \bigg)^{1/\beta} & \lesssim \bigg(\sum_{j=2^k-1}^{2^{k+1}-2} 2^{-j d} \bigg( \sum_{\mu=0}^\infty 2^{\mu d} g^{**}(2^{(\mu- j) d})^\alpha \bigg)^{\beta/\alpha}  \bigg)^{1/\beta} \\
	&\hspace{-4cm} \leq   \bigg(\sum_{j=2^k-1}^{2^{k+1}-2} 2^{-j d(1-\beta/\alpha)}  \sum_{\mu=-j}^\infty 2^{\mu d \beta/\alpha} g^{**}(2^{\mu d})^\beta  \bigg)^{1/\beta} \\
	& \hspace{-4cm}= \bigg(\sum_{\mu=-\infty}^\infty  2^{\mu d \beta/\alpha} g^{**}(2^{\mu d})^\beta \sum_{j=-\infty}^{-\mu}  2^{-j d(1-\beta/\alpha)}    \bigg)^{1/\beta} \\
	&\hspace{-4cm} \asymp \bigg(\sum_{\mu=-\infty}^\infty 2^{\mu d} g^{**}(2^{\mu d})^\beta \bigg)^{1/\beta} \asymp \|g^{**}\|_{L_\beta(0, \infty)} \asymp \|g\|_{L_\beta(0, \infty)} \leq 1.
\end{align*}

To deal with \eqref{ThmFJ2Proof5}, we shall distinguish two possible cases. Firstly, if $p_1 \geq q_0$ (or equivalently, $\gamma \geq \beta$) then
$$
		\bigg(\sum_{j=2^k-1}^{2^{k+1}-2} 2^{-j d \gamma/\beta} \bigg(\sum_{l=0}^\infty g_{j, l}^\alpha \bigg)^{\gamma/\alpha}   \bigg)^{1/\gamma} \leq \bigg(\sum_{j=2^k-1}^{2^{k+1}-2} 2^{-j d} \bigg(\sum_{l=0}^\infty g_{j, l}^\alpha \bigg)^{\beta/\alpha}   \bigg)^{1/\beta}.
$$
Secondly, under the assumption $p_1 < q_0$ (or equivalently, $\gamma < \beta$) we can apply H\"older's inequality to obtain
$$
	\bigg(\sum_{j=2^k-1}^{2^{k+1}-2} 2^{-j d \gamma/\beta} \bigg(\sum_{l=0}^\infty g_{j, l}^\alpha \bigg)^{\gamma/\alpha}   \bigg)^{1/\gamma} \lesssim 2^{k(1/p_1-1/q_0) q_1} \bigg(\sum_{j=2^k-1}^{2^{k+1}-2} 2^{-j d } \bigg(\sum_{l=0}^\infty g_{j, l}^\alpha \bigg)^{\beta/\alpha}   \bigg)^{1/\beta}.
$$

As a combination of \eqref{ThmFJ2Proof1}--\eqref{ThmFJ2Proof4}, we arrive at
\begin{align*}
	\bigg\| \bigg(\sum_{j=2^k-1}^{2^{k+1}-2} \sum_{G \in G^j} \sum_{m \in \Z^d} 2^{j s_1 q_1} |\lambda^{j, G}_m|^{q_1} \chi_{j, m} \bigg)^{1/q_1} \bigg\|_{L_{p_1}(\R^d)} & \lesssim \\
	& \hspace{-7cm}2^{k (1/p_1 - 1/q_0)_+}  \bigg(\sum_{j = 2^k-1}^{2^{k+1}-2} 2^{j (s_0-d/p_0) q_0} \sum_{G \in G^j} \bigg(\sum_{m \in \Z^d}^\infty |\lambda^{j, G}_m|^{p_0} \bigg)^{q_0/p_0} \bigg)^{1/q_0}
\end{align*}
and thus, by Theorems \ref{ThmWaveletsNewTriebelLizorkin} and \ref{ThmWaveletsNewBesov},
\begin{align*}
	\|f\|_{T^{b_1}_{r_1} F^{s_1}_{p_1, q_1}(\R^d)} &\lesssim \\
	& \hspace{-2cm} \bigg(\sum_{k=0}^\infty 2^{k (b_1 + (1/p_1 - 1/q_0)_+) r_1} \bigg(\sum_{j = 2^k-1}^{2^{k+1}-2} 2^{j (s_0-d/p_0) q_0} \sum_{G \in G^j} \bigg(\sum_{m \in \Z^d}^\infty |\lambda^{j, G}_m|^{p_0} \bigg)^{q_0/p_0} \bigg)^{r_1/q_0}  \bigg)^{1/r_1} \\
	& \hspace{-2cm} \asymp \|f\|_{T^{b_1 + (1/p_1 - 1/q_0)_+}_{r_1} B^{s_0}_{p_0, q_0}(\R^d)}.
\end{align*}

The counterexamples provided in the proof of Theorem  \ref{ThmFJ} also show the optimality of the conditions stated in (i)--(v). Further details are left to the reader.
\end{proof}

\subsection{Franke--Jawerth embeddings for $\mathbf{B}^{0, b}_{p, q}(\R^d)$ and $\text{Lip}^{s, b}_{p, q}(\R^d)$}
The characterization of Franke--Jawerth embeddings for Lipschitz spaces obtained in \cite[Theorem 4.4 and Remark 4.5]{DominguezHaroskeTikhonov} reads as follows: Let $0 < p_0 < p < p_1 \leq \infty, 1 < p < \infty, 0 < q, u \leq \infty, s > 0, b < - 1/q$ and $\xi \in \R$. Then
\begin{equation}
 \left\{\begin{array}{cl}   \text{Lip}^{s, b}_{p, q}(\R^d) \hookrightarrow B^{s - \frac{d}{p} + \frac{d}{p_1}, b + \xi}_{p_1, q}(\R^d) & \iff \xi \leq \frac{1}{\max\{p, q\}},  \\
		\text{Lip}^{s, b}_{p, q}(\R^d) \hookrightarrow B^{s - \frac{d}{p} + \frac{d}{p_1}, b + \frac{1}{q}}_{p_1, u}(\R^d) & \iff u \geq \max\{p, q\}, \\
		 B^{s - \frac{d}{p} + \frac{d}{p_0}, b + \xi}_{p_0, q}(\R^d) \hookrightarrow  \text{Lip}^{s, b}_{p, q}(\R^d) &  \iff \xi \geq \frac{1}{\min\{p, q\}},\\
		 B^{s - \frac{d}{p} + \frac{d}{p_0}, b + \frac{1}{q}}_{p_0, u}(\R^d) \hookrightarrow \text{Lip}^{s, b}_{p, q}(\R^d) &  \iff u \leq \min\{p, q\}. \label{hdshash}
		       \end{array}
                        \right.
\end{equation}
Since $T^{b+1/q}_q F^s_{p, 2}(\R^d) = \text{Lip}^{s, b}_{p, q}(\R^d)$ and $T^b_q B^s_{p, q}(\R^d) = B^{s, b}_{p, q}(\R^d)$, these embeddings are now an immediate consequence of Theorems \ref{ThmFJ} and \ref{ThmFJ2}. Despite the fact that \eqref{hdshash} is optimal within the scale of classical Besov spaces, our next result shows that it admits improvements in terms of the new scale of truncated Besov spaces.

\begin{cor}\label{Corollary13.4}
	Let $0 < p_0 < p < p_1 \leq \infty, 1 < p < \infty, 0 < q \leq \infty,$ and $b < -1/q$. Then
	 \begin{equation*}
 	 T^{b+1/q}_q B^{s-\frac{d}{p}+ \frac{d}{p_0}}_{p_0, p}(\R^d) \hookrightarrow \emph{Lip}^{s, b}_{p, q}(\R^d) \hookrightarrow T^{b+1/q}_q B^{s-\frac{d}{p}+ \frac{d}{p_1}}_{p_1, p}(\R^d).
 \end{equation*}
\end{cor}

\begin{rem}
According to Corollary  \ref{TheoremEmbeddings1},
 \begin{equation}\label{FJL2}
 	T^{b+1/q}_q B^{s-\frac{d}{p}+ \frac{d}{p_1}}_{p_1, p}(\R^d) \hookrightarrow B^{s - \frac{d}{p} + \frac{d}{p_1}, b + \frac{1}{\max\{p, q\}}}_{p_1, q}(\R^d) \cap   B^{s - \frac{d}{p} + \frac{d}{p_1}, b + \frac{1}{q}}_{p_1, \max\{p, q\}}(\R^d),
 \end{equation}
 \begin{equation}\label{FJL3}
  B^{s - \frac{d}{p} + \frac{d}{p_0}, b + \frac{1}{\min\{p, q\}}}_{p_0, q}(\R^d) +   B^{s - \frac{d}{p} + \frac{d}{p_0}, b + \frac{1}{q}}_{p_0, \min\{p, q\}}(\R^d)	\hookrightarrow T^{b+1/q}_q B^{s-\frac{d}{p}+ \frac{d}{p_0}}_{p_0, p}(\R^d).
 \end{equation}
 Therefore Corollary \ref{Corollary13.4} sharpens \eqref{hdshash}.
\end{rem}

\begin{rem}
	All the parameters in Corollary \ref{Corollary13.4} are sharp.
\end{rem}


Concerning Franke--Jawerth embeddings for $\textbf{B}^{0, b}_{p, q}(\R^d)$, it was obtained in \cite[Theorem 3.6, Propositions 9.7 and 9.8]{DominguezTikhonov} that if $0 < p_0 < p < p_1 \leq \infty, 1 < p < \infty, 0 < q \leq \infty, b > - 1/q$ and $\xi \in \R$, then
\begin{equation}\label{FJL4}
\left\{ \begin{array}{cl}
	\textbf{B}^{0, b}_{p, q}(\R^d) \hookrightarrow B^{- \frac{d}{p} + \frac{d}{p_1}, b + \xi}_{p_1, q}(\R^d)&  \iff \xi \leq \frac{1}{\max\{p, q\}}, \\
	 B^{- \frac{d}{p} + \frac{d}{p_0}, b + \xi}_{p_0, q}(\R^d) \hookrightarrow  \textbf{B}^{0, b}_{p, q}(\R^d) & \iff \xi \geq \frac{1}{\min\{p, q\}}.
	 \end{array}
	 \right.
\end{equation}
The counterpart of Corollary \ref{Corollary13.4} for $\mathbf{B}^{0, b}_{p, q}(\R^d)$ (recall that $\mathbf{B}^{0, b}_{p, q}(\R^d) = T^{b+1/q}_q F^0_{p, 2}(\R^d)$) reads as follows.

\begin{cor}\label{Corollary13.7}
Let $0 < p_0 < p < p_1 \leq \infty, 1 < p < \infty, 0 < q \leq \infty,$ and $b > - 1/q$. Then
 $$
 	 T^{b+1/q}_q B^{-\frac{d}{p}+ \frac{d}{p_0}}_{p_0, p}(\R^d) \hookrightarrow \mathbf{B}^{0, b}_{p, q}(\R^d)  \hookrightarrow T^{b+1/q}_q B^{-\frac{d}{p}+ \frac{d}{p_1}}_{p_1, p}(\R^d).
 $$
\end{cor}

\begin{rem}
	The embeddings given in  Corollary \ref{Corollary13.7} improve \eqref{FJL4} via \eqref{FJL2} and \eqref{FJL3} with $s=0$.
\end{rem}

\begin{rem}
	All the parameters in Corollary \ref{Corollary13.7} are sharp.
\end{rem}

\newpage
\section{Embeddings in the space of locally integrable functions}\label{SectionL1}

As usual, $L_1^{\text{loc}}(\R^d)$ stands for the space of functions in $\R^d$ which are integrable on any bounded domain in $\R^d$. This space is interpreted as the set of all regular distributions in $\R^d$.\index{\bigskip\textbf{Spaces}!$L_1^{\text{loc}}(\R^d)$}\label{L1LOC}

A central question working with spaces of distributions is to characterize when they contain only regular distributions. In the setting of classical spaces $A^{s, b}_{p, q} (\R^d), A \in \{B, F\},$  this question has been completely settled in \cite{Sickel} and \cite{CaetanoLeopold13} (see also \cite{CaetanoFarkas}). For convenience of the reader, these results will be recalled in Section  \ref{ReminderBTLL1}.  Then in Sections \ref{ReminderBNew1} and \ref{ReminderFNew1} we turn our attention to the new scales of spaces $T^b_r B^s_{p, q}(\R^d)$ and $T^b_r F^s_{p, q}(\R^d)$, respectively.

\subsection{Embeddings for classical Besov and Triebel--Lizorkin spaces}\label{ReminderBTLL1}

%

\begin{thm}[{\cite[Theorem 4.3]{CaetanoLeopold13}}]\label{TheoremRegular}
 	Let $s, b \in \R$ and $0 < p, q \leq \infty$. Then
		$$
		B^{s, b}_{p, q}(\R^d) \hookrightarrow L_1^{\emph{loc}}(\R^d)
	$$
	if and only if one of the following conditions holds
	\begin{enumerate}[\upshape(i)]
		\item $0 < p \leq \infty, \qquad s > d \big(\frac{1}{p}-1 \big)_+, \qquad 0 < q \leq \infty, \qquad b \in \R$,
		\item  $0 < p \leq 1, \qquad s = d \big(\frac{1}{p}-1 \big), \qquad 0 < q \leq 1, \qquad b \geq 0$,
		\item $0 < p \leq 1, \qquad s = d \big(\frac{1}{p}-1 \big), \qquad 1 < q \leq \infty, \qquad b > 1 - \frac{1}{q}$,
		\item $1 < p \leq \infty, \qquad s = 0, \qquad 0 < q \leq \min \{p, 2\}, \qquad b \geq 0$,
		\item $1 < p \leq 2, \qquad s = 0, \qquad p < q \leq \infty, \qquad b > \frac{1}{p} - \frac{1}{q}$,
		\item $2 < p \leq \infty, \qquad s = 0, \qquad 2 < q \leq \infty, \qquad b > \frac{1}{2} - \frac{1}{q}$.
	\end{enumerate}
\end{thm}

\begin{thm}[{\cite[Theorem 4.4]{CaetanoLeopold13}}]\label{TheoremRegularFspaces}
 	Let $s, b \in \R, 0 < p < \infty$ and $0 < q \leq \infty$. Then
		$$
		F^{s, b}_{p, q}(\R^d) \hookrightarrow L_1^{\emph{loc}}(\R^d)
	$$
	if and only if one of the following conditions holds
	\begin{enumerate}[\upshape(i)]
		\item $0 < p < \infty, \qquad s > d \big(\frac{1}{p}-1 \big)_+, \qquad 0 < q \leq \infty, \qquad b \in \R$,
		\item  $0 < p < 1, \qquad s = d \big(\frac{1}{p}-1 \big), \qquad 0 < q \leq \infty, \qquad b \geq 0$,
		\item $1 \leq p < \infty, \qquad s = 0, \qquad 0 < q \leq 2, \qquad b \geq 0$,
		\item $1 \leq p < \infty, \qquad s = 0, \qquad 2 < q \leq \infty, \qquad b > \frac{1}{2} - \frac{1}{q}$.
	\end{enumerate}
\end{thm}

Dealing with embeddings for $A^s_{p, q}(\R^d), \, A \in \{B, F\}$ (i.e., $b=0$ in Theorems \ref{TheoremRegular} and \ref{TheoremRegularFspaces}), we refer to \cite[Theorem 3.3.2]{Sickel}.

\subsection{Embeddings for  $T^b_r B^s_{p, q}(\R^d)$}\label{ReminderBNew1}

\begin{thm}\label{TheoremRegular2}
	Let $s \in \R, 0 < p, q, r \leq \infty$ and $b \in \R \backslash \{0\}$. Then
	$$
		T^b_r B^s_{p, q}(\R^d) \hookrightarrow L^{\text{\emph{loc}}}_1(\R^d)
	$$
	if and only if one of the following conditions holds
	\begin{enumerate}[\upshape(i)]
	 \item $0 < p \leq \infty, \qquad s > d \big(\frac{1}{p}-1 \big)_+, \qquad 0 < q, r \leq \infty, \qquad b \in \R \backslash \{0\}$,
	  \item $0 < p \leq 1, \qquad s =  d \big(\frac{1}{p}-1 \big), \qquad 0 < q \leq 1, \qquad 0 < r \leq \infty, \qquad b > 0$,
	   \item $0 < p \leq 1, \qquad s =  d \big(\frac{1}{p}-1 \big), \qquad 1 < q \leq \infty, \qquad 0 < r \leq 1, \qquad  b \geq  1 - \frac{1}{q}$,
	      \item $0 < p \leq 1, \qquad s =  d \big(\frac{1}{p}-1 \big), \qquad 1 < q \leq \infty, \qquad 1 < r \leq \infty, \qquad  b >  1 - \frac{1}{q}$,
	 \item $1 < p \leq 2, \qquad s=0,  \qquad 0 < q \leq p, \qquad 0 < r \leq \infty, \qquad b > 0$,
	 \item $1 < p \leq 2, \qquad s=0,  \qquad p < q \leq \infty, \qquad 0 < r \leq p, \qquad b \geq  \frac{1}{p}-\frac{1}{q}$,
	 \item  $1 < p \leq 2,  \qquad s=0, \qquad p < q \leq \infty, \qquad p < r \leq \infty, \qquad b >  \frac{1}{p}-\frac{1}{q}$,
	 \item  $2 < p \leq \infty,  \qquad s=0, \qquad  0 < q \leq 2, \qquad 0 < r \leq \infty, \qquad b > 0$,
	    \item  $2 < p \leq \infty,  \qquad s=0, \qquad  2 < q \leq \infty, \qquad 0 < r \leq 2, \qquad b \geq \frac{1}{2}  -\frac{1}{q}$,
	    	    \item  $2 < p \leq \infty,  \qquad s=0, \qquad  2 < q \leq \infty, \qquad 2 < r \leq \infty, \qquad b >  \frac{1}{2} -\frac{1}{q}$.
	\end{enumerate}
\end{thm}

\begin{rem}
	(a): Setting $q=r$ in Theorem \ref{TheoremRegular2} one recovers Theorem \ref{TheoremRegular} for $B^{s, b}_{p, q}(\R^d), \, b \neq 0$ (cf. Proposition \ref{PropositionCoincidences}).
	
	(b): The conditions (v)--(x) (i.e., $1 < p < \infty, s=0$ and $b \neq 0$) given in Theorem \ref{TheoremRegular2}  can be summarized as
	$$
		b \geq \frac{1}{\min\{2, p, q\}} - \frac{1}{q} \qquad \text{if} \qquad r \leq \min\{2, p, q\}
	$$
	and
	$$
		b >  \frac{1}{\min\{2, p, q\}} - \frac{1}{q} \qquad \text{if} \qquad r > \min\{2, p, q\}.
	$$
	The conditions (ii)--(iv) (i.e., $0 < p \leq 1, s = d \big(\frac{1}{p}-1 \big)$ and $b \neq 0$) can be rewritten as
	$$
		b \geq \frac{1}{\min\{1, q\}} - \frac{1}{q} \qquad \text{if} \qquad r \leq 1
	$$
	and
	$$
		b >  \frac{1}{\min\{1, q\}} - \frac{1}{q} \qquad \text{if} \qquad r > 1.
	$$
\end{rem}

\begin{proof}[Proof of Theorem \ref{TheoremRegular2}]
\textsc{Sufficiency part:} (i): Let $s_0 \in \big(d \big(\frac{1}{p}-1 \big)_+, s \big)$. According to Corollary \ref{Corollary11.4} one has $T^b_r B^s_{p, q}(\R^d)  \hookrightarrow B^{s_0, b}_{p, q}(\R^d)$ (which still holds true in the limiting case $b=0$) and thus $T^b_r B^s_{p, q}(\R^d) \hookrightarrow L^{\text{loc}}_1(\R^d)$ (cf. Theorem \ref{TheoremRegular}).

(ii): Assume $0 < p \leq 1, s = d \big(\frac{1}{p}-1 \big)$ and $0 < q \leq 1$. It follows from the triangle inequality and the Nikolskii inequality for entire functions of exponential type that
\begin{align}
	\|f\|_{L_1(\R^d)} & \leq \sum_{\nu=0}^\infty  \|(\varphi_\nu \widehat{f})^\vee\|_{L_1(\R^d)} = \sum_{j=0}^\infty \sum_{\nu=2^j-1}^{2^{j+1}-2} \|(\varphi_\nu \widehat{f})^\vee\|_{L_1(\R^d)} \nonumber\\
	& \lesssim  \sum_{j=0}^\infty \sum_{\nu=2^j-1}^{2^{j+1}-2} 2^{\nu d (1/p-1)} \|(\varphi_\nu \widehat{f})^\vee\|_{L_p(\R^d)}. \label{ProofTheoremRegular2}
\end{align}
Since $q \leq 1$ one can estimate the last term by
$$
	 \sum_{j=0}^\infty \bigg(\sum_{\nu=2^j-1}^{2^{j+1}-2} 2^{\nu d (1/p-1) q} \|(\varphi_\nu \widehat{f})^\vee\|_{L_p(\R^d)}^q \bigg)^{1/q}.
$$
We distinguish now two possible cases. If $r \leq 1$ and $b \geq 0$ then
\begin{align*}
	\|f\|_{L_1(\R^d)} & \lesssim  \sum_{j=0}^\infty \bigg(\sum_{\nu=2^j-1}^{2^{j+1}-2} 2^{\nu d (1/p-1) q} \|(\varphi_\nu \widehat{f})^\vee\|_{L_p(\R^d)}^q \bigg)^{1/q} \\
	&\leq   \left(\sum_{j=0}^\infty 2^{j b r} \bigg(\sum_{\nu=2^j-1}^{2^{j+1}-2} 2^{\nu d (1/p-1) q} \|(\varphi_\nu \widehat{f})^\vee\|_{L_p(\R^d)}^q \bigg)^{r/q}  \right)^{1/r} \\
	& = \|f\|_{T^b_r B^s_{p, q}(\R^d)}.
\end{align*}
On the other hand, if $1 < r < \infty$ and $b > 0$, we can apply H\"older's inequality to obtain
\begin{align*}
	\|f\|_{L_1(\R^d)} & \lesssim  \sum_{j=0}^\infty \bigg(\sum_{\nu=2^j-1}^{2^{j+1}-2} 2^{\nu d (1/p-1) q} \|(\varphi_\nu \widehat{f})^\vee\|_{L_p(\R^d)}^q \bigg)^{1/q} \\
	&  \hspace{-1cm} \lesssim  \left(\sum_{j=0}^\infty 2^{j b r} \bigg(\sum_{\nu=2^j-1}^{2^{j+1}-2} 2^{\nu d (1/p-1) q} \|(\varphi_\nu \widehat{f})^\vee\|_{L_p(\R^d)}^q \bigg)^{r/q}  \right)^{1/r}  = \|f\|_{T^b_r B^s_{p, q}(\R^d)}.
\end{align*}
The case $r=\infty$ and $b > 0$ is even easier and we omit the proof.

(iii), (iv): Let $0 < p \leq 1, s= d \big(\frac{1}{p}-1 \big)$ and $q > 1$. Applying H\"older's inequality in \eqref{ProofTheoremRegular2},
\begin{equation} \label{ProofTheoremRegular2.1}
	\|f\|_{L_1(\R^d)} \lesssim \sum_{j=0}^\infty 2^{j(1-1/q)} \bigg(\sum_{\nu=2^j-1}^{2^{j+1}-2} 2^{\nu d (1/p-1) q} \|(\varphi_\nu \widehat{f})^\vee\|^q_{L_p(\R^d)} \bigg)^{1/q}
\end{equation}
(with the obvious modification if $q=\infty$). Assume first $r \leq 1$ and $b \geq  1 - 1/q$ (i.e., (iii) holds). Then, by \eqref{ProofTheoremRegular2.1},
$$
	\|f\|_{L_1(\R^d)} \lesssim \left(\sum_{j=0}^\infty 2^{j b r} \bigg(\sum_{\nu=2^j-1}^{2^{j+1}-2} 2^{\nu d (1/p-1) q} \|(\varphi_\nu \widehat{f})^\vee\|^q_{L_p(\R^d)} \bigg)^{r/q} \right)^{1/r}  = \|f\|_{T^b_r B^s_{p, q}(\R^d)}.
$$
Suppose now that $1 < r < \infty$ and $ b > 1 - 1/q$ (i.e., (iv) holds). An application of H\"older's inequality in \eqref{ProofTheoremRegular2.1} yields
$$
	\|f\|_{L_1(\R^d)} \lesssim \left( \sum_{j=0}^\infty 2^{j b r} \bigg(\sum_{\nu=2^j-1}^{2^{j+1}-2} 2^{\nu d (1/p-1) q} \|(\varphi_\nu \widehat{f})^\vee\|^q_{L_p(\R^d)} \bigg)^{r/q} \right)^{1/r} = \|f\|_{T^b_r B^s_{p, q}(\R^d)}.
$$
The case $r=\infty$ and $b > 1-1/q$ can be done similarly.

(v): Let $1 < p \leq 2, s=0,$ and $q \leq p$. In virtue of the Littlewood--Paley theorem
\begin{equation}\label{LPestim}
	\|f\|_{L_p(\R^d)} \asymp \bigg\|\bigg(\sum_{\nu=0}^\infty |(\varphi_\nu \widehat{f})^\vee(\cdot)|^2 \bigg)^{1/2}  \bigg\|_{L_p(\R^d)}, \qquad 1 < p < \infty,
\end{equation}
we obtain, by Fubini's theorem,
\begin{equation}\label{LPestim2}
	\|f\|_{L_p(\R^d)} \lesssim \bigg(\sum_{\nu=0}^\infty \|(\varphi_\nu \widehat{f})^\vee\|_{L_p(\R^d)}^p \bigg)^{1/p}, \qquad 1 < p \leq 2.
\end{equation}
Since $q \leq p$, by \eqref{LPestim2},
\begin{equation}\label{LPestim3}
	\|f\|_{L_p(\R^d)} \lesssim \bigg(\sum_{j=0}^\infty  \bigg(\sum_{\nu = 2^j-1}^{2^{j+1}-2} \|(\varphi_\nu \widehat{f})^\vee\|_{L_p(\R^d)}^q \bigg)^{p/q}  \bigg)^{1/p}.
\end{equation}
If $r \leq p$ and $b \geq 0$  then
$$
	\|f\|_{L_p(\R^d)} \lesssim \bigg(\sum_{j=0}^\infty 2^{j b r}  \bigg(\sum_{\nu = 2^j-1}^{2^{j+1}-2} \|(\varphi_\nu \widehat{f})^\vee\|_{L_p(\R^d)}^q \bigg)^{r/q}  \bigg)^{1/r} = \|f\|_{T^b_r B^{s}_{p, q}(\R^d)}.
$$
On the other hand, if $p < r < \infty$ and $b > 0$, it follows from \eqref{LPestim3} and H\"older's inequality that
$$
	\|f\|_{L_p(\R^d)} \lesssim \bigg(\sum_{j=0}^\infty 2^{j b r}  \bigg(\sum_{\nu = 2^j-1}^{2^{j+1}-2} \|(\varphi_\nu \widehat{f})^\vee\|_{L_p(\R^d)}^q \bigg)^{r/q}  \bigg)^{1/r} = \|f\|_{T^b_r B^s_{p, q}(\R^d)}.
$$
Standard modifications work for the case $r=\infty$ and $b > 0$.

(vi), (vii): Assume $1 < p \leq 2, s=0$ and $p < q \leq \infty$. According to \eqref{LPestim2} and H\"older's inequality
\begin{align}
		\|f\|_{L_p(\R^d)}& \lesssim \bigg(\sum_{j=0}^\infty \sum_{\nu=2^j-1}^{2^{j+1}-2} \|(\varphi_\nu \widehat{f})^\vee\|_{L_p(\R^d)}^p \bigg)^{1/p} \nonumber \\
		& \lesssim \bigg(\sum_{j=0}^\infty 2^{j(1/p-1/q)p}  \bigg( \sum_{\nu=2^j-1}^{2^{j+1}-2} \|(\varphi_\nu \widehat{f})^\vee\|_{L_p(\R^d)}^q  \bigg)^{p/q}\bigg)^{1/p} \label{LPestim4}
\end{align}
(where the usual interpretation is made if $q=\infty$).

If $r \leq p$ and $b \geq 1/p-1/q$ then, by \eqref{LPestim4},
$$
\|f\|_{L_p(\R^d)} \lesssim  \bigg(\sum_{j=0}^\infty 2^{j b r}  \bigg( \sum_{\nu=2^j-1}^{2^{j+1}-2} \|(\varphi_\nu \widehat{f})^\vee\|_{L_p(\R^d)}^q  \bigg)^{r/q}\bigg)^{1/r}  = \|f\|_{T^b_r B^s_{p, q}(\R^d)}.
$$
On the other hand, if $p < r \leq \infty$ and $b > 1/p-1/q$, the desired estimate follows from \eqref{LPestim4} and H\"older's inequality.

(viii): Under $2 < p < \infty$ and $s=0$: by the Littlewood--Paley theorem (cf. \eqref{LPestim}) and Minkowski's inequality
\begin{equation}\label{LPestim6}
	\|f\|_{L_p(\R^d)} \lesssim \bigg(\sum_{\nu=0}^\infty \|(\varphi_\nu \widehat{f})^\vee\|_{L_p(\R^d)}^2 \bigg)^{1/2}.
\end{equation}
Further, if $q \leq 2$ then
\begin{equation}\label{LPestim5}
\|f\|_{L_p(\R^d)} \lesssim \left(\sum_{j=0}^\infty  \bigg(\sum_{\nu=2^j-1}^{2^{j+1}-2} \|(\varphi_\nu \widehat{f})^\vee\|_{L_p(\R^d)}^q \bigg)^{2/q} \right)^{1/2}.
\end{equation}
Assume that one of the following conditions is satisfied
$$
	 \left\{\begin{array}{cl} r \leq 2 \quad  \text{and} \quad b \geq 0, & \\
		2 < r \leq \infty \quad \text{and} \quad b > 0.
		       \end{array}
                        \right.
$$
Therefore the right-hand side in \eqref{LPestim5} can be dominated by $\|f\|_{T^b_r B^{s}_{p, q} (\R^d)}$.

(ix), (x): Let $2 < p < \infty, s=0,$ and $2 < q \leq \infty$. By \eqref{LPestim6} and H\"older's inequality
$$
	\|f\|_{L_p(\R^d)} \lesssim \left(\sum_{j=0}^\infty 2^{j(1/2-1/q) 2} \bigg(\sum_{\nu=2^{j}-1}^{2^{j+1}-2}  \|(\varphi_\nu \widehat{f})^\vee\|_{L_p(\R^d)}^q \bigg)^{2/q} \right)^{1/2}.
$$
Furthermore, the last term can be dominated by $\|f\|_{T^b_r B^{s}_{p, q} (\R^d)}$: if $r \leq 2$ (i.e., (ix) holds) we use the fact that $\ell_r \hookrightarrow \ell_2$ and if $r>2$ (i.e., (x) holds) then we apply H\"older's inequality.

The limiting case $p=\infty$ in (viii)--(x) is an immediate consequence of the embeddings
\begin{equation}\label{LPestim7}
	T^b_r B^{0}_{\infty, q}(\R^d) \hookrightarrow B^0_{\infty, 2}(\R^d)
\end{equation}
and
$$
	B^0_{\infty, 2}(\R^d) \hookrightarrow L_1^{\text{loc}}(\R^d)
$$
(cf. Theorem \ref{TheoremRegular}). The proof of \eqref{LPestim7} can be derived by using similar techniques as above and it is left to the interested reader.

\textsc{Necessary part:} We first deal with the embedding
	\begin{equation}\label{EmbeddingContradiction}
		T^b_r B^{0}_{p, q}(\R^d) \hookrightarrow L^{\text{loc}}_1(\R^d), \qquad 1 < p < \infty.
	\end{equation}
\textsc{Necessary condition $b \geq 1/p-1/q$ in \eqref{EmbeddingContradiction} with $p < q$:} Assume that \eqref{EmbeddingContradiction} holds with
\begin{equation}\label{Counterexample1}
 b < 1/p-1/q
\end{equation}
(cf. (vi)). Consider
	\begin{equation}\label{AuxFunctionCounterexample}
		 f = \sum_{j \in \N_0,G \in G^j,m \in \Z^d} \lambda^{j,G}_m 2^{-j d/2}
    \Psi^j_{G,m}
	\end{equation}
	where the coefficients $\lambda^{j,G}_m$ are to be chosen. According to \eqref{Counterexample1} (recall also $p < q$) there exists $\beta$ such that
	\begin{equation}\label{Counterexample1.2}
		1 < \beta < p' \min \{1-b-1/q, 1-1/q\}.
	\end{equation}
	Let
	$$
		\kappa_j = \sum_{l=0}^j (1+l)^{-\beta}.
	$$
	Observe that
	$$\kappa = \lim_{j \to \infty} \kappa_j < \infty$$
	 since $\beta > 1$ (cf. \eqref{Counterexample1.2}) and $\kappa_j < \kappa_{j+1}$ for every $j \in \N$. Define, for each $j \in \N$,
	\begin{equation}\label{Rj}
	R_j := \{x = (x_1, x_2, \ldots, x_d) : \kappa_{j-1} \leq x_1 < \kappa_j, \quad 0 \leq x_l \leq 1, \quad l= 2, \ldots, d \}.
	\end{equation}
	We now subdivide $R_j$ into dyadic cubes $Q_{j, m}, \, m \in M_{j} \subset \Z^d,$ of side length $2^{-j}$. Note that
	\begin{equation}\label{CardinalMj}
		|M_j| = 2^{j(d-1)} [2^j (\kappa_j-\kappa_{j-1})] \asymp 2^{j d} (1+j)^{-\beta}.
	\end{equation}
	Define
		\begin{equation*}
		\lambda^{j,G}_m = \left\{\begin{array}{cl} (1+j)^{\beta/p -\varepsilon}, & \quad j \in \N, \quad m \in M_j, \quad G = (M, \ldots, M),  \\
		0, & \text{otherwise},
		       \end{array}
                        \right.
	\end{equation*}
	where
	\begin{equation}\label{Counterexample1.3}
	\max\{b +  1/q, 1/q\} < \varepsilon < -\beta/p' +1
	\end{equation}
	 (cf. \eqref{Counterexample1.2}). Invoking Theorem \ref{ThmWaveletsNewBesov} and \eqref{CardinalMj},
	\begin{align*}
		\|f\|_{T^b_r B^{0}_{p, q}(\R^d)} & \asymp  \bigg(\sum_{k=0}^\infty 2^{k b r} \bigg(\sum_{j=2^k-1}^{2^{k+1}-2} 2^{-jd q/p} (1+j)^{-\varepsilon q+ \beta q/p} |M_j|^{q/p} \bigg)^{r/q} \bigg)^{1/r} \\
		& \asymp  \bigg(\sum_{k=0}^\infty 2^{k b r} \bigg(\sum_{j=2^k-1}^{2^{k+1}-2}  (1+j)^{-\varepsilon q} \bigg)^{r/q} \bigg)^{1/r} \\
		& \asymp \bigg(\sum_{k=0}^\infty  2^{k (b  -\varepsilon  + 1/q)r} \bigg)^{1/r} < \infty
	\end{align*}
	where we have used \eqref{Counterexample1.3} in the last step. By construction, $f$ is compactly supported on $[0, \kappa] \times [0, 1]^{d-1}$. Furthermore
	\begin{equation*}
		\|f\|_{L^1([0, \kappa] \times [0, 1]^{d-1})}  = \sum_{j=1}^\infty \int_{R_j} |f(x)| \, dx \asymp \sum_{j=1}^\infty (1+j)^{-\beta +\beta/p-\varepsilon} = \infty
	\end{equation*}	
	where the divergence of the last sum follows from \eqref{Counterexample1.3}. This gives the desired counterexample.

	\textsc{Necessary condition $b > 1/p-1/q$ in \eqref{EmbeddingContradiction} with $p < \min\{q, r\}$:} According to the previous case, it is enough to consider the limiting value $b= 1/p-1/q$. We take $\beta$ and $\varepsilon$ such that
	$$
		1 < \beta < p' \Big(1-\frac{1}{r} \Big) \qquad \text{and} \qquad \beta -1 < \varepsilon < \frac{\beta}{p} - \frac{1}{r}.
	$$
	We modify the preceding counterexample as follows. Define
	$$
		\kappa_j = \sum_{l=0}^j (\log (1+l))^{-\beta} \frac{1}{1+l}.
	$$
	Note that the sequence $\{\kappa_j\}_{j \in \N_0}$ is  increasing and convergent (since $\beta > 1$). Let $\kappa = \lim_{j  \to \infty} \kappa_j$. The corresponding sets $R_j$ are given by \eqref{Rj} with
	$$
		|M_j| \asymp 2^{j d} (\log (1+j))^{-\beta} \frac{1}{1+j}.
	$$
	Consider the function $f$ given by \eqref{AuxFunctionCounterexample} where
	\begin{equation*}
		\lambda^{j,G}_m = \left\{\begin{array}{cl} (\log (1+j))^\varepsilon, & \quad j \in \N, \quad m \in M_j, \quad G = (M, \ldots, M),  \\
		0, & \text{otherwise}.
		       \end{array}
                        \right.
	\end{equation*}
	Applying Theorem \ref{ThmWaveletsNewBesov} we have
	\begin{align*}
		\|f\|_{T^b_r B^{0}_{p, q}(\R^d)} & \asymp  \bigg(\sum_{k=0}^\infty 2^{k b r} \bigg(\sum_{j=2^k-1}^{2^{k+1}-2}  (1+j)^{-q/p} (\log (1+j))^{ \varepsilon q-\beta q/p} \bigg)^{r/q} \bigg)^{1/r} \\
		& \asymp \bigg(\sum_{k=0}^\infty  (1+k)^{\varepsilon r - \beta r/p} 2^{k (b+1/q-1/p) r}   \bigg)^{1/r} < \infty
	\end{align*}
	where we have used $b = 1/p-1/q$ and $\varepsilon < \beta/p -1/r$. However $f \not \in L_1^{\text{loc}}(\R^d)$ since (recall that $\beta -1 < \varepsilon$)
		\begin{equation*}
		\|f\|_{L^1([0, \kappa] \times [0, 1]^{d-1})}  \asymp \sum_{j=1}^\infty  (\log (1+j))^{\varepsilon - \beta} \frac{1}{1+j} = \infty.
	\end{equation*}	
	
	\textsc{Necessary condition $b \geq 1/2 -1/q$ in \eqref{EmbeddingContradiction}:} 	Let $\psi \in \mathcal{S}(\R^d)$ be a fixed function, non-zero everywhere, satisfying
	\begin{equation}\label{PsiDef}
		\text{supp } \widehat{\psi} \subset \{\xi \in \R^d: |\xi| \leq 2\}
	\end{equation}
	and let $f$ be the function whose Fourier series is lacunary of the form
	\begin{equation}\label{LacunaryFSDef}
		f(x) \sim \sum_{j=3}^\infty \lambda_j e^{i (2^j-2) x_1} \psi(x), \qquad x = (x_1, \dots, x_d) \in \R^d,
	\end{equation}
	where $\{\lambda_j\}_{j \in \N}$ is a given sequence of complex numbers. It is plain to see that $(\varphi_j \widehat{f})^\vee (x) = \lambda_j e^{i (2^j-2) x_1} \psi(x)$. Accordingly
	\begin{equation}\label{ProofSobolevTheoremSubcritical7}
		     \|f\|_{T^b_r B^{0}_{p,q}(\mathbb{R}^d)}  = \|\psi\|_{L_p(\R^d)} \left(\sum_{j=3}^\infty 2^{j b r} \bigg(\sum_{\nu=2^j-1}^{2^{j+1}-2} |\lambda_\nu|^q\bigg)^{r/q}\right)^{1/r}.
	\end{equation}
	On the other hand, by Zygmund's theorem for lacunary Fourier series (see, e.g.,  \cite[Theorem 3.7.4]{Grafakos})
	\begin{equation}\label{ZygmundEstim}
		\|f\|_{L_1([0, 2 \pi]^d)} \asymp \bigg(\sum_{j=3}^\infty  |\lambda_j|^2\bigg)^{1/2}.
	\end{equation}
	Assume that \eqref{EmbeddingContradiction} holds with
	$$b + 1/q < 1/2.$$
	Then the Fourier series $f$ (cf. \eqref{LacunaryFSDef}) where
	$$
		\lambda_j = j^{-\varepsilon}, \qquad j \geq 3,
	$$
	and $b + 1/q < \varepsilon < 1/2$  satisfies (cf. \eqref{ProofSobolevTheoremSubcritical7})
	$$
	\|f\|_{T^b_r B^{0}_{p,q}(\mathbb{R}^d)}^r \asymp \sum_{j=3}^\infty 2^{j (b -\varepsilon+1/q) r}  < \infty
	$$
	but (cf. \eqref{ZygmundEstim})
	$$
		\|f\|_{L_1([0, 2 \pi]^d)}^2 \asymp \sum_{j=3}^\infty j^{-2 \varepsilon} = \infty.
	$$
	This yields the desired counterexample.
		
\textsc{Necessary condition $b > 1/2 -1/q$ under $2 <  \min\{q, r\}$ in \eqref{EmbeddingContradiction}:} According to the preceding case, it only remains to show that \eqref{EmbeddingContradiction} does not hold whenever $b = 1/2 -1/q$ and $2 < \min\{q, r\}$. In this situation, the counterexample is provided by the lacunary Fourier series \eqref{LacunaryFSDef}  with
$$
	\lambda_j = j^{-1/2} (\log j)^{-\varepsilon}, \qquad j \geq 3,
$$	
where $1/r < \varepsilon < 1/2$. Indeed, by \eqref{ProofSobolevTheoremSubcritical7}, $f \in T^b_r B^{0}_{p, q}(\R^d)$ but $f \not \in L_1([0, 2 \pi]^d)$ (cf. \eqref{ZygmundEstim}).

Next we focus on
	\begin{equation}\label{EmbeddingContradiction2}
		T^b_r B^{d(1/p-1)}_{p, q}(\R^d) \hookrightarrow L^{\text{loc}}_1(\R^d), \qquad 0 < p \leq 1.
	\end{equation}
\textsc{Necessary condition $b \geq 1-1/q$ in \eqref{EmbeddingContradiction2} with $q > 1$:} We shall proceed by contradiction, i.e., we assume that \eqref{EmbeddingContradiction2} holds with $b < 1-1/q$. Let
\begin{equation}\label{Coef}
		\lambda^{j,G}_m = \left\{\begin{array}{cl} 2^{j d} (1+j)^{-\varepsilon}, & \quad j \in \N_0, \quad m = (0, \ldots, 0), \quad G = (M, \ldots, M),  \\
		0, & \text{otherwise},
		       \end{array}
                        \right.
	\end{equation}
	where $\max\{1/q, b+1/q\} < \varepsilon < 1$. Then, by Theorem \ref{ThmWaveletsNewBesov}, the function $f$ given by \eqref{AuxFunctionCounterexample} and \eqref{Coef} satisfies
	\begin{align*}
		\|f\|_{T^b_r B^{d(1/p-1)}_{p, q}(\R^d)}^r & \asymp \sum_{k=0}^\infty 2^{k b r} \bigg(\sum_{j=2^k-1}^{2^{k+1}-2} (1+j)^{-\varepsilon q}  \bigg)^{r/q} \\
		& \asymp \sum_{k=0}^\infty 2^{k (b  -\varepsilon  + 1/q)r} < \infty.
	\end{align*}
	Furthermore
	\begin{align*}
		\|f\|_{L_1([0, 1]^d)} &= \sum_{j=0}^\infty \int_{Q_{j, (0, \ldots, 0)} \backslash Q_{j+1, (0, \ldots, 0)}} |f(x)| \, dx \\
		& \asymp \sum_{j=0}^\infty \bigg(\sum_{k=0}^j 2^{k d} (1+k)^{-\varepsilon}  \bigg) |Q_{j, (0, \ldots, 0)} \backslash Q_{j+1, (0, \ldots, 0)}| \\
		& \asymp \sum_{j=0}^\infty  (1+j)^{-\varepsilon} = \infty.
	\end{align*}
	This contradicts \eqref{EmbeddingContradiction2}.

\textsc{Necessary condition $b > 1-1/q$ in \eqref{EmbeddingContradiction2} with $\min \{q, r \} > 1$:} According to the previous case, it only remains to construct an extremal function in the limiting case $b=1-1/q$. Let $f$ be given by \eqref{AuxFunctionCounterexample} with
\begin{equation*}
		\lambda^{j,G}_m = \left\{\begin{array}{cl} 2^{j d} (1+j)^{-1} (1+ \log (1+j))^{-\beta}, & \quad j \in \N_0, \quad m = (0, \ldots, 0), \quad G = (M, \ldots, M),  \\
		0, & \text{otherwise},
		       \end{array}
                        \right.
	\end{equation*}
	where $1/r < \beta < 1$. In light of Theorem \ref{ThmWaveletsNewBesov}, since $q > 1$ and $\beta > 1/r$ we obtain
	\begin{align*}
		\|f\|_{T^b_r B^{d(1/p-1)}_{p, q}(\R^d)}^r & \asymp \sum_{k=0}^\infty 2^{k b r} \bigg(\sum_{j=2^k-1}^{2^{k+1}-2} (1+j)^{- q} (1+ \log (1+j))^{-\beta q}  \bigg)^{r/q} \\
		& \asymp \sum_{k=0}^\infty (1+k)^{-\beta r} < \infty.
	\end{align*}
	However, $f \not \in L_1^{\text{loc}}(\R^d)$ since
	\begin{align*}
		\|f\|_{L_1([0, 1]^d)} & \asymp \sum_{j=0}^\infty  2^{-j d}  \bigg(\sum_{k=0}^j 2^{k d} (1+k)^{- 1} (1+ \log (1+k))^{-\beta}  \bigg)\\
		& \asymp \sum_{j=0}^\infty (1+ \log (1+k))^{-\beta} \frac{1}{1+j}  = \infty
	\end{align*}
	where we have used that $\beta < 1$ in the last step. This yields the desired counterexample to \eqref{EmbeddingContradiction2}.
\end{proof}

\begin{rem}\label{RemarkEmbBl1}
	A careful inspection of the above proof shows that restriction $b \neq 0$ in Theorem \ref{TheoremRegular2} is only applied in necessary assertions (cf. Remark \ref{RemDelicate}). To be more precise: Let $s, b \in \R$ and $0 < p, q, r \leq \infty$. Then
	$$
		T^b_r B^s_{p, q}(\R^d) \hookrightarrow L^{\text{loc}}_1(\R^d)
	$$
	provided that one of the following conditions holds
	\begin{enumerate}[\upshape(i)]
	 \item $0 < p \leq \infty, \qquad s > d \big(\frac{1}{p}-1 \big)_+, \qquad 0 < q, r \leq \infty, \qquad b \in \R$,
	 \item $0 < p \leq 1, \qquad s =  d \big(\frac{1}{p}-1 \big), \qquad 0 < q \leq 1, \qquad 0 < r \leq 1, \qquad b \geq 0$,
	  \item $0 < p \leq 1, \qquad s =  d \big(\frac{1}{p}-1 \big), \qquad 0 < q \leq 1, \qquad 1 < r \leq \infty, \qquad b > 0$,
	   \item $0 < p \leq 1, \qquad s =  d \big(\frac{1}{p}-1 \big), \qquad 1 < q \leq \infty, \qquad 0 < r \leq 1, \qquad  b \geq  1 - \frac{1}{q}$,
	      \item $0 < p \leq 1, \qquad s =  d \big(\frac{1}{p}-1 \big), \qquad 1 < q \leq \infty, \qquad 1 < r \leq \infty, \qquad  b >  1 - \frac{1}{q}$,
	 \item $1 < p \leq 2, \qquad s=0,  \qquad 0 < q \leq p, \qquad 0 < r \leq p, \qquad b \geq 0$,
	 \item $1 < p \leq 2, \qquad s=0,  \qquad 0 < q \leq p, \qquad p < r \leq \infty, \qquad b > 0$,
	 \item $1 < p \leq 2, \qquad s=0,  \qquad p < q \leq \infty, \qquad 0 < r \leq p, \qquad b \geq  \frac{1}{p}-\frac{1}{q}$,
	 \item  $1 < p \leq 2,  \qquad s=0, \qquad p < q \leq \infty, \qquad p < r \leq \infty, \qquad b >  \frac{1}{p}-\frac{1}{q}$,
	 \item  $2 < p \leq \infty,  \qquad s=0, \qquad  0 < q \leq 2, \qquad 0 < r \leq 2, \qquad b \geq 0$,
	 	  \item  $2 < p \leq \infty,  \qquad s=0, \qquad  0 < q \leq 2, \qquad 2 < r \leq \infty, \qquad b > 0$,
	    \item  $2 < p \leq \infty,  \qquad s=0, \qquad  2 < q \leq \infty, \qquad 0 < r \leq 2, \qquad b \geq \frac{1}{2}  -\frac{1}{q}$,
	    	    \item  $2 < p \leq \infty,  \qquad s=0, \qquad  2 < q \leq \infty, \qquad 2 < r \leq \infty, \qquad b >  \frac{1}{2} -\frac{1}{q}$.
	\end{enumerate}
\end{rem}

	As a byproduct of Theorem \ref{TheoremRegular2} we obtain the following
	
	\begin{cor}\label{CorollaryRegularBMO}
	Let $s \in \R, 0 < p, q, r \leq \infty$ and $b \in \mathbb{R} \backslash \{0\}$.
	\begin{enumerate}[\upshape(i)]
	\item Assume $p < \infty$. The following statements are equivalent
	$$
		T^b_r B^s_{p, q}(\R^d) \hookrightarrow L^{\text{\emph{loc}}}_1(\R^d)
	$$
	and
	$$
		T^b_r B^s_{p, q}(\R^d) \hookrightarrow L_{\max\{1, p\}}(\R^d).
	$$	
	\item The following statements are equivalent
	$$
		T^b_r B^{s}_{\infty, q}(\R^d) \hookrightarrow L^{\text{\emph{loc}}}_1(\R^d)
	$$
	and (cf. \cite[p. 37]{Triebel83})\index{\bigskip\textbf{Spaces}!$bmo(\R^d)$}\label{BMO}
	$$
		T^b_r B^{s}_{\infty, q}(\R^d) \hookrightarrow bmo(\R^d).
	$$
	\end{enumerate}
	\end{cor}
	
	\begin{rem}
		The analogue of Corollary \ref{CorollaryRegularBMO} for the classical spaces $B^s_{p, q}(\R^d)$ was obtained in \cite[Corollary 3.3.1]{Sickel} (see also \cite[Corollary 4.6]{CaetanoLeopold13} where Besov spaces of generalized smoothness are treated).
	\end{rem}
	
	\begin{proof}[Proof of Corollary \ref{CorollaryRegularBMO}]
		(i): It is obvious that the embedding into $L_{\max\{1, p\}}(\R^d)$ implies the corresponding one into $L^{\text{loc}}_1(\R^d)$. Concerning the converse implication, i.e., if $T^b_r B^s_{p, q}(\R^d) \hookrightarrow L^{\text{loc}}_1(\R^d)$ then, by Theorem \ref{TheoremRegular2}, one of the conditions  (i)--(x) given in Theorem \ref{TheoremRegular2} is satisfied. A perusal of the proof of Theorem \ref{TheoremRegular2} shows that, under such an assumption, $T^b_r B^s_{p, q}(\R^d) \hookrightarrow L_{\max\{1, p\}}(\R^d).$
		
		(ii): It follows from the proof of Theorem \ref{TheoremRegular2} that $T^b_r B^{s}_{\infty, q}(\R^d) \hookrightarrow L^{\text{loc}}_1(\R^d)$ implies $T^b_r B^{s}_{\infty, q}(\R^d) \hookrightarrow L_\infty(\R^d)$ if $s > 0$ and $T^b_r B^{0}_{\infty, q}(\R^d) \hookrightarrow B^0_{\infty, 2}(\R^d)$ (cf. \eqref{LPestim7}). Hence the desired result is a consequence of the well-known facts that $L_\infty(\R^d) \hookrightarrow bmo (\R^d)$ and $B^0_{\infty, 2}(\R^d) \hookrightarrow bmo (\R^d)$ (cf. \cite[p. 114]{Sickel}).
	\end{proof}

\subsection{Embeddings for $T^b_r F^s_{p, q}(\R^d)$}\label{ReminderFNew1}

\begin{thm}\label{ThmFL1loc}
	Let $s \in \R, b \in \R\backslash \{0\}, 0 < p < \infty,$ and $0 < q, r \leq \infty$. Then
	\begin{equation}\label{ThmFL1loc1}
		T^b_r F^s_{p, q}(\R^d) \hookrightarrow L^{\text{\emph{loc}}}_1(\R^d)
	\end{equation}
	provided that one of the following conditions holds
	\begin{enumerate}[\upshape(i)]
	 \item $0 < p < \infty, \qquad s > d \big(\frac{1}{p}-1 \big)_+, \qquad 0 < q, r \leq \infty, \qquad b \in \R$,
	  \item $0 < p < 1, \qquad s =  d \big(\frac{1}{p}-1 \big), \qquad 0 < q\leq \infty, \qquad 0 < r \leq 1, \qquad b \geq 0$,
	  \item  $0 < p < 1, \qquad s =  d \big(\frac{1}{p}-1 \big), \qquad 0 < q\leq \infty, \qquad  1 < r \leq \infty, \qquad b > 0$,
	 \item $1 \leq p < \infty, \qquad s=0,  \qquad 0 < q \leq 2, \qquad 0 < r \leq \min \{2, p\}, \qquad b \geq 0$,
	 \item $1 \leq p < \infty, \qquad s=0,  \qquad 0 < q \leq 2, \qquad  \min \{2, p\} < r \leq  \infty, \qquad b > 0$,
	 \item $1 \leq p < \infty, \qquad s=0,  \qquad 2 < q \leq \infty, \qquad 0 < r \leq \min \{2, p\}, \qquad b \geq  \frac{1}{2}-\frac{1}{q}$,
	 \item  $1 \leq p < \infty,  \qquad s=0, \qquad 2 < q \leq \infty, \qquad \min\{2, p\} < r \leq \infty, \qquad b >  \frac{1}{2}-\frac{1}{q}$.
	\end{enumerate}
	Furthermore, the conditions \emph{(i)--(vi)} are necessary. If, in addition, $r > 2$ then the condition \emph{(vii)} is also necessary.
\end{thm}

\begin{proof}
	\textsc{Sufficiency part:} (i): Assume $0 < p < \infty$ and $s > d \big(\frac{1}{p}-1 \big)_+$. Let $s_0 \in \big(d \big(\frac{1}{p}-1 \big)_+, s\big)$. Therefore, by Theorems \ref{TheoremFB1new} and \ref{TheoremRegular2},
	$$
		T^b_r F^s_{p, q}(\R^d) \hookrightarrow T^b_r B^{s_0}_{p, q}(\R^d) \hookrightarrow  L^{\text{loc}}_1(\R^d).
	$$
	
	(ii), (iii): Let $0 < p < 1$ and $s =  d \big(\frac{1}{p}-1 \big)$. In virtue of Theorems \ref{ThmFJ} and Remark \ref{RemarkEmbBl1}, under assumptions given in (ii) and (iii), we have
	$$
		T^b_r F^{d (\frac{1}{p}-1)}_{p, q}(\R^d) \hookrightarrow T^b_r B^{0}_{1, p}(\R^d) \hookrightarrow  L^{\text{loc}}_1(\R^d).
	$$
	
	(iv)--(vii): We concern with $p=1$ and $s=0$. According to Corollary \ref{CorollaryEmbFBNew}, we have
	$$
		T^b_r F^{0}_{1, q}(\R^d) \hookrightarrow T^b_r B^{0}_{1, \max\{1, q\}}(\R^d).
	$$
	Therefore the desired result can be reduced to study the validity of $T^b_r B^{0}_{1, \max\{1, q\}}(\R^d) \hookrightarrow L^{\text{loc}}_1(\R^d)$. Accordingly, this embedding holds true provided that one of the conditions (iv)
--(vii) is satisfied.
	
	(iv)--(vii): Assume $1 \leq p < \infty$. By the (quasi-)triangle inequality in $L_{p/2}(\R^d)$, we obtain
	\begin{align}
		\|f\|_{F^0_{p, 2}(\R^d)} &= \bigg\| \bigg(\sum_{j=0}^\infty |(\varphi_j \widehat{f})^\vee|^2  \bigg)^{1/2} \bigg\|_{L_p(\R^d)} \nonumber \\
		& = \bigg\| \sum_{k=0}^\infty \sum_{j=2^k-1}^{2^{k+1}-2} |(\varphi_j \widehat{f})^\vee|^2  \bigg\|_{L_{p/2}(\R^d)}^{1/2} \nonumber \\
		& \leq \bigg(\sum_{k=0}^\infty \bigg\| \bigg(\sum_{j=2^k-1}^{2^{k+1}-2} |(\varphi_j \widehat{f})^\vee|^2 \bigg)^{1/2}  \bigg\|_{L_{p}(\R^d)}^{\min\{2, p\}}  \bigg)^{1/\min\{2, p\}}. \label{ThmFL1locProof1}
	\end{align}
	Furthermore, we claim
	\begin{equation}\label{ThmFL1locProof2}
	\bigg(\sum_{k=0}^\infty \bigg\| \bigg(\sum_{j=2^k-1}^{2^{k+1}-2} |(\varphi_j \widehat{f})^\vee|^2 \bigg)^{1/2}  \bigg\|_{L_{p}(\R^d)}^{\min\{2, p\}}  \bigg)^{1/\min\{2, p\}} \lesssim \|f\|_{T^b_r F^{0}_{p, q}(\R^d)}.
	\end{equation}
	Assuming momentarily the validity of \eqref{ThmFL1locProof2}, it follows from \eqref{ThmFL1locProof1} that $T^b_r F^{0}_{p, q}(\R^d) \hookrightarrow F^0_{p, 2}(\R^d)$, which in turn gives $T^b_r F^{0}_{p, q}(\R^d) \hookrightarrow L_1^{\text{loc}}(\R^d)$ (cf. Theorem \ref{TheoremRegularFspaces}).
	
	Next we show \eqref{ThmFL1locProof2}. Suppose first $0 < q \leq 2, 0 < r \leq \min \{2, p\}$ and $b \geq 0$. Then
	\begin{align*}
		\bigg(\sum_{k=0}^\infty \bigg\| \bigg(\sum_{j=2^k-1}^{2^{k+1}-2} |(\varphi_j \widehat{f})^\vee|^2 \bigg)^{1/2}  \bigg\|_{L_{p}(\R^d)}^{\min\{2, p\}}  \bigg)^{1/\min\{2, p\}} & \leq \\
		& \hspace{-6cm} \bigg(\sum_{k=0}^\infty 2^{k b r} \bigg\| \bigg(\sum_{j=2^k-1}^{2^{k+1}-2} |(\varphi_j \widehat{f})^\vee|^q \bigg)^{1/q}  \bigg\|_{L_{p}(\R^d)}^{r}  \bigg)^{1/r}  = \|f\|_{T^b_r F^{0}_{p, q}(\R^d)}.
	\end{align*}
	
	Secondly, assume $0 < q \leq 2, \min\{2, p\} < r \leq \infty$ and $b > 0$. By H\"older's inequality,
	\begin{align*}
				\bigg(\sum_{k=0}^\infty \bigg\| \bigg(\sum_{j=2^k-1}^{2^{k+1}-2} |(\varphi_j \widehat{f})^\vee|^2 \bigg)^{1/2}  \bigg\|_{L_{p}(\R^d)}^{\min\{2, p\}}  \bigg)^{1/\min\{2, p\}} & \lesssim \\
				& \hspace{-6cm}	\bigg(\sum_{k=0}^\infty 2^{k b r} \bigg\| \bigg(\sum_{j=2^k-1}^{2^{k+1}-2} |(\varphi_j \widehat{f})^\vee|^q \bigg)^{1/q}  \bigg\|_{L_{p}(\R^d)}^r  \bigg)^{1/r} = \|f\|_{T^b_r F^{0}_{p, q}(\R^d)}.
	\end{align*}
	
	Thirdly, assume $2 < q \leq \infty, 0 < r \leq \min \{2, p\}$ and $b \geq 1/2-1/q$. By H\"older's inequality,
	\begin{align*}
		\bigg(\sum_{k=0}^\infty \bigg\| \bigg(\sum_{j=2^k-1}^{2^{k+1}-2} |(\varphi_j \widehat{f})^\vee|^2 \bigg)^{1/2}  \bigg\|_{L_{p}(\R^d)}^{\min\{2, p\}}  \bigg)^{1/\min\{2, p\}} & \lesssim \\
		& \hspace{-6cm} \bigg(\sum_{k=0}^\infty 2^{k (1/2-1/q) \min \{2, p\}} \bigg\| \bigg(\sum_{j=2^k-1}^{2^{k+1}-2} |(\varphi_j \widehat{f})^\vee|^q \bigg)^{1/q}  \bigg\|_{L_{p}(\R^d)}^{\min\{2, p\}}  \bigg)^{1/\min\{2, p\}}  \\
		&  \hspace{-6cm} \leq \bigg(\sum_{k=0}^\infty 2^{k b r} \bigg\| \bigg(\sum_{j=2^k-1}^{2^{k+1}-2} |(\varphi_j \widehat{f})^\vee|^q \bigg)^{1/q}  \bigg\|_{L_{p}(\R^d)}^{r}  \bigg)^{1/r} = \|f\|_{T^b_r F^{0}_{p, q}(\R^d)}.
	\end{align*}
	
	Fourthly, if $2 < q \leq \infty, \min \{2, p\} < r \leq \infty$ and $b > 1/2-1/q$. Applying twice H\"older's inequality, the following is achieved
	\begin{align*}
		\bigg(\sum_{k=0}^\infty \bigg\| \bigg(\sum_{j=2^k-1}^{2^{k+1}-2} |(\varphi_j \widehat{f})^\vee|^2 \bigg)^{1/2}  \bigg\|_{L_{p}(\R^d)}^{\min\{2, p\}}  \bigg)^{1/\min\{2, p\}} & \lesssim \\
		& \hspace{-6cm} \bigg(\sum_{k=0}^\infty 2^{k (1/2-1/q) \min \{2, p\}} \bigg\| \bigg(\sum_{j=2^k-1}^{2^{k+1}-2} |(\varphi_j \widehat{f})^\vee|^q \bigg)^{1/q}  \bigg\|_{L_{p}(\R^d)}^{\min\{2, p\}}  \bigg)^{1/\min\{2, p\}}  \\
		&  \hspace{-6cm} \lesssim \bigg(\sum_{k=0}^\infty 2^{k b r} \bigg\| \bigg(\sum_{j=2^k-1}^{2^{k+1}-2} |(\varphi_j \widehat{f})^\vee|^q \bigg)^{1/q}  \bigg\|_{L_{p}(\R^d)}^{r}  \bigg)^{1/r} = \|f\|_{T^b_r F^{0}_{p, q}(\R^d)}.
	\end{align*}

	\textsc{Necessary condition $b > 0$ if $0 < p \leq 1$ and $s= d \big(\frac{1}{p} -1\big)$ in \eqref{ThmFL1loc1}:} Given $0 < p_0 < p$, by Theorem \ref{ThmFJ2}, we have
	$$
		T^b_r B^{d(1/p_0-1)}_{p_0, p}(\R^d) \hookrightarrow T^b_r F^{d(1/p-1)}_{p, q}(\R^d)
	$$
	and thus \eqref{ThmFL1loc1} yields
	$$
		T^b_r B^{d(1/p_0-1)}_{p_0, p}(\R^d) \hookrightarrow L^{\text{loc}}_1(\R^d).
	$$
	This implies $b > 0$ (cf. Theorem \ref{TheoremRegular2}).
	
	\textsc{Necessary condition $b > 0$ if $1 < p < \infty$ and $s=0$ in \eqref{ThmFL1loc1}:} We shall proceed by contradiction, i.e., we assume \eqref{ThmFL1loc1} holds for some $b < 0$. Then choose $\beta$ such that
	\begin{equation}\label{AuxFunctionCounterexampleF0}
		1 < \beta < -b p' + 1.
	\end{equation}
	Let
	$$
		\kappa_j = \sum_{l=0}^j (1+j)^{-\beta} \qquad \text{and} \qquad \kappa = \lim_{j \to \infty} \kappa_j.
	$$
	and consider the rectangles $R_j$ defined by \eqref{Rj}. For each $j$, we denote by $\{Q_{j, m}: m \in M_j\}$ the collection of all (dyadic) cubes of side length $2^{-j}$ contained in $R_j$. Recall that $|M_j| \asymp 2^{j d} (1 + j)^{-\beta}$  (cf. \eqref{CardinalMj}). Let
	\begin{equation}\label{AuxFunctionCounterexampleF}
		 f = \sum_{j \in \N_0,G \in G^j,m \in \Z^d} \lambda^{j,G}_m 2^{-j d/2}
    \Psi^j_{G,m}
	\end{equation}
	where
		\begin{equation*}
		\lambda^{j,G}_m = \left\{\begin{array}{cl} \eta_j, & \quad j \in \N, \quad m \in M_j, \quad G = (M, \ldots, M),  \\
		0, & \text{otherwise},
		       \end{array}
                        \right.
	\end{equation*}
	with the scalar-valued sequence $\{\eta_j\}_{j \in \N}$ to be chosen.
	
	According to Theorem \ref{ThmWaveletsNewTriebelLizorkin} and using the disjointness of supports of dyadic cubes with fixed side length, as well as the disjointness of the $R_j$'s, we obtain
	\begin{align}
		\|f\|_{T^b_r F^{0}_{p, q}(\R^d)} & \asymp \bigg(\sum_{k=0}^\infty 2^{k b r} \bigg\| \bigg(\sum_{j=2^k-1}^{2^{k+1}-2} \sum_{m \in M_j} |\lambda^{j, (M, \ldots, M)}_m|^q \chi_{j, m} \bigg)^{1/q} \bigg\|_{L_p(\R^d)}^r \bigg)^{1/r} \nonumber \\
		& = \bigg(\sum_{k=0}^\infty 2^{k b r} \bigg\| \sum_{j=2^k-1}^{2^{k+1}-2} \sum_{m \in M_j} |\lambda^{j, (M, \ldots, M)}_m| \chi_{j, m}  \bigg\|_{L_p(\R^d)}^r \bigg)^{1/r} \nonumber  \\
		& \asymp \bigg(\sum_{k=0}^\infty 2^{k b r} \bigg( \sum_{j=2^k-1}^{2^{k+1}-2}  |\eta_j|^p 2^{-j d} |M_j| \bigg)^{r/p} \bigg)^{1/r} \nonumber  \\
		& \asymp \bigg(\sum_{k=0}^\infty 2^{k (b-\beta/p) r} \bigg( \sum_{j=2^k-1}^{2^{k+1}-2}  |\eta_j|^p \bigg)^{r/p} \bigg)^{1/r}. \label{AuxFunctionCounterexampleF2}
	\end{align}
	On the other hand
	\begin{equation}\label{AuxFunctionCounterexampleF3}
		\|f\|_{L^1([0, \kappa] \times [0, 1]^{d-1})}  = \sum_{j=1}^\infty \int_{R_j} |f(x)| \, dx \asymp \sum_{j=1}^\infty |\eta_j| (1 + j)^{-\beta}.
	\end{equation}
	
	Set
	$$\eta_j = (1 + j)^{\varepsilon} \qquad \text{where} \qquad \beta - 1 < \varepsilon < -b -\frac{1}{p}+ \frac{\beta}{p}.$$
	Here we recall that $\beta < -b p'+1$ (cf. \eqref{AuxFunctionCounterexampleF0}). In light of \eqref{AuxFunctionCounterexampleF2} and \eqref{AuxFunctionCounterexampleF3}, we derive
	$$
		\|f\|_{T^b_r F^{0}_{p, q}(\R^d)} \asymp  \bigg(\sum_{k=0}^\infty 2^{k (b-\beta/p + \varepsilon + 1/p) r} \bigg)^{1/r} < \infty
	$$
	and
	$$
	\|f\|_{L^1([0, \kappa] \times [0, 1]^{d-1})} \asymp \sum_{j=1}^\infty (1 + j)^{\varepsilon -\beta} = \infty.	
	$$
	This gives the desired counterexample.
	
	\textsc{Necessary condition $b \geq \frac{1}{2}-\frac{1}{q}$ if $2 < q, \, 1 \leq p < \infty,$ and $s= 0$ in \eqref{ThmFL1loc1}:} We shall construct extremal functions showing that \eqref{ThmFL1loc1} is no longer true provided that  $b < 1/2-1/q$. Indeed, consider the lacunary Fourier series
\begin{equation}\label{LacFS}
		f(x) \sim \sum_{j=3}^\infty \lambda_j e^{i (2^j-2) x_1} \psi(x), \qquad x = (x_1, \dots, x_d) \in \R^d,
	\end{equation}
	where $\{\lambda_j\}_{j \in \N}$ is a sequence of complex numbers to be chosen and  $\psi \in \mathcal{S}(\R^d) \backslash \{0\}$ satisfies \eqref{PsiDef}. Elementary computations lead to
	\begin{equation}\label{AuxFunctionCounterexampleF4}
		\|f\|_{T^b_r F^{0}_{p, q}(\R^d)} \asymp \bigg(\sum_{j=3}^\infty 2^{j b r} \bigg(\sum_{k=2^j-1}^{2^{j+1}-2} |\lambda_k|^q \bigg)^{r/q} \bigg)^{1/r}.
	\end{equation}
	On the other hand, by \eqref{ZygmundEstim},
	\begin{equation}\label{AuxFunctionCounterexampleF5}
		\|f\|_{L_1([0, 2 \pi]^d)} \asymp  \bigg(\sum_{j=3}^\infty  |\lambda_j|^2\bigg)^{1/2}.
	\end{equation}
	
	Let $\lambda_j = (1+j)^{-\varepsilon}$ where $b + 1/q < \varepsilon < 1/2$. It follows from \eqref{AuxFunctionCounterexampleF4} and \eqref{AuxFunctionCounterexampleF5} that
	$$
	\|f\|_{T^b_r F^{0}_{p, q}(\R^d)} \asymp \bigg(\sum_{j=3}^\infty 2^{j (b-\varepsilon + 1/q) r} \bigg)^{1/r}	< \infty
	$$
	but
	$$
	\|f\|_{L_1([0, 2 \pi]^d)} \asymp  \bigg(\sum_{j=3}^\infty  (1 + j)^{-\varepsilon 2} \bigg)^{1/2} = \infty.
	$$
	
	\textsc{Necessary condition $b > \frac{1}{2}-\frac{1}{q}$ if $2 <  \min\{q, r \}, \, 1 \leq p < \infty,$ and $s= 0$ in \eqref{ThmFL1loc1}:} According to the previous case, it only remains to deal with the limiting value $b = 1/2 - 1/q$.  Let $\lambda_j = (1 + j)^{-1/2} (1 + \log (1 + j))^{-\beta}$ where $1/r < \beta < 1/2$ and the related Fourier series $f$ given by \eqref{LacFS}. It follows from \eqref{AuxFunctionCounterexampleF4} and \eqref{ZygmundEstim} that
	$$
			\|f\|_{T^{1/2-1/q}_r F^{0}_{p, q}(\R^d)} \asymp \bigg(\sum_{j=3}^\infty j^{-\beta r} \bigg)^{1/r} < \infty
	$$
	and
	$$
	\|f\|_{L_1([0, 2 \pi]^d)} \asymp  \bigg(\sum_{j=3}^\infty (\log j)^{-2 \beta}  \frac{1}{j} \bigg)^{1/2} = \infty.
	$$
	This shows $T^{1/2-1/q}_r F^{0}_{p, q}(\R^d) \not \hookrightarrow L_1^{\text{loc}}(\R^d)$.
	
\end{proof}

The analog of Corollary \ref{CorollaryRegularBMO} for $F$-spaces reads as follows.

\begin{cor}
	Let $s \in \R, 0 < p < \infty, 0 < q \leq \infty, b \in \mathbb{R} \backslash \{0\}$ and $0 < r \leq \infty$ (in case the parameters satisfy condition \emph{(vii)} given in Theorem \ref{ThmFL1loc} we assume that $r > 2$).  Then the following statements are equivalent
		$$
		T^b_r F^s_{p, q}(\R^d) \hookrightarrow L^{\text{\emph{loc}}}_1(\R^d)
	$$
	and
	$$
		T^b_r F^s_{p, q}(\R^d) \hookrightarrow L_{\max\{1, p\}}(\R^d).
	$$	
\end{cor}

\newpage
\section{Embeddings in the space of continuous functions}

Let $C(\R^d)$ be the space of all complex-valued uniformly continuous functions on $\R^d$\index{\bigskip\textbf{Spaces}!$C(\R^d)$}\label{CU}, equipped with the usual sup-norm. It is a well-known result that
\begin{equation}\label{ClassicBesovC}
	B^{s, b}_{p, q}(\R^d) \hookrightarrow C(\R^d) \iff  \left\{\begin{array}{lc}
                            0 <p \leq \infty, \quad 0 < q \leq \infty, \quad -\infty < b < \infty, \quad \text{and} \quad s > \frac{d}{p}, \\
                            0 < p \leq \infty, \quad 0 < q \leq 1, \quad b \geq 0, \quad \text{and} \quad s = \frac{d}{p},
                            \\
                            0 < p \leq \infty, \quad 1 < q \leq \infty, \quad b > \frac{1}{q'}, \quad \text{and} \quad s = \frac{d}{p}.
            \end{array}
            \right.
	\end{equation}
	The analogue for $F$-spaces reads as follows
	\begin{equation}\label{ClassicTLC}
		F^{s, b}_{p, q}(\R^d) \hookrightarrow C(\R^d) \iff \left\{\begin{array}{lc}
		0 < p < \infty, \quad 0 < q \leq \infty, \quad - \infty < b < \infty, \quad \text{and} \quad s > \frac{d}{p}, \\
		1 < p < \infty, \quad 0 < q \leq \infty, \quad b > \frac{1}{p'}, \quad \text{and} \quad s = \frac{d}{p}, \\
		0 < p \leq 1, \quad 0 < q \leq \infty, \quad b \geq 0, \quad \text{and} \quad s = \frac{d}{p}.
		\end{array}
		\right.
	\end{equation}
	For the proof of these results, we refer to \cite[Theorem 3.3.1]{Sickel} if $b=0$, \cite{Kalyabin} and \cite[Proposition 3.13, Example 3.14]{CaetanoMoura} for $b \in \R$ (even in the more general setting of function spaces with generalised smoothness).

 The goal of this section is to extend \eqref{ClassicBesovC}-\eqref{ClassicTLC} into the scale of truncated spaces $T^b_r A^{s}_{p, q}(\R^d)$.

 \begin{thm}\label{ThmC}
 	Let $p, q, r \in (0, \infty], \, b \in \R \backslash \{0\}$, and $s \in \R$. Then
	\begin{equation}\label{embeddingC}
	T^b_r B^s_{p, q}(\R^d) \hookrightarrow C(\R^d)
	\end{equation}
	if and only if one of the following conditions is satisfied
	\begin{enumerate}[\upshape(i)]
\item $0 <p \leq \infty, \quad 0 < q \leq \infty, \quad 0 < r \leq \infty,$ \quad and \quad $s > \frac{d}{p}$.
                         \item $0 < p \leq \infty, \quad 0 < q \leq 1, \quad 0 < r \leq 1, \quad b \geq 0,$ \quad and \quad $s = \frac{d}{p}$.
                           \item $0 < p \leq \infty, \quad 1 < q \leq \infty, \quad 0 < r \leq 1, \quad b \geq  \frac{1}{q'}$, \quad and \quad $s = \frac{d}{p}$.
                           \item $0 < p \leq \infty, \quad 0 < q \leq 1, \quad 1 < r \leq \infty, \quad b > 0$, \quad and \quad $s = \frac{d}{p}$.
                           \item  $0 < p \leq \infty, \quad 1 < q \leq \infty, \quad 1 < r \leq \infty, \quad b >  \frac{1}{q'}$, \quad and \quad $s = \frac{d}{p}$.
	\end{enumerate}
	The space $C(\R^d)$ can be replaced by $L_\infty(\R^d)$ in \eqref{embeddingC}.
 \end{thm}

 Specializing Theorem \ref{ThmC} with $q=r$ one recovers \eqref{ClassicBesovC} (cf. Proposition \ref{PropositionCoincidences}).

\begin{rem}\label{RemIfPart}
	The if-part of Theorem \ref{ThmC} works with $b \in \R$. On the other hand, the proof of the corresponding only-if part will rely on certain extremal functions constructed via wavelets (cf. Theorem \ref{ThmWaveletsNewBesov}), so that the restriction $b \neq 0$ appears naturally in this methodology; see Remark \ref{RemDelicate}.
\end{rem}

 \begin{rem}
 	Note that the conditions (ii)--(v) in Theorem \ref{ThmC} can be unified, for $0 < p, q, r \leq \infty$, as
	$$
	  b >  1 - \frac{1}{\max\{1, q\}}
	$$
	(where the limiting value $b =  1 - \frac{1}{\max\{1, q\}}$ is also admissible if $r \leq 1$).
  \end{rem}

 \begin{proof}[Proof of Theorem \ref{ThmC}]
 	\textsc{Sufficiency part:} (i): Let $s_0 \in (\frac{d}{p}, s)$. By Corollary \ref{Corollary11.4}, $T^b_r B^s_{p, q}(\R^d) \hookrightarrow B^{s_0, b} _{p, q}(\R^d)$. Thus the desired result follows from the classical embedding $B^{s_0, b}_{p, q}(\R^d) \hookrightarrow C(\R^d)$ (cf. \eqref{ClassicBesovC}).

	 Next we deal with the limiting case $s=d/p$. According to the Nikolski$\breve{\text{\i}}$ inequality for entire functions of exponential type (cf. \cite[p. 126]{Nikolskii} and \cite[p. 18]{Triebel83})
	$$
		\|(\varphi_\nu \widehat{f})^\vee\|_{L_\infty(\R^d)} \lesssim 2^{\nu d/p} 	\|(\varphi_\nu \widehat{f})^\vee\|_{L_p(\R^d)}, \qquad \nu \in \mathbb{N}_0.
	$$
	Therefore
	\begin{equation}\label{7.2}
		\sum_{\nu=0}^\infty \|(\varphi_\nu \widehat{f})^\vee\|_{L_\infty(\R^d)}  \lesssim \sum_{j=0}^\infty
 \sum_{\nu=2^j-1}^{2^{j+1}-2} 2^{\nu d/p} 	\|(\varphi_\nu \widehat{f})^\vee\|_{L_p(\R^d)}.
 	\end{equation}
	
	(ii): If $q, r \in (0, 1]$ and $b \geq 0$ then the right-hand side of \eqref{7.2} can be estimated from above by
	$$
		\left(\sum_{j=0}^\infty 2^{j b r} \bigg(\sum_{\nu=2^j-1}^{2^{j+1}-2} 2^{\nu d q/p} 	\|(\varphi_\nu \widehat{f})^\vee\|_{L_p(\R^d)}^q \bigg)^{r/q} \right)^{1/r} = \|f\|_{T^b_r B^s_{p, q}(\R^d)}.
	$$
	
	(iii): Let $q \in (1, \infty], \, r \in (0, 1]$ and $b  \geq 1/q'$. By H\"older's inequality
	\begin{equation}\label{7.3}
		\sum_{\nu=2^j-1}^{2^{j+1}-2} 2^{\nu d/p} 	\|(\varphi_\nu \widehat{f})^\vee\|_{L_p(\R^d)} \lesssim 2^{j/q'} \bigg(\sum_{\nu=2^j-1}^{2^{j+1}-2} 2^{\nu d q/p} 	\|(\varphi_\nu \widehat{f})^\vee\|_{L_p(\R^d)}^q \bigg)^{1/q}
	\end{equation}
	(with the usual modification if $q=\infty$). Inserting this estimate into \eqref{7.2}, we obtain
	\begin{align*}
			\sum_{\nu=0}^\infty \|(\varphi_\nu \widehat{f})^\vee\|_{L_\infty(\R^d)} & \lesssim \sum_{j=0}^\infty 2^{j/q'} \bigg(\sum_{\nu=2^j-1}^{2^{j+1}-2} 2^{\nu d q/p} 	 \|(\varphi_\nu \widehat{f})^\vee\|_{L_p(\R^d)}^q \bigg)^{1/q} \\
			&\hspace{-2.5cm} \leq  \left(\sum_{j=0}^\infty 2^{j b r} \bigg(\sum_{\nu=2^j-1}^{2^{j+1}-2} 2^{\nu d q/p} 	\|(\varphi_\nu \widehat{f})^\vee\|_{L_p(\R^d)}^q \bigg)^{r/q} \right)^{1/r} = \|f\|_{T^b_r B^s_{p, q}(\R^d)}.
	\end{align*}
	
	(iv): Let $q \in (0,1], \, r \in (1, \infty]$ and $b > 0$. Applying H\"older's inequality with exponent $r$, it follows from \eqref{7.2} that
	$$
	\sum_{\nu=0}^\infty \|(\varphi_\nu \widehat{f})^\vee\|_{L_\infty(\R^d)}  \lesssim \left(\sum_{j=0}^\infty 2^{j b r}
\bigg( \sum_{\nu=2^j-1}^{2^{j+1}-2} 2^{\nu d q/ p} 	\|(\varphi_\nu \widehat{f})^\vee\|_{L_p(\R^d)}^q \bigg)^{r/q} \right)^{1/r}.
	$$
	This gives the desired embedding.
	
	(v): Let $q, r \in (1, \infty]$ and $b > 1/q'$. By \eqref{7.2} and \eqref{7.3}, one has
	\begin{align*}
		\sum_{\nu=0}^\infty \|(\varphi_\nu \widehat{f})^\vee\|_{L_\infty(\R^d)} & \lesssim \sum_{j=0}^\infty 2^{j/q'} \bigg(\sum_{\nu=2^j-1}^{2^{j+1}-2} 2^{\nu d q/p} 	\|(\varphi_\nu \widehat{f})^\vee\|_{L_p(\R^d)}^q \bigg)^{1/q}  \\
		&\hspace{-3.25cm} \lesssim \left(\sum_{j=0}^\infty 2^{j b r}
\bigg( \sum_{\nu=2^j-1}^{2^{j+1}-2} 2^{\nu d q/ p} 	\|(\varphi_\nu \widehat{f})^\vee\|_{L_p(\R^d)}^q \bigg)^{r/q} \right)^{1/r} \left(\sum_{j=0}^\infty 2^{-j(b-1/q')r'} \right)^{1/r'} \\
& \hspace{-3.25cm} \asymp  \|f\|_{T^b_r B^s_{p, q}(\R^d)}.
	\end{align*}
	
	\textsc{Necessary condition $b > 0$ in \eqref{embeddingC} with $r > 1$ and $s = d/p$:} Consider
	\begin{equation}\label{AuxFunctionCounterexampleC}
		 f = \sum_{j \in \N_0,G \in G^j,m \in \Z^d} \lambda^{j,G}_m 2^{-j d/2}
    \Psi^j_{G,m}
	\end{equation}
	where the coefficients $\lambda^{j,G}_m$ are given by
		\begin{equation*}
		\lambda^{j,G}_m = \left\{\begin{array}{cl} (1+k)^{-\varepsilon}, & \quad j = 2^k, \quad k \in \N_0, \quad m = (0, \ldots, 0), \quad G = (M, \ldots, M),  \\
		0, & \text{otherwise},
		       \end{array}
                        \right.
	\end{equation*}
	and $1/r < \varepsilon \leq 1$. By Theorem \ref{ThmWaveletsNewBesov}, if $b < 0$ then
	$$
		\|f\|_{T^b_r B^{d/p}_{p, q}(\R^d)}^r  \asymp \sum_{k=0}^\infty 2^{k b r} (1+k)^{-\varepsilon r} < \infty,
	$$
	but $f \not \in C(\R^d)$ since
	$$
		f(0, \ldots, 0) \asymp \sum_{k=0}^\infty (1+k)^{-\varepsilon} = \infty.
	$$
	
	\textsc{Necessary condition $b \geq 1/q'$ in \eqref{embeddingC} with $q > 1$ and $s = d/p$:} Assume $b <  1/q'$ and let $\varepsilon$ be such that $b + 1/q < \varepsilon \leq 1$. Consider $f$ given by \eqref{AuxFunctionCounterexampleC} with
		\begin{equation*}
		\lambda^{j,G}_m = \left\{\begin{array}{cl} (1+j)^{-\varepsilon}, & \quad j  \in \N_0, \quad m = (0, \ldots, 0), \quad G = (M, \ldots, M),  \\
		0, & \text{otherwise}.
		       \end{array}
                        \right.
	\end{equation*}
	Applying Theorem \ref{ThmWaveletsNewBesov}, we have
	$$
		\|f\|_{T^b_r B^{d/p}_{p, q}(\R^d)}^r  \asymp \sum_{k=0}^\infty 2^{k(b-\varepsilon + 1/q) r} < \infty.
	$$
	However $f$ is not continuous since
	$$
		f(0, \ldots, 0) \asymp \sum_{j=0}^\infty (1+j)^{-\varepsilon} = \infty.
	$$
	
	\textsc{Necessary condition $b >  1/q'$ in \eqref{embeddingC} with $q > 1, r > 1$ and $s = d/p$:} In light of the preceding counterexample, it only remains to construct an extremal function in the limiting case $b =  1/q'$ under the additional assumption $r > 1$. To do this, we let $\varepsilon$ satisfying $1/r < \varepsilon \leq 1$ and take $f$ defined by \eqref{AuxFunctionCounterexampleC} with
		\begin{equation*}
		\lambda^{j,G}_m = \left\{\begin{array}{cl} (1+j)^{-1} (1 + \log (1 + j))^{-\varepsilon}, & \quad j  \in \N_0, \quad m = (0, \ldots, 0), \quad G = (M, \ldots, M),  \\
		0, & \text{otherwise}.
		       \end{array}
                        \right.
	\end{equation*}
	According to Theorem \ref{ThmWaveletsNewBesov},
	$$
		\|f\|_{T^{1/q'}_r B^{d/p}_{p, q}(\R^d)}^r  \asymp \sum_{k=0}^\infty 2^{k(1/q'-1 + 1/q) r} (1 + k)^{-\varepsilon r} = \sum_{k=0}^\infty (1+k)^{-\varepsilon r} < \infty,
	$$
	but
	$$
		f(0, \ldots, 0) \asymp \sum_{j=0}^\infty  (1+j)^{-1} (1 + \log (1 + j))^{-\varepsilon} = \infty.
	$$
	
	\textsc{Necessary condition $s \geq d/p$ in  \eqref{embeddingC}:} We can proceed by contradiction, i.e., assume that there exists $s < d/p$ such that
	$$
	T^b_r B^s_{p, q}(\R^d) \hookrightarrow C(\R^d).
	$$
	We choose $s_0$ such that $s < s_0 < d/p$. It follows from Corollary \ref{CorollaryBNewB} that
	$$
		B^{s_0, b}_{p, q}(\R^d) \hookrightarrow C(\R^d),
	$$
	but this contradicts \eqref{ClassicBesovC}. One can also apply Theorem \ref{ThmWaveletsNewBesov} to construct explicit extremal functions showing the necessity of $s \geq d/p$ in  \eqref{embeddingC}. Indeed, if $s < d/p$ then the function $f$ defined by \eqref{AuxFunctionCounterexampleC} with
		\begin{equation*}
		\lambda^{j,G}_m = \left\{\begin{array}{cl} 1, & \quad j  \in \N_0, \quad m = (0, \ldots, 0), \quad G = (M, \ldots, M),  \\
		0, & \text{otherwise},
		       \end{array}
                        \right.
	\end{equation*}
	satisfies $f \in T^b_r B^s_{p, q}(\R^d)$ but $f \not \in C(\R^d)$.
	
 \end{proof}

  \begin{thm}\label{ThmCF}
 	Let $p \in (0, \infty), q, r \in (0, \infty], \, b \in \R \backslash \{0\}$, and $s \in \R$. Then
	\begin{equation}\label{embeddingCF}
	T^b_r F^s_{p, q}(\R^d) \hookrightarrow C(\R^d)
	\end{equation}
	if and only if one of the following conditions is satisfied
	\begin{enumerate}[\upshape(i)]
\item $0 <p < \infty, \quad 0 < q \leq \infty, \quad 0 < r \leq \infty,$ \quad and \quad $s > \frac{d}{p}$.
                         \item $0 < p \leq 1, \quad 0 < q \leq \infty, \quad 1 < r \leq \infty, \quad b >  0,$ \quad and \quad $s = \frac{d}{p}$.
                           \item $0 < p \leq 1, \quad 0 < q \leq \infty, \quad 0 < r \leq 1, \quad b \geq 0$, \quad and \quad $s = \frac{d}{p}$.
                           \item $1 < p < \infty, \quad 0 < q \leq \infty, \quad 1 < r \leq \infty, \quad b > \frac{1}{p'} $, \quad and \quad $s = \frac{d}{p}$.
                           \item  $1 < p < \infty, \quad 0 < q \leq \infty, \quad 0 < r \leq 1, \quad b \geq \frac{1}{p'}$, \quad and \quad $s = \frac{d}{p}$.
	\end{enumerate}
	The space $C(\R^d)$ can be replaced by $L_\infty(\R^d)$ in \eqref{embeddingCF}.
 \end{thm}
 \begin{proof}
 	\textsc{Sufficiency part:} (i): By Corollary \ref{CorTruncBTLFixedps},
	$$
		T^b_r F^{s}_{p, q} (\R^d) \hookrightarrow T^b_r B^{s}_{p, \max\{p, q\}}(\R^d)
	$$
	and thus Theorem \ref{ThmC} yields
	$$
		T^b_r F^{s}_{p, q} (\R^d) \hookrightarrow C(\R^d)
	$$
	for  $s > d/p$.
	
	(ii)--(v): Assume $s=d/p$. In virtue of Theorems \ref{ThmFJ} and \ref{ThmC}, the following holds
	$$
		T^b_r F^{d/p}_{p, q} (\R^d) \hookrightarrow T^b_r B^{0}_{\infty, p}(\R^d) \hookrightarrow C(\R^d).
	$$
	
	\textsc{Necessary conditions:} We first deal with the case $s=d/p$, i.e., we assume
	\begin{equation}\label{ThmCFProof1}
		T^b_r F^{d/p}_{p, q}(\R^d) \hookrightarrow C(\R^d).
	\end{equation}
	Let $0 < p_1 < p$. By Theorem \ref{ThmFJ2}, we have
	\begin{equation}\label{ThmCFProof2}
	T^b_r B^{d/p_1}_{p_1, p}(\R^d)	\hookrightarrow T^b_r F^{d/p}_{p, q}(\R^d).
	\end{equation}
	Putting together \eqref{ThmCFProof1} and \eqref{ThmCFProof2}, we obtain
	$$
		T^b_r B^{d/p_1}_{p_1, p}(\R^d)	\hookrightarrow  C(\R^d)
	$$
	and, by Theorem \ref{ThmC}, this implies the validity of one of the conditions stated in (ii)--(v).
	
	It remains to show the failure of the embedding
	$$
		T^b_r F^s_{p, q}(\R^d) \hookrightarrow C(\R^d), \qquad s < \frac{d}{p}.
	$$
	Indeed, this is an immediate consequence of the fact that $T^b_r B^{s_0}_{p, q}(\R^d) \hookrightarrow T^b_r F^s_{p, q}(\R^d)$ for $s_0 > s$ (cf. Theorem \ref{ThmFJ2}) and Theorem  \ref{ThmC}.
 \end{proof}

 \begin{rem}
 	The proof above shows that the if-part of Theorem \ref{ThmCF} still remains true in the borderline case $b=0$; cf. Remark \ref{RemIfPart}.
 \end{rem}

 A special case of Theorem \ref{ThmCF} refers to embeddings of the Lipschitz spaces $\text{Lip}^{s, b}_{p, q}(\R^d) = F^{b+1/q}_q F^{s}_{p, 2}(\R^d)$ (cf. Proposition \ref{PropositionCoincidences}) into the space of continuous functions.

 \begin{cor}
 	Let $1 < p < \infty, 0 < q \leq \infty, s > 0$ and $b < -1/q$. Then
	$$
		\emph{Lip}^{s, b}_{p, q}(\R^d) \hookrightarrow C(\R^d) \iff s > \frac{d}{p}
	$$
	where $C(\R^d)$ can be replaced by $L_\infty(\R^d)$.
 \end{cor}

 In particular, the above result exhibits another striking difference between $\text{Lip}^{s, b}_{p, q}(\R^d)$ and the standard Besov--Triebel--Lizorkin scale $A^{s, b}_{p, q}(\R^d)$, namely, the tuning parameter $b$ enables to establish embeddings from $A^{d/p, b}_{p, q}(\R^d)$  into $C(\R^d)$ (cf. \eqref{ClassicBesovC} and \eqref{ClassicTLC}), but this never happens in the Lipschitz setting.

 \newpage
 \section{Characterizations and embeddings for general monotone functions}

 \subsection{Definition and basic properties of general monotone functions}

We recall the definition of the general monotone functions given in \cite{LiflyandTikhonov, Tikhonov}. A complex-valued function $\varphi (z), z
>0,$ is called \emph{general monotone}  if it is locally of bounded variation and for some  constant $C > 1$ the following is true
\begin{equation}\label{3.1}
    \int_z^{2z}  |d \varphi (u)| \leq C |\varphi (z)|
\end{equation}
for all $z > 0$. The set of all general monotone functions is denoted by $GM$\index{\bigskip\textbf{Sets}!$GM$}\label{GM}. Examples of general monotone functions are:
decreasing functions,
 quasi-monotone
functions $\varphi$ (i.e, $\varphi(t) t^{-\alpha}$ is non-increasing for some $\alpha \geq 0$), and increasing functions $\varphi$ such that
$ \varphi(2z) \lesssim \varphi(z)$. It is readily seen that (\ref{3.1}) implies
\begin{equation}\label{3.2}
    |\varphi (u)| \lesssim |\varphi (z)|\quad \text{ for any } \quad z \leq u \leq
    2z,
\end{equation}
which subsequently gives
\begin{equation}\label{3.3}
    |\varphi (z)| \lesssim \int_{z/c}^\infty \frac{|\varphi(u)|}{u}
    du, \qquad c>1.
\end{equation}

For later use we recall the following lemma on multipliers of general monotone functions (see \cite[Remark 5.5]{LiflyandTikhonov}).

\begin{lem}\label{Lemma 3.1}
    Let $\varphi, \alpha \in GM$, then $\varphi \alpha \in GM$.
\end{lem}
Now we are in a position to give the main definition in this section.
First, we recall that
the Fourier transform of a radial function $f(x) = f_0(|x|)$ is also radial, $\widehat{f}(\xi)=F_{0}(|\xi|)$
(see, e.g., \cite[Appendix B5]{Grafakos}) and it can be written as the Fourier--Hankel transform
\begin{equation}\label{FourierHankel}
F_{0}(s)= \frac{2 \pi^{d/2}}{\Gamma\left(\frac{d}{2}\right)} \int_{0}^{\infty}f_{0}(t)j_{d/2-1}(st)t^{d-1}\,dt,
\end{equation}
where $j_{\alpha}(t)=\Gamma(\alpha+1)(t/2)^{-\alpha}J_{\alpha}(t)$\index{\bigskip\textbf{Functionals and functions}!$j_\alpha$}\label{NORBESS} is the
normalized Bessel function ($j_{\alpha}(0)=1$), $\alpha\ge -1/2$, with $J_\alpha(t)$ the classical Bessel function\index{\bigskip\textbf{Functionals and functions}!$J_\alpha$}\label{BESS} of the first kind of order $\alpha$.


Let $\widehat{GM}^d$ \index{\bigskip\textbf{Sets}!$\widehat{GM}^d$}\label{GMF} be the collection of all radial functions $f(x) = f_0(|x|), \, x \in \mathbb{R}^d,$ such that the corresponding $F_0$ given by (\ref{FourierHankel}) belongs to the class $GM$, is positive and satisfies the condition
\begin{equation}\label{3.4new}
    \int_0^1 u^{d-1} F_0(u) du + \int_1^\infty u^{(d-1)/2} |d
    F_0(u)| < \infty;
\end{equation}
see \cite{GorbachevTikhonov}.
In other words,
 $\widehat{GM}^{d}$ consists of radial functions
$f(x)=f_{0}(|x|)$, $x\in \mathbb{R}^{d}$, which are defined in terms of the inverse Fourier--Hankel transform
\begin{equation}\label{3.4new+}
f_{0}(z)=\frac{2}{\Gamma\left(\frac{d}{2}\right) (2 \sqrt{\pi})^d}\int_{0}^{\infty}F_{0}(s)j_{d/2-1}(zs)s^{d-1}\,ds,
\end{equation}
where the function $F_{0}\in GM$ and satisfies  condition (\ref{3.4new}).
We note that, by  \cite[Lemma 1]{GorbachevLiflyandTikhonov}, the integral in \eqref{3.4new}
converges in the improper sense and therefore $f_{0}(z)$ is continuous for $z>0$.
 If $d=1$ we simply write $\widehat{GM}$.

In the discrete case, $\widehat{GM}$ coincides with the  well investigated class of $L_1(\mathbb{T})$-functions $f(x) \sim \sum_{n=1}^\infty (a_n \cos n x + b_n \sin nx)$ such that the sequences of their Fourier coefficients $\{a_n\}_{n \in \mathbb{N}}, \{b_n\}_{n \in \mathbb{N}} \in GM$, the \emph{discrete general monotone condition}, that is,
\begin{equation}\label{DiscreteGM}
 \sum_{k=n}^{2n-1} |\Delta d_k| \leq C |d_n|
 \end{equation}
 for all $n\in \mathbb{N}$ ($\Delta d_k := d_k - d_{k+1}$\index{\bigskip\textbf{Functionals and functions}!$\Delta$}\label{DIFFSEQ}) (cf. (\ref{3.1})); see \cite{Tikhonov, LiflyandTikhonov} and the references therein.
 In particular, the widely studied class (see, e.g., \cite[Chapters V, XII]{Zygmund}) of Fourier series with monotonic coefficients belongs to $\widehat{GM}$.


 \subsection{Characterization of spaces $T^b_r B^{s}_{p,q}(\mathbb{R}^d)$}\label{SubsectionContinuous}
Working with the $\widehat{GM}^d$ class, we provide an effective criterion which characterizes
functions from $T^b_r B^{s}_{p,q}(\mathbb{R}^d)$ in terms of their Fourier transform.

\begin{thm}\label{Theorem 3.2}
    Let $\frac{2d}{d+1} < p < \infty, 0 < q, r \leq \infty, s \in \R$, and $b \in \R \backslash \{0\}$. Let $f \in \widehat{GM}^d$. Then
    \begin{align}
          \|f\|_{T^b_r B^{s}_{p, q}(\R^d)}&\asymp   \bigg( \sum_{j=0}^\infty 2^{-j ( p -1) d} F_0^p(2^{-j}) \bigg)^{1/p} + \nonumber \\
          & \hspace{1cm}  \left(\sum_{j=0}^\infty 2^{j b r} \bigg(\sum_{\nu=2^j}^{2^{j+1}-1} 2^{\nu (s + d -d/p) q} F_0^q(2^\nu) \bigg)^{r/q}    \right)^{1/r} \label{GMB}
    \end{align}
    (with the usual modifications if $q=\infty$ and/or $r=\infty$).
\end{thm}

\begin{rem}\label{RemarkContinuous}
The analogue of Theorem \ref{Theorem 3.2} for the Besov spaces $B^{s, b}_{p, q}(\R^d)$ was obtained in \cite[Theorem 4.6]{DominguezTikhonov}: Let $\frac{2d}{d+1} < p < \infty, 0 < q \leq \infty$, and $- \infty < s, b < \infty$. Let $f \in \widehat{GM}^d$. Then
    \begin{equation*}
          \|f \|_{B^{s, b}_{p,q}(\mathbb{R}^d)} \asymp   \bigg( \sum_{j=0}^\infty 2^{-j ( p -1) d} F_0^p(2^{-j}) \bigg)^{1/p}+ \bigg(\sum_{j=0}^\infty 2^{j (s  + d  -d /p) q} (1 + j)^{b q} F_0^q(2^j) \bigg)^{1/q}.
    \end{equation*}
    This also follows by setting $r=q$ in \eqref{GMB} (see Proposition \ref{PropositionCoincidences}).
    On the other hand, dealing with $\mathbf{B}^{0, b}_{p, q}(\R^d)$, the following was obtained in \cite[Theorem 4.2]{DominguezTikhonov}: Let $\frac{2d}{d+1} < p < \infty, 0 < q \leq \infty$, and $b > -1/q$. Let $f \in \widehat{GM}^d$. Then
    \begin{align}
          \|f \|_{\mathbf{B}^{0, b}_{p,q}(\mathbb{R}^d)} &\asymp   \bigg( \sum_{j=0}^\infty 2^{-j ( p -1) d} F_0^p(2^{-j}) \bigg)^{1/p} \nonumber \\
          &\hspace{-.5cm} + \bigg(\sum_{j=0}^\infty 2^{j(b+1/q) q} \bigg( \sum_{\nu=2^j}^{2^{j+1}-1} 2^{\nu (d  -d /p) p}  F_0^p(2^\nu) \bigg)^{q/p} \bigg)^{1/q}.\label{GMBesovZero}
    \end{align}
    The limiting case $b=-1/q$ in $\mathbf{B}^{0, b}_{p,q}(\mathbb{R}^d)$ is also covered by \cite[Theorem 4.2]{DominguezTikhonov}, but the outcome does not coincide with the right-hand side of \eqref{GMBesovZero}.
   We refer to \cite{DominguezTikhonov} for further results on logarithmic Besov spaces and $\widehat{GM}^d$ functions.

\end{rem}

\begin{proof}[Proof of Theorem \ref{Theorem 3.2}]
      We shall use two important properties of function in $\widehat{GM}^d$. First, the following description of the modulus of smoothness (cf. \eqref{DefModuli}) in terms of the Fourier transform for functions in the $\widehat{GM}^d$ class (see \cite[Corollary 4.1 and
    (7.6)]{GorbachevTikhonov})
    \begin{equation}\label{GTNew}
        \omega^p_{k}(f,t)_p \asymp t^{ k p} \int_0^{1/t} u^{ k p + dp - d
        -1} F_0^p(u) du + \int_{1/t}^\infty u^{dp - d -1} F^p_0(u)
        du
    \end{equation}
    where
    \begin{equation}\label{Conditionsk}
    k \in \mathbb{N} \quad \text{if} \quad d=1 \qquad \text{and} \qquad k \quad \text{is even if}  \quad d \geq 2.
    \end{equation}
    Second, the Hardy-Littlewood-type estimate for functions $f \in \widehat{GM}^d$ obtained in \cite[Theorem 1]{GorbachevLiflyandTikhonov} (see also \cite[(4.10)]{GorbachevTikhonov}) which asserts that
\begin{equation}\label{HL}
	\|f\|_{L_p(\mathbb{R}^d)} \asymp \left(\int_0^\infty t^{d p - d - 1} F_0^p(t) dt\right)^{1/p},\qquad  \frac{2d}{d+1} < p < \infty.
\end{equation}
 For the one-dimensional case and monotone functions, see \cite{Tit, Boas, Sagher}.

  \textsc{The case $s > 0$ and $b > 0$:} Choose $k > s$ with \eqref{Conditionsk}. Under these assumptions, it was already shown in \eqref{jsajajsajs1} that
  \begin{align}
  \|f\|_{T^b_r B^s_{p, q}(\mathbb{R}^d)}  &\asymp \|f\|_{L_p(\R^d)} \nonumber \\
  &  \hspace{1cm}+ \left(\int_0^1 (1-\log t)^{(b-1/r) r} \bigg(\int_0^{t} (u^{-s} \omega_k(f,u)_p)^q \frac{du}{u}\bigg)^{r/q}
        \frac{dt}{t}\right)^{1/r}. \label{New1711}
  \end{align}

   We claim that
        \begin{align}
    	 \left(\int_0^1 (1-\log t)^{(b-1/r) r} \left(\int_0^t (u^{-s} \omega_k(f,u)_p)^q \frac{du}{u}\right)^{r/q}
        \frac{dt}{t}\right)^{1/r} & \asymp  \left( \int_0^{1} t^{k p + dp -d-1} F_0^p(t) \, dt\right)^{1/p} \nonumber  \\
        & \hspace{-7cm}+  \left(\int_1^\infty (1+\log t)^{(b-1/r) r} \left(\int_{t}^\infty u^{s q + d q -d q/p} F_0^q(u) \right)^{r/q}  \frac{dt}{t}\right)^{1/r}. \label{814}
    \end{align}
Indeed, by \eqref{GTNew} and a simple changes of variables, we get
        \begin{align}
         \left(\int_0^1 (1-\log t)^{(b-1/r) r} \left(\int_0^t (u^{-s} \omega_k(f,u)_p)^q \frac{du}{u}\right)^{r/q}
        \frac{dt}{t}\right)^{1/r} & \asymp \nonumber \\
        & \hspace{-8.75cm}  \left(\int_1^\infty (1+\log t)^{(b-1/r) r} \left(\int_{t}^\infty u^{-(k-s)q} \left( \int_0^{u} v^{k p + dp -d-1} F_0^p(v) \, dv\right)^{q/p} \frac{du}{u}\right)^{r/q}
        \frac{dt}{t}\right)^{1/r}  \nonumber  \\
        & \hspace{-8.5cm} +  \left(\int_1^\infty (1+\log t)^{(b-1/r) r} \left(\int_{t}^\infty u^{s q} \left( \int_{u}^\infty v^{ dp -d-1} F_0^p(v) \, dv\right)^{q/p} \frac{du}{u}\right)^{r/q}  \frac{dt}{t}\right)^{1/r} \nonumber   \\
           &  \hspace{-8.75cm} \asymp  \left( \int_0^{1} v^{k p + dp -d-1} F_0^p(v) \, dv\right)^{1/p}     \nonumber     \\
        &  \hspace{-8.5cm} +  \left(\int_1^\infty  t^{-(k-s) r}  (1+\log t)^{(b-1/r) r} \left( \int_1^{t} v^{k p + dp -d-1} F_0^p(v) \, dv\right)^{r/p}         \frac{dt}{t}\right)^{1/r}  \nonumber  \\
         & \hspace{-8.5cm}  +  \left(\int_1^\infty (1+\log t)^{(b-1/r) r} \left(\int_{t}^\infty u^{-(k-s)q} \left( \int_t^{u} v^{k p + dp -d-1} F_0^p(v) \, dv\right)^{q/p} \frac{du}{u}\right)^{r/q}
        \frac{dt}{t}\right)^{1/r}  \nonumber  \\
              & \hspace{-8.5cm} +  \left(\int_1^\infty (1+\log t)^{(b-1/r) r} \left(\int_{t}^\infty u^{s q} \left( \int_{u}^\infty v^{ dp -d-1} F_0^p(v) \, dv\right)^{q/p} \frac{du}{u}\right)^{r/q}  \frac{dt}{t}\right)^{1/r}  \nonumber   \\
              & \hspace{-8cm} =: I + II + III+ IV. \label{IIIIIIIv}
    \end{align}

    Next we show that
        \begin{equation}\label{ClaimII}
    	I + II \asymp I + \left(\int_1^\infty t^{s r + d r -d r /p} (1 + \log t)^{(b-1/r) r}  F_0^r(t) \frac{dt}{t} \right)^{1/r}
    \end{equation}
    and
       \begin{equation}\label{ClaimIIIIV}
    	I + III  + IV \asymp I+ \left(\int_1^\infty (1 + \log t)^{(b-1/r) r} \left(\int_t^\infty u^{s q + d q -d q/p} F_0^q(u) \frac{du}{u} \right)^{r/q} \frac{dt}{t} \right)^{1/r}.
    \end{equation}

    Concerning \eqref{ClaimII}, we make use of monotonicity property $F_0(x) \lesssim F_0(y)$ for $y \leq x \leq 2 y$ (cf. \eqref{3.2}) so that
    $$
     F_0(1) + II \asymp \left(\sum_{j=0}^\infty  2^{-j(k-s) r}  (1+ j)^{(b-1/r) r} \left( \sum_{l=0}^{j} 2^{l(k + d -d/p) p} F_0^p(2^l) \right)^{r/p}      \right)^{1/r}
    $$
    and thus Hardy's inequality \eqref{H1} implies
    $$
    F_0(1) + II \asymp  \left(\sum_{j=0}^\infty  2^{j(s + d -d/p) r}  (1+ j)^{(b-1/r) r} F_0^r(2^j)     \right)^{1/r},
    $$
    i.e., \eqref{ClaimII} holds. We note that the additional term $F_0(1)$ appears in the previous estimates in order to get equivalence constants independent of $F_0$.

%

     Similar ideas as above also work with \eqref{ClaimIIIIV}. Specifically, applying monotonicity properties, we have
    \begin{align*}
    	F_0(1) + III + IV & \asymp \left(\sum_{j=0}^\infty (1+j)^{(b-1/r) r} \left(\sum_{l=j}^\infty  2^{-l(k-s)q} \left(\sum_{\nu=j}^l  2^{\nu(k p + dp -d)} F_0^p(2^\nu)\right)^{q/p}\right)^{r/q} \right)^{1/r}  \\
              &  \hspace{1cm}+  \left(\sum_{j=0}^\infty (1+j)^{(b-1/r) r} \left(\sum_{l=j}^\infty 2^{l s q} \left( \sum_{\nu=l}^\infty 2^{ \nu (dp -d)} F_0^p(2^\nu) \right)^{q/p} \right)^{r/q} \right)^{1/r}.
    \end{align*}
    Therefore Hardy's inequalities \eqref{H1} and \eqref{H2} yield
    $$
    	F_0(1) + III+IV \asymp  \left(\sum_{j=0}^\infty (1+j)^{(b-1/r) r} \left( \sum_{\nu=j}^\infty  2^{\nu(s q+ d q -d q/p)} F_0^q(2^\nu) \right)^{r/q} \right)^{1/r},
    $$
    which gives the desired estimate \eqref{ClaimIIIIV}.

    Inserting \eqref{ClaimII} and \eqref{ClaimIIIIV} into \eqref{IIIIIIIv} and using basic monotonicity properties (cf. \eqref{3.2}), we obtain
    \begin{align*}
    	   \left(\int_0^1 (1-\log t)^{(b-1/r) r} \left(\int_0^t (u^{-s} \omega_k(f,u)_p)^q \frac{du}{u}\right)^{r/q}
        \frac{dt}{t}\right)^{1/r} & \asymp I  + \\
        & \hspace{-7cm}   \left(\int_1^\infty t^{s r + d r -d r /p} (1 + \log t)^{(b-1/r) r}  F_0^r(t) \frac{dt}{t} \right)^{1/r} \\
        & \hspace{-7cm} + \left(\int_1^\infty (1 + \log t)^{(b-1/r) r} \left(\int_t^\infty u^{s q + d q -d q/p} F_0^q(u) \frac{du}{u} \right)^{r/q} \frac{dt}{t} \right)^{1/r} \\
        & \hspace{-7.5cm} \asymp  I +\left(\int_1^\infty (1 + \log t)^{(b-1/r) r} \left(\int_t^\infty u^{s q + d q -d q/p} F_0^q(u) \frac{du}{u} \right)^{r/q} \frac{dt}{t} \right)^{1/r}.
    \end{align*}
    This proves \eqref{814}.

    In virtue of \eqref{New1711}, \eqref{HL} and \eqref{814}, we have
    \begin{align}
          \|f\|_{T^b_r B^{s}_{p, q}(\R^d)}\asymp  \left(\int_1^\infty t^{d p - d - 1} F_0^p(t) \, dt\right)^{1/p} +    \left( \int_0^{1} t^{d p -d-1} F_0^p(t) \, dt\right)^{1/p} \nonumber  \\
        & \hspace{-8cm}+  \left(\int_1^\infty (1+\log t)^{(b-1/r) r} \left(\int_{t}^\infty u^{s q + d q -d q/p} F_0^q(u) \frac{du}{u} \right)^{r/q}  \frac{dt}{t}\right)^{1/r}. \label{1717new}
    \end{align}
    Furthermore
    \begin{equation}\label{LpBGMPositive}
    	  \left(\int_1^\infty t^{d p - d - 1} F_0^p(t) \, dt\right)^{1/p} \lesssim  \left(\int_1^\infty (1+\log t)^{(b-1/r) r} \left(\int_{t}^\infty u^{s q + d q -d q/p} F_0^q(u) \frac{du}{u} \right)^{r/q}  \frac{dt}{t}\right)^{1/r}.
    \end{equation}
    This is a simple consequence of monotonicity properties (cf. \eqref{3.2}), H\"older inequality (if $r \geq p$) and the fact that $\ell_r \hookrightarrow \ell_p$ with $r < p$. Further details are left to the reader. As a combination of \eqref{1717new} and \eqref{LpBGMPositive}, we find
    \begin{align}
      \|f\|_{T^b_r B^{s}_{p, q}(\R^d)}&\asymp \left( \int_0^{1} t^{d p -d-1} F_0^p(t) \, dt\right)^{1/p} \nonumber\\
       & \hspace{4mm}+  \left(\int_1^\infty (1+\log t)^{(b-1/r) r} \left(\int_{t}^\infty u^{s q + d q -d q/p} F_0^q(u) \frac{du}{u} \right)^{r/q}  \frac{dt}{t}\right)^{1/r}.  \label{1717new2}
    \end{align}

      Let $\mu_j = 2^{2^j}, \, j \geq 0$. Applying Hardy's inequality \eqref{H2} (recall that $b > 0$)
    \begin{align}
    F_0(1) + \left(\int_1^\infty (1+\log t)^{(b-1/r) r} \left(\int_{t}^\infty u^{s q + d q -d q/p} F_0^q(u) \frac{du}{u} \right)^{r/q}  \frac{dt}{t}\right)^{1/r} & \asymp \nonumber \\
     &\hspace{-9cm}   F_0(1) + \left(\sum_{j=0}^\infty 2^{j b r} \left(\int_{\mu_j}^\infty u^{s q + d q -d q/p} F_0^q(u) \frac{du}{u} \right)^{r/q}    \right)^{1/r} \nonumber \\
     & \hspace{-9cm}  \asymp   F_0(1) +  \left(\sum_{j=0}^\infty 2^{j b r} \left( \int_{\mu_j}^{\mu_{j+1}-1} u^{s q + d q -d q/p} F_0^q(u) \frac{du}{u} \right)^{r/q}    \right)^{1/r} \nonumber \\
     & \hspace{-9cm}  \asymp  \left(\sum_{j=0}^\infty 2^{j b r} \left(\sum_{\nu= 2^j}^{2^{j+1}-1} 2^{\nu(s+d-d/p) q} F_0^q(2^\nu)  \right)^{r/q}    \right)^{1/r}.   \label{1717new3}
    \end{align}
    On the other hand, elementary monotonicity properties yield
    \begin{equation}\label{1717new4}
    \left( \int_0^{1} t^{d p -d-1} F_0^p(t) \, dt\right)^{1/p}  \asymp \bigg( \sum_{j=0}^\infty 2^{-j ( p -1) d} F_0^p(2^{-j}) \bigg)^{1/p}.
    \end{equation}

    Putting together \eqref{1717new2}--\eqref{1717new4}, the desired estimate \eqref{GMB} is achieved.

\textsc{The case $s > 0$ and $b < 0$:} The proof follows similar lines as the case $s > 0$ and $b > 0$, but now relying on the characterization of $T^b_r B^s_{p, q}(\R^d)$ with $b < 0$ provided by \eqref{jsajajsajs2}, namely,
  \begin{align*}
  \|f\|_{T^b_r B^s_{p, q}(\mathbb{R}^d)}  &\asymp \|f\|_{L_p(\R^d)} \nonumber \\
  &  \hspace{1cm}+ \left(\int_0^1 (1-\log t)^{(b-1/r) r} \bigg(\int_{t}^1 (u^{-s} \omega_k(f,u)_p)^q \frac{du}{u}\bigg)^{r/q}
        \frac{dt}{t}\right)^{1/r}.
  \end{align*}
  In this case, following similar arguments as in \eqref{1717new} and \eqref{1717new3} (but using now \eqref{H1}), one can prove
    \begin{align}
    	  \left(\int_1^\infty t^{d p - d - 1} F_0^p(t) \, dt\right)^{1/p} &\lesssim  \left(\int_1^\infty (1+\log t)^{(b-1/r) r} \left(\int_{1}^t u^{s q + d q -d q/p} F_0^q(u) \frac{du}{u} \right)^{r/q}  \frac{dt}{t}\right)^{1/r} \nonumber \\
	  & \lesssim   \left(\sum_{j=0}^\infty 2^{j b r} \left(\sum_{\nu= 2^j}^{2^{j+1}-1} 2^{\nu(s+d-d/p) q} F_0^q(2^\nu)  \right)^{r/q}    \right)^{1/r}. \label{bcbsb}
    \end{align}
 Further details are left to the reader.

 \textsc{The case $s \leq 0$:} We choose $\sigma$ such that $\sigma < s$.  Notice that $I_\sigma f \in \widehat{GM}^d$ (cf. \eqref{LiftingDef} and Lemma \ref{Lemma 3.1}). According to Theorem \ref{TheoremLifting} and the previous case, we have
 \begin{align*}
 	\|f\|_{T^b_r B^s_{p, q}(\R^d)} &\asymp \|I_\sigma f\|_{T^b_r B^{s-\sigma}_{p, q}(\R^d)} \\
	& \hspace{-1cm}\asymp  \bigg( \sum_{j=0}^\infty 2^{-j (p -1)d} (1 + 2^{-2 j})^{\sigma p/2} F_0^p(2^{-j}) \bigg)^{1/p} \\
	& \hspace{-0.5cm}+  \left(\sum_{j=0}^\infty 2^{j b r} \bigg(\sum_{\nu=2^j}^{2^{j+1}-1} 2^{\nu (s-\sigma + d -d/p) q} (1 + 2^{2 \nu})^{\sigma q/2} F_0^q(2^\nu) \bigg)^{r/q}    \right)^{1/r}  \\
	& \hspace{-1cm}\asymp  \bigg( \sum_{j=0}^\infty 2^{-j (p -1)d} F_0^p(2^{-j}) \bigg)^{1/p} \\
	& \hspace{-0.5cm}+   \left(\sum_{j=0}^\infty 2^{j b r} \bigg(\sum_{\nu=2^j}^{2^{j+1}-1} 2^{\nu (s + d -d/p) q} F_0^q(2^\nu) \bigg)^{r/q}    \right)^{1/r}.
 \end{align*}


\end{proof}

\begin{rem}\label{Remark 3.3}
    Let us consider a slightly wider  class (cf. \cite{GorbachevLiflyandTikhonov, LiflyandTikhonov}) of general monotone functions $\varphi(z)$  which are locally of bounded variation, vanishes at
    infinity, and there is a constant $c > 1$ depending on $\varphi$
    such that
    \begin{equation}\label{3.12}
        \int_z^\infty |d \varphi(u)| \lesssim \int_{z/c}^\infty
        \frac{|\varphi(u)|}{u} du < \infty,\qquad  z >0.
    \end{equation}
Noting that such  functions  satisfy the monotonicity condition (\ref{3.3}), we claim that  Theorem \ref{Theorem 3.2}  holds in the case $\min\{q, r\} \geq p$
 if we consider the  class formed
    by all
    radial functions $f(x) = f_0(|x|), x \in \mathbb{R}^d,$ with the Fourier transform $F_0 = \widehat{f} \geq 0$
    satisfying
    conditions (\ref{3.12}) and (\ref{3.4new}).
\end{rem}

 \subsection{Embeddings of $T^b_r B^s_{p, q}(\R^d)$ into $L_p(\R^d)$}\label{Section83}
 In Theorem \ref{TheoremRegular2} (see also Corollary \ref{CorollaryRegularBMO}) a complete characterization of the embedding
 \begin{equation}\label{BLpGen}
 	T^b_r B^s_{p, q}(\R^d) \hookrightarrow L_p(\R^d)
 \end{equation}
 in terms of the involved parameters was provided. However, from the point of view of applications, it is of considerable interest to deal with sufficiently rich classes of functions (e.g. monotone-type functions) which enable to sharpen the general statement provided by \eqref{BLpGen}. Accordingly, in this section we prove that Theorem \ref{TheoremRegular2} can be sharpened if we restrict ourselves to work with monotone-type functions. Specifically, the following result establishes necessary and sufficient conditions for which the embedding
 \begin{equation}\label{BLpMon}
 	T^b_r B^s_{p, q}(\R^d) \cap \widehat{GM}^d \hookrightarrow L_p(\R^d)
 \end{equation}
 holds and, in particular, it shows that the corresponding range of parameters is bigger than those related to \eqref{BLpGen}.

  \begin{thm}\label{ThmBLpGM}
 	Let $\frac{2d}{d+1} < p < \infty, 0 < q < \infty, 0 < r \leq \infty, s \in \R$, and $ b \in \R \backslash \{0\}$. Then \eqref{BLpMon} holds if and only if one of the following conditions is satisfied
	\begin{enumerate}[\upshape(i)]
	\item $s > 0$,
	\item $s=0, \qquad p < q,  \qquad 0 < r \leq \infty, \qquad b>  \frac{1}{p} - \frac{1}{q}$,
	\item $s= 0, \qquad p < q, \qquad  r \leq p, \qquad b = \frac{1}{p} - \frac{1}{q}$,
\item $s=0, \qquad q \leq p, \qquad 0 < r \leq \infty, \qquad b > 0$.
	 \end{enumerate}
 \end{thm}

 \begin{rem}
 	The previous result shows a bigger range of parameters for the validity of \eqref{BLpMon} in comparison with the general embedding \eqref{BLpGen} treated in Theorem \ref{TheoremRegular2}. For instance, under the assumptions (ix) in Theorem \ref{TheoremRegular2}, the embedding \eqref{BLpGen} requires $b \geq 1/2-1/q$ but \eqref{BLpMon} holds under the weaker condition $b \geq 1/p-1/q$ if $p < q$ and any $b > 0$ if $p \geq q$. Similar improvements can also be obtained under (x) in Theorem  \ref{TheoremRegular2}.
 \end{rem}

 \begin{proof}[Proof of Theorem \ref{ThmBLpGM}]
 \textsc{Sufficiency part:} Applying Theorem \ref{Theorem 3.2} and \eqref{HL}, the inequality
 \begin{equation}\label{ProofThmBLpGM1}
 	\|f\|_{L_p(\R^d)} \lesssim \|f\|_{T^b_r B^{s}_{p, q}(\R^d)}, \qquad f \in  \widehat{GM}^d,
 \end{equation}
 turns out to be equivalent
 \begin{align}
 	\bigg(\sum_{j=0}^\infty 2^{-j (p-1) d} F_0^p(2^{-j})\bigg)^{1/p} +  \bigg(\sum_{j=0}^\infty 2^{j (p-1) d} F_0^p(2^{j})\bigg)^{1/p} &\lesssim \nonumber \\
	 & \hspace{-8cm}    \bigg( \sum_{j=0}^\infty 2^{-j ( p -1) d} F_0^p(2^{-j}) \bigg)^{1/p} +
         \left(\sum_{j=0}^\infty 2^{j b r} \bigg(\sum_{\nu=2^j}^{2^{j+1}-1} 2^{\nu (s + d -d/p) q} F_0^q(2^\nu) \bigg)^{r/q}    \right)^{1/r} \label{ProofThmBLpGM1*}
 \end{align}
 where $F_0$ is given by \eqref{FourierHankel}. In particular, \eqref{ProofThmBLpGM1} will be satisfied whenever
 \begin{equation}\label{ProofThmBLpGM1*1}
  \bigg(\sum_{j=0}^\infty 2^{j (p-1) d} F_0^p(2^{j})\bigg)^{1/p} \lesssim    \left(\sum_{j=0}^\infty 2^{j b r} \bigg(\sum_{\nu=2^j}^{2^{j+1}-1} 2^{\nu (s + d -d/p) q} F_0^q(2^\nu) \bigg)^{r/q}    \right)^{1/r}.
 \end{equation}

Assume $s > 0$ (i.e., (i) holds). Then \eqref{ProofThmBLpGM1*1} follows immediately from \eqref{LpBGMPositive}, \eqref{1717new3} and \eqref{bcbsb}.

Next we investigate the case $s=0$ in \eqref{ProofThmBLpGM1*1}. It is clear that
 \begin{equation}\label{ProofThmBLpGM1*2}
 	\bigg(\sum_{j=0}^\infty 2^{j (p-1) d} F_0^p(2^{j})\bigg)^{1/p} = \bigg(\sum_{j=0}^\infty \sum_{\nu=2^j}^{2^{j+1}-1} 2^{\nu (1-1/p) d p} F_0^p(2^{\nu})  \bigg)^{1/p}.
 \end{equation}
 Assume $q \leq p$ (i.e., (iv) holds), then
 \begin{equation}\label{ProofThmBLpGM1*3}
 	\bigg(\sum_{j=0}^\infty 2^{j (p-1) d} F_0^p(2^{j})\bigg)^{1/p}  \leq \bigg(\sum_{j=0}^\infty \bigg( \sum_{\nu=2^j}^{2^{j+1}-1} 2^{\nu (1-1/p) d q} F_0^q(2^{\nu}) \bigg)^{p/q}  \bigg)^{1/p}.
 \end{equation}
 Therefore $r \leq p$ and $b \geq 0$ imply
 $$
 		\bigg(\sum_{j=0}^\infty 2^{j (p-1) d} F_0^p(2^{j})\bigg)^{1/p}  \leq \bigg(\sum_{j=0}^\infty 2^{j b r}\bigg( \sum_{\nu=2^j}^{2^{j+1}-1} 2^{\nu (1-1/p) d q} F_0^q(2^{\nu}) \bigg)^{r/q}  \bigg)^{1/r}.
 $$
 On the other hand, if $r > p$ and $b > 0$, applying H\"older's inequality in \eqref{ProofThmBLpGM1*3}, the previous estimate (modulo replacing $\leq$ by $\lesssim$) is also achieved. Hence \eqref{ProofThmBLpGM1*1}  holds.

 Suppose now $q > p$ (i.e., the cases (ii) and (iii)). Using H\"older's inequality in \eqref{ProofThmBLpGM1*2}, we have
 $$
 	\bigg(\sum_{j=0}^\infty 2^{j (p-1) d} F_0^p(2^{j})\bigg)^{1/p} \lesssim \bigg(\sum_{j=0}^\infty 2^{j (1/p-1/q) p} \bigg(\sum_{\nu=2^j}^{2^{j+1}-1} 2^{\nu (1-1/p) d q} F_0^q(2^{\nu}) \bigg)^{p/q}   \bigg)^{1/p}.
 $$
 There are two possibilities. Firstly, if $r \leq p$ and $b \geq 1/p -1/q$ then
 $$
 		\bigg(\sum_{j=0}^\infty 2^{j (p-1) d} F_0^p(2^{j})\bigg)^{1/p} \lesssim \bigg(\sum_{j=0}^\infty 2^{j b r} \bigg(\sum_{\nu=2^j}^{2^{j+1}-1} 2^{\nu (1-1/p) d q} F_0^q(2^{\nu}) \bigg)^{r/q}   \bigg)^{1/r}.
 $$
 Secondly, the range $r > p$ and $b > 1/p-1/q$ follows from H\"older's inequality.


 \textsc{Neccessary part:} Next we construct counterexamples showing that \eqref{ProofThmBLpGM1} (or equivalently, \eqref{ProofThmBLpGM1*}) becomes false if none of the conditions (i)--(v) is satisfied.

 \textsc{Necessary condition $s \geq 0$ in \eqref{ProofThmBLpGM1}:} For $s < 0$, we let
	\begin{equation*}
		F_0(t) = \left\{\begin{array}{cl}   1, & \quad \text{if} \quad t \in (0, 1],  \\
		t^{-d+d/p-\varepsilon}, &\quad \text{if} \quad t \in (1, \infty),
		       \end{array}
                        \right.
	\end{equation*}
	where $s < \varepsilon < 0$ and
	\begin{equation}\label{ExtremalFunctionGMLp}
	f(x) = f_0(|x|), \quad x \in \R^d, \quad \text{with} \quad f_0 \quad \text{given by \eqref{3.4new+}}.
	\end{equation}
	 Note that $f \in \widehat{GM}^{d}$ since $F_0 \in GM$. In light of Theorem \ref{Theorem 3.2}, we have
	    \begin{align*}
          \|f \|_{T^b_r B^{s}_{p,q}(\mathbb{R}^d)} &\asymp   \bigg( \sum_{j=0}^\infty 2^{-j ( p -1) d} \bigg)^{1/p} +   \left(\sum_{j=0}^\infty 2^{j b r} \bigg(\sum_{\nu=2^j}^{2^{j+1}-1} 2^{\nu (s  -\varepsilon) q}  \bigg)^{r/q}    \right)^{1/r}    \\
          & \asymp  \bigg( \sum_{j=0}^\infty 2^{-j ( p -1) d} \bigg)^{1/p}  + \left(\sum_{j=0}^\infty 2^{j b r} 2^{2^j (s-\varepsilon) r}   \right)^{1/r} < \infty
    \end{align*}
    and, on the other hand, by \eqref{HL},
    $$
    	\|f\|_{L_p(\mathbb{R}^d)} \gtrsim  \left(\int_1^\infty t^{-\varepsilon p - 1} dt\right)^{1/p} = \infty.
    $$

    \textsc{Necessary condition $b \geq 1/p -1/q$ in  \eqref{ProofThmBLpGM1} with $s=0$ and $p < q$:} If $b <  1/p -1/q$, we let
    \begin{equation*}
		F_0(t) = \left\{\begin{array}{cl}   1, & \quad \text{if} \quad t \in (0, 1],  \\
		t^{-d+d/p} (1 + \log t)^{-\varepsilon}, &\quad \text{if} \quad t \in (1, \infty),
		       \end{array}
                        \right.
	\end{equation*}
	where $\max \{b + 1/q, 1/q\}  < \varepsilon < 1/p$  and the corresponding $f$ defined by \eqref{ExtremalFunctionGMLp}. According to Theorem \ref{Theorem 3.2} and \eqref{HL}, we derive
	    \begin{align*}
          \|f \|_{T^b_r B^{0}_{p,q}(\mathbb{R}^d)} &\asymp  \bigg( \sum_{j=0}^\infty 2^{-j ( p -1) d} \bigg)^{1/p} +   \left(\sum_{j=0}^\infty 2^{j b r} \bigg(\sum_{\nu=2^j}^{2^{j+1}-1} \nu^{-\varepsilon q} \bigg)^{r/q}    \right)^{1/r}    \\
          & \asymp    \bigg( \sum_{j=0}^\infty 2^{-j ( p -1) d} \bigg)^{1/p}  + \left(\sum_{j=0}^\infty 2^{j (b -\varepsilon + 1/q) r}    \right)^{1/r}  < \infty
    \end{align*}
    and
    \begin{equation*}
    		\|f\|_{L_p(\mathbb{R}^d)} \gtrsim  \left(\int_1^\infty (1 + \log t)^{-\varepsilon p} \frac{dt}{t}\right)^{1/p} = \infty,
    \end{equation*}
    respectively.

        \textsc{Necessary condition $b > 1/p -1/q$ in  \eqref{ProofThmBLpGM1} with $s=0$ and $p < \min\{q, r\}$:} According to the previous case, it only remains to deal with the limiting value $b = 1/p -1/q$. This can be done in a similar fashion as above by taking
    \begin{equation*}
		F_0(t) = \left\{\begin{array}{cl}   1, & \quad \text{if} \quad t \in (0, 1],  \\
		t^{-d+d/p} (1 + \log t)^{-1/p} (1 + \log (1 + \log t))^{-\varepsilon}, &\quad \text{if} \quad t \in (1, \infty),
		       \end{array}
                        \right.
	\end{equation*}
	where $1/r < \varepsilon < 1/p$.
	
	\textsc{Necessary condition $b \geq 0$ in  in  \eqref{ProofThmBLpGM1} with $s=0$:} Evaluating \eqref{ProofThmBLpGM1*} (with $s=0$) for the decreasing function $F_0$ given by the characteristic function $\chi_{(0, 2^{2^N})}, \, N \geq 2$, we have
	$$
	L + 2^{2^N d (1-1/p)}  \lesssim L +    2^{2^N d (1-1/p)} 2^{N b}
	$$
	where $L^p = \sum_{j=0}^\infty 2^{-j (p-1) d} < \infty$. Hence
	$$
	 2^{-2^N d (1-1/p)} L +1  \lesssim 2^{-2^N d (1-1/p)}  L +   2^{N b}, \qquad N \geq 2,
	$$
	and, taking limits as $N \to \infty$, we conclude that necessarily $b \geq 0$.

 \end{proof}

 Specializing Theorem \ref{ThmBLpGM} with $q=r$ we obtain the following result for classical Besov spaces.

 \begin{cor}
  	Let $\frac{2d}{d+1} < p < \infty, 0 < q < \infty, s \in \R$ and $b \in \R \backslash \{0\}$. Then
	$$
	B^{s, b}_{p, q}(\R^d) \cap \widehat{GM}^d \hookrightarrow L_p(\R^d)
	$$
	 holds if and only if one of the following conditions is satisfied
	\begin{enumerate}[\upshape(i)]
	\item $s > 0$,
	\item $s=0, \qquad p < q, \qquad b>  \frac{1}{p} - \frac{1}{q}$,
	\item $s=0, \qquad p \geq q, \qquad b  > 0$.
	 \end{enumerate}
 \end{cor}

 An alternative approach to the previous result was proposed in \cite[Theorem 4.10]{DominguezTikhonov}. There the additional restriction $b \neq 0$ (which naturally arises dealing with truncated function spaces) does not appear.

 \subsection{Embeddings between $T^b_r B^s_{p, q}(\R^d)$}\label{Section84}

 The full characterization of the embeddings
 \begin{equation}\label{EmbeddingBBNew}
	T^{b_0}_{r_0} B^{s_0}_{p_0,q_0}(\mathbb{R}^d) \hookrightarrow T^{b_1}_{r_1} B^{s_1}_{p_1,q_1}(\mathbb{R}^d)
\end{equation}
 was obtained in Theorem \ref{TheoremEmbeddingsBBCharacterization}. Next we show that, working with general monotone functions, $\hookrightarrow$ can be sharpened by $=$ in \eqref{EmbeddingBBNew}. Before we state the precise statement, we introduce the following notation: for $f \in \widehat{GM}^d, \, f(x) = f_0(|x|)$, and any $\tau > 0$, set
 $$
 	J_f (\tau) := \bigg(\int_0^1 t^{d \tau -d-1} F_0^\tau(t) \, dt  \bigg)^{1/\tau},
 $$
 where $F_0$ is defined by \eqref{FourierHankel}. Notice that, by \eqref{3.2},
 \begin{equation}\label{Jf}
 	J_f(\tau) \lesssim J_f(\kappa) \quad \text{for} \quad \kappa < \tau.
 \end{equation}

 \begin{thm}\label{ThmEquivBGM1Statement}
 	Let $\frac{2 d}{d+1} < p_0 < p_1 < \infty, 0 < q, r \leq \infty,$ and $s_0, s_1, b\in \R$.
	\begin{enumerate}[\upshape(i)]
	\item Assume that
	\begin{equation}\label{AssumptionsBCoincidence}
		s_0 - \frac{d}{p_0} = s_1 - \frac{d}{p_1}.
	\end{equation}
	Let $f \in \widehat{GM}^d$. If $J_f(p_0) < \infty$ then
	\begin{equation}\label{ThmEquivBGM1}
		f \in T^b_r B^{s_0}_{p_0,q}(\mathbb{R}^d) \iff f \in T^b_r B^{s_1}_{p_1,q}(\mathbb{R}^d).
	\end{equation}
	\item	Conversely, if
	$$
		\widehat{GM}^d \cap T^b_r B^{s_0}_{p_0,q}(\mathbb{R}^d) = \widehat{GM}^d \cap T^b_{r} B^{s_1}_{p_1,q}(\mathbb{R}^d)
	$$
	then \eqref{AssumptionsBCoincidence} holds.
	\end{enumerate}
 \end{thm}

 \begin{proof}
 	(i): Under the assumption given in \eqref{AssumptionsBCoincidence}, we obviously have
	$$
	 \int_{t}^\infty u^{s_0 q + d q -d q/p_0 -1} F_0^{q}(u) \, du  =    \int_{t}^\infty u^{s_1 q + d q -d q/p_1 -1} F_0^{q}(u) \, du
	$$
	and thus \eqref{ThmEquivBGM1} follows immediately from Theorem \ref{Theorem 3.2} and \eqref{Jf}.
	
	(ii): We shall proceed by contradiction, i.e., if we assume that \eqref{AssumptionsBCoincidence} is false then $	\widehat{GM}^d \cap T^b_r B^{s_0}_{p_0,q}(\mathbb{R}^d) \neq \widehat{GM}^d \cap T^b_r B^{s_1}_{p_1,q}(\mathbb{R}^d)$. Indeed, suppose $s_0 - d/p_0 > s_1 - d/p_1$ and let
	\begin{equation*}
		F_0(t) = \left\{\begin{array}{cl}   1, & \quad \text{if} \quad t \in (0, 1],  \\
		t^{-s_1 - d + d/p_1 -\varepsilon}, &\quad \text{if} \quad t \in (1, \infty),
		       \end{array}
                        \right.
	\end{equation*}
	where $0 < \varepsilon < s_0 -\frac{d}{p_0}-s_1+ \frac{d}{p_1}$ and $f$ defined by \eqref{ExtremalFunctionGMLp}. Let
	$$
	I_i = 	  \left(\int_1^\infty (1+\log t)^{b r} \left(\int_{t}^\infty u^{s_i q + d q -d q/p_i -1} F_0^q(u) \, du \right)^{r/q}  \frac{dt}{t}\right)^{1/r} \quad \text{for} \quad i=0, 1.
	$$
	Elementary computations show that $I_1 < \infty$ but $I_0 = \infty$. Applying Theorem \ref{Theorem 3.2} we arrive at the desired contradiction. The case $s_0 - d/p_0 < s_1 - d/p_1$ can be treated analogously.
	
 \end{proof}

 \begin{rem}
 	The assumption $J_f(p_0) < \infty$ in Theorem \ref{ThmEquivBGM1Statement}(i) is necessary. To see this we consider the function $f$ with \eqref{ExtremalFunctionGMLp} and
	  \begin{equation*}
		F_0(t) = \left\{\begin{array}{cl}   t^{-\beta}, & \quad \text{if} \quad t \in (0, 1],  \\
		t^{-s_1-d + d/p_1 -\varepsilon}, &\quad \text{if} \quad t \in (1, \infty),
		       \end{array}
                        \right.
	\end{equation*}
	where $\varepsilon > 0$ and $d- \frac{d}{p_0} < \beta < d - \frac{d}{p_1}$. By Theorem \ref{Theorem 3.2}, we have $f \in \widehat{GM}^d \cap T^b_r B^{s_1}_{p_1,q}(\mathbb{R}^d)$ but $f \not \in T^b_r B^{s_0}_{p_0,q}(\mathbb{R}^d)$.
 \end{rem}

 \subsection{Characterizations of spaces $T^b_r B^{s}_{p, q}(\mathbb{T})$.} The analogue of Theorem \ref{Theorem 3.2} for periodic functions and $GM$ sequences (cf. \eqref{DiscreteGM}) reads as follows.

 \begin{thm}\label{ThmBesovGMPer}
 	Let $1 < p < \infty, 0 < q, r \leq \infty, s \in \R$ and $b \in \R \backslash \{0\}$. Let $f \in L_1(\mathbb{T})$ be such that
	$$f(x) \sim \sum_{n=1}^\infty (a_n \cos n x + b_n \sin nx)$$
	 with $\{a_n\}_{n \in \N}$ and $\{b_n\}_{n \in \N}$ nonnegative general monotone sequences.
	Then
	$$
		\|f\|_{T^b_r B^{s}_{p, q}(\mathbb{T})} \asymp  \left( \sum_{k=0}^\infty 2^{k b r} \bigg(\sum_{\nu= 2^k-1}^{2^{k+1}-2} 2^{\nu (s + 1-1/p) q} (a_{2^\nu} + b_{2^\nu})^q \bigg)^{r/q}  \right)^{1/r}.
	$$
 \end{thm}

 \begin{proof}
 	We may assume without loss of generality that $f(x) \sim \sum_{n=1}^\infty a_n \cos n x$ with nonnegative $\{a_n\}_{n \in \N} \in GM$.
	
	By Corollary \ref{CorollaryLizorkinRep},
		\begin{equation}\label{ThmBesovGMPer1}
		\|f\|_{T^b_r B^{s}_{p, q}(\mathbb{T})} \asymp  \left( \sum_{k=0}^\infty 2^{k b r} \bigg(\sum_{\nu= 2^k-1}^{2^{k+1}-2} 2^{\nu s q} \bigg\|\sum_{m = 2^\nu}^{2^{\nu+1}-1} a_m \cos mx \bigg\|_{L_p(\mathbb{T})}^q \bigg)^{r/q}  \right)^{1/r}.
	\end{equation}
	Furthermore, since $\{a_n\}_{n \in \N} \in GM$, the following estimates hold
	\begin{equation}\label{ThmBesovGMPer2}
		a_{2^{\nu+1}} 2^{\nu(1-1/p)} \lesssim \bigg\|\sum_{m = 2^\nu}^{2^{\nu+1}-1} a_m \cos mx \bigg\|_{L_p(\mathbb{T})} \lesssim a_{2^\nu} 2^{\nu(1-1/p)},
	\end{equation}
	cf. \cite{Tikhonov}. Combining \eqref{ThmBesovGMPer1} and \eqref{ThmBesovGMPer2} we achieve the desired result.
 \end{proof}

 \begin{rem}
 	Relying on Theorem \ref{ThmBesovGMPer}, all the results given in Sections \ref{Section83} and \ref{Section84} admit counterparts in the periodic setting. Further details are left to the interested reader.
 \end{rem}

 \subsection{Characterizations of spaces $\text{Lip}^{s, b}_{p, q}(\R^d)$}
 The goal of this section is to establish the characterization for Lipschitz spaces.

\begin{thm}[\bf{Characterization of Lipschitz norms for $\widehat{GM}^{d}$ functions}]\label{TheoremGMLip}
	Let $\frac{2d}{d+1} < p < \infty, 0 < q \leq \infty, s > 0$, and $b < - 1/q$ . Assume that $f \in \widehat{GM}^d$. Then
	\begin{align}
		\|f\|_{\emph{Lip}^{s,b}_{p,q}(\mathbb{R}^d)} & \asymp  \bigg( \sum_{j=0}^\infty 2^{-j ( p -1) d} F_0^p(2^{-j}) \bigg)^{1/p} \nonumber \\
		&\hspace{1cm}+ \left(\sum_{j=0}^\infty 2^{j (b+1/q) q} \bigg(\sum_{\nu=2^j}^{2^{j+1}-1} 2^{\nu (s + d -d/p) p} F_0^p(2^\nu) \bigg)^{q/p} \right)^{1/q}. \label{LipGM}
	\end{align}
\end{thm}

\begin{rem}
(i) If the right-hand side of (\ref{LipGM}) is finite then $f \in L_p(\R^d)$, i.e.,
\begin{equation}\label{LpEstim1Rem}
	\sum_{j=-\infty}^\infty 2^{-j ( p -1) d} F_0^p(2^{-j})  < \infty
\end{equation}
(see \eqref{HL}). Indeed, invoking H\"older's inequality if $q \geq p$ and the fact that $\ell_q \hookrightarrow \ell_p$ if $q < p$, we can estimate, for $s > 0$ and $b \in \R$,
\begin{equation}\label{LpEstim1Rem2}
	\bigg(\sum_{j=0}^\infty 2^{j ( p -1) d} F_0^p(2^{j})  \bigg)^{1/p} \lesssim \bigg(\sum_{j=0}^\infty 2^{j(s+d-d/p) q} (1 + j)^{b q} F_0^q(2^j) \bigg)^{1/q}.
\end{equation}
Furthermore, basic monotonicity properties and Hardy's inequality \eqref{H2} yield
\begin{align*}
	\bigg(\sum_{j=0}^\infty 2^{j(s+d-d/p) q} (1 + j)^{b q} F_0^q(2^j) \bigg)^{1/q} & \lesssim \bigg(\sum_{j=0}^\infty 2^{j (b +1/q)q} \bigg(\sum_{\nu=1}^{2^j} 2^{\nu(s+d-d/p) p} F_0^p(2^\nu) \bigg)^{q/p} \bigg)^{1/q} \\
	&\hspace{-5cm} \asymp  \left(\sum_{j=0}^\infty 2^{j (b+1/q) q} \bigg(\sum_{\nu=2^j}^{2^{j+1}-1} 2^{\nu (s + d -d/p) p} F_0^p(2^\nu) \bigg)^{q/p} \right)^{1/q}.
\end{align*}
Combining this estimate and \eqref{LpEstim1Rem2},
\begin{equation}\label{ñlñlñ}
\bigg(\sum_{j=0}^\infty 2^{j ( p -1) d} F_0^p(2^{j})  \bigg)^{1/p} \lesssim \left(\sum_{j=0}^\infty 2^{j (b+1/q) q} \bigg(\sum_{\nu=2^j}^{2^{j+1}-1} 2^{\nu (s + d -d/p) p} F_0^p(2^\nu) \bigg)^{q/p} \right)^{1/q}.
\end{equation}
 Then the convergence of \eqref{LpEstim1Rem} results as a combination of \eqref{ñlñlñ} and  (\ref{LipGM}).

	(ii) The borderline case $b=0$ and $q=\infty$ in the previous theorem refers to the Sobolev space $H^s_p(\R^d) = \L^{s,0}_{p,\infty}(\R^d)$ (see \eqref{LipSobFract}). This case was already investigated in \cite[Theorem 4.8]{DominguezTikhonov}, namely,
	\begin{equation*}
		\|f\|_{H^s_p(\R^d)} \asymp   \bigg( \sum_{j=0}^\infty 2^{-j ( p -1) d} F_0^p(2^{-j}) \bigg)^{1/p} + \left(\sum_{j=0}^\infty   2^{j (s + d -d/p) p} F_0^p(2^{j}) \right)^{1/p}, \quad f \in \widehat{GM}^d.
	\end{equation*}
\end{rem}

\begin{proof}[Proof of Theorem \ref{TheoremGMLip}]
	    We  make use of the following description (see \cite[Corollary 4.1]{GorbachevTikhonov}) of the modulus of smoothness in terms of the Fourier transform for functions in the $\widehat{GM}^d$ class
    \begin{equation*}
        \omega_{s}(f,t)_p \asymp t^{s} \left(\int_0^{1/t} u^{ s p + d p - d} F_0^p(u) \frac{d u}{u} \right)^{1/p} + \left(\int_{1/t}^\infty u^{d p - d} F^p_0(u)
        \frac{d u}{u}\right)^{1/p}.
    \end{equation*}
    Therefore,
    \begin{equation}\label{TheoremGMLip1}
    	\left(\int_0^{1} (t^{-s} (1 - \log t)^{-b} \omega_s(f,t)_p)^q \frac{d t}{t} \right)^{1/q} = I + II,
\end{equation}
	 where
	 \begin{equation*}
	I =  \left(\int_0^1 (1-\log t)^{b q} \left( \int_0^{1/t} u^{ s p + d p - d} F_0^p(u) \frac{d u}{u} \right)^{q/p} \frac{d t}{t} \right)^{1/q}
	\end{equation*}
	and
	\begin{equation*}
	II = \left(\int_0^1 t^{-s q} (1 - \log t)^{b q} \left(\int_{1/t}^\infty u^{d p - d} F_0^p(u) \frac{d u}{u} \right)^{q/p} \frac{d t}{t} \right)^{1/q}.
    \end{equation*}
    A simple change of variables yields that
    \begin{align*}
    	I & =  \left(\int_1^\infty  (1 + \log t)^{b q} \left(\int_0^t u^{s p  + d p - d} F_0^p(u) \frac{d u}{u} \right)^{q/p} \frac{d t}{t}\right)^{1/q} \\
	& \asymp \left(\int_0^1 u^{s p  + d p - d} F_0^p(u) \frac{d u}{u} \right)^{1/p} \\
	& \hspace{1cm} +  \left(\int_1^\infty  (1 + \log t)^{b q} \left(\int_1^t u^{s p  + d p - d} F_0^p(u) \frac{d u}{u} \right)^{q/p} \frac{d t}{t}\right)^{1/q}
    \end{align*}
    where we have also used that $b < -1/q$. Note that, by \eqref{H1},
    \begin{align*}
     F_0(1) + \left(\int_1^\infty  (1 + \log t)^{b q} \left(\int_1^t u^{s p  + d p - d} F_0^p(u) \frac{d u}{u} \right)^{q/p} \frac{d t}{t}\right)^{1/q} &\asymp  \\
     & \hspace{-6cm}\left(\sum_{j=0}^\infty 2^{j (b+1/q) q} \bigg(\sum_{\nu=2^j}^{2^{j+1}-1} 2^{\nu (s + d -d/p) p} F_0^p(2^\nu) \bigg)^{q/p} \right)^{1/q}.
    \end{align*}
    Hence, we have
	\begin{align}
	\left(\sum_{j=0}^\infty 2^{j (b+1/q) q} \bigg(\sum_{\nu=2^j}^{2^{j+1}-1} 2^{\nu (s + d -d/p) p} F_0^p(2^\nu) \bigg)^{q/p} \right)^{1/q} & \lesssim I \nonumber \\
	&\hspace{-7cm} \asymp  \bigg( \sum_{j=0}^\infty 2^{-j ( p -1) d} F_0^p(2^{-j}) \bigg)^{1/p} \nonumber \\
	& \hspace{-6cm}+ \left(\sum_{j=0}^\infty 2^{j (b+1/q) q} \bigg(\sum_{\nu=2^j}^{2^{j+1}-1} 2^{\nu (s + d -d/p) p} F_0^p(2^\nu) \bigg)^{q/p} \right)^{1/q}. \label{TheoremGMLip2}
    \end{align}

   Furthermore, a simple change of variables, basic monotonicity properties, \eqref{H2} lead to
    \begin{align}
    	F_0(1) + II  & =F_0(1)+ \left(\int_1^\infty t^{s q} (1 + \log t)^{b q} \left(\int_t^\infty u^{d p - d} F_0^p(u) \frac{d u}{u} \right)^{q/p} \frac{d t}{t} \right)^{1/q} \nonumber \\
	& \asymp \bigg(\sum_{j=0}^\infty 2^{j s q} (1 + j)^{b q} \bigg(\sum_{k=j}^\infty 2^{k (d p -d)} F_0^p(2^k) \bigg)^{q/p} \bigg)^{1/q} \nonumber \\
	& \asymp \bigg(\sum_{j=0}^\infty 2^{j (s + d -d/p)q} (1 + j)^{b q} F_0^q(2^j) \bigg)^{1/q} \nonumber \\
	& \lesssim \bigg(\sum_{j=0}^\infty 2^{j (b+1/q) q} \bigg(\sum_{\nu=2^{j}}^{2^{j+1}-1} 2^{\nu(s+d-d/p) p} F_0^p(2^\nu) \bigg)^{q/p} \bigg)^{1/q}. \label{TheoremGMLip3}
    \end{align}

	According to (\ref{TheoremGMLip1}), (\ref{TheoremGMLip2}), (\ref{TheoremGMLip3}), we estimate
	\begin{align}
			\left(\sum_{j=0}^\infty 2^{j (b+1/q) q} \bigg(\sum_{\nu=2^j}^{2^{j+1}-1} 2^{\nu (s + d -d/p) p} F_0^p(2^\nu) \bigg)^{q/p} \right)^{1/q} & \nonumber\\
			& \hspace{-7cm} \lesssim  \left(\int_0^{1} (t^{-s} (1 - \log t)^{b} \omega_s(f,t)_p)^q \frac{d t}{t} \right)^{1/q}   \label{TheoremGMLip3*}\\
			& \hspace{-7cm} \lesssim  \bigg( \sum_{j=0}^\infty 2^{-j ( p -1) d} F_0^p(2^{-j}) \bigg)^{1/p} \nonumber \\
	& \hspace{-6cm}+ \left(\sum_{j=0}^\infty 2^{j (b+1/q) q} \bigg(\sum_{\nu=2^j}^{2^{j+1}-1} 2^{\nu (s + d -d/p) p} F_0^p(2^\nu) \bigg)^{q/p} \right)^{1/q}. \nonumber
	\end{align}
	
	Finally, the desired characterization (\ref{LipGM}) follows from \eqref{HL} and (\ref{TheoremGMLip3*}).
\end{proof}

The next result is a direct consequence of Theorems \ref{Theorem 3.2} and \ref{TheoremGMLip}.

\begin{cor}
	Let $s > 0, \frac{2 d}{d + 1} < p < \infty, 0 < q \leq \infty$ and $b < -1/q$. Then
	$$
	\emph{Lip}^{s,b}_{p,q}(\R^d) \cap \widehat{GM}^{d} = T^{b+1/q}_q B^{s}_{p,p}(\R^d) \cap \widehat{GM}^{d}.
	$$
	In particular (cf. Proposition \ref{PropositionCoincidences})
	\begin{equation*}
		\emph{Lip}^{s,b}_{p,p}(\R^d) \cap \widehat{GM}^{d} = B^{s,b + 1/p}_{p,p}(\R^d) \cap \widehat{GM}^{d}.
	\end{equation*}
\end{cor}

	All the results from this section admit periodic counterparts in terms of general monotone sequences. For example
(see Theorem \ref{TheoremGMLip})
	
	 \begin{thm}\label{ThmLipGMPer}
 	Let $1 < p < \infty, 0 < q \leq \infty, s>0$ and $b <-1/q$. Let $f \in L_1(\mathbb{T})$ be such that
	$$f(x) \sim \sum_{n=1}^\infty (a_n \cos n x + b_n \sin nx)$$
	 with $\{a_n\}_{n \in \N}$ and $\{b_n\}_{n \in \N}$ nonnegative general monotone sequences.
	Then
	$$
		\|f\|_{\emph{Lip}^{s, b}_{p, q}(\mathbb{T})} \asymp  \left( \sum_{j=0}^\infty 2^{j (b+1/q) q} \bigg(\sum_{\nu= 2^j-1}^{2^{j+1}-2} 2^{\nu (s + 1-1/p) q} (a_{2^\nu} + b_{2^\nu})^p \bigg)^{q/p}  \right)^{1/q}.
	$$
 \end{thm}

\newpage
\section{Some applications of truncated function spaces}

\subsection{Trace operator}

Let us consider the trace operator $\text{Tr}$\index{\bigskip\textbf{Operators}!$\text{Tr}$}\label{TRACE}, which restricts to
the hyperplane
\begin{equation*}
    \mathbb{R}^{d-1} = \{x \in \mathbb{R}^d : x = (x',0)\}
\end{equation*}
in $\mathbb{R}^d$ for $d \geq 2$, that is,
\begin{equation*}
    (\text{Tr } f)(x) = (\text{Tr }_{\{x_d=0\}} f)(x) = f(x',0).
\end{equation*}
For a rigorous interpretation of this operator, we refer to \cite[Section 2.7.2]{Triebel83}.

Let
\begin{equation}\label{AssTExt}
1 \leq p \leq \infty, \, 0 < q \leq \infty \qquad \text{and} \qquad s > 1/p.
\end{equation}
 Under these assumptions, it is well known that the trace $\text{Tr}$ is
continuous
\begin{equation}\label{2.1}
    \text{Tr }: B^s_{p,q}(\mathbb{R}^d) \longrightarrow
    B^{s-1/p}_{p,q}(\mathbb{R}^{d-1})
\end{equation}
and admits a right inverse $\text{Ex}$ which is a
bounded linear operator from $B^{s-1/p}_{p,q}(\mathbb{R}^{d-1})$ onto
$B^s_{p,q}(\mathbb{R}^d)$. Importantly, one can construct $\text{Ex}$ to be a common extension operator for all spaces $B^{s}_{p, q}(\R^d)$ with \eqref{AssTExt}. Further details may be found in \cite[Section 2.2, pp. 29--38]{Triebel20}.

The borderline case $s=1/p$ is of special interest. Let $1 \leq p <
\infty$ and $0 < q \leq 1$, then
\begin{equation}\label{2.2}
    \text{Tr }: B^{1/p}_{p,q}(\mathbb{R}^d) \longrightarrow L_p(\mathbb{R}^{d-1})
\end{equation}
and there exists an extension operator which is bounded but not
linear. This result has a substantial history. The case $p=2$ and $q=1$ was investigated by Agmon and
H\"ormander \cite{AgmonHormander} in their study of the asymptotic properties of the solutions of differential
equations; cf. also \cite{Nikolskii51}. For $1 \leq p < \infty$  and $q=1$ the result was stated by
Peetre \cite{Peetre} and a more general assertion covering anisotropic spaces may be found in Burenkov and Gol'dman \cite{BurenkovGoldman}. The extension to $0 <
q \leq 1$ (and more generally,  $0 < p < \infty$ and $0 < q \leq \min \{1, p\}$) was obtained by Frazier and Jawerth \cite{FrazierJawerth} via atomic decompositions. However
there does not exist a bounded linear extension operator from
$L_p(\mathbb{R}^{d-1})$ into $B^{1/p}_{p,1}(\mathbb{R}^d)$, see \cite{Peetre} and \cite{BurenkovGoldman}. This is in  contrast to \eqref{2.1}.

The trace embedding \eqref{2.2} with $p=1$ can be complemented via the famous Gagliardo's theorem \cite{Gagliardo} which establishes that
\begin{equation}\label{TraceGagliardo}
	\text{Tr}: \text{BV}(\R^{d+1}_+) \to L_1(\R^d)
\end{equation}
and every function in $L_1(\R^d)$ has a bounded extension to $\text{BV}(\R^{d+1}_+)$\index{\bigskip\textbf{Spaces}!$\text{BV}(\Omega)$}\label{BV}, the space of \emph{bounded variation functions} on the upper half-space $\R^{d+1}_+ = \R^d \times (0, \infty)$\index{\bigskip\textbf{Sets}!$\R^{d+1}_+$}\label{UPHS}. Analogous statements hold when $\R^{d+1}_+$ is replaced by a bounded Lipschitz domain $\Omega$. In the same vein as \eqref{2.2}, Peetre \cite{Peetre79} showed that there is no a bounded linear extension operator from  $L_1(\R^d)$ to $\text{BV}(\R^{d+1}_+)$.

Both \eqref{2.2} and \eqref{TraceGagliardo} tell us that integrability properties of functions are not enough to construct bounded linear extension operators. However, it seems to be quite plausible that working with functions having at least rough smoothness the situation may improve. Indeed, very recently Mal\'y, Shanmugalingam and Snipes \cite{MSS} showed the validity of the previous assertion dealing with the Dini space $\mathbf{B}^{0}_{1,1}$ (cf. \eqref{BesovDifDef}). To be more precise, in the more general setting of metric spaces equipped with a doubling measure and $\Omega$ a bounded domain satisfying certain regularity assumptions, they proved the existence of a bounded linear extension operator\index{\bigskip\textbf{Operators}!$\text{Ex}$}\label{EXT}
\begin{equation}\label{MSS}
\text{Ex} : \mathbf{B}^{0}_{1,1} (\partial \Omega) \to \text{BV}(\Omega)
\end{equation}
such that $\text{Tr} \circ \text{Ex} = \text{id}$ (identity operator)\index{\bigskip\textbf{Operators}!$\text{id}$}\label{IDEN}. As a consequence of this result, a (nonlinear) extension operator from $L_1(\partial \Omega)$ to $\text{BV}(\Omega)$ can be constructed (cf. \eqref{TraceGagliardo}).

Comparing now \eqref{TraceGagliardo} with \eqref{MSS}, one observes that the non-linearity of bounded extension operators can be circumvented by restricting the target space in \eqref{TraceGagliardo}  (i.e. substituting $\mathbf{B}^0_{1, 1}$ for $L_1$), however both \eqref{TraceGagliardo} and \eqref{MSS} are formulated in terms of $\text{BV}$ (in particular, the trace of $\text{BV}$ function is not necessarily Dini of type $\mathbf{B}^0_{1, 1}$). This leads us to ask whether one can  improve \eqref{MSS}. Let us also mention that related results to \eqref{MSS} for $L_p$-functions may be found in \cite{Koskela}.


The main result of this section gives a substantial improvement of \eqref{MSS} in terms of truncated Besov spaces. Namely, we establish the following

\begin{thm}\label{Theorem 2.1}
    Let $1 \leq p \leq \infty, 0 < q \leq \infty$ and $b > -1/q$. Then
    \begin{enumerate}[\upshape(i)]
        \item the operator $\emph{Tr}$ acts boundedly from
        $T^{b+1/q}_q B^{1/p}_{p, 1}(\mathbb{R}^d)$ into
        $\mathbf{B}^{0,b}_{p,q}(\mathbb{R}^{d-1})$,
        \item there exists a linear extension operator $\emph{Ex}$ which is
        continuous from $\mathbf{B}^{0,b}_{p,q}(\mathbb{R}^{d-1})$
        onto $T^{b+1/q}_q B^{1/p}_{p,1}(\mathbb{R}^d)$ and
           $ \emph{Tr } \circ \emph{Ex } =
            \emph{id}.$
    \end{enumerate}
\end{thm}

In light of Proposition \ref{PropositionCoincidences}(i) and Theorem \ref{Theorem 2.1} with $q=1$, optimal bounded linear extension operators for $\mathbf{B}^{0}_{p, 1}$ are given in terms of the classical logarithmic Besov spaces $B^{1/p, 1}_{p, 1}$, i.e., a logarithmic derivative is gained. To be more precise,

\begin{cor}\label{CorollaryMSS}
	 Let $1 \leq p \leq \infty$ and $b > -1$. Then
	     \begin{enumerate}[\upshape(i)]
        \item the operator $\emph{Tr}$ acts boundedly from
        $B^{1/p, b+1}_{p, 1}(\mathbb{R}^d)$ into
        $\mathbf{B}^{0,b}_{p,1}(\mathbb{R}^{d-1})$,
        \item there exists a linear extension operator $\emph{Ex}$ which is
        continuous from $\mathbf{B}^{0,b}_{p,1}(\mathbb{R}^{d-1})$
        onto $B^{1/p, b+1}_{p,1}(\mathbb{R}^d)$ and
            $\emph{Tr } \circ \emph{Ex } =
            \emph{id}.
$    \end{enumerate}
\end{cor}

We write down Corollary \ref{CorollaryMSS} with $p=1$ and $b=0$.

\begin{cor}\label{CorNEwMMS}
	The operator $\text{Tr}$ acts boundedly from $B^{1, 1}_{1, 1}(\R^d)$ into $\mathbf{B}^0_{1, 1}(\R^{d-1})$ and there exists a bounded linear extension operator
\begin{equation}\label{ExtensionSpecialp=1}
	\emph{Ex}: \mathbf{B}^{0}_{1,1}(\mathbb{R}^{d-1}) \to B^{1,1}_{1, 1}(\mathbb{R}^d).
\end{equation}
with $\emph{Tr } \circ \emph{Ex } =
            \emph{id}.$
\end{cor}

Note that
$$B^{1,1}_{1, 1}(\mathbb{R}^d) \subsetneq  B^{1}_{1, 1}(\mathbb{R}^d) \subsetneq W^1_1(\R^d) \subsetneq \text{BV}(\R^d).$$
Therefore, Corollary \ref{CorNEwMMS} drastically  improves \eqref{MSS}.

%

The method of proof of \eqref{MSS} in \cite{MSS} relies, among other tools, on partitions of unity related to Whitney coverings. On the other hand, the proof of Theorem \ref{Theorem 2.1} uses completely different machinery, specifically, we will make a strong use of several techniques already studied in previous sections: limiting interpolation (cf. Section \ref{SectionLimIntGeneral}), approximation techniques (Jackson's inequality) and certain decomposition methods for truncated Besov spaces (cf. Section \ref{SectionApproximation}).

\begin{proof}[Proof of Theorem \ref{Theorem 2.1}]
    (i): We first observe that trace operator is well defined on $T^{b+1/q}_q B^{1/p}_{p, 1}(\R^d)$. Indeed, if $p < \infty$ and $q < \infty$ then $\mathcal{S}(\mathbb{R}^d)$ is dense in
    $T^{b+1/q}_q B^{1/p}_{p,1}(\mathbb{R}^d)$ (cf. \eqref{DualPairing}) so that the the trace operator $\text{Tr}$ on
    this space should be understood in a classical sense via completion. If $p=\infty$ then, by Theorem \ref{ThmC} (since $b + 1/q > 0$),  $T^{b+1/q}_q B^{0}_{\infty,1}(\mathbb{R}^d)$ is formed only by continuous functions and $\text{Tr}$ makes sense pointwise. If $q=\infty$ then $T^b_\infty B^{1/p}_{p,1}(\mathbb{R}^d), \, b > 0,$ is embedded into $B^{1/p}_{p, 1}(\R^d)$, see Proposition \ref{PropositionElementary}(i). Since traces are independent of the source spaces and of the target spaces (cf. \cite[Remark 6.62]{Triebel08}), one can define the trace on $T^b_\infty B^{1/p}_{p,1}(\mathbb{R}^d)$ via restriction of the trace on  $B^{1/p}_{p, 1}(\R^d)$.

    Let $s > 1/p$. Applying the limiting interpolation method with $\theta = 0$ (cf. \eqref{DefLimInterpolation}) to (\ref{2.1})-(\ref{2.2}) we obtain
    \begin{equation}\label{2.3}
        \text{Tr }: (B^{1/p}_{p,1}(\mathbb{R}^d),
        B^s_{p,1}(\mathbb{R}^d))_{(0,b),q} \rightarrow (L_p(\mathbb{R}^{d-1}),
        B^{s-1/p}_{p,1}(\mathbb{R}^{d-1}))_{(0,b),q}.
    \end{equation}
    Next we determine the interpolation spaces appearing in
    (\ref{2.3}). It follows from Theorem \ref{TheoremInterpolation}(i) (recall $b > -1/q$) that
    \begin{equation*}
        (B^{1/p}_{p,1}(\mathbb{R}^d),
        B^s_{p,1}(\mathbb{R}^d))_{(0,b),q} =
        T^{b+1/q}_q B^{1/p}_{p,1}(\mathbb{R}^d).
    \end{equation*}
    Concerning the target space in
    (\ref{2.3}), we invoke the interpolation formula \eqref{IntLimB0} (see also Remark \ref{RemarkLimBesov}) so that
    \begin{equation*}
        (L_p(\mathbb{R}^{d-1}),
        B^{s-1/p}_{p,1}(\mathbb{R}^{d-1}))_{(0,b),q} =
        \mathbf{B}^{0,b}_{p,q}(\mathbb{R}^{d-1}).
    \end{equation*}
    Hence
        \begin{equation*}
        \text{Tr }: T^{b+1/q}_q B^{1/p}_{p,1}(\mathbb{R}^d) \rightarrow \mathbf{B}^{0,b}_{p,q}(\mathbb{R}^{d-1}).
    \end{equation*}

    (ii): Let $f \in \mathbf{B}^{0, b}_{p, q}(\R^{d-1})$. Recall that $V_t f, \, t > 0$, denotes the de la Valle\'e Poussin means (cf. \eqref{DelaValleePoussin}). Let $\mu_k := 2^{2^k}, \, k \geq 0$. Therefore $f$ admits the linear decomposition
    \begin{equation}\label{firstDecomposition}
    	f = \sum_{k=0}^\infty R_k f \qquad \text{(convergence in
        $L_p(\mathbb{R}^{d-1})$)}
    \end{equation}
    where $R_k := V_{\mu_{k+1}} - V_{\mu_k}$ for $k \geq 1$ and $R_0:= V_{2}$. Indeed, it follows from \eqref{BestApproximation} and the Jackson's inequality (see e.g. \cite[5.2.1]{Nikolskii}) that
    \begin{align*}
    	\Big\|f- \sum_{k=0}^N R_k f \Big\|_{L_p(\R^{d-1})} &= \|f - V_{\mu_{N+1}} f\|_{L_p(\R^{d-1})}  \\
	&\hspace{-2cm}\lesssim E_{\mu_{N+1}}(f)_{L_p(\R^{d-1})} \lesssim \omega_1\Big(f, \frac{1}{\mu_{N+1}} \Big)_p
    \end{align*}
    and $\omega_1(f,t)_p \to 0$ as $t \to 0+$ since $f \in \mathbf{B}^{0, b}_{p, q}(\R^{d-1})$.

Let $\psi \in \mathcal{S}(\mathbb{R})$ be such that
    \begin{equation}\label{new}
        \text{supp } \widehat{\psi} \subseteq [-1,1] \qquad \text{and} \qquad
        \psi(0)=1.
    \end{equation}
    Let us define the linear operator
    \begin{equation}\label{2.8}
        \text{Ex} f (x) = \sum_{k=0}^\infty R_k f (x') \,  \psi(\mu_k
        x_d), \quad x=(x',x_d), \quad x' \in \mathbb{R}^{d-1}, \quad x_d \in
        \mathbb{R}.
    \end{equation}
    Note that $R_k f (x') \,  \psi(\mu_k
        x_d)$ is an entire function of exponential type $\mu_k$.  Assume momentarily that \eqref{2.8} is well-defined for $f \in \mathbf{B}^{0, b}_{p, q}(\R^{d-1})$. Then invoking the representation theorems in $T^{b+1/q}_q B^{1/p}_{p, 1}(\mathbb{R}^d)$ and $B^{1/p}_{p, 1}(\R^d)$ given in \eqref{Decomposition0BesovNew} and \eqref{Decomposition0nenene}, respectively,  together with Fubini's theorem, we get
        \begin{align}
            \|\text{Ex} f \|_{T^{b+1/q}_q B^{1/p}_{p, 1}(\mathbb{R}^d)} & \lesssim  \left( \sum_{k=0}^\infty 2^{k (b+ \frac{1}{q}) q} \|R_k f (x') \,  \psi(\mu_k
        x_d)\|_{B^{1/p}_{p, 1}(\R^d)}^q  \right)^{\frac{1}{q}} \nonumber \\
            & \lesssim   \left( \sum_{k=0}^\infty 2^{k (b+ \frac{1}{q}) q} \mu_k^{q/p} \|R_k f (x') \,  \psi(\mu_k
        x_d)\|_{L_p(\R^d)}^q  \right)^{\frac{1}{q}} \label{68}  \\
        & =   \left( \sum_{k=0}^\infty 2^{k (b+ \frac{1}{q}) q} \mu_k^{q/p} \|R_k f \|_{L_p(\R^{d-1})}^q   \|\psi(\mu_k
        \cdot)\|_{L_p(\R)}^q  \right)^{\frac{1}{q}} \nonumber \\
        & \asymp \left( \sum_{k=0}^\infty 2^{k (b+ \frac{1}{q}) q} \|R_k f \|_{L_p(\R^{d-1})}^q   \right)^{\frac{1}{q}}.\nonumber
        \end{align}
        Applying the triangle inequality, the best approximation property of the de la Valle\'e Poussin means (cf. \eqref{BestApproximation}) and the Jackson's inequality, we find
        \begin{align*}
        		\|R_k f\|_{L_p(\R^{d-1})} & \leq \|f- V_{\mu_{k+1}} f \|_{L_p(\R^{d-1})} + \| f- V_{\mu_k}f \|_{L_p(\R^{d-1})} \\
		& \lesssim E_{\mu_k} (f)_{L_p(\R^{d-1})} \lesssim \omega_1\Big(f, \frac{1}{\mu_k} \Big)_p
        \end{align*}
       and thus, making use of the monotonicity properties of the modulus of smoothness,
       \begin{equation}\label{69}
        	   \left( \sum_{k=0}^\infty 2^{k (b+ \frac{1}{q}) q} \|R_k f \|_{L_p(\R^{d-1})}^q   \right)^{\frac{1}{q}} \lesssim  \left( \sum_{k=0}^\infty 2^{k (b+ \frac{1}{q}) q}  \omega_1\Big(f, \frac{1}{\mu_k} \Big)_p^q   \right)^{\frac{1}{q}} \lesssim \|f\|_{\mathbf{B}^{0,b}_{p,q}(\mathbb{R}^{d-1})}.
       \end{equation}
       We can combine \eqref{68} and \eqref{69} to establish
       $$
       	 \|\text{Ex} f \|_{T^{b+1/q}_q B^{1/p}_{p, 1}(\mathbb{R}^d)} \lesssim \|f\|_{\mathbf{B}^{0,b}_{p,q}(\mathbb{R}^{d-1})}
       $$
       which shows that $\text{Ex}$ is bounded from
        $\mathbf{B}^{0,b}_{p,q}(\mathbb{R}^{d-1})$ into
        $T^{b+1/q}_q B^{1/p}_{p, 1}(\mathbb{R}^d)$. Furthermore, it is clear from \eqref{firstDecomposition}--\eqref{2.8}  that $\text{Tr } \circ \text{Ex } =
            \text{id}.$

            It remains to show that \eqref{2.8} is well-defined for $f \in \mathbf{B}^{0, b}_{p, q}(\R^{d-1})$. Indeed, we will prove that the series given in the right-hand side of \eqref{2.8} converges in $B^{1/p}_{p, 1}(\R^d).$ This assertion is an immediate consequence of
            $$
            	\sum_{k=0}^\infty \|R_k f(x') \,  \psi(\mu_k
        x_d) \|_{B^{1/p}_{p, 1}(\R^d)} \lesssim  \left( \sum_{k=0}^\infty 2^{k (b+ \frac{1}{q}) q} \|R_k f (x') \,  \psi(\mu_k
        x_d)\|_{B^{1/p}_{p, 1}(\R^d)}^q  \right)^{\frac{1}{q}}
            $$
    and \eqref{68}-\eqref{69}. The validity of the previous estimate is an application of  H\"older's inequality if $q \geq 1$ and the fact that $\ell_q \hookrightarrow \ell_1$ if $q < 1$.
    \end{proof}

Unlike the intriguing case $s = 1/p$ considered in Theorem \ref{Theorem 2.1}, a description of traces for $T^b_r B^s_{p, q}(\R^d)$ with $s > 1/p$ follows easily just from the interpolation properties studied in Section \ref{Section:LimitingInterpolation}. More precisely, we establish the following result.

\begin{thm}\label{Thm175}
    Let
    \begin{equation}\label{Thm175As1}
    1 \leq p \leq \infty, \qquad 0 < q, r \leq \infty, \qquad s > \frac{1}{p} \qquad \text{and} \qquad b \neq 0.
    \end{equation}
     Then
    \begin{enumerate}[\upshape(i)]
        \item the operator $\emph{Tr}$ acts boundedly from
        $T^{b}_r B^{s}_{p, q}(\mathbb{R}^d)$ onto
        $T^b_r B^{s-\frac{1}{p}}_{p,q}(\mathbb{R}^{d-1})$,
        \item there exists a linear extension operator $\emph{Ex}$ which is
        continuous from $T^b_r B^{s-\frac{1}{p}}_{p,q}(\mathbb{R}^{d-1})$
        onto $T^{b}_r B^{s}_{p, q}(\mathbb{R}^d)$ and
           $ \emph{Tr } \circ \emph{Ex } =
            \emph{id}.$
    \end{enumerate}
\end{thm}

\begin{proof}
	Assume first $b > 0$.  According to \eqref{2.1}, given any  $\varepsilon > 0$,
	\begin{equation}\label{Fr2+}
		 \text{Tr }: B^s_{p,q}(\mathbb{R}^d) \longrightarrow
    B^{s-1/p}_{p,q}(\mathbb{R}^{d-1}) \qquad \text{and} \qquad  \text{Tr }: B^{s+\varepsilon}_{p,q}(\mathbb{R}^d) \longrightarrow
    B^{s + \varepsilon-1/p}_{p,q}(\mathbb{R}^{d-1}).
	\end{equation}
	Moreover, there exists a linear operator $\text{Ex}$ such that
	\begin{equation}\label{Fr2-}
		\text{Ex }: B^{s-1/p}_{p,q}(\mathbb{R}^{d-1}) \longrightarrow B^s_{p,q}(\mathbb{R}^d)  \qquad \text{and} \qquad \text{Ex }: B^{s+\varepsilon-1/p}_{p,q}(\mathbb{R}^{d-1}) \longrightarrow B^{s+\varepsilon}_{p,q}(\mathbb{R}^d)
	\end{equation}
	with      $ \text{Tr } \circ \text{Ex } =
            \text{id}.$
	
	Applying limiting interpolation to \eqref{Fr2+} and \eqref{Fr2-}, we derive
	\begin{equation}\label{Fr2}
		 \text{Tr }: (B^s_{p,q}(\mathbb{R}^d),  B^{s+\varepsilon}_{p,q}(\mathbb{R}^d))_{(0, b-1/r), r} \longrightarrow   (B^{s-1/p}_{p,q}(\mathbb{R}^{d-1}),  B^{s+\varepsilon-1/p}_{p,q}(\mathbb{R}^{d-1}))_{(0, b-1/r), r}
	\end{equation}
	and
	\begin{equation}\label{Fr2.1}
		\text{Ex }: (B^{s-1/p}_{p,q}(\mathbb{R}^{d-1}),  B^{s+\varepsilon-1/p}_{p,q}(\mathbb{R}^{d-1}))_{(0, b-1/r), r} \longrightarrow  (B^s_{p,q}(\mathbb{R}^d),  B^{s+\varepsilon}_{p,q}(\mathbb{R}^d))_{(0, b-1/r), r}.
	\end{equation}
	Furthermore, in virtue of Theorem \ref{Thm4.2}(i), we have
	$$
		(B^s_{p,q}(\mathbb{R}^d),  B^{s+\varepsilon}_{p,q}(\mathbb{R}^d))_{(0, b-1/r), r}  = T^b_r B^s_{p, q}(\R^d)
	$$
	and
	$$
		(B^{s-1/p}_{p,q}(\mathbb{R}^{d-1}),  B^{s+\varepsilon-1/p}_{p,q}(\mathbb{R}^{d-1}))_{(0, b-1/r), r} = T^b_r B^{s-1/p}_{p,q}(\mathbb{R}^{d-1}),
	$$
	which combined with \eqref{Fr2} and \eqref{Fr2.1} give
	$$
		 \text{Tr }: T^b_r B^s_{p, q}(\R^d) \longrightarrow T^b_r B^{s-1/p}_{p,q}(\mathbb{R}^{d-1})
	$$
	and
	$$
		\text{Ex }: T^b_r B^{s-1/p}_{p,q}(\mathbb{R}^{d-1}) \longrightarrow T^b_r B^s_{p, q}(\R^d).
	$$
	
	The case $b < 0$ can be obtained in an analogous fashion, but now invoking part (ii) in Theorem \ref{Thm4.2} with $s_0 = s -\varepsilon$ and $0 < \varepsilon < s-1/p$.
\end{proof}

\begin{rem}
	The extension operator $\text{Ex}$ in Theorem \ref{Thm175} is common for all spaces $T^b_r B^s_{p, q}(\R^d)$ with \eqref{Thm175As1}.
\end{rem}

\begin{rem}
	The method of proof of Theorem \ref{Thm175} still works with limiting case $b=0$, however the outcome is now formulated in terms of the spaces $T^*_r B^s_{p, q}(\R^d)$; cf. \eqref{33} and Remark \ref{Remark4.4}. Specifically, if $1 \leq p \leq \infty, 0 < q, r \leq \infty$ and $s > 1/p$, then
	$$
		\text{Tr }: T^*_r B^s_{p, q}(\R^d) \longrightarrow T^*_r B^{s-1/p}_{p, q}(\R^{d-1})
	$$
	and there exists a linear extension operator $\text{Ex}$ satisfying
	$$
		\text{Ex }: T^*_r B^{s-1/p}_{p, q}(\R^{d-1}) \longrightarrow T^*_r B^s_{p, q}(\R^d)
	$$
	with $ \text{Tr } \circ \text{Ex } =
            \text{id}.$
\end{rem}

Specialising Theorem \ref{Thm175} with $r=q$ and taking into account  Proposition \ref{PropositionCoincidences}(i), we are able to extend \ref{2.1} to the setting of $B^{s, b}_{p, q}(\R^d)$ spaces.

\begin{cor}
	    Let
    \begin{equation*}
    1 \leq p \leq \infty, \qquad 0 < q \leq \infty, \qquad s > \frac{1}{p} \qquad \text{and} \qquad b \in \R.
    \end{equation*}
     Then
    \begin{enumerate}[\upshape(i)]
        \item the operator $\emph{Tr}$ acts boundedly from
        $B^{s, b}_{p, q}(\mathbb{R}^d)$ onto
        $B^{s-\frac{1}{p}, b}_{p,q}(\mathbb{R}^{d-1})$,
        \item there exists a linear extension operator $\emph{Ex}$ which is
        continuous from $B^{s-\frac{1}{p}, b}_{p,q}(\mathbb{R}^{d-1})$
        onto $B^{s, b}_{p, q}(\mathbb{R}^d)$ and
           $ \emph{Tr } \circ \emph{Ex } =
            \emph{id}.$
    \end{enumerate}
\end{cor}

\begin{rem}
	The previous result may be viewed as the non-limiting case (i.e., $s >1/p$) of Corollary \ref{CorollaryMSS}. In particular, this shows an interesting switch from $B^{0, b}_{p, 1}(\R^{d-1})$ as expected trace scale of Besov spaces with critical smoothness $s=1/p$ to the more surprising scale given by spaces $\mathbf{B}^{0, b}_{p, 1}(\R^{d-1})$ (see \eqref{EmbBBzero}).
\end{rem}

\begin{rem}
	All the results from this subsection can be extended to deal with traces on hyperplanes of dimension $1, 2, \ldots, d-2$. But this is a technical matter and will not be done in detail here. We just mention that the corresponding extension of  Theorem \ref{Theorem 2.1}  reads as follows: Let $d > m \in \N$ and $\text{Tr}_{\R^m} = \text{Tr}$. Let $1 \leq p \leq \infty, 0 < q \leq \infty$ and $b > -1/q$. Then $\text{Tr}: T^{b + 1/q}_q B^{\frac{d-m}{p}}_{p, 1}(\mathbb{R}^d) \to \mathbf{B}^{0,b}_{p,q}(\mathbb{R}^{m})$ and there exists a bounded linear extension operator $\text{Ex}: \mathbf{B}^{0,b}_{p,q}(\mathbb{R}^{m}) \to T^{b + 1/q}_q B^{\frac{d-m}{p}}_{p, 1}(\mathbb{R}^d)$ such that $\text{Tr } \circ \text{Ex } =
            \text{id}.$
\end{rem}

\subsection{Embeddings with critical smoothness}\label{SectionEmbCritical} The goal of this section is to apply the theory of truncated Besov spaces developed in previous sections in order to improve classical embedding theorems in the critical case $s=d/p$. We divide our analysis into two parts. First, we shall concentrate on Trudinger-type inequalities (i.e., the case $b < -1/q$). Second, we turn our attention to embeddings into H\"older-type spaces (i.e., the case $b \geq -1/q$). Before we go further, we briefly recall the definition of Lorentz--Zygmund spaces.

Let $0 < p, q \leq \infty$ and $b \in \R$. Then  $L_{p, q}(\log L)_{b}(\mathbb{T}^d)$ stands for the \emph{Lorentz--Zygmund space} formed by all measurable functions $f$ on $\mathbb{T}^d$ such that\index{\bigskip\textbf{Spaces}!$L_{p, q}(\log L)_b(\mathbb{T}^d)$}\label{LZ}
\begin{equation}\label{DefLZ}
	\|f\|_{L_{p, q}(\log L)_{b}(\mathbb{T}^d)} = \bigg(\int_0^1 [t^{1/p}(1-\log t)^{b } f^*(t)]^q \frac{dt}{t} \bigg)^{1/q} < \infty
\end{equation}
(with the usual modification if $q=\infty$).   As usual, $f^*$ stands for the \emph{non-increasing rearrangement of $f$}. We shall consider the rearrangement which is left-continuous. It can be expressed uniquely by the equality\index{\bigskip\textbf{Functionals and functions}!$f^\ast$}\label{REARRANGEMENT}
\begin{equation}\label{Rearrengement}
	f^*(t) = \sup_{|E| = t} \inf_{x \in E} |f(x)|
\end{equation}
(see \cite{ChongRice}). A basic reference for Lorentz--Zygmund spaces is \cite{BennettRudnick}. This scale contains many classical spaces as distinguished elements: Letting $p=q$ in  $L_{p, q}(\log L)_{b}(\mathbb{T}^d)$ one obtains the \emph{Zygmund space} $L_p(\log L)_b(\mathbb{T}^d)$\index{\bigskip\textbf{Spaces}!$L_p(\log L)_b(\mathbb{T}^d)$}\label{ZYG} and if, additionally, $b=0$ then $L_p(\mathbb{T}^d)$. On the other hand, setting $b=0$ in  $L_{p, q}(\log L)_{b}(\mathbb{T}^d)$, one recovers \emph{Lorentz spaces} $L_{p, q}(\mathbb{T}^d)$\index{\bigskip\textbf{Spaces}!$L_{p, q} (\mathbb{T}^d)$}\label{LOR}. Dealing with the limiting value $p=\infty$, the additional assumption $b < -1/q$ ($b \leq 0$ if $q=\infty$) is required in order to get meaningful spaces. In particular, it is well known that $L_{\infty}(\log L)_{b}(\mathbb{T}^d) = L_{\infty, \infty}(\log L)_{b}(\mathbb{T}^d)$ is the classical \emph{Orlicz space} $\text{exp}\, L^{-\frac{1}{b}}(\mathbb{T}^d)$ of exponentially integrable functions\index{\bigskip\textbf{Spaces}!$\text{exp} \, L^{-\frac{1}{b}}(\mathbb{T}^d)$}\label{OR}, cf. \cite[Theorem 10.3]{BennettRudnick} and \cite[Chapter 4, Lemma 6.2]{BennettSharpley}.

\vspace{2mm}
\textbf{Trudinger's embeddings: the case $b < -1/q$.} To avoid technical issues, we will switch momentarily to function spaces defined on the $d$-dimensional torus $\mathbb{T}^d$. We refer to Remark \ref{RemarkRn} below for further details.

The celebrated Trudinger's inequality \cite{Trudinger, Yudovich} can be rephrased  in terms of the following embeddings
\begin{equation}\label{Trudinger}
	W^{d/p}_p(\mathbb{T}^d) \hookrightarrow L_\infty (\log L)_{-1/p'} (\mathbb{T}^d), \qquad p \in (1, \infty).
\end{equation}
This result has shown to be very useful in the study of elliptic and parabolic equations and Gaussian curvatures on sphere. Furthermore, \eqref{Trudinger} is sharp within the class of Orlicz spaces, but it can be improved when dealing with the larger class of Lorentz--Zygmund spaces \eqref{DefLZ}. Indeed, the Maz'ya--Hansson--Brezis--Wainger embedding claims that
\begin{equation}\label{MHBW}
	W^{d/p}_p(\mathbb{T}^d)  \hookrightarrow L_{\infty, p}(\log L)_{-1}(\mathbb{T}^d),
\end{equation}
cf. \cite{BrezisWainger, Hansson, Mazya}. Further, it is plain to see that (cf. \cite[Theorem 9.5]{BennettRudnick})
$$
	 L_{\infty, p}(\log L)_{-1}(\mathbb{T}^d) \hookrightarrow L_\infty (\log L)_{-1/p'} (\mathbb{T}^d).
$$
Accordingly \eqref{MHBW} yields an improvement of \eqref{Trudinger}.

In fact, not only the target space in \eqref{Trudinger} can be sharpened but also the domain space can be sharpened in terms of Besov spaces. Indeed, the following improvement of \eqref{MHBW} holds, for $0 < p < \infty$ and $1 < q \leq \infty$,
\begin{equation}\label{CriticalEmbedding}
	B^{d/p}_{p,q}(\mathbb{T}^d) \hookrightarrow L_{\infty, q}(\log L)_{-1}(\mathbb{T}^d);
\end{equation}
cf. \cite[Theorem 13.2]{Triebel01},  \cite[Theorem 8.16]{Haroske} and the references within. Note that, by Franke--Jawerth embeddings \eqref{FJClas} (recall that Littlewood--Paley theorem asserts that $W^{d/p}_p(\mathbb{T}^d) = F^{d/p}_{p, 2}(\mathbb{T}^d), \, p > 1$),
$$
	W^{d/p}_p(\mathbb{T}^d)  \hookrightarrow B^{d/p_0}_{p_0,p}(\mathbb{T}^d), \qquad p < p_0,
$$
so that \eqref{CriticalEmbedding} sharpens \eqref{MHBW}.

The embedding \eqref{CriticalEmbedding} can be placed naturally in the more general setting of Besov spaces of logarithmic smoothness. Namely, if $0 < p < \infty, 0 < q \leq \infty$ and $b < -1/q$ then
\begin{equation}\label{CriticalEmbeddingLog}
	B^{d/p, b + 1/\min\{1, q\}}_{p,q}(\mathbb{T}^d) \hookrightarrow L_{\infty, q}(\log L)_{b}(\mathbb{T}^d).
\end{equation}
In particular, setting $b=-1$ in \eqref{CriticalEmbeddingLog} gives back \eqref{CriticalEmbedding}.
The embedding \eqref{CriticalEmbeddingLog} together with its generalizations have been the object of intensive research. In this regard, we only mention here \cite{DeVore}, \cite{CaetanoMoura}, \cite{CaetanoFarkas}, \cite{CaetanoLeopold}, \cite{Martin} and \cite{MouraNevesPiotrowski}. In particular, it is well known that \eqref{CriticalEmbeddingLog} is optimal within the scale of Besov spaces of logarithmic smoothness. However, as a byproduct of the following result, we can conclude that the known embedding \eqref{CriticalEmbeddingLog} can be sharpened with the help of the truncated Besov spaces $T^b_r B^{s}_{p, q}(\mathbb{T}^d)$.

\begin{thm}\label{TheoremEmbedinngCritical}
	Let $0 < p < \infty, 0 < q \leq \infty$ and $b < -1/q$. Then
	\begin{equation}\label{TheoremEmbedinngCritical1}
		T^{b+1/q}_q B^{d/p}_{p, 1}(\mathbb{T}^d) \hookrightarrow L_{\infty, q}(\log L)_b(\mathbb{T}^d).
	\end{equation}
	Furthermore, the embedding is optimal in the following sense: let $0 < u \leq \infty$, then
	\begin{equation}\label{SharpTheoremEmbedinngCritical}
		T^{b+1/q}_q B^{d/p}_{p, u}(\mathbb{T}^d) \hookrightarrow L_{\infty, q}(\log L)_b(\mathbb{T}^d) \iff u \leq 1.
	\end{equation}
\end{thm}

Before we proceed with the proof of this theorem, we show that \eqref{TheoremEmbedinngCritical1} is a non-trivial improvement of \eqref{CriticalEmbeddingLog} (and in particular \eqref{CriticalEmbedding}).

\begin{rem}\label{Remark6.5}
Note that as a consequence of the periodic counterpart of Corollary \ref{TheoremEmbeddings1}, we derive
	$$
		B^{d/p, b + 1/\min\{1, q\}}_{p,q}(\mathbb{T}^d) + B^{d/p, b + 1/q}_{p, \min\{1, q\}}(\mathbb{T}^d) \hookrightarrow T^{b+1/q}_q B^{d/p}_{p, 1}(\mathbb{T}^d).
	$$
Thus, by \eqref{TheoremEmbedinngCritical1},
	$$	
		B^{d/p, b + 1/\min\{1, q\}}_{p,q}(\mathbb{T}^d) + B^{d/p, b + 1/q}_{p, \min\{1, q\}}(\mathbb{T}^d) \hookrightarrow  L_{\infty, q}(\log L)_{b}(\mathbb{T}^d),
	$$
	which implies \eqref{CriticalEmbeddingLog}. In fact, by Proposition \ref{PropositionCoincidences}, \eqref{TheoremEmbedinngCritical1} coincides with \eqref{CriticalEmbeddingLog} if $q=1$. However, if $q \neq 1$ then \eqref{TheoremEmbedinngCritical1} gives a non-trivial improvement of \eqref{CriticalEmbeddingLog} since one can construct $f \in T^{b+1/q}_q B^{d/p}_{p, 1}(\mathbb{T}^d)$ such that $f \not \in B^{d/p, b + 1/\min\{1, q\}}_{p,q}(\mathbb{T}^d)$. The construction of these extremal functions can be done using the orthonormal system $\{ \Psi^{j, L, \text{per}}_{G,m}\}$ in $L_2(\mathbb{T^d})$\index{\bigskip\textbf{Functionals and functions}!$\Psi^{j, L, \text{per}}_{G, m}$}\label{PERWAV} obtained from standard periodization arguments for (cf. \eqref{DefinitionWavelet})
	\begin{equation*}
\Psi^{j, L}_{G,m} (x) := 2^{(j+L) d/2} \prod^n_{r=1} \psi_{G_r} (2^{j+L} x_r - m_r
),
\end{equation*}
where $L \in \N$ is a fixed parameter such that $\text{supp } \Psi^{j, L}_{G,m} \subset \{x \in \R^d: |x| < 1/2\}$. Here $j \in \N_0, G \in G^j$ and $m \in \mathbb{P}^d_j = \{m \in \Z^d: 0 \leq m_r < 2^{j+L}\}$. We refer the interested reader to \cite[Section 1.3.2]{Triebel08} for further details. Then obvious modifications show that the periodic counterpart of Theorem \ref{ThmWaveletsNewBesov} holds (in particular the related sequence spaces are defined in the same fashion as \eqref{DefBesSeq} but now $m$ runs over the index set $\mathbb{P}^d_j$). Assume $q > 1$ and let
		\begin{equation}\label{DecompositionCounterexample}
	 f = \sum_{j \in \N_0,G \in G^j,m \in \mathbb{P}^d_j} \lambda^{j,G}_m 2^{-(j + L) d/2}
    \Psi^{j, L, \text{per}}_{G,m}
    \end{equation}
    where
    	\begin{equation*}
		\lambda^{j,G}_m = \left\{\begin{array}{cl}   2^{-k \varepsilon}, & \quad j = 2^k, \quad k \in \N_0, \quad m= (0, \ldots, 0),  \\
		& \qquad G = (M, \ldots, M),  \\
		0, & \text{otherwise},
		       \end{array}
                        \right.
	\end{equation*}
	and $b+ \frac{1}{q} < \varepsilon < b + 1$. According to Theorem \ref{ThmWaveletsNewBesov}, we have
	$$
		\|f\|_{T^{b+1/q}_q B^{d/p}_{p, 1}(\mathbb{T}^d) } \asymp  \bigg(\sum_{k=0}^\infty 2^{k(b+1/q-\varepsilon) q}  \bigg)^{1/q} < \infty
	$$
	and, by Theorem \ref{ThmWaveletsBesovClassic} (more precisely, its periodic counterpart),
	$$
		\|f\|_{B^{d/p, b + 1}_{p,q}(\mathbb{T}^d)} \asymp \left(\sum_{k=0}^\infty  2^{k(b+1-\varepsilon) q} \right)^{1/q} = \infty.
	$$
	
	Assume now $q < 1$. Then let $f$ be given by \eqref{DecompositionCounterexample} with
		\begin{equation*}
		\lambda^{j,G}_m = \left\{\begin{array}{cl}  (1 + j)^{-\varepsilon}, & \quad j \in \N_0, \quad m= (0, \ldots, 0), \quad G = (M, \ldots, M),  \\
		0, & \text{otherwise},
		       \end{array}
                        \right.
	\end{equation*}
	and $b + \frac{1}{q} +1 < \varepsilon < b+ \frac{2}{q}$. Invoking again Theorems \ref{ThmWaveletsBesovClassic} and  \ref{ThmWaveletsNewBesov}, we obtain
	\begin{align*}
	\|f\|_{T^{b+1/q}_q B^{d/p}_{p, 1}(\mathbb{T}^d) } &\asymp \bigg(\sum_{k=0}^\infty 2^{k (b+1/q) q} \bigg(\sum_{j=2^k-1}^{2^{k+1}-2} (1 + j)^{-\varepsilon} \bigg)^{q} \bigg)^{1/q} \\
	&\asymp \bigg(\sum_{k=0}^\infty 2^{k(b - \varepsilon + 1+1/q) q} \bigg)^{1/q} < \infty
	\end{align*}
	and
	$$
	\|f\|_{B^{d/p, b + 1/q}_{p,q}(\mathbb{T}^d)} \asymp \left(\sum_{j=0}^\infty  (1 + j)^{(b+1/q-\varepsilon) q} \right)^{1/q} = \infty.
	$$
	
\end{rem}

\begin{proof}[Proof of Theorem \ref{TheoremEmbedinngCritical}]
Applying the limiting interpolation method \eqref{DefLimInterpolation} with $\theta=1$ to the classical embeddings
\begin{equation}\label{BasicEmbedding}
	B^{d/p}_{p, 1}(\mathbb{T}^d) \hookrightarrow L_\infty(\mathbb{T}^d)
\end{equation}
and
$$
	B^{0}_{p, \min\{p, 1\}}(\mathbb{T}^d) \hookrightarrow L_p(\mathbb{T}^d),
$$
we derive
\begin{equation}\label{TheoremEmbedinngCriticalProof1}
	(B^0_{p, \min\{p, 1\}}(\mathbb{T}^d), B^{d/p}_{p, 1}(\mathbb{T}^d))_{(1, b), q} \hookrightarrow (L_p(\mathbb{T}^d), L_\infty(\mathbb{T}^d))_{(1, b), q}.
\end{equation}
On the one hand, invoking Theorem \ref{TheoremInterpolation}(ii) (to be more precise, its periodic counterpart), we have
\begin{equation}\label{TheoremEmbedinngCriticalProof2}
	(B^0_{p, \min\{p, 1\}}(\mathbb{T}^d), B^{d/p}_{p, 1}(\mathbb{T}^d))_{(1, b), q} = T^{b+1/q}_q B^{d/p}_{p, 1}(\mathbb{T}^d). 	
\end{equation}
On the other hand, since (cf. \cite[Theorem 5.2.1]{BerghLofstrom})
\begin{equation}\label{KfunctLp}
	K(t, f; L_p(\mathbb{T}^d), L_\infty(\mathbb{T}^d)) \asymp \bigg(\int_0^{t^p} (f^*(u))^p \, du \bigg)^{1/p},
\end{equation}
we can apply basic monotonicity properties and Hardy's inequality \eqref{H2} so that
\begin{align*}
	\|f\|_{(L_p(\mathbb{T}^d), L_\infty(\mathbb{T}^d))_{(1, b), q}} & \asymp \bigg(\int_0^1 t^{-q/p} (1-\log t)^{b q} \bigg(\int_0^t (f^*(u))^p \, du \bigg)^{q/p}  \frac{dt}{t} \bigg)^{1/q} \\
	& \hspace{-3cm} \asymp \bigg(\sum_{j=0}^\infty 2^{j q/p} (1 + j)^{b q} \bigg(\sum_{k=j}^\infty (f^*(2^{-k}))^p 2^{-k} \bigg)^{q/p} \bigg)^{1/q} \\
	& \hspace{-3cm} \asymp \bigg(\sum_{j=0}^\infty  (1 + j)^{b q} (f^*(2^{-j}))^q \bigg)^{1/q} \\
	& \hspace{-3cm}\asymp \bigg(\int_0^1 ((1-\log t)^{b} f^*(t))^q \frac{dt}{t} \bigg)^{1/q} = \|f\|_{L_{\infty, q}(\log L)_b(\mathbb{T}^d)}.
\end{align*}
This estimate together with \eqref{TheoremEmbedinngCriticalProof1} and \eqref{TheoremEmbedinngCriticalProof2} enable us to conclude that
$$
	 T^{b+1/q}_q B^{d/p}_{p, 1}(\mathbb{T}^d) \hookrightarrow L_{\infty, q}(\log L)_b(\mathbb{T}^d).
$$

Next we concentrate on the sharpness assertion \eqref{SharpTheoremEmbedinngCritical}. If $u \leq 1$ then
$$
	T^{b+1/q}_q B^{d/p}_{p, u}(\mathbb{T}^d) \hookrightarrow T^{b+1/q}_q B^{d/p}_{p, 1}(\mathbb{T}^d)
$$
(cf. Theorem \ref{TheoremEmbeddingsBBCharacterization}) and thus the embedding given in \eqref{SharpTheoremEmbedinngCritical} follows trivially from \eqref{TheoremEmbedinngCritical1}. Conversely, we assume that
$$
		T^{b+1/q}_q B^{d/p}_{p, u}(\mathbb{T}^d) \hookrightarrow L_{\infty, q}(\log L)_b(\mathbb{T}^d)
$$
holds true and we will show that necessarily $u \leq 1$. Indeed, suppose that $u > 1$ and let
	\begin{equation}\label{TheoremEmbedinngCriticalProof3}
	 f = \sum_{j \in \N_0,G \in G^j,m \in \mathbb{P}^d_j} \lambda^{j,G}_m 2^{-(j + L) d/2}
    \Psi^{j, L, \text{per}}_{G,m}
    \end{equation}
	where
	\begin{equation}\label{TheoremEmbedinngCriticalProof4}
		\lambda^{j,G}_m = \left\{\begin{array}{cl}  (1 + j)^{-\varepsilon}, & \quad j \in \N_0, \quad m= (0, \ldots, 0), \quad G = (M, \ldots, M),  \\
		0, & \text{otherwise},
		       \end{array}
                        \right.
	\end{equation}
	and $b + \frac{1}{q} + \frac{1}{u} < \varepsilon < b + \frac{1}{q} +1$. According to the periodic counterpart of Theorem \ref{ThmWaveletsNewBesov}, we have
	\begin{align*}
	\|f\|_{	T^{b+1/q}_q B^{d/p}_{p, u}(\mathbb{T}^d) } &\asymp \bigg(\sum_{k=0}^\infty 2^{k (b + 1/q) q} \bigg(\sum_{j=2^k-1}^{2^{k+1}-2} (1 + j)^{-\varepsilon u} \bigg)^{q/u} \bigg)^{1/q} \\
	& \asymp \bigg(\sum_{k=0}^\infty 2^{k (b+1/q-\varepsilon + 1/u) q} \bigg)^{1/q} < \infty.
	\end{align*}
	Given $k \in \N_0$, by \eqref{Rearrengement}, \eqref{TheoremEmbedinngCriticalProof3} and \eqref{TheoremEmbedinngCriticalProof4},
	\begin{equation}\label{RearrangementTrick}
		f^*(2^{-k d}) \geq \inf_{x \in Q_{k, (0, \ldots, 0)}} |f(x)| \gtrsim \sum_{j=0}^k (1 + j)^{-\varepsilon}
	\end{equation}
	and thus
	\begin{align*}
		\|f\|_{L_{\infty, q}(\log L)_b(\mathbb{T}^d)} &\asymp \bigg(\sum_{k=0}^\infty ((1 + k)^b f^*(2^{-k d}))^q \bigg)^{1/q} \\
		& \gtrsim  \bigg(\sum_{k=0}^\infty (1 + k)^{b q} \bigg( \sum_{j=0}^k (1 + j)^{-\varepsilon} \bigg)^q \bigg)^{1/q} \\
		& \asymp \bigg(\sum_{k=0}^\infty (1 + k)^{b q -\varepsilon q + q} \bigg)^{1/q} = \infty.
	\end{align*}
	This yields the desired contradiction.
\end{proof}

\begin{rem}\label{RemarkRn}
	As already mentioned at the beginning of this section, Theorem \ref{TheoremEmbedinngCritical} is stated for periodic functions. In particular, we make use of the fact that the couple $(L_p(\mathbb{T}^d), L_\infty(\mathbb{T}^d))$ is ordered (i.e., $L_\infty(\mathbb{T}^d) \hookrightarrow L_p(\mathbb{T}^d)$) when we apply the limiting interpolation method in \eqref{TheoremEmbedinngCriticalProof1}; see Section \ref{Section:LimitingInterpolation}. Working with functions defined on $\R^d$, the proof of Theorem \ref{TheoremEmbedinngCritical} can be easily adapted to obtain the following local result: Let $0 < p < \infty, 0 < q, u \leq \infty$ and $b < -1/q$. Then
	$$
		\bigg(\int_0^1 ((1-\log t)^{b } f^*(t))^q \frac{dt}{t} \bigg)^{1/q} \lesssim \|f\|_{T^{b+1/q}_q B^{d/p}_{p, u}(\R^d)}
	$$
	if and only if $u \leq 1$.
\end{rem}

\textbf{Embedding into H\"older spaces: the case $b >-1/q$}.  In this section we study  $\BB_{\infty, q}^{0, b}(\R^d)$ spaces, which sometimes denoted by $\Lambda_{\infty, q}^{(1-\log t)^{-b}}(\R^d)$ and called H\"older spaces, see \cite{Haroske}, \cite{MouraNevesPiotrowski09}, \cite{MouraNevesSchneider11}, \cite{MouraNevesSchneider14}. Embedding theorems into these spaces have been  intensively studied in recent years (occasionally, under the name of continuity envelopes and even in the setting of generalized smoothness) in the above mentioned papers and the references within. In this direction, the best known embedding theorem reads as follows: let $0 < p < \infty, 0 < q \leq \infty$ and $b > -\frac{1}{q}$, then
\begin{equation}\label{ContinuityEnvelope1}
	B^{d/p, b + \frac{1}{\min\{1,q\}}}_{p, q}(\R^d) \hookrightarrow \BB_{\infty, q}^{0,b}(\R^d).
\end{equation}
Furthermore, the embedding \eqref{ContinuityEnvelope1} is optimal within the scale of classical spaces $B^{d/p, b}_{p, q}(\R^d).$

Next we show that \eqref{ContinuityEnvelope1} can be improved applying the spaces $T^b_r B^s_{p, q}(\R^d)$. To be more precise, we establish the following

\begin{thm}\label{TheoremEmbedinngCriticalContinuity}
	Let $0 < p, q \leq \infty$ and $b > -1/q$. Then
	\begin{equation}\label{TheoremEmbedinngCritical1Continuity}
		T^{b+1/q}_q B^{d/p}_{p, 1}(\R^d) \hookrightarrow \BB^{0,b}_{\infty, q}(\R^d).
	\end{equation}
	Furthermore, the embedding is optimal in the following sense: let $0 < u \leq \infty$, then
	\begin{equation}\label{SharpTheoremEmbedinngCriticalContinuity}
		T^{b+1/q}_q B^{d/p}_{p, u}(\R^d) \hookrightarrow \BB^{0,b}_{\infty, q}(\R^d) \iff u \leq 1.
	\end{equation}
\end{thm}

\begin{rem}
	In view of Theorem \ref{TheoremEmbeddings1}, one has
	$$
				B^{d/p, b + 1/\min\{1, q\}}_{p,q}(\R^d) + B^{d/p, b + 1/q}_{p, \min\{1, q\}}(\R^d) \hookrightarrow T^{b+1/q}_q B^{d/p}_{p, 1}(\R^d)
	$$
	and thus \eqref{TheoremEmbedinngCritical1Continuity} implies, in particular, \eqref{ContinuityEnvelope1}. Furthermore, the embedding $B^{d/p, b + 1/\min\{1, q\}}_{p,q}(\R^d) \hookrightarrow$ $T^{b+1/q}_q B^{d/p}_{p, 1}(\R^d)$ is strict if $q \neq 1$ (by Proposition \ref{PropositionCoincidences}, these spaces coincide of $q=1$). Indeed, counterexamples given in Remark \ref{Remark6.5} work.
\end{rem}

\begin{proof}[Proof of Theorem \ref{TheoremEmbedinngCriticalContinuity}]
	It follows from  \eqref{BasicEmbedding} (where $\mathbb{T}^d$ is now replaced by $\R^d$) that
	$$
		B^{d/p + 1}_{p, 1}(\R^d) \hookrightarrow W^1_\infty(\R^d).
	$$
	To be more precise, by Proposition \ref{SobProp},
	$$
		\|f\|_{W^1_\infty(\R^d)} \lesssim \|f\|_{B^{d/p}_{p, 1}(\R^d)} + \|\nabla f\|_{B^{d/p}_{p, 1}(\R^d)} \asymp \|f\|_{B^{d/p+1}_{p, 1}(\R^d)}.
	$$
	Applying limiting interpolation with $\theta =0$ (cf. \eqref{DefLimInterpolation}) we infer
	\begin{equation}\label{TheoremEmbedinngCriticalContinuityProof1}
		(B^{d/p}_{p, 1}(\R^d), B^{d/p + 1}_{p, 1}(\R^d))_{(0, b), q} \hookrightarrow (L_\infty(\R^d), W^1_\infty(\R^d))_{(0, b), q}.
	\end{equation}
	According to Theorem \ref{TheoremInterpolation}, the left-hand side space can be equivalently characterized as
	\begin{equation}\label{TheoremEmbedinngCriticalContinuityProof2}
		(B^{d/p}_{p, 1}(\R^d), B^{d/p + 1}_{p, 1}(\R^d))_{(0, b), q} = T^{b+1/q}_q B^{d/p}_{p, 1}(\R^d).
	\end{equation}
	On the other hand, the well-known formula (see e.g. \cite[Chapter 5, Theorem 4.12]{BennettSharpley})
	$$
	K(t^k, f; L_\infty(\R^d), W^k_\infty(\R^d)) \asymp t^k \|f\|_{L_\infty(\R^d)} + \omega_k(f, t)_\infty
	$$
	for $t \in (0, 1)$ implies
	\begin{equation}\label{TheoremEmbedinngCriticalContinuityProof3}
		(L_\infty(\R^d), W^1_\infty(\R^d))_{(0, b), q} = \BB^{0, b}_{\infty, q}(\R^d).
	\end{equation}
	 Putting together \eqref{TheoremEmbedinngCriticalContinuityProof1}--\eqref{TheoremEmbedinngCriticalContinuityProof3} we obtain
	$$
	T^{b+1/q}_q B^{d/p}_{p, 1}(\R^d) \hookrightarrow 	\BB^{0, b}_{\infty, q}(\R^d),
	$$
	i.e., \eqref{TheoremEmbedinngCritical1Continuity} holds.
	
	Conversely, we let
	\begin{equation}\label{TheoremEmbedinngCriticalContinuityProof4}
			T^{b+1/q}_q B^{d/p}_{p, u}(\R^d) \hookrightarrow \BB^{0,b}_{\infty, q}(\R^d)
	\end{equation}
	for some $0 < u \leq \infty$. Then necessarily $u \leq 1$. Indeed, assume momentarily that \eqref{TheoremEmbedinngCriticalContinuityProof4} holds for some $u > 1$. Define
	\begin{equation}\label{newnewnew}
f = \sum^\infty_{j=0} \sum_{G\in G^j} \sum_{m \in \mathbb{Z}^d}
\lambda^{j,G}_m \, 2^{-j d/2} \, \Psi^j _{G,m},
\end{equation}
where
	\begin{equation}\label{newnewnew2}
		\lambda^{j,G}_m = \left\{\begin{array}{cl}  (1 + j)^{-\varepsilon}, & \quad j \in \N_0, \quad m= (0, \ldots, 0), \quad G = (M, \ldots, M),  \\
		0, & \text{otherwise},
		       \end{array}
                        \right.
	\end{equation}
	and $\max\{b+ \frac{1}{q} + \frac{1}{u}, 1 \} < \varepsilon < b + \frac{1}{q} + 1$. By Theorem \ref{ThmWaveletsNewBesov},
	$$
		\|f\|_{T^{b+1/q}_q B^{d/p}_{p, u}(\R^d)} \asymp \bigg(\sum_{k=0}^\infty 2^{k (b+1/q-\varepsilon +1/u) q} \bigg)^{1/q} < \infty.
	$$
	We may assume, without loss of generality, that $\psi_M (0) = \max_{z \in \R} \psi_M(z)$ (in particular, $\psi_M(0) > 0$) and there exists $c \in (0, 1)$ (depending on $\psi_M$) such that $\psi_M(z) \leq c \,  \psi_M(0)$ for all $z \geq 2$. Let $e_1 = (1, 0, \ldots, 0) \in \R^d$. Given $k \in \N_0$, we can estimate $\omega_1(f,2^{-k})_\infty$ (cf. \eqref{DefModuli}) as follows (cf. \eqref{DefinitionWavelet})
	\begin{align*}
		\omega_1(f,2^{-k})_\infty & \geq |f(0) - f(2^{-k} e_1)| \\
		& = \bigg| \sum^\infty_{j=0} (1+ j)^{-\varepsilon} \, (\psi_{M} (0))^d  - \sum^\infty_{j=0} (1 + j)^{-\varepsilon} \, \psi_{M}(2^{j-k}) (\psi_{M} (0))^{d-1} \bigg| \\
		& =  (\psi_{M} (0))^{d-1} \sum_{j=0}^\infty (1+j)^{-\varepsilon}  (\psi_M(0) - \psi_{M}(2^{j-k})) \\
		& \geq (\psi_{M} (0))^{d-1} \sum_{j=k+1}^\infty (1+j)^{-\varepsilon}  (\psi_M(0) - \psi_{M}(2^{j-k})) \\
		& \geq (1-c)  (\psi_{M} (0))^{d} \sum_{j=k+1}^\infty (1+j)^{-\varepsilon} \\
		& \asymp (1+k)^{-\varepsilon + 1}.
	\end{align*}
	Therefore, one can estimate
	\begin{align*}
		\|f\|_{\BB^{0,b}_{\infty, q}(\R^d)} & \gtrsim \bigg(\sum_{k=0}^\infty \big((1+k)^{b} \omega_1(f,2^{-k})_\infty\big)^q  \bigg)^{1/q} \\
		& \gtrsim \bigg(\sum_{k=0}^\infty (1+k)^{b q -\varepsilon  q+ q} \bigg) = \infty.
	\end{align*}
	This yields the desired contradiction of \eqref{TheoremEmbedinngCriticalContinuityProof4} if $b > -1/q$.
\end{proof}

\begin{rem}
	It remains to investigate the borderline setting $b=-1/q$ in Theorem \ref{TheoremEmbedinngCriticalContinuity}. In this case, the method of proof proposed for Theorem \ref{TheoremEmbedinngCriticalContinuity} still works, but the expected spaces $T_q B^{d/p}_{p, 1}(\R^d)$ do not arise. Instead, the smaller spaces $T^*_q B^{d/p}_{p, 1}(\R^d)$ (cf. \eqref{33}) play a key role. To be more precise, if $0< p, q < \infty$ and $0 < u \leq \infty$ then
	$$
	T^*_q B^{d/p}_{p, u}(\R^d) \hookrightarrow \BB^{0, -1/q}_{\infty, q}(\R^d)	 \iff u \leq 1.
	$$
	The proof of the implication $\Rightarrow$ follows similar ideas as in Theorem \ref{TheoremEmbedinngCriticalContinuity} but now relying on Remark \ref{RemDelicate}. More precisely, consider $f$ given by \eqref{newnewnew} with \eqref{newnewnew2} and $\varepsilon = 1$.
	\end{rem}

\subsection{Sharp embeddings into $L_\infty$} It is the well-known fact that the Sobolev space $W^{d/p}_p(\R^d), \, p > 1,$ is not formed by bounded functions. As already mentioned in Section \ref{SectionEmbCritical}, this defect can be overcome if we replace $L_\infty(\R^d)$ by the slightly larger exponential class (cf. \eqref{Trudinger}). Another alternative is to fix $L_\infty(\R^d)$ as a target space and restrict the domain space $W^{d/p}_p(\R^d)$ to the smaller Lorentz-Sobolev space $W^{d/p} L_{p, 1}(\R^d)$ (cf. \cite{Stein81} for a precise statement; see also \cite{DeVoreSharpley}). Next we propose a third methodology in terms of Besov spaces $\mathbf{B}^{0, b}_{p, q}(\R^d)$ using the embeddings for truncated Triebel-Lizorkin spaces.

\begin{thm}
Let $1 < p < \infty$ and $d/p \in \N$. Then
		$$
		\|f\|_{L_\infty(\R^d)} \lesssim  \|f\|_{L_p(\R^d)} + \sum_{l=1}^d  \bigg\|\frac{\partial^{d/p} f}{\partial x_l^{d/p}} \bigg\|_{\mathbf{B}_{p, 1}^{0, -1/p}(\R^d)}.
	$$
\end{thm}
\begin{proof}
	According to Theorem \ref{ThmCF},
	\begin{equation}\label{App1}
		T^{1/p'}_1 F^{d/p}_{p, 2} (\R^d) \hookrightarrow L_\infty(\R^d).
	\end{equation}
	Furthermore, by Propositions \ref{SobProp2} and \ref{PropositionCoincidences},
	\begin{align}
		\|f\|_{T^{1/p'}_1 F^{d/p}_{p, 2} (\R^d)} &\asymp \|f\|_{T^{1/p'}_1 F^{0}_{p, 2} (\R^d)} +  \sum_{l=1}^d  \bigg\|\frac{\partial^{d/p} f}{\partial x_l^{d/p}} \bigg\|_{T^{1/p'}_1 F^{0}_{p, 2} (\R^d)} \nonumber\\
		& \asymp \|f\|_{\mathbf{B}^{0, -1/p}_{p, 1}(\R^d)} + \sum_{l=1}^d  \bigg\|\frac{\partial^{d/p} f}{\partial x_l^{d/p}} \bigg\|_{\mathbf{B}^{0, -1/p}_{p, 1}(\R^d)}. \label{App2}
	\end{align}
	It follows from \eqref{App1} and \eqref{App2} that
	\begin{equation*}
		\|f\|_{L_\infty(\R^d)} \lesssim \|f\|_{\mathbf{B}^{0, -1/p}_{p, 1}(\R^d)} + \sum_{l=1}^d  \bigg\|\frac{\partial^{d/p} f}{\partial x_l^{d/p}} \bigg\|_{\mathbf{B}^{0, -1/p}_{p, 1}(\R^d)}.
	\end{equation*}
	Thus the proof will be finished once we show
	\begin{equation}\label{App3}
		\|f\|_{\mathbf{B}^{0, -1/p}_{p, 1}(\R^d)} \lesssim \|f\|_{L_p(\R^d)}   + \sum_{l=1}^d  \bigg\|\frac{\partial^{d/p} f}{\partial x_l^{d/p}} \bigg\|_{\mathbf{B}^{0, -1/p}_{p, 1}(\R^d)}.
	\end{equation}
	Indeed, according to \eqref{BesovDifDef} (with $k = d/p$)
	$$
		\|f\|_{\mathbf{B}^{0, -1/p}_{p, 1}(\R^d)} \asymp \|f\|_{L_p(\R^d)}  + \int_0^1 (1-\log t)^{-1/p} \omega_{d/p}(f, t)_p \frac{dt}{t}
	$$
	and the estimates $\omega_{d/p}(f, t)_p \lesssim t^{d/p} \|f\|_{W^{d/p}_p(\R^d)}$ (cf. \eqref{LipSobFract}), we get via \eqref{Der0}
	\begin{align*}
		\|f\|_{\mathbf{B}^{0, -1/p}_{p, 1}(\R^d)}& \lesssim \|f\|_{W^{d/p}_p(\R^d)} \asymp \|f\|_{L_p(\R^d)} + \sum_{l=1}^d  \bigg\|\frac{\partial^{d/p} f}{\partial x_l^{d/p}} \bigg\|_{L_p(\R^d)} \\
		& \leq \|f\|_{L_p(\R^d)} + \sum_{l=1}^d  \bigg\|\frac{\partial^{d/p} f}{\partial x_l^{d/p}} \bigg\|_{\mathbf{B}_{p, 1}^{0, -1/p}(\R^d)}.
	\end{align*}
	Hence \eqref{App3} holds.
\end{proof}

\subsection{Sobolev-type embeddings in the subcritical case} In this section we investigate optimal Sobolev-type embeddings in the sub-critical case $\sigma_p \leq s < \frac{d}{p}$. Recall that $\sigma_p = d (\frac{1}{p}-1)_+$. This question
is nowadays completely understood for Besov and Triebel--Lizorkin spaces (with generalized smoothness) when $\sigma_p < s < \frac{d}{p}$, see \cite{DeVore}, \cite{BricchiMoura}, \cite{CaetanoMoura04}, \cite{CaetanoFarkas}, \cite{CaetanoLeopold}, \cite{Haroske}, \cite{Martin}, \cite{MouraNevesPiotrowski} and the references within.

In the borderline case $s= \sigma_p$ and classical spaces $A^s_{p, q}(\R^d), \, A \in \{B, F\}$,  the optimality (in terms of growth envelopes) of the related Sobolev inequalities may be found in \cite{Vybiral} (with \cite{Haroske} as a forerunner). However, to the best of our knowledge, the corresponding extension to generalized smoothness (even for the prototype given by  logarithmic smoothness) remains as an open problem. This question was explicitly stated by Triebel \cite{Triebel12}, where first results in this direction were obtained. Before we proceed further, we recall that $A^{\sigma_p, b}_{p, q}(\R^d)$ is formed by distributions, but not necessarily regular distributions (cf. Theorems \ref{TheoremRegular} and \ref{TheoremRegularFspaces}). In particular, if $1 < p < \infty$ then
$$
	B^{0, b}_{p, q}(\R^d) \hookrightarrow L_1^{\text{\text{loc}}}(\R^d) \iff     \left\{\begin{array}{lcl}
                            b \geq 0 & \text{ if }  & 0 < q \leq
                            \min\{2,p\}, \\
                            b > 1/p - 1/q & \text{ if } & 1 < p \leq
                            2 \quad  \text{and} \quad p < q \leq \infty, \\
                            b > 1/2 -1/q & \text{ if } & 2 < p <
                            \infty \quad \text{and} \quad 2 < q \leq \infty.
            \end{array}
            \right.
$$
Therefore, under the above assumptions on the parameters, it makes sense to study Sobolev-type embeddings related to $B^{0, b}_{p, q}(\R^d)$. To the best of our knowledge, the sharpest result so far appears in \cite{Dominguez17} (in particular, improving earlier results in \cite{Triebel12}). Namely, if $1 < p < \infty, 0 < q \leq \infty$ and $b > \frac{1}{\min\{2, p, q\}}-\frac{1}{q}$ then
\begin{equation}\label{Dominguez}
	\|f\|_{L^{(p, b- \frac{1}{\min\{2, p, q\}}, q}(0,1)} \lesssim \|f\|_{B^{0, b}_{p, q}(\R^d)}.
\end{equation}
Here, for $0 < p < \infty, 0 < q \leq \infty$ and $b \geq -1/q \, (b > 0 \text{ if } q=\infty)$, by $L^{(p, b, q}(0,1)$ we mean the \emph{small Lebesgue space} equipped with\index{\bigskip\textbf{Spaces}!$L^{(p, b, q}(0, 1)$}\label{SMALLLEB}
\begin{equation}\label{DefSmallLebesgueSpaces}
	\|f\|_{L^{(p, b, q}(0,1)} := \bigg(\int_0^1 (1-\log t)^{b q} \bigg(\int_0^t (f^*(u))^p \, du \bigg)^{q/p} \frac{dt}{t} \bigg)^{1/q}
\end{equation}
(where the usual change is made if $q=\infty$); see the recent survey paper \cite{Fiorenza} and the extensive list of references given there.  Note that $b \geq -1/q$ is natural; otherwise, the space $L^{(p, b, q}(0,1)$ coincides with $L_p(0, 1)$. Clearly $L^{(p, b, p}(0, 1) = L_p(\log L)_{b + 1/p}(0,1)$ if $b > -1/p$, but in general the spaces $L^{(p, b, q}(0, 1)$ with $p \neq q$ does not fit into the scale of Lorentz--Zygmund spaces \eqref{DefLZ}. However, there are some known relations between these two scales of function spaces, e.g.,
$$
L_{p, q}(\log L)_{b + 1/\min\{p, q\}}(0, 1) \hookrightarrow 	L^{(p, b, q}(0,1) \hookrightarrow L_{p, q}(\log L)_{b + 1/\max\{p, q\}}(0, 1)
$$
provided that $0 < p < \infty, 0 < q \leq \infty$ and $b > -1/q$. Applying these relations, one immediately gets from \eqref{Dominguez} (and under the same assumptions on the parameters as there) that
$$
	\|f\|_{L_{p, q} (\log L)_{b - \frac{1}{\min\{2, p, q\}} + \frac{1}{\max\{p, q\}} }(0,1)} \lesssim \|f\|_{B^{0, b}_{p, q}(\R^d)}
$$
 (cf. \cite[Corollary 3.2]{Dominguez17}). However, the methodology in \cite{Dominguez17}, which is based on extrapolation techniques, is not strong enough to achieve the optimality of \eqref{Dominguez} in the full range of parameters. Below, we propose a novel methodology relying on the spaces $T^b_r B^s_{p, q}(\R^d)$ which gives an answer to the Triebel's question and, in particular, improves \eqref{Dominguez}.

 Before we state our result, it is worthwhile to mention that the corresponding question for $\BB^{0, b}_{p, q}(\R^d)$ (cf. \eqref{BesovDifDef}) has been  now completely solved as a combination of the papers \cite{CaetanoGogatishvili, CaetanoGogatishvili11, Dominguez16, NevesOpic} and the references given there. In this regard, let us emphasize that $\BB^{0, b}_{p, q}(\R^d) \neq B^{0, b}_{p, q}(\R^d)$ and, in fact, there are substantial distinctions between these two Besov spaces of smoothness near zero. A detailed account on this point may be found in \cite{DominguezTikhonov}.

 \begin{thm}\label{SobolevTheoremSubcritical}
 	Let $1 < p < \infty, 0 < q \leq \infty$ and $b > -1/q$. Then
	\begin{equation}\label{SobolevTheoremSubcritical2123}
		\|f\|_{L^{(p, b, q}(0, 1)} \lesssim \|f\|_{T^{b+1/q}_q B^{0}_{p, \min\{2, p\}}(\R^d)}.
	\end{equation}
	The embedding is optimal in the following sense. Let $0 < u \leq \infty$. If $p \leq 2$ then
	\begin{equation}\label{SobolevTheoremSubcritical4324}
		\|f\|_{L^{(p, b, q}(0, 1)} \lesssim \|f\|_{T^{b+1/q}_q B^{0}_{p, u}(\R^d)} \iff u \leq p.
	\end{equation}
	Moreover,  if $2 < \min \{p, u\}$ then the inequality given in \eqref{SobolevTheoremSubcritical4324} is not true whenever one of the following conditions holds:
	\begin{enumerate}[\upshape(i)]
	\item $b \in [-\frac{1}{q}, -\frac{1}{q} + \frac{1}{2}-\frac{1}{u})$,
	\item $b=-\frac{1}{q} + \frac{1}{2}-\frac{1}{u}$ \qquad and \qquad $q > 2$.
	\end{enumerate}
 \end{thm}

\begin{rem}
	Note that \eqref{SobolevTheoremSubcritical2123} gives a non-trivial improvement of \eqref{Dominguez} since
	$$
		B^{0, b + 1/\min\{2, p, q\}}_{p, q}(\R^d) + B^{0, b+1/q}_{p, \min\{2, p, q\}}(\R^d) \hookrightarrow  T^{b+1/q}_q B^{0}_{p, \min\{2, p\}}(\R^d)
	$$
	provided that $b > -1/q$ (cf. Theorem \ref{TheoremEmbeddings1}) and one can construct $f$ such that
	\begin{equation}\label{Claim}
		f \in T^{b+1/q}_q B^{0}_{p, \min\{2, p\}}(\R^d) \quad \text{but} \quad f \not \in B^{0, b + 1/\min\{2, p, q\}}_{p, q}(\R^d)
	\end{equation}
	whenever $q \neq \min\{2, p\}$. Indeed, if $q < \min\{2, p\}$ then we let
		\begin{equation*}
f = \sum^\infty_{j=0} \sum_{G\in G^j} \sum_{m \in \mathbb{Z}^d}
\lambda^{j,G}_m \, 2^{-j d/2} \, \Psi^j _{G,m}
\end{equation*}
where
	\begin{equation*}
		\lambda^{j,G}_m = \left\{\begin{array}{cl}  2^{j d/p} (1 + j)^{-\varepsilon}, & \quad j \in \N_0, \quad m= (0, \ldots, 0), \quad G = (M, \ldots, M),  \\
		0, & \text{otherwise},
		       \end{array}
                        \right.
	\end{equation*}
	and $b + \frac{1}{\min\{2, p\}} + \frac{1}{q} < \varepsilon < b + \frac{2}{q}$. Applying Theorems \ref{ThmWaveletsBesovClassic} and \ref{ThmWaveletsNewBesov}, it is plain to see that $f$ satisfies the desired claim \eqref{Claim}. On the other hand, if $q > \min \{2, p\}$ then we consider $f$ as given above with
    	\begin{equation*}
\lambda^{j,G}_m = \left\{\begin{array}{cl}   2^{\frac{2^k d}{p}} 2^{-k \varepsilon}, & \quad j = 2^k, \quad k \in \N_0, \quad m= (0, \ldots, 0),  \\
		& \qquad G = (M, \ldots, M),  \\
		0, & \text{otherwise},
		       \end{array}
                        \right.
	\end{equation*}
	where $b + \frac{1}{q} < \varepsilon < b + \frac{1}{\min\{2, p\}}$.
\end{rem}

\begin{proof}[Proof of Theorem \ref{SobolevTheoremSubcritical}]
	We start by showing \eqref{SobolevTheoremSubcritical2123}. We wish to interpolate by the limiting method with $\theta = 0$ the well-known embeddings
	\begin{equation}\label{ProofSobolevTheoremSubcritical1}
		B^0_{p, \min\{p, 2\}}(\R^d) \hookrightarrow L_p(\R^d) \qquad \text{and} \qquad B^{d/p}_{p, 1}(\R^d) \hookrightarrow L_\infty(\R^d).
	\end{equation}
	However, when we try to accomplish this task, a first technical issue already appears, namely, the limiting interpolation method $(A_0, A_1)_{(0, b), q}$ as given in \eqref{DefLimInterpolation} is only defined for ordered couples  and the particular choice $(A_0, A_1)=(L_p(\R^d), L_\infty(\R^d))$ is not ordered. To overcome this obstruction, we shall modify the original definition of the limiting interpolation \eqref{DefLimInterpolation} to make it available for general couples of quasi-Banach spaces. Specifically, let $(A_0, A_1)$ be a quasi-Banach couple (not necessarily ordered) and let $0 < q \leq \infty$ and $b, \eta \in \R$, we introduce the limiting interpolation space $(A_0, A_1)_{(0,(b, \eta)),q}$ as the set of all those $f \in A_0 + A_1$ such that\index{\bigskip\textbf{Spaces}!$(A_0, A_1)_{(0, (b, \eta)), q}$}\label{LIMINTBROKE}
	\begin{equation}\label{ProofSobolevTheoremSubcritical2}
	\|f\|_{(A_0,A_1)_{(0,(b, \eta)),q}} := \bigg(\int_0^1 ((1-\log t)^{b} K(t,f))^q \frac{dt}{t} \bigg)^{1/q} +  \bigg(\int_1^\infty ((1+\log t)^{\eta} K(t,f))^q \frac{dt}{t} \bigg)^{1/q}
\end{equation}
is finite (where the usual change has to be made if $q=\infty$). Here we shall assume $\eta < -1/q$ so that $(A_0, A_1)_{(0,(b, \eta)),q} \neq \{0\}$. Note that working with ordered couples $(A_0, A_1)$, since $K(t, f) \asymp \|f\|_{A_0} \asymp K(1,f)$ for $t > 1$, we have
\begin{align*}
	\|f\|_{(A_0,A_1)_{(0,(b, \eta)),q}} &\asymp \bigg(\int_0^1 ((1-\log t)^{b} K(t,f))^q \frac{dt}{t} \bigg)^{1/q} + K(1,f) \\
	& \asymp \bigg(\int_0^1 ((1-\log t)^{b} K(t,f))^q \frac{dt}{t} \bigg)^{1/q}
	\end{align*}
	and thus one recovers the original definition (cf. \eqref{DefLimInterpolation}).
	
	Let $\eta < -1/q$. Applying the limiting interpolation method \eqref{ProofSobolevTheoremSubcritical2} to \eqref{ProofSobolevTheoremSubcritical1} we get
	\begin{equation}\label{ProofSobolevTheoremSubcritical3}
		(B^0_{p, \min\{p, 2\}}(\R^d), B^{d/p}_{p, 1}(\R^d))_{(0, (b, \eta)), q} \hookrightarrow (L_p(\R^d), L_\infty(\R^d))_{(0, (b, \eta)), q}.
	\end{equation}
	Next we compute these interpolation spaces. On the one hand, it follows from \eqref{DefLimInterpolation} and  \eqref{KfunctLp} that
	\begin{align}
	\|f\|_{(L_p(\R^d), L_\infty(\R^d))_{(0, (b, \eta)), q}} &\gtrsim \bigg(\int_0^1 (1-\log t)^{b q}  \bigg(\int_0^{t} (f^*(u))^p \, du \bigg)^{q/p} \frac{dt}{t} \bigg)^{1/q} \nonumber \\
	& = \|f\|_{L^{(p, b, q}(0,1)}. \label{ProofSobolevTheoremSubcritical4}
	\end{align}
	On the other hand, since $B^{d/p}_{p, 1}(\R^d)  \hookrightarrow B^0_{p, \min\{p, 2\}}(\R^d)$, it follows from Theorem \ref{TheoremInterpolation} that
	\begin{align}
		(B^0_{p, \min\{p, 2\}}(\R^d), B^{d/p}_{p, 1}(\R^d))_{(0, (b, \eta)), q}  &= (B^0_{p, \min\{p, 2\}}(\R^d), B^{d/p}_{p, 1}(\R^d))_{(0, b), q} \nonumber  \\
		& =T^{b+1/q}_q B^{0}_{p, \min\{p, 2\}}(\R^d). \label{ProofSobolevTheoremSubcritical5}
	\end{align}
	As a combination of \eqref{ProofSobolevTheoremSubcritical3}-\eqref{ProofSobolevTheoremSubcritical5} we establish
	$$
		\|f\|_{L^{(p, b, q}(0,1)} \lesssim \|f\|_{T^{b+1/q}_q B^{0}_{p, \min\{p, 2\}}(\R^d)}.
	$$
	
	Next we deal with \eqref{SobolevTheoremSubcritical4324}. More precisely, we will show that if $p \leq 2$ and the inequality
	\begin{equation}\label{ProofSobolevTheoremSubcritical6}
		\|f\|_{L^{(p, b, q}(0, 1)} \lesssim \|f\|_{T^{b+1/q}_q B^{0}_{p, u}(\R^d)}
	\end{equation}
	holds, then necessarily $u \leq p$. To do this, we shall assume $u > p$ and then we will arrive at a contradiction. Indeed, we let
	\begin{equation*} 
f = \sum^\infty_{j=0} \sum_{G\in G^j} \sum_{m \in \mathbb{Z}^d}
\lambda^{j,G}_m \, 2^{-j d/2} \, \Psi^j _{G,m}
\end{equation*}
where
	\begin{equation*}
		\lambda^{j,G}_m = \left\{\begin{array}{cl}  2^{j d/p} (1 + j)^{-\varepsilon}, & \quad j \in \N_0, \quad m= (0, \ldots, 0), \quad G = (M, \ldots, M),  \\
		0, & \text{otherwise},
		       \end{array}
                        \right.
	\end{equation*}
	and $\max\{b + \frac{1}{q} + \frac{1}{u}, \frac{1}{p} \} < \varepsilon < b + \frac{1}{q}+\frac{1}{p}$ (recall $b > -1/q$). By Theorem \ref{ThmWaveletsNewBesov}, we have
	\begin{align*}
		\|f\|_{T^{b+1/q}_q B^{0}_{p, u}(\R^d)} &\asymp  \bigg(\sum_{k=0}^\infty 2^{k (b+1/q) q} \bigg(\sum_{j=2^{k}-1}^{2^{k+1}-2}  (1+j)^{-\varepsilon u}  \bigg)^{q/u} \bigg)^{1/q} \\
		& \asymp \bigg(\sum_{k=0}^\infty 2^{k(b + 1/q  -\varepsilon  + 1/u)q} \bigg)^{1/q} < \infty.
	\end{align*}
	On the other hand, by basic monotonicity properties and \eqref{RearrangementTrick}, we find that
	\begin{align*}
		\|f\|_{L^{(p, b, q}(0, 1)} &  \asymp \bigg(\sum_{k=0}^\infty (1+k)^{b q} \bigg(\sum_{j=k}^\infty (f^*(2^{-j d}))^p 2^{-j d} \bigg)^{q/p} \bigg)^{1/q} \\
		& \gtrsim \bigg(\sum_{k=0}^\infty (1+k)^{b q} \bigg(\sum_{j=k}^\infty (2^{j d/p} (1+j)^{-\varepsilon})^p 2^{-j d} \bigg)^{q/p} \bigg)^{1/q} \\
		& \asymp \bigg(\sum_{k=0}^\infty (1+k)^{b q - \varepsilon q + q/p} \bigg)^{1/q} = \infty.
	\end{align*}
	
	
	It remains to show the optimality of \eqref{ProofSobolevTheoremSubcritical6} in the range  $2 < \min\{p, u\}$. Next we prove that \eqref{ProofSobolevTheoremSubcritical6} fails to be true if one of the conditions (i) or (ii) holds.

	Let $\psi \in \mathcal{S}(\R^d) \backslash \{0\}$ be a fixed function with \eqref{PsiDef} and let $f$ be the function whose Fourier series is lacunary of the form  \eqref{LacunaryFSDef}. According to \eqref{ProofSobolevTheoremSubcritical7},
	\begin{equation}\label{ProofSobolevTheoremSubcritical72}
		     \|f\|_{T^{b+1/q}_q B^{0}_{p,u}(\mathbb{R}^d)}  = \|\psi\|_{L_p(\R^d)} \left(\sum_{j=3}^\infty 2^{j (b+1/q) q} \bigg(\sum_{\nu=2^j-1}^{2^{j+1}-2} |\lambda_\nu|^u\bigg)^{q/u}\right)^{1/q}.
	\end{equation}
	
	On the other hand, by monotonicity properties and the Hardy--Littlewood inequality for rearrangements (cf. \cite[Chapter 2, Lemma 2.1, p. 44]{BennettSharpley}), we can estimate
	\begin{align*}
		 \|f\|_{L^{(p, b, q}(0, 1)} &=  \bigg(\int_0^1 t^{q/p}(1-\log t)^{b q} \bigg( \frac{1}{t}\int_0^t (f^*(u))^p \, du \bigg)^{q/p} \frac{dt}{t} \bigg)^{1/q} \\
		 & \gtrsim \bigg(\int_0^1 (f^*(t))^p \, dt \bigg)^{1/p} \\
		 &\geq \sup_{\substack{E \subset \R^d \\ |E| \leq 1}} \bigg(\int_E |f(x)|^p \, dx \bigg)^{1/p}
	\end{align*}
	which implies, by the Zygmund property for lacunary Fourier series (see, e.g., \cite[Theorem 3.7.4]{Grafakos}), that
	\begin{equation}\label{ProofSobolevTheoremSubcritical8}
	 \|f\|_{L^{(p, b, q}(0, 1)} \gtrsim \bigg(\sum_{j=3}^\infty |\lambda_j|^2 \bigg)^{1/2}
	\end{equation}
	where the implicit constant depends on $\psi$.
	
	It follows from \eqref{ProofSobolevTheoremSubcritical6}--\eqref{ProofSobolevTheoremSubcritical8} that
	\begin{equation}\label{ProofSobolevTheoremSubcritical9}
	\bigg(\sum_{j=3}^\infty |\lambda_j|^2 \bigg)^{1/2} \lesssim   \left(\sum_{j=3}^\infty 2^{j (b+1/q) q} \bigg(\sum_{\nu=2^j-1}^{2^{j+1}-2} |\lambda_\nu|^u\bigg)^{q/u}\right)^{1/q}
	\end{equation}
	for all $\{\lambda_j\}_{j \in \N}$. However, in general, this inequality fails to be true. If $b \in [-\frac{1}{q}, -\frac{1}{q} + \frac{1}{2}-\frac{1}{u})$ (i.e., (i) holds), we consider the sequence $\lambda_j = (1+j)^{-\varepsilon}$ where $b + \frac{1}{u}+\frac{1}{q} < \varepsilon < \frac{1}{2}$.  Therefore the right-hand side of \eqref{ProofSobolevTheoremSubcritical9} is finite since
	$$
	   \sum_{j=3}^\infty 2^{j (b+1/q) q} \bigg(\sum_{\nu=2^j-1}^{2^{j+1}-2} (1+\nu)^{-\varepsilon u}\bigg)^{q/u} \asymp \sum_{j=3}^\infty 2^{j (b+1/q-\varepsilon + 1/u) q} < \infty
	$$
	but the corresponding left-hand side can be estimated as
	$$
	\sum_{j=3}^\infty |\lambda_j|^2 = \sum_{j=3}^\infty (1+j)^{-2 \varepsilon} = \infty.
	$$
	If $b = -\frac{1}{q} + \frac{1}{2}-\frac{1}{u}$ and $q > 2$ (i.e., the condition (ii) is fulfilled) then we can take $\lambda_j = (1+j)^{-1/2} (1+ \log (1+j))^{-\beta}$ with $\beta \in (\frac{1}{q}, \frac{1}{2})$ in the above argument.
\end{proof}

\begin{rem}
	One can also deal with the limiting value $b=-1/q$ in Theorem \ref{SobolevTheoremSubcritical}, after replacing the space $T_q B^0_{p, \min\{2, p\}}(\R^d)$ by the smaller one $T^*_q B^0_{p, \min\{p, 2\}}(\R^d)$ (cf. \eqref{33}).
\end{rem}

\subsection{Bourgain--Brezis--Mironescu formulas for Besov and Triebel--Lizorkin spaces}
 A well-known defect of the fractional Sobolev seminorms\index{\bigskip\textbf{Spaces}!$\dot{W}^{s, p}(\R^d)$}\label{HOMFRACTSOB}
 \begin{equation*}
		\|f\|_{\dot{W}^{s, p}(\R^d)} := \ \bigg(\int_{\R^d} \int_{\R^d} \frac{|f(x)-f(y)|^p}{|x-y|^{d + s p}} \, dx \, dy \bigg)^{1/p}, \quad s \in (0,1), \quad p \in [1, \infty),
\end{equation*}
(see also \eqref{GagliardoNorm535353}) is that they do not converge to the classical Sobolev seminorm $\|\nabla f\|_{L_p(\R^d)}$ as $s \to 1-$. This defect can be fixed under a certain normalization of the seminorms $\|\cdot\|_{\dot{W}^{s, p}(\R^d)}$. Namely, a celebrated result by Bourgain--Brezis--Mironescu \cite{Bourgain} asserts that
\begin{equation}\label{BBM}
	\lim_{s \to 1-} (1-s)^{1/p} \|f\|_{\dot{W}^{s, p}(\R^d)} = c_{d, p} \, \|\nabla f\|_{L_p(\R^d)}
\end{equation}
for $p \in (1, \infty)$. Here, the explicit value of the constant $c_{d, p}$ is known, depending only on $d$ and $p$. This result has been extended in many different ways and in several contexts, e.g., Milman showed in \cite{Milman} that \eqref{BBM} is a special case of a more general phenomenon based on interpolation theory and this method can be applied, in particular, to extend \eqref{BBM} to Besov seminorms of higher order  (cf. \cite{KaradzhovMilmanXiao}).

The aim of this section is to show that the family of truncated norms on Besov and Triebel--Lizorkin spaces  satisfy the Bourgain--Brezis--Mironescu phenomenon, in the sense that under a certain normalization, one can attain the classical Besov norm $\|\cdot\|_{B^s_{p,q}(\R^d)}$ via limits. In this regard, it is important  to mention that the chosen norms on $T^b_r B^s_{p, q}(\R^d)$ and $T^b_r F^s_{p, q}(\R^d)$ play a key role. Indeed, there are equivalence constants depending on the involved parameters when comparing different quasi-norms. In this case, the correct normalizations of truncated norms are provided by $\|\cdot\|^*_{T^b_r B^{s}_{p, q }(\R^d)}$ and $\|\cdot\|^*_{T^b_r F^{s}_{p, q }(\R^d)}$ (cf. Proposition \ref{PropEquiQN}).

\begin{thm}\label{ThmLimitsBBM}
	Let $0 < p, q\leq \infty, 0 < r < \infty$ and $-\infty < s < \infty$. Then, for every $f \in \mathcal{S}(\R^d)$,
	\begin{equation}\label{BBMNewSpaces}
		\lim_{b \to 0-}  (1-2^{b r})^{1/r} \, \|f\|^*_{T^b_r B^{s}_{p, q }(\R^d)} = \|f\|_{B^s_{p, q}(\R^d)}
	\end{equation}
	and if, additionally, $p < \infty$ then
		\begin{equation}\label{BBMNewSpacesTL}
		\lim_{b \to 0-}  (1-2^{b r})^{1/r} \, \|f\|^*_{T^b_r F^{s}_{p, q }(\R^d)} = \|f\|_{F^s_{p, q}(\R^d)}.
	\end{equation}
\end{thm}

\begin{rem}	
	The philosophy behind the existence of the limits in \eqref{BBMNewSpaces} and \eqref{BBMNewSpacesTL} relies on the facts that
	\begin{equation}\label{1757}
		 \sum_{j=0}^\infty \bigg(\sum_{\nu= 0}^{2^j} 2^{\nu s q} \|(\varphi_\nu \widehat{f})^\vee\|_{L_p(\mathbb{R}^d)}^q\bigg)^{r/q} < \infty \implies f=0
	\end{equation}
	and
	\begin{equation}\label{1758}
		\sum_{j=0}^\infty  \bigg\| \bigg(\sum_{\nu= 0}^{2^j} 2^{\nu s q} |(\varphi_\nu \widehat{f})^\vee|^q \bigg)^{1/q} \bigg\|_{L_p(\R^d)}^r < \infty \implies f=0
	\end{equation}
(cf. Remark \ref{Remark36}). Thus, given a non-zero function $f$, the blow-up of the functionals in \eqref{1757} and \eqref{1758} can be compensated with the decay given by the factor $1-2^{b r}$ as $b \to 0-$. Note that a similar phenomenon can not be expected for $b \to 0+$, since $T^*_r B^s_{p, q}(\R^d)$ and $T^*_r F^s_{p, q}(\R^d)$ are non-trivial spaces, cf. Remark \ref{Remark36}.
\end{rem}

\begin{proof}[Proof of Theorem \ref{ThmLimitsBBM}]
	We start by proving \eqref{BBMNewSpaces}. Let $\varepsilon > 0$. Accordingly, there exists $j_0 \in \N$ such that, for every $j \geq j_0$,
	\begin{equation}\label{ThmLimitsBBMProof1}
		\|f\|_{B^s_{p, q}(\R^d)}^r - \bigg(\sum_{\nu=0}^{2^j} 2^{\nu s q} \|(\varphi_\nu \widehat{f})^\vee\|_{L_p(\mathbb{R}^d)}^q   \bigg)^{r/q} < \varepsilon.
	\end{equation}
	Assume $b < 0$. We have
	\begin{align}
		| (1-2^{b r}) \, \|f\|^{* r}_{T^b_r B^{s}_{p, q }(\R^d)} -  \|f\|_{B^s_{p, q}(\R^d)}^r| & =  \nonumber  \\
		& \hspace{-5cm} \bigg|\frac{1}{ \sum_{j=0}^\infty 2^{j b r} }\, \sum_{j=0}^\infty 2^{j b r} \bigg(\sum_{\nu=0}^{2^j} 2^{\nu s q} \|(\varphi_\nu \widehat{f})^\vee\|_{L_p(\mathbb{R}^d)}^q\bigg)^{r/q} - \frac{\sum_{j=j_0}^\infty 2^{j b r}}{\sum_{j=j_0}^\infty 2^{j b r}} \|f\|_{B^s_{p, q}(\R^d)}^r  \bigg|  \nonumber \\
		& \hspace{-5cm} \leq \bigg|\frac{1}{ \sum_{j=0}^\infty 2^{j b r} }\, \sum_{j=j_0}^\infty 2^{j b r} \bigg[\bigg(\sum_{\nu=0}^{2^j} 2^{\nu s q} \|(\varphi_\nu \widehat{f})^\vee\|_{L_p(\mathbb{R}^d)}^q\bigg)^{r/q} - \|f\|_{B^s_{p, q}(\R^d)}^r \bigg]  \bigg|  \nonumber \\
		& \hspace{-4cm} + \bigg|\bigg( \frac{1}{ \sum_{j=0}^\infty 2^{j b r} } - \frac{1}{\sum_{j=j_0}^\infty 2^{j b r}} \bigg) \sum_{j=j_0}^\infty 2^{j b r} \|f\|_{B^s_{p, q}(\R^d)}^r  \bigg|  \nonumber  \\
		& \hspace{-4cm} + \frac{1}{ \sum_{j=0}^\infty 2^{j b r} }\, \sum_{j=0}^{j_0-1} 2^{j b r} \bigg(\sum_{\nu=0}^{2^j} 2^{\nu s q} \|(\varphi_\nu \widehat{f})^\vee\|_{L_p(\mathbb{R}^d)}^q\bigg)^{r/q} \nonumber \\
		& \hspace{-4cm} =: I+ II+ III.\label{ThmLimitsBBMProof1*}
	\end{align}
	
	To estimate $I$, we can invoke \eqref{ThmLimitsBBMProof1} so that
	\begin{align}
		I &\leq  \frac{1}{ \sum_{j=0}^\infty 2^{j b r} }\, \sum_{j=j_0}^\infty 2^{j b r} \bigg[\|f\|_{B^s_{p, q}(\R^d)}^r-\bigg(\sum_{\nu=0}^{2^j} 2^{\nu s q} \|(\varphi_\nu \widehat{f})^\vee\|_{L_p(\mathbb{R}^d)}^q\bigg)^{r/q}  \bigg] \nonumber \\
		& \leq  \frac{\varepsilon}{ \sum_{j=0}^\infty 2^{j b r} }\, \sum_{j=j_0}^\infty 2^{j b r} = \varepsilon \, 2^{j_0 b r}. \label{ThmLimitsBBMProof2}
	\end{align}

	Concerning $II$, we can estimate
	\begin{align}
		II &= \bigg(\frac{1}{\sum_{j=j_0}^\infty 2^{j b r}} - \frac{1}{ \sum_{j=0}^\infty 2^{j b r} } \bigg) \sum_{j=j_0}^\infty 2^{j b r} \|f\|_{B^s_{p, q}(\R^d)}^r  \nonumber \\
		& = \frac{\sum_{j=0}^{j_0-1} 2^{j b r}}{\sum_{j=0}^\infty 2^{j b r}} \|f\|_{B^s_{p, q}(\R^d)}^r = (1-2^{j_0 b r}) \|f\|_{B^s_{p, q}(\R^d)}^r.  \label{ThmLimitsBBMProof3}
	\end{align}

	With the aim of estimate $III$, we consider an auxiliary parameter $b_0 > 0$ (without loss of generality we may assume $b_0 > -b$) such that
	\begin{align}
		III &= \frac{1}{ \sum_{j=0}^\infty 2^{j b r} }\, \sum_{j=0}^{j_0-1} 2^{j b r} \bigg(\sum_{\nu=0}^{2^j} 2^{\nu s q} (1+\nu)^{b_0 q} (1+\nu)^{-b_0 q}  \|(\varphi_\nu \widehat{f})^\vee\|_{L_p(\mathbb{R}^d)}^q\bigg)^{r/q} \nonumber \\
		&\hspace{1cm} \leq 2^{b_0 r}(1-2^{b r})  \|f\|_{B^{s, -b_0}_{p, q}(\R^d)}^r \sum_{j=0}^{j_0-1} 2^{j (b+b_0) r} \nonumber \\
		&\hspace{1cm}  =  \frac{2^{b_0 r}(1-2^{b r})(1-2^{j_0 (b+b_0) r})}{1-2^{(b+b_0) r}} \,  \|f\|_{B^{s, -b_0}_{p, q}(\R^d)}^r. \label{ThmLimitsBBMProof4}
	\end{align}

	Combining now \eqref{ThmLimitsBBMProof1*}--\eqref{ThmLimitsBBMProof4} we achieve
	\begin{align*}
		| (1-2^{b r}) \, \|f\|^{* r}_{T^b_r B^{s}_{p, q }(\R^d)} -  \|f\|_{B^s_{p, q}(\R^d)}^r| & \leq \\
		& \hspace{-5.5cm} \varepsilon \,   2^{j_0 b r}+   (1-2^{j_0 b r}) \|f\|_{B^s_{p, q}(\R^d)}^r +  \frac{2^{b_0 r}(1-2^{b r})(1-2^{j_0 (b+b_0) r})}{1-2^{(b+b_0) r}} \,  \|f\|_{B^{s, -b_0}_{p, q}(\R^d)}^r.
	\end{align*}
	Taking limits on both sides of the previous inequality as $b \to 0-$ we conclude that
	$$
	\lim_{b \to 0-} | (1-2^{b r}) \, \|f\|^{* r}_{T^b_r B^{s}_{p, q }(\R^d)} -  \|f\|_{B^s_{p, q}(\R^d)}^r| \leq \varepsilon.
	$$
	
	The proof of \eqref{BBMNewSpacesTL} follows similar ideas as above. Let $\varepsilon > 0$ and $b < 0$. Choose $j_0$ such that
		\begin{equation}\label{ThmLimitsBBMProof1T}
		\|f\|_{F^s_{p, q}(\R^d)}^r - \bigg\|\bigg(\sum_{\nu=0}^{2^j}  2^{\nu s q} |(\varphi_\nu \widehat{f})^\vee|^q\bigg)^{1/q} \bigg\|_{L_p(\R^d)}^r < \varepsilon.
	\end{equation}
	We have
	\begin{align}
		| (1-2^{b r}) \, \|f\|^{* r}_{T^b_r F^{s}_{p, q }(\R^d)} -  \|f\|_{F^s_{p, q}(\R^d)}^r| & =  \nonumber  \\
		& \hspace{-5cm} \bigg|\frac{1}{ \sum_{j=0}^\infty 2^{j b r} }\, \sum_{j=0}^\infty 2^{j b r} \bigg\|\bigg(\sum_{\nu=0}^{2^j}  2^{\nu s q} |(\varphi_\nu \widehat{f})^\vee|^q\bigg)^{1/q} \bigg\|^r_{L_p(\R^d)}- \frac{\sum_{j=j_0}^\infty 2^{j b r}}{\sum_{j=j_0}^\infty 2^{j b r}} \|f\|_{F^s_{p, q}(\R^d)}^r  \bigg|  \nonumber \\
		& \hspace{-5cm} \leq \bigg|\frac{1}{ \sum_{j=0}^\infty 2^{j b r} }\, \sum_{j=j_0}^\infty 2^{j b r} \bigg[ \bigg\|\bigg(\sum_{\nu=0}^{2^j}  2^{\nu s q} |(\varphi_\nu \widehat{f})^\vee|^q\bigg)^{1/q} \bigg\|^r_{L_p(\R^d)} - \|f\|_{F^s_{p, q}(\R^d)}^r \bigg]  \bigg|  \nonumber \\
		& \hspace{-4cm} + \bigg|\bigg( \frac{1}{ \sum_{j=0}^\infty 2^{j b r} } - \frac{1}{\sum_{j=j_0}^\infty 2^{j b r}} \bigg) \sum_{j=j_0}^\infty 2^{j b r} \|f\|_{F^s_{p, q}(\R^d)}^r  \bigg|  \nonumber  \\
		& \hspace{-4cm} + \frac{1}{ \sum_{j=0}^\infty 2^{j b r} }\, \sum_{j=0}^{j_0-1} 2^{j b r} \bigg\|\bigg(\sum_{\nu=0}^{2^j}  2^{\nu s q} |(\varphi_\nu \widehat{f})^\vee|^q\bigg)^{1/q} \bigg\|^r_{L_p(\R^d)}\nonumber \\
		& \hspace{-4cm} =: J+ JJ+ JJJ.\label{ThmLimitsBBMProof1*TL}
	\end{align}
	
	To estimate $J$, we can make use of  \eqref{ThmLimitsBBMProof1T}:
		\begin{equation}
		J  \leq  \frac{\varepsilon}{ \sum_{j=0}^\infty 2^{j b r} }\, \sum_{j=j_0}^\infty 2^{j b r} = \varepsilon \, 2^{j_0 b r}. \label{ThmLimitsBBMProof2TL}
	\end{equation}
		On the other hand, elementary computations lead to
	\begin{equation}
		JJ = (1-2^{j_0 b r}) \|f\|_{F^s_{p, q}(\R^d)}^r.  \label{ThmLimitsBBMProof3TL}
	\end{equation}
		It remains to estimate $JJJ$. We can proceed as follows
	\begin{align}
		JJJ &\leq \frac{2^{b_0 r}}{ \sum_{j=0}^\infty 2^{j b r} }\, \sum_{j=0}^{j_0-1} 2^{j (b+b_0) r} \bigg\|\bigg(\sum_{\nu=0}^{2^j}  2^{\nu s q} (1 + \nu)^{-b_0 q} |(\varphi_\nu \widehat{f})^\vee|^q\bigg)^{1/q} \bigg\|^r_{L_p(\R^d)} \nonumber \\
		&\hspace{1cm} \leq 2^{b_0 r} (1-2^{b r})  \|f\|_{F^{s, -b_0}_{p, q}(\R^d)}^r \sum_{j=0}^{j_0-1} 2^{j (b+b_0) r} \nonumber \\
		&\hspace{1cm}  =  \frac{2^{b_0 r} (1-2^{b r})(1-2^{j_0 (b+b_0) r})}{1-2^{(b+b_0) r}} \,  \|f\|_{F^{s, -b_0}_{p, q}(\R^d)}^r. \label{ThmLimitsBBMProof4TL}
	\end{align}
	Putting together \eqref{ThmLimitsBBMProof1*TL}--\eqref{ThmLimitsBBMProof4TL}, we achieve
	\begin{align*}
		| (1-2^{b r}) \, \|f\|^{* r}_{T^b_r F^{s}_{p, q }(\R^d)} -  \|f\|_{F^s_{p, q}(\R^d)}^r|  & \\
		&\hspace{-5cm} \leq  \varepsilon \, 2^{j_0 b r} + (1-2^{j_0 b r}) \|f\|_{F^s_{p, q}(\R^d)}^r  + \frac{2^{b_0 r} (1-2^{b r})(1-2^{j_0 (b+b_0) r})}{1-2^{(b+b_0) r}} \,  \|f\|_{F^{s, -b_0}_{p, q}(\R^d)}^r.
	\end{align*}
	The proof is finished by taking limit as $b \to 0$ in the last estimate.
\end{proof}

We can also get Bourgain--Brezis--Mironescu formulas for the standard Besov norms given by differences\index{\bigskip\textbf{Spaces}!$\dot{B}^{s}_{p, q}(\R^d)$}\label{HOMBES}
$$
	|f|_{\dot{B}^s_{p, q}(\R^d)} := \bigg(\int_{|h| < 1} |h|^{-s q-d} \|\Delta^k_h f\|_{L_p(\R^d)}^q \, dh \bigg)^{1/q}
$$
in terms of truncated Besov norms. Here $0 < s < k, \, k \in \N,$ and $0 < p, q \leq \infty$. To be more precise, consider the truncated Besov norms  (cf. \eqref{TheoremModuli2})\index{\bigskip\textbf{Spaces}!$T^b_r \dot{B}^s_{p, q}(\R^d)$}\label{TRUNHOMBES}
$$
	|f|^*_{T^b_r \dot{B}^s_{p, q}(\R^d)} := \bigg( \int_0^1 (1-\log t)^{b r-1} \bigg(\int_{t < |h| < 1} |h|^{-s q-d} \|\Delta^k_h f\|_{L_p(\R^d)}^q \, dh \bigg)^{r/q} \frac{dt}{t} \bigg)^{1/r},
$$
then the following holds.

\begin{thm}\label{Thm1717}
	Let $0 < s < k, \, k \in \N, 0 < p, q \leq \infty$ and $0 < r < \infty$. Then, for every $f \in C^\infty_0(\R^d)$,
	$$
		\lim_{b \to 0-} \,  (-b r)^{1/r} |f|^*_{T^b_r \dot{B}^s_{p, q}(\R^d)} = |f|_{\dot{B}^s_{p, q}(\R^d)}.
	$$
\end{thm}

\begin{rem}
	The interesting case in Theorem \ref{Thm1717} occurs when $q \neq r$, otherwise a simple application of Fubini's theorem gives
	$$
		|f|^{* q}_{T^b_q \dot{B}^s_{p, q}(\R^d)} = - \frac{1}{b q} \int_{|h| < 1} |h|^{-s q -d} (1-\log |h|)^{b q} \|\Delta^k_h f\|_{L_p(\R^d)}^q \, dh
	$$
	and thus
	$$
		\lim_{b \to 0-} \,  (-b q)^{1/q} |f|^*_{T^b_q \dot{B}^s_{p, q}(\R^d)} = |f|_{\dot{B}^s_{p, q}(\R^d)}.
	$$
\end{rem}

\begin{proof}[Proof of Theorem \ref{Thm1717}]
Given any $\varepsilon > 0$, there exists $t_0 \in (0, 1)$ such that
\begin{equation}\label{1769}
	|f|_{\dot{B}^s_{p, q}(\R^d)}^r -  \bigg(\int_{t < |h| < 1} |h|^{-s q-d} \|\Delta^k_h f\|_{L_p(\R^d)}^q \, dh \bigg)^{r/q} < \varepsilon
\end{equation}
for $t \in (0, t_0)$. Let $b < 0$. Note that
\begin{equation}\label{intTriv}
	\int_0^{t_0} (1-\log t)^{b r - 1} \frac{dt}{t} = - \frac{(1-\log t_0)^{b r}}{b r}.
\end{equation}
Therefore
\begin{align}
	|(-br)  |f|^{* r}_{T^b_r \dot{B}^s_{p, q}(\R^d)} - |f|_{\dot{B}^s_{p, q}(\R^d)}^r | &\leq   \nonumber \\
	& \hspace{-5cm} (-b r)  \int_0^{t_0} (1-\log t)^{b r -1} \bigg| \bigg(\int_{t < |h| < 1} |h|^{-s q-d} \|\Delta^k_h f\|_{L_p(\R^d)}^q \, dh \bigg)^{r/q}  - \bigg(\int_{|h| < 1} |h|^{-s q-d} \|\Delta^k_h f\|_{L_p(\R^d)}^q \, dh \bigg)^{r/q} \bigg| \frac{dt}{t} \nonumber \\
	& \hspace{-4cm} + (1-(1-\log t_0)^{b r}) \bigg(\int_{|h| < 1} |h|^{-s q-d} \|\Delta^k_h f\|_{L_p(\R^d)}^q \, dh \bigg)^{r/q} \nonumber \\
	&  \hspace{-4cm} +  (-b r) \int_{t_0}^1 (1-\log t)^{b r -1}  \bigg(\int_{t < |h| < 1} |h|^{-s q-d} \|\Delta^k_h f\|_{L_p(\R^d)}^q \, dh \bigg)^{r/q} \frac{dt}{t} \nonumber \\
	&  \hspace{-4cm} =: \mathcal{I} + \mathcal{II} + \mathcal{III}.\label{1769new}
\end{align}
By \eqref{1769} and \eqref{intTriv},
\begin{equation}\label{1770}
	\mathcal{I} \leq \varepsilon (1-\log t_0)^{b r}.
\end{equation}
To estimate $\mathcal{III}$, we choose $b_0 > 0$ satisfying $b_0 > -b$, thus
\begin{align}
	\mathcal{III} &\leq (-b r)   \bigg(\int_{|h| < 1} |h|^{-s q-d} (1 - \log |h|)^{-b_0 q} \|\Delta^k_h f\|_{L_p(\R^d)}^q \, dh \bigg)^{r/q} \int_{t_0}^1 (1-\log t)^{(b+b_0) r -1} \frac{dt}{t} \nonumber \\
	& \leq  \frac{-b  ((1-\log t_0)^{(b+b_0) r} -1)}{b+b_0}  \bigg(\int_{|h| < 1} |h|^{-s q-d} (1 - \log |h|)^{-b_0 q} \|\Delta^k_h f\|_{L_p(\R^d)}^q \, dh \bigg)^{r/q}. \label{1771}
\end{align}
Inserting \eqref{1770} and \eqref{1771} into \eqref{1769new}, we achieve
\begin{align}
				|(-br)  |f|^{* r}_{T^b_r \dot{B}^s_{p, q}(\R^d)} - |f|_{\dot{B}^s_{p, q}(\R^d)}^r |  & \leq \nonumber \\
		&\hspace{-4cm}   \varepsilon  (1-\log t_0)^{b r}+   (1-(1-\log t_0)^{b r}) |f|_{\dot{B}^s_{p, q}(\R^d)}^r \nonumber\\
		& \hspace{-3.5cm} -   \frac{b  ((1-\log t_0)^{(b+b_0) r} -1)}{b+b_0}  \bigg(\int_{|h| < 1} |h|^{-s q-d} (1 - \log |h|)^{-b_0 q} \|\Delta^k_h f\|_{L_p(\R^d)}^q \, dh \bigg)^{r/q}.\label{1774new}
\end{align}
Note that the last integral is convergent since $ \|\Delta^k_h f\|_{L_p(\R^d)} \lesssim |h|^k \|f\|_{\dot{W}^k_p(\R^d)}$ (see e.g. \cite[Proposition 3, p. 139]{Stein}) and $k > s$. Taking now successive limits in \eqref{1774new} first as $b \to 0-$ and then an $\varepsilon \to 0$, we get
\begin{equation*}
	\lim_{b \to 0-} (-br)  |f|^{* r}_{T^b_r \dot{B}^s_{p, q}(\R^d)} =  |f|_{\dot{B}^s_{p, q}(\R^d)}^r.
\end{equation*}
\end{proof}

\subsection{Fourier series}

We consider the Fourier coefficients of $f: \mathbb{T}^d \to \mathbb{C}, \, f \in L_1(\mathbb{T}^d),$
$$
	\widehat{f}(m) = \frac{1}{(2 \pi)^d} \int_{\mathbb{T}^d} f(x) e^{-i m \cdot x} \, dx, \qquad m \in \Z^d,
$$
and the Fourier transform\index{\bigskip\textbf{Operators}!$\mathfrak{F}$}\label{FTM}
$$\mathfrak{F}(f) := \{\widehat{f}(m)\}_{m \in \Z^d}.$$ The mapping properties of the Fourier transform of Besov functions are well known, cf. \cite{Zygmund}, \cite{Peetre76} and \cite{Pietsch}. In more detail, if $1 \leq p \leq 2, 0 < q \leq \infty, s > 0$ and $\frac{1}{u} = \frac{s}{d} + \frac{1}{p'}$ then
\begin{equation}\label{FTB}
	\mathfrak{F} : B^s_{p, q}(\mathbb{T}^d) \to \ell_{u, q}(\Z^d).
\end{equation}
By Remark \ref{Remark10.4}, note that the best possible result is given by $p=2$, i.e.,
$$
	\mathfrak{F} : B^{d (\frac{1}{u} - \frac{1}{2})}_{2, q}(\mathbb{T}^d) \to \ell_{u, q}(\Z^d), \qquad 0 < u < 2.
$$
In particular, for $u = q$, we recover the classical Sz\'asz result on absolute convergence of Fourier series. Moreover, if $u=q=1$ then we arrive at the Bernstein result:
$$
	\mathfrak{F} : B^{d/2}_{2, 1}(\mathbb{T}^d) \to \ell_{1}(\Z^d).
$$
Here $\ell_{u, q}(\Z^d)$ is the \emph{Lorentz sequence space} endowed with the quasi-norm\index{\bigskip\textbf{Spaces}!$\ell_{u, q}(\Z^d)$}\label{LSEQ}
	\begin{equation}\label{DefLorentzSpace2}
		\|\lambda\|_{\ell_{u, q}(\Z^d)} := \bigg(\sum_{n=1}^\infty (n^{1/u} \lambda_n^*)^q \frac{1}{n}\bigg)^{1/q} < \infty
	\end{equation}
(with the usual modification if $q=\infty$). As usual, $\{\lambda_n^*\}_{n \in \N}$ denotes the \emph{non-increasing rearrangement}\index{\bigskip\textbf{Functionals and functions}!$\{\lambda_n^*\}_{n \in \N}$}\label{REARRSEQ} of the sequence $\{|\lambda_m|\}_{m \in \Z^d}$ relative to $\lambda = \{\lambda_m\}_{m \in \Z^d}$. In particular, if $u=q$ then $\ell_{u, q} (\Z^d) = \ell_u (\Z^d)$. Note that \eqref{DefLorentzSpace2} corresponds to the discrete version of \eqref{DefLZ}.

Dealing with Triebel--Lizorkin spaces, it is an immediate consequence of embeddings \eqref{FJClas} and \eqref{FTB} that
\begin{equation}\label{FTF}
	\mathfrak{F} : F^s_{p, q}(\mathbb{T}^d) \to \ell_{u, p}(\Z^d)
\end{equation}
provided that $1 \leq p < 2, 0 < q \leq \infty, s > 0$ and $\frac{1}{u} = \frac{s}{d} + \frac{1}{p'}$.

\vspace{2mm}
\textbf{Fourier transform on $T^b_r A^s_{p, q}$.} Our next goal is  to investigate the mapping properties of the Fourier transform in the setting of truncated function spaces. To proceed with, we first need to introduce truncated counterparts of Lorentz spaces $\ell_{u, q}(\Z^d)$.

\begin{defn}
	Let $0 < u < \infty, 0 < q, r \leq \infty$ and $-\infty < b < \infty$. The truncated Lorentz sequence space $T^b_r \ell_{u, q}(\Z^d)$ is formed by all those $\lambda = \{\lambda_m\}_{m \in \Z^d}$ such that\index{\bigskip\textbf{Spaces}!$T^b_r \ell_{u, q}(\Z^d)$}\label{TRUNLSEQ}
	$$
		\|\lambda\|_{T^b_r \ell_{u, q}(\Z^d)} = \left( \sum_{j=0}^\infty 2^{j b r} \bigg(\sum_{n=2^j-1}^{2^{j+1}-2} (2^{n/u} \lambda_{2^{n}}^*)^q  \bigg)^{r/q} \right)^{1/r} < \infty
	$$
	(with the standard modifications if $q=\infty$ and/or $r=\infty$).
\end{defn}

\begin{rem}\label{RemarkGenLZS}
 Note that (cf. \eqref{DefLorentzSpace2})
	\begin{equation}\label{DefTrunLorentz}
		\|\lambda\|_{\ell_{u, q}(\Z^d)} \asymp \bigg(\sum_{n=0}^\infty (2^{n/u} \lambda_{2^n}^*)^q \bigg)^{1/q}.
	\end{equation}
	Then the definition of $T^b_r \ell_{u, q}(\Z^d)$ via appropriate truncations of \eqref{DefTrunLorentz} imitates the definition of truncated function spaces given in Definition \ref{DefinitionNewBesov}.
\end{rem}

 Clearly $T^b_q \ell_{u, q}(\Z^d) = \ell_{u, q} (\log \ell)_b(\Z^d)$, the \emph{Lorentz--Zygmund sequence space},\index{\bigskip\textbf{Spaces}!$\ell_{u, q} (\log \ell)_b(\Z^d)$}\label{LZSEQ}
	$$
		\|\lambda\|_{ \ell_{u, q} (\log \ell)_b(\Z^d)}:=  \bigg(\sum_{n=1}^\infty (n^{1/u} (1 + \log n)^b \lambda_n^*)^q \frac{1}{n}\bigg)^{1/q},
	$$
	see \eqref{DefLZ}. In particular, $ \ell_{u, q} (\log \ell)_0(\Z^d) = \ell_{u, q}(\Z^d)$. More generally, the relationships between $T^b_r \ell_{u, q}(\Z^d)$ and $\ell_{u, q}(\log \ell)(\Z^d)$ are given in the following

\begin{prop}\label{Prop1719}
	Let $0 < u < \infty, 0 < q, r \leq \infty$ and $-\infty < b < \infty$. Then
	\begin{equation}\label{EmbTLCL}
		  \ell_{u, r} (\log \ell)_{b-1/r + 1/\min\{q, r\}}(\Z^d)  \hookrightarrow T^b_r \ell_{u, q}(\Z^d) \hookrightarrow \ell_{u, r} (\log \ell)_{b-1/r + 1/\max\{q, r\}}(\Z^d)
	\end{equation}
	and
		\begin{equation}\label{EmbTLCL2}
		  \ell_{u, \min\{q, r\}} (\log \ell)_{b}(\Z^d)  \hookrightarrow T^b_r \ell_{u, q}(\Z^d) \hookrightarrow \ell_{u, \max\{q, r\}} (\log \ell)_{b}(\Z^d).
	\end{equation}
\end{prop}
\begin{proof}
We will only prove the first embeddings in \eqref{EmbTLCL} and \eqref{EmbTLCL2}; the corresponding second one follows similar ideas. Assume first $r \geq q$. By H\"older's inequality, for every $j \in \N_0$,
$$
	\bigg(\sum_{n=2^j-1}^{2^{j+1}-2} (2^{n/u} \lambda_{2^{n}}^*)^q \bigg)^{1/q} \lesssim 2^{j(1/q-1/r)} \bigg(\sum_{n=2^j-1}^{2^{j+1}-2} (2^{n/u} \lambda_{2^{n}}^*)^r  \bigg)^{1/r}.
$$
Accordingly
	\begin{align*}
		\|\lambda\|_{T^b_r \ell_{u, q}(\Z^d)} & \lesssim \left( \sum_{j=0}^\infty 2^{j (b-1/r+1/q) r} \sum_{n=2^j-1}^{2^{j+1}-2} (2^{n/u} \lambda_{2^{n}}^*)^r   \right)^{1/r} \\
		&\hspace{-1.5cm} \asymp \left(\sum_{n=0}^\infty (2^{n/u} (1+n)^{b-1/r+1/q} \lambda_{2^n}^*)^r \right)^{1/r} \asymp \|\lambda\|_{\ell_{u, r} (\log \ell)_{b-1/r + 1/q} (\Z^d)}.
	\end{align*}
	On the other hand, by Minkowski's inequality,
		\begin{align*}
		\|\lambda\|_{T^b_r \ell_{u, q}(\Z^d)} & \lesssim \left( \sum_{n=0}^\infty (2^{n/u} (1+n)^b \lambda_{2^{n}}^*)^q  \right)^{1/q}  \asymp \|\lambda\|_{\ell_{u, q}(\log \ell)_b(\Z^d)}.
	\end{align*}
	
	Secondly, suppose $r \leq q$. Then
	\begin{align*}
		\|\lambda\|_{T^b_r \ell_{u, q}(\Z^d)} & \leq \left( \sum_{j=0}^\infty 2^{j b r} \sum_{n=2^j-1}^{2^{j+1}-2} (2^{n/u} \lambda_{2^{n}}^*)^r   \right)^{1/r} \\
		&\hspace{-1.5cm} \asymp \left(\sum_{n=0}^\infty (2^{n/u} (1+n)^{b} \lambda_{2^n}^*)^r \right)^{1/r} \asymp \|\lambda\|_{\ell_{u, r} (\log \ell)_{b} (\Z^d)}.
	\end{align*}
\end{proof}

The next result shows that \eqref{FTB} can be extended to the truncated setting.

\begin{thm}\label{ThmFTB}
	Let $1 \leq p \leq 2, 0 < q, r \leq \infty, s > 0, b \in \R \backslash \{0\}$ and $\frac{1}{u} = \frac{s}{d} + \frac{1}{p'}$. Then
	\begin{equation*}
	\mathfrak{F} : T^b_r B^{s}_{p, q}(\mathbb{T}^d) \to T^b_r \ell_{u, q} (\Z^d).
\end{equation*}
As a consequence
$$
	\mathfrak{F} : T^b_r B^{s}_{p, q}(\mathbb{T}^d) \to \ell_{u, r} (\log \ell)_{b-1/r + 1/\max\{q, r\}} (\Z^d) \cap \ell_{u, \max\{q, r\}} (\log \ell)_{b}(\Z^d).
$$
\end{thm}

\begin{rem}
	(i) By Theorem \ref{TheoremEmbeddingsBBCharacterization}, the optimal result in Theorem \ref{ThmFTB} is
	$$
		\mathfrak{F} : T^b_r B^{d (\frac{1}{u} - \frac{1}{2})}_{2, q}(\mathbb{T}^d) \to T^b_r \ell_{u, q} (\Z^d), \qquad 0 < u < 2.
	$$
	
	(ii) Letting $q=r$ in Theorem \ref{ThmFTB} one recovers (cf. \cite[Corollary 7.3(i)]{DeVore})
$$
	\mathfrak{F} : B^{s, b}_{p, q}(\mathbb{T}^d) \to \ell_{u, q}(\log \ell)_b(\Z^d),
$$
which is an extension of \eqref{FTB}.
\end{rem}

The proof of Theorem \ref{ThmFTB} relies on limiting interpolation techniques. In particular, the next result shows that the spaces $T^b_r \ell_{u, q}(\Z^d)$ can be generated from classical Lorentz spaces $\ell_{u, q}(\Z^d)$ via limiting interpolation, i.e., it establishes the analogue of Theorems \ref{TheoremInterpolation} and \ref{TheoremInterpolationF} for Lorentz  sequence spaces.

\begin{thm}\label{ThmLorentzSeqLim}
	Let $0 < u_0 < u_1 < \infty, 0 < q_0, q_1, r \leq \infty, b_1 > 0$ and $b_0 < 0$. Then (with equivalence of quasi-norms)
	\begin{equation*}
		T^{b_i}_r \ell_{u_i, q_i}(\Z^d) = (\ell_{u_1, q_1}(\Z^d), \ell_{u_0, q_0}(\Z^d))_{(1-i, b_{i}-1/r), r} \qquad \text{for} \qquad i \in \{0, 1\}.
	\end{equation*}
\end{thm}

\begin{rem}
The assumption $u_0 < u_1$ guarantees that the pair $(\ell_{u_1, q_1}(\Z^d), \ell_{u_0, q_0}(\Z^d))$ is ordered (i.e., $\ell_{u_0, q_0}(\Z^d) \hookrightarrow \ell_{u_1, q_1}(\Z^d)$).
\end{rem}

\begin{proof}[Proof of Theorem \ref{ThmLorentzSeqLim}]
	For simplicity, we set $K(t, \lambda) = K(t, \lambda; \ell_{u_1, q_1}(\Z^d), \ell_{u_0, q_0}(\Z^d))$ for $t > 0$ and $\lambda \in \ell_{u_1, q_1}(\Z^d)$. Let $1/\alpha = 1/u_0 - 1/u_1$. According to the Holmstedt's formula \cite[Theorem 4.2]{Holmstedt}
	$$
		K(t, \lambda) \asymp t \bigg(\int_0^{t^{-\alpha}} (\xi^{1/u_0} \lambda^*(\xi))^{q_0} \frac{d \xi}{\xi} \bigg)^{1/q_0} + \bigg(\int_{t^{-\alpha}}^\infty (\xi^{1/u_1} \lambda^*(\xi))^{q_1} \frac{d \xi}{\xi} \bigg)^{1/q_1}
	$$
	where
	$$
		\lambda^*(\xi) = \lambda_n^* \qquad \text{if} \qquad \xi \in (n-1, n], \qquad n \in \N.
	$$
	By \eqref{DefLimInterpolation} and a simple change of variables, we get
	\begin{align}
		\|\lambda\|_{(\ell_{u_1, q_1}(\Z^d), \ell_{u_0, q_0}(\Z^d))_{(1-i, b_i-1/r), r} } & = \bigg(\int_0^1 (t^{-1+i} (1-\log t)^{b_i-1/r} K(t, \lambda))^r \frac{dt}{t} \bigg)^{1/r} \nonumber \\
		&\hspace{-4cm} \asymp \left(\int_1^\infty t^{-i r/\alpha} (1+\log t)^{(b_i-1/r) r}  \bigg(\int_0^{t} (\xi^{1/u_0} \lambda^*(\xi))^{q_0} \frac{d \xi}{\xi} \bigg)^{r/q_0}  \frac{dt}{t} \right)^{1/r}  \nonumber\\
		& \hspace{-3.5cm} + \left(\int_1^\infty t^{(1-i) r/\alpha} (1+\log t)^{(b_i-1/r) r} \bigg(\int_{t}^\infty (\xi^{1/u_1} \lambda^*(\xi))^{q_1} \frac{d \xi}{\xi} \bigg)^{r/q_1} \frac{dt}{t}  \right)^{1/r} \nonumber \\
		& \hspace{-4cm} =: I + II. \label{ProofLimIntLorentz}
	\end{align}
	
	Assume $i=1$. By basic monotonicity properties and Hardy's inequality \eqref{H2} (note that $b_1 > 0$)
	\begin{equation*}
		\lambda_1^* + II  \asymp \left(\sum_{j=0}^\infty 2^{j b_1 r} \bigg(\sum_{n=2^{j}-1}^{2^{j+1}-2} (2^{n/u_1} \lambda_{2^{n}}^*)^{q_1}\bigg)^{r/q_1}  \right)^{1/r} = \|\lambda\|_{T^{b_1}_r \ell_{u_1, q_1}(\Z^d)}.
	\end{equation*}
	On the other hand, by Hardy's inequality \eqref{H1},
	\begin{align*}
		I &=  \left(\int_1^\infty t^{-r/\alpha} (1+\log t)^{(b_1-1/r) r}  \bigg(\int_0^{t} (\xi^{1/u_0} \lambda^*(\xi))^{q_0} \frac{d \xi}{\xi} \bigg)^{r/q_0}  \frac{dt}{t} \right)^{1/r} \\
		& \asymp \bigg(\int_0^{1} (\xi^{1/u_0} \lambda^*(\xi))^{q_0} \frac{d \xi}{\xi} \bigg)^{1/q_0} + \bigg(\sum_{j=0}^\infty 2^{-j r/\alpha} (1 + j)^{(b_1-1/r) r} \bigg(\sum_{n=0}^{j } (2^{n/u_0} \lambda_{2^n}^*)^{q_0} \bigg)^{r/q_0}  \bigg)^{1/r} \\
		& \asymp \bigg(\sum_{j=0}^\infty [2^{j/u_1} (1 + j)^{b_1-1/r} \lambda_{2^j}^*]^r  \bigg)^{1/r} \lesssim  \bigg(\sum_{j=0}^\infty  2^{j b_1 r} \bigg(\sum_{n=2^{j}}^\infty (2^{n/u_1}\lambda_{2^n}^*)^{q_1} \bigg)^{r/q_1}  \bigg)^{1/r} \\
		& \asymp \|\lambda\|_{T^{b_1}_r \ell_{u_1, q_1}(\Z^d)}
	\end{align*}
	where the last step follows from \eqref{H2}.
	Hence \eqref{ProofLimIntLorentz} implies
	$$
	\|\lambda\|_{(\ell_{u_1, q_1}(\Z^d), \ell_{u_0, q_0}(\Z^d))_{(1-i, b_i-1/r), r} }  \asymp  \|\lambda\|_{T^{b_1}_r \ell_{u_1, q_1}(\Z^d)}.
	$$
	
	The case $i=0$ can be carried out similarly and is left to the interested reader.
\end{proof}

We are now ready to give the proof of

\begin{proof}[Proof of Theorem \ref{ThmFTB}]
Assume first $b > 0$. Let $s_0 > s$ and  $1/u_0 = s_0/d + 1/p'$ (and so $u_0 < u$). It follows from \eqref{FTB} that
\begin{equation*}
	\mathfrak{F} : B^s_{p, q}(\mathbb{T}^d) \to \ell_{u, q}(\Z^d) \qquad \text{and} \qquad \mathfrak{F} : B^{s_0}_{p, q}(\mathbb{T}^d) \to \ell_{u_0, q}(\Z^d).
\end{equation*}
Applying the limiting interpolation method \eqref{DefLimInterpolation} with $\theta = 0$, we derive
\begin{equation}\label{ProofThmFTB1}
	\mathfrak{F} : ( B^s_{p, q}(\mathbb{T}^d),  B^{s_0}_{p, q}(\mathbb{T}^d) )_{(0, b-1/r), r} \to ( \ell_{u, q}(\Z^d),  \ell_{u_0, q}(\Z^d)  )_{(0, b-1/r), r}.
\end{equation}
On the one hand, by Theorem \ref{TheoremInterpolation},
\begin{equation}\label{ProofThmFTB2}
( B^s_{p, q}(\mathbb{T}^d),  B^{s_0}_{p, q}(\mathbb{T}^d) )_{(0, b-1/r), r}  = T^b_r B^{s}_{p, q}(\mathbb{T}^d)
\end{equation}
and, on the other hand, by Theorem \ref{ThmLorentzSeqLim} (with $i=1$),
\begin{equation}\label{ProofThmFTB3}
( \ell_{u, q}(\Z^d),  \ell_{u_0, q}(\Z^d)  )_{(0, b-1/r), r} = T^b_r \ell_{u, q}(\Z^d).
\end{equation}
Combining \eqref{ProofThmFTB1}--\eqref{ProofThmFTB3} we arrive at
$$
	\mathfrak{F}: T^b_r B^{s}_{p, q}(\mathbb{T}^d) \to T^b_r \ell_{u, q}(\Z^d).
$$

The case $b < 0$ follows similar ideas as above, but now taking $s_0 \in (0, s)$ and invoking the limiting interpolation method \eqref{DefLimInterpolation} with $\theta = 1$ related to the couple $(B^{s_0}_{p, q}(\mathbb{T}^d), B^s_{p, q}(\mathbb{T}^d))$.
\end{proof}

Combining Theorem \ref{ThmFTB} and the Franke--Jawerth embedding obtained from Theorem \ref{ThmFJ},
$$
	T^b_r F^s_{p, q}(\mathbb{T}^d) \hookrightarrow T^b_r B^{s_0}_{p_0, p}(\mathbb{T}^d)
$$
for $p < p_0$ and $s-d/p = s_0-d/p_0$, we derive the analogue of \eqref{FTF} for truncated Triebel--Lizorkin spaces.

\begin{cor}\label{CorollaryFTF}
	Let $1 \leq p < 2, 0 < q, r \leq \infty, s > 0, b \in \R \backslash \{0\}$ and $\frac{1}{u} = \frac{s}{d} + \frac{1}{p'}$. Then
	\begin{equation*}
	\mathfrak{F} : T^b_r F^{s}_{p, q}(\mathbb{T}^d) \to T^b_r \ell_{u, p} (\Z^d).
\end{equation*}
As a consequence
$$
	\mathfrak{F} : T^b_r F^{s}_{p, q}(\mathbb{T}^d) \to \ell_{u, r} (\log \ell)_{b-1/r + 1/\max\{p, r\}} (\Z^d) \cap \ell_{u, \max\{p, r\}} (\log \ell)_b(\Z^d).
$$
\end{cor}

In the remainder of this section we study the sharpness of Theorem \ref{ThmFTB}: First, we extend the range of the integrability $p$ and second, we obtain two-sided estimates for the Fourier transform. To avoid delicate issues, we will focus only on the case $d=1$.

\begin{thm}
	Let $1 < p < \infty, 0 < q, r \leq \infty, s \in \R, b \in \R \backslash \{0\}$ and $\frac{1}{u} = s + \frac{1}{p'}$. Let $f \in L_1(\mathbb{T})$ be such that
	$$f(x) \sim \sum_{n=1}^\infty (a_n \cos n x + b_n \sin nx)$$
	 with $\{a_n\}_{n \in \N}$ and $\{b_n\}_{n \in \N}$ nonnegative general monotone sequences. Then
	 $$
	 	\|f\|_{T^b_r B^{s}_{p, q}(\mathbb{T})} \asymp \|\mathfrak{F}(f)\|_{T^b_r \ell_{u, q}(\Z)}.
	 $$
\end{thm}

\begin{proof}
	Without loss of generality, we may assume that $b_n = 0$ for all $n \in \N$. According to Theorem \ref{ThmBesovGMPer}, we have
		\begin{equation}\label{dadada1}
		\|f\|_{T^b_r B^{s}_{p, q}(\mathbb{T})} \asymp  \left( \sum_{k=0}^\infty 2^{k b r} \bigg(\sum_{\nu= 2^k-1}^{2^{k+1}-2} 2^{\nu (s + 1-1/p) q} a_{2^\nu}^q \bigg)^{r/q}  \right)^{1/r}.
	\end{equation}
		
		To complete the proof, it  only remains to verify that
	\begin{equation}\label{bvbvb1213}
		\|\{a_n\}_{n \in \N}\|_{T^b_r \ell_{u, q}(\Z)} \asymp   \left(\sum_{k=0}^\infty 2^{k b r} \bigg(\sum_{\nu=2^k-1}^{2^{k+1}-2} (2^{\nu/u} a_{2^\nu})^q \bigg)^{r/q} \right)^{1/r},
	\end{equation}
that is
	\begin{equation}\label{dadada2}
	 \left(\sum_{k=0}^\infty 2^{k b r} \bigg(\sum_{\nu=2^k-1}^{2^{k+1}-2} (2^{\nu/u} a_{2^\nu}^*)^q \bigg)^{r/q} \right)^{1/r} \asymp  \left(\sum_{k=0}^\infty 2^{k b  r} \bigg(\sum_{\nu=2^k-1}^{2^{k+1}-2} (2^{\nu/u} a_{2^\nu})^q \bigg)^{r/q} \right)^{1/r}
	\end{equation}
	for nonnegative sequences $\{a_n\}_{n \in \N} \in GM$. Indeed, since $a_k \lesssim a_n$ for $n \leq k \leq 2n$ (cf. \eqref{3.2}), it is plain to see that $a_n^* \gtrsim a_n$. This gives the estimate $\gtrsim$ in \eqref{dadada2}. On the other hand, the estimate $\lesssim$ in \eqref{dadada2} is a consequence of \cite[Theorem 368, p. 261]{Hardy}, which implies
	$$
		\sum_{\nu=2^k-1}^{2^{k+1}-2} (2^{\nu/u} a_{2^\nu}^*)^q  \leq \sum_{\nu=2^k-1}^{2^{k+1}-2} (2^{\nu/u} a_{2^\nu})^q.
	$$
\end{proof}

\vspace{2mm}
\textbf{Fourier transform on $\text{Lip}^{s, b}_{p, q}$.} As a distinguished example of Corollary \ref{CorollaryFTF}, we determine the mapping properties of the Fourier transform acting on Lipschitz spaces (cf. Proposition \ref{PropositionCoincidences}).

\begin{cor}
	Let $1 < p < 2, 0 < q \leq \infty, s > 0, b < -1/q$ and $\frac{1}{u} = \frac{s}{d} + \frac{1}{p'}$. Then
	$$
		\mathfrak{F} : \emph{Lip}^{s, b}_{p, q}(\mathbb{T}^d) \to T^{b+1/q}_q \ell_{u, p}(\Z^d).
	$$
	In particular
	$$
		\mathfrak{F} : \emph{Lip}^{s, b}_{p, q}(\mathbb{T}^d) \to \ell_{u, q}(\log \ell)_{b + 1/\max\{p, q\}}(\Z^d) \cap \ell_{u, \max\{p, q\}}(\log \ell)_{b+1/q}(\Z^d)
	$$
	and if, in addition, $p=q$ then
	$$
	\mathfrak{F} : \emph{Lip}^{s, b}_{p, p}(\mathbb{T}^d) \to \ell_{u, p} (\log \ell)_{b+1/p}(\Z^d).
	$$
\end{cor}

The optimality of the previous result can be shown as a combination of Theorem \ref{ThmLipGMPer} and \eqref{bvbvb1213}.

\begin{cor}
	Let $1 < p < \infty, 0 < q \leq \infty, s >0, b <-1/q$ and $\frac{1}{u} = s + \frac{1}{p'}$. Let $f \in L_1(\mathbb{T})$ be such that
	$$f(x) \sim \sum_{n=1}^\infty (a_n \cos n x + b_n \sin nx)$$
	 with $\{a_n\}_{n \in \N}$ and $\{b_n\}_{n \in \N}$ nonnegative general monotone sequences. Then
	 $$
	 	\|f\|_{\emph{Lip}^{s, b}_{p, q}(\mathbb{T})} \asymp \|\mathfrak{F}(f)\|_{T^{b+1/q}_q \ell_{u, p}(\Z)}.
	 $$
\end{cor}

\vspace{2mm}
\textbf{Fourier transform on $\mathbf{B}^{0, b}_{p, q}$.} Let $1 \leq p \leq 2, 0 < q \leq \infty$ and $b > -1/q$. The mapping properties of Fourier transform on $\mathbf{B}^{0, b}_{p, q}(\mathbb{T})$ were already investigated by DeVore, Riemenschneider and Sharpley \cite{DeVore} via weak-type interpolation techniques.  Namely, they obtained (cf. \cite[Corollary 7.3]{DeVore})
$$
	\mathfrak{F} : \mathbf{B}^{0, b}_{p, q}(\mathbb{T}) \to \ell_{p', q}(\log \ell)_{b}(\Z).
$$
However, this result is not sharp. Indeed, applying reiteration formulas for limiting approximation spaces, it was shown in \cite[Theorem 5.1]{CobosDominguez} that
\begin{equation}\label{CoDo}
	\mathfrak{F} : \mathbf{B}^{0, b}_{p, q}(\mathbb{T}) \to \ell_{p', q}(\log \ell)_{b+ 1/\max\{p', q\}}(\Z).
\end{equation}
Note that $\ell_{p', q}(\log \ell)_{b+ 1/\max\{p', q\}}(\Z) \subsetneq \ell_{p', q}(\log \ell)_{b}(\Z)$ if $q \neq \infty$ and $p \neq 1$.

Next we apply the theory of truncated spaces in order to get a significant improvement of \eqref{CoDo} (even in the classical scale formed by Lorentz--Zygmund spaces).

\begin{thm}\label{ThmCoDoIm}
	Let $1 < p \leq 2, 0 < q \leq \infty$ and $b > -1/q$. Then
	$$
		\mathfrak{F} : \mathbf{B}^{0, b}_{p, q}(\mathbb{T}^d) \to T^{b+1/q}_q \ell_{p', p}(\Z^d).
	$$
	In particular (cf. Proposition \ref{Prop1719})
	$$
		\mathfrak{F} : \mathbf{B}^{0, b}_{p, q}(\mathbb{T}^d) \to  \ell_{p', q}(\log \ell)_{b+ 1/\max\{p, q\}}(\Z^d) \cap \ell_{p', \max\{p, q\}} (\log \ell)_{b+1/q}(\Z^d).
	$$
\end{thm}

\begin{rem}
	In general $\ell_{p', q}(\log \ell)_{b+ 1/\max\{p, q\}}(\Z^d) \subsetneq   \ell_{p', q}(\log \ell)_{b+ 1/\max\{p', q\}}(\Z^d)$, so that the previous result sharpens \eqref{CoDo}.
\end{rem}

\begin{proof}[Proof of Theorem \ref{ThmCoDoIm}]
	According to the Paley inequality (see e.g. \cite[Chapter 6]{Peetre76})
	$$
		\mathfrak{F} : L_p(\mathbb{T}^d)  \to \ell_{p',p}(\Z^d)
	$$
	and \eqref{FTB}
	\begin{equation*}
	\mathfrak{F} : B^s_{p, q}(\mathbb{T}^d) \to \ell_{u, q}(\Z^d), \qquad \frac{1}{u} = \frac{s}{d} + \frac{1}{p'}.
\end{equation*}
By limiting interpolation
\begin{equation}\label{ThmCoDoImProof1}
	\mathfrak{F} : (L_p(\mathbb{T}^d), B^s_{p, q}(\mathbb{T}^d))_{(0, b), q}  \to (\ell_{p',p}(\Z^d),  \ell_{u, q}(\Z^d))_{(0, b), q}.
\end{equation}
Since $L_p(\mathbb{T}^d) = F^0_{p, 2}(\mathbb{T}^d)$, Theorem \ref{TheoremInterpolationF} and Proposition \ref{PropositionCoincidences} give
\begin{equation}\label{ThmCoDoImProof2}
	(L_p(\mathbb{T}^d), B^s_{p, q}(\mathbb{T}^d))_{(0, b), q} = T^{b+1/q}_q F^0_{p, 2}(\mathbb{T}^d) = \mathbf{B}^{0, b}_{p, q}(\mathbb{T}^d).
\end{equation}
On the other hand, by Theorem \ref{ThmLorentzSeqLim},
\begin{equation}\label{ThmCoDoImProof3}
	(\ell_{p',p}(\Z^d),  \ell_{u, q}(\Z^d))_{(0, b), q} = T^{b+1/q}_q \ell_{p', p}(\Z^d).
\end{equation}
Putting together \eqref{ThmCoDoImProof1}--\eqref{ThmCoDoImProof3},
$$
	\mathfrak{F} :\mathbf{B}^{0, b}_{p, q}(\mathbb{T}^d) \to T^{b+1/q}_q \ell_{p', p}(\Z^d).
$$
\end{proof}

The optimality assertion related to Theorem \ref{ThmCoDoIm} reads as follows.

\begin{cor}
	Let $1 < p < \infty, 0 < q \leq \infty,$ and $b >-1/q$. Let $f \in L_1(\mathbb{T})$ be such that
	$$f(x) \sim \sum_{n=1}^\infty (a_n \cos n x + b_n \sin nx)$$
	 with $\{a_n\}_{n \in \N}$ and $\{b_n\}_{n \in \N}$ nonnegative general monotone sequences. Then
	 $$
	 	\|f\|_{\mathbf{B}^{0, b}_{p, q}(\mathbb{T})} \asymp \|\mathfrak{F}(f)\|_{T^{b+1/q}_q \ell_{p', p}(\Z)}.
	 $$
\end{cor}
\begin{proof}
	It follows from \eqref{GMBesovZero} (more precisely, its periodic counterpart given in \cite[Theorem 4.27]{DominguezTikhonov}) that
	$$
		\|f\|_{\mathbf{B}^{0, b}_{p, q}(\mathbb{T})} \asymp \bigg(\sum_{j=0}^\infty 2^{j(b+1/q) q} \bigg( \sum_{\nu=2^j}^{2^{j+1}-1} 2^{\nu (p  -1)}  F_0^p(2^\nu) \bigg)^{q/p} \bigg)^{1/q}.
	$$
	On the other hand, by \eqref{bvbvb1213},
	$$
		\|\mathfrak{F}(f)\|_{T^{b+1/q}_q \ell_{p', p}(\Z)} \asymp \bigg(\sum_{j=0}^\infty 2^{j(b+1/q) q} \bigg( \sum_{\nu=2^j}^{2^{j+1}-1} 2^{\nu (p  -1)}  F_0^p(2^\nu) \bigg)^{q/p} \bigg)^{1/q}.
	$$
\end{proof}

\newpage
\appendix
\section{Fourier-analytical description of Lipschitz spaces}\label{Appendix}

This appendix complements characterizations of Lipschitz spaces $\L^{s, b}_{p,q}(\R^d)$ obtained in \cite{DominguezHaroskeTikhonov}. To be more precise, we deal here with their Fourier-analytical descriptions.

\begin{thm}\label{ThmLipFourier}
	Let $s > 0, 1 < p < \infty, 0 < q \leq \infty$ and $b < -1/q$. Then
	\begin{equation}\label{ThmLipFourier*}
		\|f\|_{\emph{\L}^{s,b}_{p,q}(\R^d)} \asymp  \left(\sum_{j=0}^\infty 2^{j (b+1/q) q} \Big\| \Big(\sum_{\nu=2^j-1}^{2^{j+1}-2} 2^{\nu s 2} |(\varphi_\nu \widehat{f})^\vee|^2 \Big)^{1/2} \Big\|_{L_p(\R^d)}^q \right)^{1/q}.
	\end{equation}
\end{thm}

\begin{proof}
	We have
	\begin{equation}\label{ThmLipFourier1}
		\|f\|_{(L_p(\R^d), H^s_p(\R^d))_{(1,b),q}} \asymp \|((\varphi_\nu \widehat{f})^\vee(\cdot))\|_{(L_p(\R^d; \ell_2), L_p(\R^d; \ell_2^s))_{(1,b),q}}.
	\end{equation}
	Indeed, this interpolation formula follows from the well-known facts that $L_p(\R^d)$ and $H^s_p(\R^d)$ are retracts of $L_p(\R^d; \ell_2)$ and $L_p(\R^d; \ell_2^s)$, respectively, with co-retraction operator $\mathfrak{J}(f) = ((\varphi_\nu \widehat{f})^\vee(\cdot))$ (see \eqref{Retraction}).
	
	Now according to Lemma \ref{LemmaNew23}(ii), we derive
	\begin{equation}\label{hsahasg}
		\|f\|_{(L_p(\R^d), H^s_p(\R^d))_{(1,b),q}} \asymp \left(\sum_{j=0}^\infty 2^{j (b+1/q) q} \Big\| \Big(\sum_{\nu=2^j-1}^{2^{j+1}-2} 2^{\nu s 2} |(\varphi_\nu \widehat{f})^\vee|^2 \Big)^{1/2} \Big\|_{L_p(\R^d)}^q \right)^{1/q}.
	\end{equation}
	
	On the other hand, the $K$-functional relative to the pair $(L_p(\R^d), H^s_p(\R^d))$ can be characterized through modulus of smoothness (cf. \cite[(4.2)]{Wilmes} and \cite[(1.36)]{Kolomoitsev}), more precisely, for $t \in (0, 1)$,
	$$
		K(t^s, f; L_p(\R^d), H^s_p(\R^d)) \asymp t^s \|f\|_{L_p(\R^d)} + \omega_s(f,t)_p
	$$
	which automatically yields
	\begin{equation}\label{shhasha}
		\|f\|_{(L_p(\R^d), H^s_p(\R^d))_{(1,b),q}} \asymp \|f\|_{\L^{s, b}_{p, q}(\R^d)}.
	\end{equation}
	
	Putting together \eqref{hsahasg} and \eqref{shhasha}, we achieve
	$$
	 \|f\|_{\L^{s, b}_{p, q}(\R^d)} \asymp  \left(\sum_{j=0}^\infty 2^{j (b+1/q) q} \Big\| \Big(\sum_{\nu=2^j-1}^{2^{j+1}-2} 2^{\nu s 2} |(\varphi_\nu \widehat{f})^\vee|^2 \Big)^{1/2} \Big\|_{L_p(\R^d)}^q \right)^{1/q}.
	$$
\end{proof}

We present 
  another characterization of Lipschitz norm in terms of
  Fourier means. Let $\Psi_n(f)$ stand for any  of the following
 means:

\begin{enumerate}

\item the $\ell_r$-\emph{Fourier means} given by\index{\bigskip\textbf{Operators}!$S_{n, r}$}\label{FME}
 $$
\widehat{S_{n,r}} f(\xi):=\chi_{\{\xi\in\R^d\,:\,\Vert \xi\Vert_{\ell_r}\le n\}}(\xi) \widehat{f}(\xi),\qquad r=1,\infty.
$$
Here, $\chi_{\{\xi\in\R^d\,:\,\Vert \xi\Vert_{\ell_r}\le n\}}$ denotes the characteristic function relative to the $\ell_r$-ball centered at the origin and radii $n$;

\item the \emph{de~la~Vall\'ee Poussin-type means} $V_{n}f$  (cf. \eqref{DelaValleePoussin}; see also \cite{KolomoitsevTikhonov*});

\item the \emph{Riesz spherical means} $R_n^{\beta,\delta} f$ given by
$$
\widehat{R_n^{\beta,\delta} f} (\xi)=\bigg(1- \bigg(\frac{\left|\xi\right|}n\bigg)^\beta\bigg)_+^\delta \widehat{f}(\xi)
$$
for $\beta>0$ and $\delta>d|\frac1p-\frac12 |-\frac12$.\index{\bigskip\textbf{Operators}!$R^{\beta, \delta}_n$}\label{RME}
\end{enumerate}
\begin{thm}\label{ThmLipFourier+}
	Let $s > 0, 1 < p < \infty, 0 < q \leq \infty$ and $b < -1/q$. Then,
	\begin{equation*}
		\|f\|_{\emph{\L}^{s, b}_{p,q}(\R^d)} \asymp \|f\|_{L_p(\R^d)} + \left(\sum_{j=0}^\infty (1 + j)^{b q} \|
 (-\Delta)^{s/2}{\Psi_{2^j}} f
\|_{L_p(\R^d)}^q \right)^{1/q}.
	\end{equation*}
\end{thm}

The proof immediately  follows  from Hardy's inequalities and the following estimates
  \begin{align*}
    \bigg(\sum_{\nu=j+1}^\infty 2^{-\nu s \tau}\Vert (-\Delta)^{s/2}{\Psi_{2^\nu}} f\Vert_{L_p(\R^d)}^\tau\bigg)^\frac1\tau&\lesssim \omega_s (f,2^{-j})_{p}\\
   & \lesssim  \bigg(\sum_{\nu=j+1}^\infty 2^{-\nu s \theta}\Vert (-\Delta)^{s/2}{\Psi_{2^\nu}} f\Vert_{L_p(\R^d)}^\theta\bigg)^{\frac{1}{\theta}},
  \end{align*}
  where $\tau=\max(2,p)$ and $\theta=\min(2,p)$, see
   \cite[Theorem 6.3]{KolomoitsevTikhonov*}.

\newpage
\section{Some auxiliary inequalities}

\begin{lem}\label{LemmaB1}
	Let $\lambda_j = 2^{-2^{j-1}}$ for $j \geq 1$, $\lambda_0 = 1$ and $R > 0$. Let $F$ be a nonnegative function defined on $(0, 1)$. Then
	\begin{equation}\label{Disc1}
		\sum_{j=0}^\infty \bigg(\int_{\lambda_{j+1}}^{\lambda_j} F(z) \, dz \bigg)^R \asymp \int_0^1 \bigg(\int_{t^2}^t F(z) \, dz \bigg)^R \, \frac{dt}{t (1-\log t)}.
	\end{equation}
	The corresponding multivariate result also holds: Let $F$ be a nonnegative function defined on the unit ball $B(0, 1)$ in $\R^d$. Then
		\begin{equation}\label{Disc2}
		\sum_{j=0}^\infty \bigg(\int_{\lambda_{j+1} < |h| < \lambda_j} F(h) \, dh \bigg)^R \asymp \int_0^1 \bigg(\int_{t^2 < |h| < t} F(h) \, dh \bigg)^R \, \frac{dt}{t (1-\log t)}.
	\end{equation}
	\end{lem}

\begin{proof}
To show \eqref{Disc1} we can proceed as follows. We have
	\begin{align*}
		\int_0^1 \bigg(\int_{t^2}^t F(z) \, dz \bigg)^R \, \frac{dt}{t (1-\log t)} &= \sum_{j=0}^\infty \int_{\lambda_{j+1}}^{\lambda_j}  \bigg(\int_{t^2}^t F(z) \, dz \bigg)^R \, \frac{dt}{t (1-\log t)} \\
		& \leq  \sum_{j=0}^\infty  \bigg(\int_{\lambda_{j+2}}^{\lambda_j} F(z) \, dz \bigg)^R \int_{\lambda_{j+1}}^{\lambda_j}  \, \frac{dt}{t (1-\log t)} \\
		& \asymp  \sum_{j=0}^\infty  \bigg(\int_{\lambda_{j+2}}^{\lambda_j} F(z) \, dz \bigg)^R \\
		& \asymp  \sum_{j=0}^\infty  \bigg(\int_{\lambda_{j+1}}^{\lambda_j} F(z) \, dz \bigg)^R.
	\end{align*}
	Concerning the converse estimate, we have
	\begin{align*}
		\sum_{j=0}^\infty  \bigg(\int_{\lambda_{j+1}}^{\lambda_j} F(z) \, dz \bigg)^R & \asymp \sum_{j=0}^\infty  \bigg(\int_{\lambda_{j+1}}^{\lambda_j} F(z) \, dz \bigg)^R  \int_{\lambda_{j+1}}^{\lambda_j}  \, \frac{dt}{t (1-\log t)} \\
		& \leq \sum_{j=0}^\infty \int_{\lambda_{j+1}}^{\lambda_j} \bigg(\int_{t^2}^{\sqrt{t}} F(z) \, dz \bigg)^R \frac{dt}{t(1-\log t)}  \\
		& \lesssim \int_0^1  \bigg(\int_{t^2}^{t} F(z) \, dz \bigg)^R \frac{dt}{t(1-\log t)} \\
		& \hspace{1cm}+  \int_0^1  \bigg(\int_{t}^{\sqrt{t}} F(z) \, dz \bigg)^R \frac{dt}{t(1-\log t)} \\
		& \asymp \int_0^1  \bigg(\int_{t^2}^{t} F(z) \, dz \bigg)^R \frac{dt}{t(1-\log t)}
	\end{align*}
	where the last estimate follows from a simple change of variables.  This completes the proof of \eqref{Disc1}. The proof of \eqref{Disc2} follows along the same lines as \eqref{Disc1}.
\end{proof}

For the convenience of the reader, we collect below some Hardy-type inequalities invoked frequently in the paper. For  complete treatments, we refer the reader to \cite{KufnerPersson}, \cite{OpicKufner}, and \cite{potapov}.

\begin{lem}
	Let $\lambda > 0, \, 0 < q \leq \infty,$ and $-\infty < b < \infty$. Let $\xi_j \geq 0$ for $j \in \N_0$. Then, for every $j_0 \in \N_0$,
	\begin{equation}\label{H1}
		\sum_{j=j_0}^\infty 2^{-j \lambda q} (1 + j)^{b q} \bigg(\sum_{k=j_0}^j \xi_k \bigg)^q \asymp \sum_{j=j_0}^\infty 2^{-j \lambda q} (1 + j)^{b q} \xi_j^q
	\end{equation}
	and
	\begin{equation}\label{H2}
		\sum_{j=j_0}^\infty 2^{j \lambda q} (1 + j)^{b q} \bigg(\sum_{k=j}^\infty \xi_k \bigg)^q \asymp \sum_{j=j_0}^\infty 2^{j \lambda q} (1 + j)^{b q} \xi_j^q
	\end{equation}
	where the hidden equivalence constants are independent of $j_0$.
\end{lem}

\newpage
\section{List of symbols}

\bigskip

  \textbf{Sets}

\medskip

    $\mathbb{R}^d$, $d$-dimensional real Euclidean space,  \pageref{SETR}

    $\R$, real line, \pageref{SETR1}

    $\mathbb{C}$, complex numbers, \pageref{SETC}

    $\mathbb{T}^d$, $d$-dimensional torus, \pageref{SETTD}

    $\mathbb{T}$, unit circle, \pageref{SETT}

    $\mathbb{Z}^d$, $d$-dimensional integer lattice, \pageref{SETZ}

    $\Z$, integer lattice, \pageref{SETZ1}

    $\mathbb{N}$, natural numbers, \pageref{SETN}

    $\mathbb{N}_0$, natural numbers with $0$, \pageref{SETN0}

    $\mathfrak{U}_p$, class of exponential type functions, \pageref{EXPFUN}

         $Q_{j, m}$, dyadic cubes, \pageref{DYACUB}

    $GM$, set of general monotone functions, \pageref{GM}

    $\widehat{GM}^d$, set of radial functions whose Fourier transform is general monotone, \pageref{GMF}

     $\R^{d+1}_+$, upper half-space of $\R^{d+1}$, \pageref{UPHS}

\bigskip

\textbf{Numbers and relations}

\medskip

  $A \lesssim B$, the estimate $A \leq C B,$ where $C$ is a positive constant, \pageref{AB}

    $A \asymp B$, the estimates $C^{-1} B \leq A \leq C B$, where $C > 1$ is a constant, \pageref{ASYMP}

    $X \hookrightarrow Y$, continuous embedding, \pageref{XY}


    $p'$, dual exponent of $p$, \pageref{p'}


%
    $a_+=\max\{a,0\}$, \pageref{a_+}

    $\sigma_p = d (\frac{1}{p}-1)_+$, \pageref{sigmap}
%

\bigskip

 \textbf{Spaces}

\medskip

$X'$, dual space of $X$, \pageref{DUAL}

$L_p(\R^d; X)$, Bochner space, \pageref{LEBX}

$L_p(\R^d)$, Lebesgue space, \pageref{LEB}

$\ell^s_q(X), \ell^s_q$,  weighted $\ell_q$ spaces, \pageref{LEBSEQ}

$\mathcal{S}(\mathbb{R}^d)$, Schwartz space, \pageref{Schwartz}

$\mathcal{S}'(\mathbb{R}^d)$, space of tempered distributions, \pageref{S'}

$B^{s,b}_{p,q}(\mathbb{R}^d)$, Besov space defined by Fourier analytic tools, \pageref{BESOVF}

$F^{s,b}_{p,q}(\mathbb{R}^d)$, Triebel-Lizorkin space, \pageref{TL}

$H^{s}_p(\mathbb{R}^d)$,  Bessel potential space, \pageref{SOB}

$W^k_p(\mathbb{R}^d)$, classical Sobolev space, \pageref{SOBCLAS}

$\mathbf{B}^{s,b}_{p,q}(\mathbb{R}^d)$, Besov space defined by differences, \pageref{BESOVDIFF}

$\text{Lip}^{s, b}_{p,q}(\mathbb{R}^d)$, Lipschitz space, \pageref{LOGLIPSCHITZ}

$T^{b}_r B^s_{p, q}(\R^d)$, truncated Besov space, \pageref{TRUNB}

$T^{b}_r F^s_{p, q}(\R^d)$, truncated Triebel--Lizorkin space, \pageref{TRUNF}

$\mathfrak{T}^{b}_r F^s_{p, q}(\R^d)$, inner truncated Triebel--Lizorkin space, \pageref{TRUNF2}

$T^{b}_r B^s_{p, q}(\mathbb{T}^d), T^{b}_r F^s_{p, q}(\mathbb{T}^d), \mathfrak{T}^{b}_r F^s_{p, q}(\mathbb{T}^d)$, truncated periodic smooth function spaces, \pageref{TRUNBPER}

$T^{*}_r B^s_{p, q}(\R^d)$, limiting truncated Besov space, \pageref{TRUNBL}

$T^{*}_r F^s_{p, q}(\R^d)$, limiting truncated Triebel--Lizorkin space, \pageref{TRUNFL}

$B^{s,b,\xi}_{p,q}(\mathbb{R}^d)$, Besov space of iterated logarithmic smoothness defined by Fourier analytic tools, \pageref{BESOVFLOG}

$(A_0,A_1)_{\theta,q, b}, (A_0, A_1)_{\theta, q}$, real interpolation spaces, \pageref{REAL}

$(A_0,A_1)_{(\theta,b),q}$, limiting interpolation space, \pageref{REALLIM}

$W^{s, b, p}(\R^d)$, fractional Sobolev space, \pageref{FRACTSOB}

$b^{s, \xi}_{p, q}$, Besov sequence spaces, \pageref{BESOVSEQ}

$f^s_{p, q}$, Triebel--Lizorkin sequence spaces, \pageref{TLSEQ}

$T^\xi_r b^s_{p, q}$, truncated Besov sequence spaces, \pageref{TRUBESOVSEQ}

$T^*_r b^s_{p, q}$, limiting truncated Besov sequence spaces, \pageref{LIMTRUNBESEQ}

$T^\xi_r f^s_{p, q}$, truncated Triebel--Lizorkin sequence spaces, \pageref{TTLSEQ}

$\mathbf{b}^{0, \xi}_{p, q}$, Besov sequence spaces with related smoothness near zero, \pageref{BESOVSEQZERO}

$\text{lip}^{s, \xi}_{p, q}$, Lipschitz sequence spaces, \pageref{LIPSEQ}

$T^*_r f^s_{p, q}$, limiting truncated Triebel--Lizorkin sequence spaces, \pageref{LIMTRUNTLSEQ}

$L_1^{\text{loc}}(\R^d)$, space of locally integrable functions on $\R^d$, \pageref{L1LOC}

$bmo(\R^d)$, bounded mean oscillation space, \pageref{BMO}

$C(\R^d)$, space of uniformly continuous functions, \pageref{CU}

$\text{BV}(\Omega)$, the space of bounded variation functions, \pageref{BV}

$L_{p, q}(\log L)_b(\mathbb{T}^d)$,  Lorentz--Zygmund space, \pageref{LZ}

$L_p(\log L)_b(\mathbb{T}^d)$, Zygmund space, \pageref{ZYG}

$L_{p, q}(\mathbb{T}^d)$, Lorentz spaces, \pageref{LOR}

$\text{exp} \, L^{-\frac{1}{b}} (\mathbb{T}^d)$, Orlicz space of exponentially integrable functions, \pageref{OR}

$L^{(p, b, q}(0, 1)$, small Lebesgue space, \pageref{SMALLLEB}

$(A_0, A_1)_{(0, (b, \eta)), q}$, limiting interpolation space with broken logarithmic weights, \pageref{LIMINTBROKE}

$\dot{W}^{s, p}(\R^d)$, homogeneous fractional Sobolev space, \pageref{HOMFRACTSOB}

$\dot{B}^s_{p, q}(\R^d)$, homogeneous Besov space, \pageref{HOMBES}

$T^b_r \dot{B}^s_{p, q}(\R^d)$, truncated homogeneous Besov space, \pageref{TRUNHOMBES}

$\ell_{u, q}(\Z^d)$, Lorentz sequence space, \pageref{LSEQ}

$T^b_r \ell_{u, q}(\Z^d)$, truncated Lorentz sequence space, \pageref{TRUNLSEQ}

$\ell_{u, q}(\log \ell)_b(\Z^d)$, Lorentz--Zygmund sequence space, \pageref{LZSEQ}

%

\bigskip

\textbf{Functionals and functions}

\medskip

     $\omega_{k}(f,t)_p$, modulus of smoothness of integer order $k$, \pageref{MODK}

     $\Delta^k_h$, difference of integer order $k$ with step $h$, \pageref{DELTA}

     $\Delta^s_h$, difference of fractional order $s$ with step $h$, \pageref{DELTAFRAC}

      $\omega_{s}(f,t)_{p}$, modulus of smoothness of fractional order $s$, \pageref{MODA}

     $K(t,f)$, Peetre's $K$-functional, \pageref{K}

    $E_j(f)_{X}$, $X$-best approximation of the function $f$ by entire functions of exponential type $j$, \pageref{ERROR}

     $\psi_F, \psi_M$, scaling function (father wavelet) and associated wavelet (mother wavelet), \pageref{FW}

     $\Psi^j_{G, m}$,  wavelets, \pageref{WAV}

     $\chi_{j, m}$, characteristic function related to $Q_{j, m}$, \pageref{CHARDYA}

     $H^j_{G, m}$, Haar wavelets, \pageref{HAAR}

     $j_\alpha, J_\alpha$, Bessel functions, \pageref{NORBESS}

     $\Delta$, difference for sequence components, \pageref{DIFFSEQ}

          $f^\ast$, non-increasing rearrangement of $f$, \pageref{REARRANGEMENT}

     $\Psi^{j, L, \text{per}}_{G, m}$, periodic wavelets, \pageref{PERWAV}

     $\{\lambda^*_n\}_{n \in \N}$, non-increasing rearrangement of the sequence $\{|\lambda_m|\}_{m \in \Z^d}$, \pageref{REARRSEQ}

\bigskip

\textbf{Operators}

\medskip

    $\widehat{\varphi}$, Fourier transform of $\varphi$, \pageref{FT}

    $\varphi^\vee$, inverse Fourier transform of $\varphi$, \pageref{IFT}

$P_k f$, almost best approximants of $f$, \pageref{BA}

   $V_t$, de la Vall\'ee-Poussin operators, \pageref{VALLEEPOUSSIN}

    $\lambda^{j, G}_m$, wavelet coefficients, \pageref{WAVCOE}

    $I_\sigma, \mathfrak{I}_{\sigma}$, lifting operators, \pageref{LIFT}

    $\text{Tr}$, trace operator, \pageref{TRACE}

    $\text{Ex}$, extension operator, \pageref{EXT}

    $\text{id}$, identity operator, \pageref{IDEN}

    $\mathfrak{F}$, Fourier transform map, \pageref{FTM}

    $S_{n, r}$, Fourier means, \pageref{FME}

    $R^{\beta, \delta}_n$, Riesz means, \pageref{RME}
%

\newpage

\end{document}